\documentclass[11pt]{amsart}

\usepackage[T1]{fontenc}
\usepackage[utf8]{inputenc}
\usepackage[english]{babel}
\usepackage{amscd,amsmath,amsthm,amssymb}
\usepackage{mathtools}
\usepackage{mathrsfs}
\usepackage{comment}
\usepackage{xcolor}
\usepackage{array}
\usepackage{caption}
\usepackage{mathabx}
\usepackage{standalone}
\usepackage{scalerel}
\usepackage{csquotes}

\usepackage{bm}
\usepackage{accents} %%% for the circle above Z

\usepackage{enumitem}
\setlist[enumerate]{label=\textnormal{(\arabic*)}}

\usepackage{xr-hyper} 
\definecolor{cadmiumgreen}{rgb}{0.0, 0.42, 0.24}
\definecolor{darkred}{rgb}{.65,0,0}
\definecolor{aqua}{rgb}{0.0, 1.0, 1.0}

\usepackage[
  colorlinks,
  linkcolor=darkred,
  urlcolor=blue,
  citecolor=cadmiumgreen,
  pdfstartview ={FitV},
]{hyperref}
\hypersetup{
  pdfauthor={Omid Amini, Matthieu Piquerez, Lyuhui Wu},
  pdftitle={Solutions to tropical Monge--Ampere}}

\usepackage[
  alphabetic,
  msc-links,
  nobysame,
  lite,
]{amsrefs}

\usepackage{tikz, tikz-cd}
\usetikzlibrary{calc, backgrounds, quotes, arrows.meta}

\usepackage[left=3cm,top=3.5cm,right=3cm]{geometry}
\def\xx{1}
\def\xy{0}
\def\yx{0}
\def\yy{1}
\def\zx{.3}
\def\zy{.6}

\def\ratio{1.3}
\def\ratioo{.5}
\def\ratioe{.6}

\newcommand{\BFcoordinates}{
  \coordinate (I) at (0,0);
  \coordinate (O) at (0,0);

  \coordinate (e1) at (\ratioe*\xx, \ratioe*\xy);
  \coordinate (e2) at (\ratioe*\yx, \ratioe*\yy);
  \coordinate (e3) at (\ratioe*\zx, \ratioe*\zy);
  \coordinate (e0) at ($(0,0)-(e1)-(e2)-(e3)$);

  \foreach \i/\j in {1/2, 1/3, 1/0, 2/3, 2/0, 3/0} {
    \coordinate (e\i\j) at ($(e\i)+(e\j)$);
  }

  \foreach \i/\j/\ax/\ay/\bx/\by in {1/2/\xx/\xy/\yx/\yy, 1/3/\xx/\xy/\zx/\zy, 2/3/\yx/\yy/\zx/\zy} {
    \pgftransformcm{\ax}{\ay}{\bx}{\by}{\pgfpoint{0}{0}}
    \coordinate (I\i\i\j) at (2,1);
    \coordinate (I\j\i\j) at (1,2);
    \coordinate (I\i\j) at ($(0,0)!(I\i\i\j)!(1,1)$);
    \coordinate (I\i) at ($(0,0)!(I\i\i\j)!(1,0)$);
    \coordinate (I\j) at ($(0,0)!(I\j\i\j)!(0,1)$);
  }

  \foreach \i/\ax/\ay in {1/\xx/\xy, 2/\yx/\yy, 3/\zx/\zy} {
    \pgftransformcm{\ax}{\ay}{-\xx-\yx-\zx+\ax}{-\xy-\yy-\zy+\ay}{\pgfpoint{0}{0}}
    \coordinate (I\i0) at (0,\ratio);
    \coordinate (I\i\i0) at (2,\ratio);
    \coordinate (I0\i0) at (-\ratioo,\ratio);
  }

  \coordinate (I0) at ($.333*(I010)+.333*(I020)+.333*(I030)$);
}

\def\roundface#1#2[#3]{\fill[#3] (I) -- (I#1) .. controls ($.3*(I#1)+.7*(I#1#1#2)$) and ($.3*(I#1#2)+.7*(I#1#1#2)$) .. (I#1#2) .. controls ($.3*(I#1#2)+.7*(I#2#1#2)$) and ($.3*(I#2)+.7*(I#2#1#2)$) .. (I#2) -- cycle}

\newtheorem{thm}{Theorem}[section]

\newtheorem{lemma}[thm]{Lemma}

\newtheorem{prop}[thm]{Proposition}

\newtheorem{claim}[thm]{Claim}
\newtheorem{cor}[thm]{Corollary}

\newtheorem{conj}[thm]{Conjecture}

\newtheorem*{theoremA}{Theorem A}
\newtheorem*{theoremB}{Theorem B}
\newtheorem*{theoremC}{Theorem C}
\newtheorem*{theoremD}{Theorem D}
\newtheorem*{theoremD'}{Theorem D'}
\newtheorem*{theoremE}{Theorem E}

\theoremstyle{definition}
\newtheorem{defii}[thm]{Definition}
\newtheorem{definition}[thm]{Definition}
\newtheorem{construction}[thm]{Construction}
\newenvironment{defi}
  {\pushQED{\qed}\defii}
  {\popQED\enddefii}
\newtheorem{remm}[thm]{Remark}
\newenvironment{remark}
  {\pushQED{\qed}\remm}
  {\popQED\endremm}
\newtheorem{exx}[thm]{Example}
\newenvironment{example}
  {\pushQED{\qed}\exx}
  {\popQED\endexx}

\numberwithin{equation}{section}

\newcommand{\resp}{\textit{resp.} }

\newcommand{\Q}{\mathbb{Q}}
\newcommand{\N}{\mathbb N}
\newcommand{\R}{\ssub{\mathbb{R}}}
\newcommand{\Z}{\ssub{\mathbb{Z}}}

\newcommand{\bul}{\bullet} % to use as index for complexes

\renewcommand{\emptyset}{\varnothing}

\newcommand{\rquot}[2]{#1\big/#2}
\newcommand{\rest}[1]{\raisebox{-1pt}{$\vert$}_{#1}}

\let\oldchi\chi
  \newcommand{\raisechi}[2]{\raisebox{.4ex}{$#1#2$}}
  \renewcommand{\chi}{{\mathpalette\raisechi\oldchi}}

\makeatletter
\let\oldsum\sum
\renewcommand{\sum}{\@ifnextchar_\@mysum\oldsum}
\def\@mysum_#1{\oldsum_{\substack{#1}}}
\let\oldbigoplus\bigoplus
\renewcommand{\bigoplus}{\@ifnextchar_\@mybigoplus\oldbigoplus}
\def\@mybigoplus_#1{\oldbigoplus_{\substack{#1}}}
\let\oldprod\prod
\renewcommand{\prod}{\@ifnextchar_\@myprod\oldprod}
\def\@myprod_#1{\oldprod_{\substack{#1}}}
\makeatother

\let\oldbigwedge\bigwedge
\renewcommand{\bigwedge}{{\textstyle\oldbigwedge\!}}

\DeclareMathOperator{\sed}{sed} % sedentarity
\DeclareMathOperator{\SedT}{Sed} % sedentarity
\newcommand{\reg}{\mathrm{reg}}
\DeclareMathOperator{\gyss}{Gys} % gysin map
\newcommand{\gys}{\ssub{\gyss}!}

\DeclareMathOperator{\ress}{Res} % gysin map
\newcommand{\res}{\ssub{\ress}!}

\DeclareMathOperator{\id}{id} % identity map
\renewcommand{\v}{\mathbf v} % a multivector
\newcommand{\eR}{\mathbb T} % tropical semiring
\newcommand{\e}{\ssub{\mathfrak e}!} % primitive vector of a ray
\newcommand{\f}{{\textnormal{\textsl{\textsf f}}}} % finite for subscript
\renewcommand{\P}{\mathbb P} % projective space
\newcommand{\TP}{\mathbb{TP}} % tropical projective space
\newcommand{\bbU}{\ssub[-1pt]{\mathbb U}}
\newcommand{\bbO}{\ssub[-1pt]{\mathbb O}}
\renewcommand{\div}{\mathrm{div}} % a divisor
\newcommand{\ssT}{\ssub{T}!} % espace tangent

\newcommand{\dual}{\star} % dual
\newcommand{\CC}{\mathcal C} % C^0, C^1, etc. functions
\let\i\relax
\newcommand{\i}{\ssub{\mathrm i}!} % inclusion
\renewcommand{\d}{\ssub{{\partial}}!} %  differential Dolbeaul
\newcommand{\parr}{{\scaleto{/\!/}{6pt}}} % symbol parallel
\newcommand{\nvect}{\ssub{\mathfrak n}!} % normal unit vector to gamma in delta

\newcommand{\x}{\textsc{x}} % polynomial variable
\newcommand{\y}{\textsc{y}} % polynomial variable
\newcommand{\z}{\textsc{z}} % polynomial variable
\newcommand{\card}[1]{\lvert#1\rvert} % cardinality
\newcommand{\st}{\bigm|} % such that in sets
\newcommand{\Hom}{\ssub{\mathrm{Hom}}} % Hom
\newcommand{\RpMod}{\mathcal M} % category of \R_+ modules

\newcommand{\abs}[1]{\lvert #1\rvert} % absolute value
\newcommand{\shiftcomp}[2][0]{{}\mkern#1mu\overline{\mkern-#1mu#2}}
\newcommand{\comp}[1]{\ifx#1N \shiftcomp[3]{#1}\else\ifx#1X \shiftcomp[3]{#1}\else\ifx#1Z \shiftcomp[3]{#1}\else\ifx#1F \shiftcomp[3]{#1} \else \shiftcomp{#1}\fi\fi\fi\fi} % canonical compactification

\newcommand{\X}{\mathfrak X}

\renewcommand{\bar}{\comp}
\renewcommand{\dh}{\d}
\newcommand{\dhb}{\ssub{\bar\d}!}
\newcommand{\dm}{\ssub{\bm \partial}!}
\newcommand{\dmb}{\ssub{\bar\dm}!}

  \newcommand{\dims}[1]{\ss{d}_{#1}}
\newcommand{\suppaux}[2]{\scalebox{1}[1.4]{$#1\lvert$}#2\scalebox{1}[1.4]{$#1\rvert$}}
  \newcommand{\supp}[1]{\mathpalette\suppaux{#1}}
\newcommand{\conezero}{{\underline0}} % cone 0

\newcommand{\subface}{\prec}
\newcommand{\subfaceq}{\preceq}
\newcommand{\ssubface}{\mathbin{\mathchoice
  {\subface\!\!\!\cdot}%
  {\subface\!\!\!\cdot}%
  {\subface\!\cdot}%
  {\subface\!\cdot}%
}} % subface of codimension one codimension 1
\newcommand{\supface}{\succ}
\newcommand{\supfaceq}{\succeq}
\newcommand{\ssupface}{\mathbin{\mathchoice
  {\cdot\!\!\!\supface}%
  {\cdot\!\!\!\supface}%
  {\cdot\!\supface}%
  {\cdot\!\supface}%
}}

\DeclareMathOperator{\conv}{conv} % convex hull

\newcommand{\Trop}{\mathrm{Trop}}

\newcommand{\uniform}{\mathsf{t}}

\newcommand{\weight}{\ssub{\varpi}!}

\newcommand{\Spec}{\operatorname{Spec}}

\newcommand{\wpsh}{\mathrm{PSH}}

\newcommand{\NN}{{\mathpalette\doscaleNN\relax}}
\newcommand{\doscaleNN}[2]{\ssub{\scalebox{.9}{$#1\mathcal N$}}!}

\renewcommand{\a}{\ssub{\mathfrak a}!}

\newcommand{\rec}{\mathrm{rec}}

\newcommand{\Radonp}{\ssub{\mathrm{Radon}}!_{\ge 0}}
\newcommand{\MA}{\ssub{\mathrm{MA}}!}

\newcommand{\an}{{\mathrm{an}}}
\newcommand{\val}{{\mathrm{val}}}

\newcommand{\NA}{{\mathrm{NA}}}

\newcommand{\env}{\ssub[-2pt]{\mathrm{P}}!} %%%enveloppe

\NewDocumentCommand{\ssub}{O{0pt} O{.9} m t! e{_^}}{
  #3%
  \IfValueT{#5}{
    \IfBooleanTF{#4}{\sb{\hspace{#1}\scaleobj{#2}{#5}}}{\sb{#5}}
  }
  \IfValueT{#6}{\sp{\scaleobj{#2}{#6}}}
}

\ExplSyntaxOn
\NewDocumentCommand{\tossub}{o o m}{
  \expandafter\let\csname old\cs_to_str:N #3\endcsname#3
  \renewcommand#3%
  {\ssub[#1][#2]{\csname old\cs_to_str:N #3\endcsname}}
}
\ExplSyntaxOff

\tossub\varpi
\tossub\deg
\tossub\omega
\tossub\nu
\tossub\rho
\newcommand{\ssM}{\ssub{M}!}
\newcommand{\ssN}{\ssub[-1pt]{N}!}
\newcommand{\ssA}{\ssub{A}!}
\newcommand{\ssbarN}{\ssub[-1pt]{\bar N}!}
\newcommand{\ssbarT}{\ssub[-1pt]{\bar T}!}

\newcommand{\ssW}{\ssub[-1pt]{W}!}

\newcommand{\ssSigma}{\ssub{\Sigma}!}
\newcommand{\sssigma}{\ssub[-2pt]{\sigma}}
\newcommand{\sstau}{\ssub[-2pt]{\tau}}

\newcommand{\ssOmega}{\ssub{\Omega}!}

\newcommand{\varX}{\mathbf X}
\newcommand{\varH}{\mathbf H}

\newcommand{\varY}{\mathbf Y}
\newcommand{\varT}{\mathbf T}
\newcommand{\varL}{\mathbf L}
\newcommand{\varG}{\mathbf G}

\newcommand{\Sk}{\mathrm{Sk}}

\newcommand{\E}{\mathrm{E}_{\omega}}
\newcommand{\FE}{\ssub[-2pt]{\mathrm{F}}!}
\newcommand{\GE}{\ssub[-2pt]{\mathrm{G}}!}

\renewcommand{\setminus}{\smallsetminus}

\newcommand{\std}{\mathrm{std}}

\newcommand{\be}{\begin{equation}}
\newcommand{\ee}{\end{equation}}
\newcommand{\bes}{\begin{equation*}}
\newcommand{\ees}{\end{equation*}}
\newcommand{\ben}{\begin{eqnarray}}
\newcommand{\een}{\end{eqnarray}}
\newcommand{\bens}{\begin{eqnarray*}}
\newcommand{\eens}{\end{eqnarray*}}

\newcommand{\lms}{\longmapsto}
\newcommand{\ol}{\overline}
\newcommand{\vol}{\mathrm{vol}}
\newcommand{\Supp}{\mathrm{Supp}}

\newcommand{\MongeA}{Monge--Ampère }

\newcommand{\PSH}{\mathrm{PSH}}
\newcommand{\CPSH}{\mathrm{CPSH}}
\newcommand{\dCPSH}{\mathrm{dCPSH}}
\newcommand{\appCPSH}{\CPSH^{\mathrm{app}}}

\newcommand{\pCPSH}{\mathrm{CPSH}^{+}}

\newcommand{\Conv}{\mathrm{Conv}}

\newcommand{\LConv}{\mathrm{LConv}}
\newcommand{\appLConv}{\LConv^{\mathrm{app}}}
\newcommand{\extLConv}{\mathrm{PSH}}

\newcommand{\PL}{\mathrm{PL}}
\newcommand{\Aff}{\mathrm{Aff}}
\newcommand{\PSm}{\mathrm{PC}^{\infty}}

\newcommand{\dom}{\mathrm{dom}}

\newcommand{\aff}{\mathrm{aff}}
\newcommand{\interior}{\mathrm{int}}
\newcommand{\relint}{\mathrm{relint}}

\newcommand{\ext}{\mathrm{ext}}

\newcommand{\xo}{X_o}

\newcommand{\MWp}{\mathrm{MW}^{\ge 0}}
\newcommand{\uinfty}{\underline{\infty}}

\newcommand{\kahler}{K\"ahler }
\newcommand{\prekahler}{prek\"ahler }

\newcommand{\MAo}{\MA_{\omega}}
\newcommand{\envomega}{\env_{\omega}}

\newcommand{\BC}{{\mathbb {C}}} \newcommand{\BD}{{\mathbb {D}}}
 
\newcommand{\BG}{{\mathbb {G}}}

 \newcommand{\BN}{{\mathbb {N}}}
\newcommand{\BO}{{\mathbb {O}}} \newcommand{\BP}{{\mathbb {P}}}
\newcommand{\BQ}{{\mathbb {Q}}} \newcommand{\BR}{{\mathbb {R}}}
\newcommand{\BS}{{\mathbb {S}}} \newcommand{\BT}{{\mathbb {T}}}

 \newcommand{\BZ}{{\mathbb {Z}}}

\newcommand{\CE}{{\mathcal {E}}} \newcommand{\CF}{{\mathcal {F}}}
 \newcommand{\CH}{{\mathcal {H}}}
 
\newcommand{\CK}{{\ssub{\mathcal {K}}!}} \newcommand{\CL}{{\mathcal {L}}}
 
\newcommand{\CO}{{\mathcal {O}}} 
 \newcommand{\CR}{\ssub{\mathcal {R}}!}
\newcommand{\CS}{{\mathcal {S}}}

\newcommand{\nC}{\ssub{C}!}

\newcommand{\fe}{{\mathfrak{e}}}

\newcommand{\fm}{{\mathfrak{m}}}

\newcommand{\fs}{{\mathfrak{s}}} \newcommand{\ft}{{\mathfrak{t}}}

\newcommand{\fC}{{\mathfrak{C}}}

 \newcommand{\fL}{{\mathfrak{L}}}

\newcommand{\fS}{{\mathfrak{S}}} 
 
 \newcommand{\fX}{{\mathfrak{X}}}

\newcommand{\su}{\mathsf{u}}
\newcommand{\sv}{\mathsf{v}}
\newcommand{\sw}{\mathsf{w}}
\newcommand{\sC}{\mathsf{C}}

\newcommand{\Lip}{\mathrm{Lip}}

\newcommand{\valfield}{\mathbb{K}}
\newcommand{\valring}{\mathbb{K}^{\circ}}
\newcommand{\resfield}{\kappa}

\DeclareMathOperator{\DL}{DL}

\newcommand{\olPi}{\ol{\Pi}}

\makeatletter
\let\@oldinfty\infty
\newcommand{\@sminfty}{{\hspace{-.02cm}\scaleto{\@oldinfty}{2.8pt}}} % infinity used as index
\newcommand{\@smsminfty}{_{\hspace{-.04cm}\scaleto{\@oldinfty}{2.8pt}}} % infinity used as second index
\renewcommand{\infty}{{\mathchoice%
  {\displaystyle{\@oldinfty}}%
  {\textstyle{\@oldinfty}}%
  {\scriptstyle{\@sminfty}}%
  {\scriptscriptstyle{\@smsminfty}}}
}
\makeatother

\begin{document}
\title{On solutions to tropical Monge--Ampère equations}

\allowdisplaybreaks

\author{Omid Amini}
\address{CNRS - LMO, Université Paris-Saclay, Orsay, France}
\email{\href{omid.amini@universite-paris-saclay.fr}{omid.amini@universite-paris-saclay.fr}}

\author{Matthieu Piquerez}
\address{Nantes, France}
\email{\href{matthieupi2@yahoo.fr}{matthieupi2@yahoo.fr}}

\author{Lyuhui Wu}
\address{\shortstack[l]{IMJ--PRG, Sorbonne Universit\'e, Paris, France\\
LMO, Université Paris-Saclay, Orsay, France}}
\email{\href{lyuhui.wu@imj-prg.fr}{lyuhui.wu@imj-prg.fr}}

\date{October 15, 2026}

\begin{abstract}
This paper develops convex analysis on tropical varieties and investigates the properties of the tropical \MongeA equations. 

We start by introducing notions of plurisubharmonic functions and \kahler forms on arbitrary tropical varieties. We show that many classical results of pluripotential theory extend to this setting, including an analogue of Demailly's regularization theorem. Next, we study the solutions to the tropical \MongeA equations. While we establish their uniqueness, we show that the existence is not guaranteed in general. More precisely, we prove that the tropical analogue of the Calabi--Yau theorem fails for any pair $(X, \theta)$ consisting of a compact tropical variety $X$ of dimension at least two satisfying a mild genericity assumption, and a tropical \kahler form $\theta$ on $X$. In particular, there exist measures compactly supported in the open part of $X$ that admit no solution. In contrast, for tropical varieties that are regular in codimension one, we show the existence of solutions. Finally, via strongly faithful tropicalizations whose existence is established in this paper, we relate the tropical solutions, when they exist, to the solutions of non-Archimedean \MongeA equations on Berkovich spaces. 
\end{abstract}

\maketitle

\setcounter{tocdepth}{1}

\tableofcontents

\newpage

\section{Introduction} 

The study of convex analysis and \MongeA equations on $\mathbb{R}^n$ has a long and fruitful history, with applications across many branches of mathematics. On the other hand, the theory of plurisubharmonic (psh) functions, viewed as the complex counterparts of convex functions, and of complex \MongeA equations on complex manifolds has proved fundamental in complex and Kähler geometry. 
The purpose of this paper is to develop and study a theory of convex analysis and \MongeA equations on tropical varieties.

\vspace{.2cm}

We provide a broader background to our work before going into more details. Let $M$ be a smooth projective complex variety of dimension $d$.  Let $\omega$ be a \kahler form on $M$ of volume $\deg(\omega^d)=\int_M \omega^d$, and let $\mu$ be a strictly positive smooth probability measure on $M$. The celebrated Calabi--Yau theorem asserts that the complex \MongeA equation $(\omega+dd^c \varphi)^d=\deg(\omega^d) \mu$ has a unique solution $\varphi \in C^{\infty}(M)$, up to an additive constant, such that $\omega+dd^c \varphi \ge 0$. A function satisfying this positivity condition  is called an $\omega$-psh function on $M$. Geometrically, $\omega+dd^c\varphi$ constitutes a new \kahler form on $M$ within the \kahler class $[\omega]$. Consequently, solving this \MongeA equation is equivalent to finding a \kahler metric with a prescribed volume form, which in turn prescribes its Ricci curvature. A prominent consequence of the Calabi--Yau theorem is the existence of Ricci-flat \kahler metrics on Calabi--Yau manifolds 
 (\textit{i.e.} varieties with a trivial canonical bundle).

Understanding the degeneration of \kahler metrics, particularly Ricci-flat metrics, is a central problem in \kahler geometry. To make this precise, let $\fX \to \BC \setminus \{0\}$ be a smooth projective morphism, so that each fiber $\fX_t$ is a $d$-dimensional projective manifold for $t \in \BC \setminus \{0\}$. Let $\fL$ be an ample line bundle on $\fX$. For each $t$, we equip the fiber $\fX_t$ with a (Ricci-flat) \kahler metric $\omega_t$ in the class $c_1(\fL_t)$. A fundamental question is to determine the asymptotic behavior of $\omega_t$ as $t \to 0$.

A powerful approach to studying the limits of degenerating \kahler metrics relies on non-Archimedean geometry. This perspective was originally introduced by Kontsevich and Soibelman~\cite{KS06} in the context of the SYZ conjecture, and was subsequently developed by Boucksom, Favre, and Jonsson~\cites{BFJ15, BJ17}. 
To characterize these limits, Kontsevich and Soibelman proposed studying the non-Archimedean polarized Calabi--Yau variety $(\varX^{\an}, \varL^{\an})$ over the complete non-Archimedean field of Laurent series $\valfield = \mathbb{C}((\mathsf{t}))$. This space is obtained as the Berkovich analytification of the polarized $\valfield$-variety $(\varX, \varL)$ induced by the generic fiber of the family $(\fX, \fL)$. By a theorem of Payne \cite{Pay09}, the Berkovich space $\varX^{\an}$ can be realized as the inverse limit of all its tropicalizations, where each tropicalization is a finite integral polyhedral complex endowed with a weight structure. In this sense, $\varX^{\an}$ can be naturally viewed as an ``infinite'' polyhedral space.

Following an idea of Shou-Wu Zhang~\cite{Zha95}, non-Archimedean psh functions are defined using the algebraic geometric notion of nefness, which, via the toric dictionary, serves as an analogue of convexity. A non-Archimedean \MongeA operator was introduced by Chambert-Loir \cite{CL06}, and a non-Archimedean analogue of the Calabi--Yau theorem was subsequently established by Boucksom, Favre, and Jonsson \cite{BFJ15}. Specifically, they proved that certain non-Archimedean \MongeA equations admit unique solutions (up to additive constants) within the class of non-Archimedean psh functions. In particular, for a degenerating family $\fX$ of Calabi--Yau manifolds, there exists a non-Archimedean Calabi--Yau potential $\varphi_{\mathrm{CY}, \mathrm{NA}}$ on $\varX^{\an}$. 

It was later shown by Yang Li \cite{Li25C0convergence} that the Archimedean Calabi--Yau potentials $\varphi_{\mathrm{CY},t}$ converge in a suitable sense to $\varphi_{\mathrm{CY}, \mathrm{NA}}$. Furthermore, when $\fX$ is a maximally degenerating family, this non-Archimedean potential satisfies a ``comparison property'' \cite{Li26Valuative}, which ultimately leads to a version of the SYZ conjecture \cite{Li23WeakComparison}.

However, because they are defined on an ``infinite polyhedral space'', non-Archimedean psh functions are extremely difficult to describe explicitly. Therefore, despite the relatively complete non-Archimedean picture, it remains beneficial from a complex geometric perspective to study convex functions and \MongeA equations on a single tropicalizaliation, which is a finite polyhedral shadow of the Berkovich space, as a way to seek a more explicit description of the limit metrics. The tropical \MongeA equations have previously been studied for specific examples of tropical hypersurfaces \cites{HJMM22,Li23Fano, AH23} (although under disguised languages). The goal of this paper is to extend this study to general tropical varieties.

\subsection{Convex analysis on tropical varieties}
In this paper, a tropical variety is defined as an extended polyhedral space endowed with a tropical weight function that satisfies the balancing condition (see Section \ref{subsec:tropical-varieties}). 
More precisely, a polyhedral chart is an open subset of the canonical compactification\footnote{For a simplicial cone $\sigma \cong [0, +\infty)^k$, for the coordinates given by the generators of the rays of $\sigma$, 
its canonical compactification is isomorphic to $[0, +\infty]^k$, obtained by allowing the coordinates to take the value infinity. 
The canonical compactification of a simplicial fan is obtained by gluing the canonical compactifications of its cones along the common faces.} of a simplicial rational fan, and an extended polyhedral space is a Hausdorff space obtained by gluing polyhedral charts via $\Z$-linear maps. 
A tropical weight function on an extended polyhedral space assigns a positive integer to each top-dimensional face.  It satisfies the balancing condition if the weighted sum of the normal vectors in the star fan of each codimension-one face vanishes.
 For a tropical variety $X$, we use $\xo$ to denote its open part which is, by definition, the union of the open parts of the polyhedral charts of $X$\footnote{For an open subset $U$ in the canonical compactification $\comp{\Sigma}$ of a simplicial fan $\Sigma$, its open part is the intersection of $U$ with the support of $\Sigma$.}.

In order to be able to work on tropical \MongeA equations, the first task is to specify a proper class of convex functions on tropical varieties. 
When the tropical variety $X$ is embedded in $\R^n$, a natural choice to consider is the restriction of convex functions in the ambient space, the so-called class of ``globally convex functions''. However, this class of convex functions turns out to be too restrictive for our purpose. Indeed, in the setting of the metric SYZ conjectures for Calabi--Yau hypersurfaces in toric Fano varieties, 
the optimal transport problem introduced and studied in \cites{HJMM22, Li23Fano, AH23} turns out to be equivalent to solving a certain tropical \MongeA equation inside the class of globally convex functions mentioned above. 
 However, in \cite{AH23}, the authors showed that for many examples (actually, most examples they examined), this optimal transport problem does not admit any solution. This, along with recently developed Hodge theoretic results in the combinatorial and tropical geometric~\cite{AHK, AP-hodge-fan,AP-tht} settings, motivated us to introduce a larger class of convex functions. 
\begin{definition} [Definition \ref{def:locally_convex_paper_2}] \label{introdef:locally_convex_functions}
  Let $X$ be a tropical variety, and let $U$ be an open subset of $X$. We say that a function $u \colon U \to \R$ is \emph{locally convex} on $U$ if
    \begin{enumerate}
      \item $u \in C^0(U)$.
      \item For any $x \in U$, there exist an open neighborhood $U'$ of $x$ and an affine function $l \colon U' \to \R$ such that $u(x) = l(x)$ and $u(y) \ge l(y)$ for all  $y \in U'$.
    \end{enumerate}
  We use $\LConv(U)$ to denote the space of locally convex functions on $U$. 
\end{definition}
A locally convex function is called \emph{strictly locally convex} if in addition we have that $u(y)>l(y)$ for all $y \in U' \setminus \{x\}$ in Definition \ref{introdef:locally_convex_functions}\,(2). 

\smallskip

When $X=\R^n$, the above class of convex functions coincides with the class of usual convex functions (with finite values). In general, we show that over the open part $\xo$ of the tropical variety, this class of convex functions is equivalent to a class of convex functions introduced in \cite{BBS22} based on a notion of polyhedral convex combination.

\vspace{.2cm}

In \cite{APW1}, we introduce a formalism of mixed differential forms. In this formalism, mixed $(0,0)$-forms are just piecewise smooth functions\footnote{A function $f$ is called piecewise smooth if it is continuous, and there exists a polyhedral structure such that the restriction of $f$ on each face is smooth.}. In general, this differential calculus incorporates both discrete and continuous features. The discrete contribution is encoded in the Chow rings of local fans (isomorphic to the Chow rings of the associated toric varieties), while the continuous part arises from differential calculus of superforms on real vector spaces developed by Largerberg~\cite{Lag12}. The mixed differential operators $\dm, \dmb$, as well as the wedge products for mixed differential forms are defined, and enjoy the nice properties of their complex-geometric analogues. The framework of mixed differential calculus provides a generalization of smooth superforms of~\cites{Lag12,CLD} and $\delta$-forms of~\cites{GK17,Mih24,BGK25}. In particular, the differential operators $\dm, \dmb$  are the generalization of the operators $d', d''$ in the setting of smooth superforms and $\delta$-forms.

\vspace{.2cm}

Based on the class of locally convex functions, and using the mixed differential calculus,  we give the definition of  \kahler forms and psh functions on tropical varieties. 

A real closed mixed $(1,1)$-form $\omega$ on $X$ is a mixed $(1,1)$-form satisfying the following condition: for every $x \in X$, there exist an open neighborhood $U$ of $x$ and a piecewise smooth function $h$ on $U$ such that $\omega=\dm\dmb h$\footnote{Readers who prefer not to delve into our formalism of mixed differential forms may instead view $\omega=d'd''h$ as a real closed $(1,1)$-supercurrent (resp. $(1,1)$-$\delta$-current) in the sense of \cite{CLD} (resp. \cite{GK17}). This does not affect the understanding of any of the main results presented in this paper.}. In this case, $h$ is called a \emph{local potential} of $\omega$.

\begin{definition} [Definition \ref{def:kahler_forms_and_strictly_convex_functions}]
  A \emph{\kahler form}  $\theta$ on a compact  tropical variety $X$ is a real closed mixed $(1,1)$-form $\theta$ such that any local potential $h$ of $\theta$ is strictly locally convex (Definition \ref{def:kahler_forms_and_strictly_convex_functions}).
\end{definition}

If we weaken the condition on the local potentials by requiring them to be locally convex (instead of strictly locally convex), we say that the real closed mixed (1,1)-form $\theta$ is \emph{semipositive} (Definition \ref{def:semipositive-form_paper_2}).

\begin{definition} [Definition \ref{def:psh}]
  Let $\theta$ be a \kahler form on $X$, and let $\varphi$ be a (real-valued) continuous function on $X$. We say that  $\varphi$  is a \emph{continuous $\theta$-psh function} if for any $x \in X$, there exists a neighborhood $U$ of $x$ and a local potential $h$ of $\theta$ on $U$ such that $h+\varphi \in \LConv(U)$. We use $\CPSH_{\theta}(X)$ to denote the space of $\theta$-psh functions on $X$.
\end{definition}

\vspace{.15cm}

The class of continuous $\theta$-psh functions introduced above has many nice properties. For example, we show that is satisfies the so-called envelope property, \textit{i.e.} the upper-semicontinuous regular upper envelope of any bounded above family of $\theta$-psh functions  is again a continuous $\theta$-psh function. If $X$ is a compact tropical variety, we also show that the class of continuous $\theta$-psh functions is uniformly Lipschitz on any compact subset of $\xo$. 

\smallskip

A tropical variety is called \emph{\prekahler}if it admits a \kahler form. A continuous function $f$ on $X$ is called \emph{$\theta$-piecewise linear} (Definition \ref{def:omega-PL}) if there exists a polyhedral structure on $X$ such that, for any local potential $h$ of $\theta$, the restriction of $h+f$ on each face is affine. The class of  functions on $X$ that are both  $\theta$-psh and $\theta$-piecewise linear plays the role of a “nice” subclass of tropical $\theta$-psh functions, analogous to the class of smooth $\theta$-psh functions in complex pluripotential theory (\cite{Dem92}), or the class of $\theta$-psh PL functions in non-Archimedean pluripotential theory (\cites{BFJ16, BJ22, BJ25}). Our first main theorem  is a regularization theorem for $\theta$-psh functions on compact \prekahler tropical varieties, which can be viewed as a tropical analogue of Demailly's regularization theorem (\cites{Dem92, BK07}). 

\begin{theoremA} [Theorem \ref{thm:LConv_regularization_PL}] \phantomsection\label{mainthm:regularization}
  Let $X$ be a compact \prekahler tropical variety. Let $\theta$ be a \kahler form on $X$.  Then, for any continuous $\theta$-psh function $\varphi$ on $X$, there exists a   sequence  of  functions $\{\varphi_i\}_{i \ge 1}$ converging uniformly to $\varphi$ such that $\varphi_i$ is $\theta$-piecewise linear and $\theta$-psh for each  $i \ge 1$. 
\end{theoremA}

\subsection{Tropical \MongeA equations}
Let $X$ be a compact \prekahler tropical variety of dimension $d$ endowed with a \kahler form $\theta$. 

In \cite{APW1}, we use the formalism of mixed differential calculus mentioned above to define tropical \MongeA measures for piecewise smooth functions. 
For a piecewise smooth $\theta$-psh function $\varphi$, we define the \MongeA measure $\MA_{\theta}(\varphi)$ to be the nonnegative Radon measure associated to the top degree form $(\theta+\dm \dmb \varphi)^{\wedge d}$.
Using  the regularization theorem (Theorem~\hyperref[mainthm:regularization]{A}) and a tropical version of the Chern--Levine--Nirenberg inequality, we extend the definition of \MongeA measures to all continuous $\theta$-psh functions.

Let now $\mu$ be a  compactly supported probability measure on $\xo$. The tropical \MongeA equation is the following. 
\begin{equation}  \label{equation:tropical_Monge_Ampere}
      \MA_{\theta}(\varphi)=\deg(\theta^d) \mu \text{   on } X, \qquad \varphi \in \CPSH_{\theta}(X).
\end{equation}

Our second main theorem is the following. 
\begin{theoremB} [Uniqueness, Theorem \ref{thm:uniqueness}] \phantomsection\label{mainthm:uniqueness}
  Let $X$ be a compact \prekahler tropical variety of dimension $d$ equipped with a \kahler form $\theta$\,\footnote{Note that we actually prove Theorem \hyperref[mainthm:uniqueness]{B} under the assumption that $\theta$ is a ``strongly semipositive form'' (see Definition \ref{def:strongly_semipositive_form} and Remark \ref{rmk:kahler_is_strong_semipositive}), which is a weaker condition than $\theta$ being a \kahler form. The same for Theorem \hyperref[mainthm:failure_of_existence]{C} and Theorem \hyperref[mainthm:existence_regular]{D} below.}. Assume that Equation \eqref{equation:tropical_Monge_Ampere} admits a solution. Then, the solution is unique up to adding a constant function.
\end{theoremB}
Regarding the existence of solutions, in dimension one, Equation \eqref{equation:tropical_Monge_Ampere} reduces to a Poisson equation on a metric graph, and it is well-known that it always admits a solution when $X$ is locally matroidal (see for example \cite{Thu05}). However, somehow surprisingly, we proved that it fails for almost every tropical variety of dimension at least $2$.  

\begin{theoremC} [Failure of existence in higher dimensions, Theorem \ref{thm:global_failure_of_orthogonality}] \phantomsection\label{mainthm:failure_of_existence}
  Let $X$ be a compact \prekahler tropical variety of dimension $d$ equipped with a \kahler form $\theta$. Suppose $\dim X \ge 2$ and $X$ satisfies some genericity condition\footnote{That is, the condition stated in Definition \ref{condition:genericity} is satisfied by the star fan of each codimension two face of some polyhedral structure of $X$.}. Then, there exists a compactly supported probability measure $\mu$ on $\xo$ such that Equation~\eqref{equation:tropical_Monge_Ampere} does not admit any solution.
\end{theoremC}

We remark that if $X$ is a projective tropical toric variety $\TP_{\Xi}$,  existence of the solutions holds (see for example \cite[\S\,9.2]{BFJ15}). Note that $\TP_{\Xi}$ does not satisfy the genericity condition stated in Definition \ref{condition:genericity}.  
We conjecture that the tropical analogue of Calabi--Yau theorem holds for $(X, \theta)$ if, and only if, the open part $\xo$ can be written locally as a product of $\R^{d-1}$ and a locally matroidal tropical curve (Conjecture \ref{conj:exact_condition_for_orthogonality_properties}). 
We establish the ``if'' direction of this conjecture in Theorem~\ref{thm:existence_of_solutions_nice_tropical_varieties} by showing the existence of solutions on compact \prekahler tropical varieties that are \emph{regular in codimension one}.

\begin{theoremD}[Existence for tropical varieties regular in codimension one, Theorem~\ref{thm:existence_of_solutions_nice_tropical_varieties}]\phantomsection\label{mainthm:existence_regular}
  Let $X$ be a compact \prekahler tropical variety of dimension $d$ equipped with a \kahler form $\theta$. Assume that $\xo$ is locally modeled as a product of a matroidal tropical curve and $\R^{d-1}$. Then, Equation~\ref{equation:tropical_Monge_Ampere} has a solution for any probability measure compactly supported on $\xo$.
\end{theoremD}

\medskip

We mention briefly the strategies for proving these theorems. 

\noindent For the uniqueness theorem, the proof is similar to the proof of uniqueness for non-Archimedean \MongeA equations given in \cite{BFJ15} using a tropical Cauchy--Schwartz inequality. For the existence, we follow the variational approach for complex and non-Archimedean \MongeA eqautions in \cites{BBGZ13, BFJ15}. 
We enlarge the space $\CPSH_{\theta}(X)$ to the space $\PSH_{\theta}(X)$ equipped with the topology of pointwise convergence of $\xo$.  Denote by $\PSH_{\theta}(X)/\R$  the quotient space of $\PSH_{\theta}(X)$ modulo additive constants.
A uniform Lipschitz result for $\PSH_{\theta}(X)$ (see Theorem \ref{thm:lipschitz} and Proposition \ref{prop:compactness_of_psh}) implies that $\PSH_{\theta}(X)/\R$ is a compact space. 

We construct an energy functional 
\[
\mathrm{E}_{\theta} \colon \PSH_{\theta}(X)/\R \to \R \cup \{-\infty\}
\]
whose derivative is the \MongeA operator, in the sense that
\[
\left.\frac{d}{dt}\right|_{t=0+} \mathrm{E}_{\theta}(\varphi+tf)=\int_X f \, \MA_{\theta}(\varphi) 
\]
for every  $\varphi \in \CPSH_{\theta}(X)/\R$ and $f \in C^0(X)$ satisfying $\varphi+tf \in \CPSH_{\theta}(X)/\R$ for any  $t \ge 0$.

Given a compactly supported probability measure $\mu$ on $\xo$, consider the Lagrangian functional $\FE_\mu \colon \PSH_{\theta}(X)/\R \to \R$ defined by 
\[
\FE_{\mu}(\varphi) \coloneqq \mathrm{E}_{\theta}(\varphi)-\int_{X} \varphi\, \mu. 
\]
The upper semicontinuity of $\FE_{\mu}$ together with the compactness of $\PSH_{\theta}(X)/\R$ guarantee the existence of a maximizer $\varphi_{\mu} \in \PSH_{\theta}(X)/\R$. Furthermore, we show that the maximizer $\varphi_{\mu}$ belongs to $\CPSH_{\theta}(X)/\R$.  However, contrary to the setting of complex and non-Archimedean pluripotential theory, $\varphi_\mu$ is not necessarily a solution to Equation \ref{equation:tropical_Monge_Ampere}. 
Actually, in our tropical setting, we show that the maximizer of $\FE_\mu$ is a solution of Equation \ref{equation:tropical_Monge_Ampere} for any compactly supported probablity measure on $\xo$ if and only if the orthogonality property holds for $(X, \theta)$. Following \cites{BBGZ13, BFJ15}, we say that $(X, \theta)$ satisfies the \emph{orthogonality property} if 
\[
\int_X (f-\env_{\theta}(f)) \MA_{\theta}(\env_{\theta}(f))=0 \qquad \text{for any } f \in C^0(X).
\]
Here, $\env_{\theta}(f)$ is the $\theta$-psh envelope defined by
\[
\env_{\theta}(f)  \coloneqq \Bigl(\sup\Bigl\{\varphi \st \varphi \in \PSH_{\theta}(X) \textrm{ and } \varphi \leq f\Bigr\}\Bigr)^{\star }.
\]
The \emph{envelope property}, that $\env_{\theta}(f)$ belongs to  $\CPSH_{\theta}(X)$, is proved in the tropical setting (Corollary \ref{cor:psh_envelope}). 

When $X$ is locally modeled by a product of a matroidal curve and $R^{d-1}$, 
we show that the orthogonality property holds for $(X,\theta)$. In contrast, if $\dim X \geq 2$ and $X$ satisfies the genericity condition alluded to above, 
the orthogonality property fails. The obstruction is that, for such $X$, one can construct, using polyhedral-geometric methods, a function $\varphi \in \CPSH_{\theta}(X)$ such that $\MA_{\theta}(\varphi)(\{ x \}) > 0$ for some $x \in \xo$, while $\varphi + h$ is not strictly locally convex at $x$, for a local potential $h$ of $\theta$ around $x$. 
An explicit example in dimension 2 is given in Section~\ref{subsec:example_failure_of_tropical_Calabi-Yau}.
We show that the orthogonality property holds in codimension one (Theorem~\ref{thm:orthogonality_codimension_1}).

\subsection{Connections to non-Archimedean solutions}
Let $(\fX_t, \fL_t)_{t \in \BD^*}$ be a maximally degenerating family of smooth polarized Calabi--Yau varieties over the punctured disc $\BD^{*} \subset \BC\setminus\{0\}$. 
Let $(\varX, \varL)$ be the generic fibre of the family over $\valfield =\BC((\mathsf{t}))$.  We assume that $\varL$ is very ample on $\varX$, and $\iota \colon \varX \hookrightarrow \P^n$ is a closed embedding satisfying $\iota^{*} \CO(1)=\varL$. 
Then, the trivial  model of $(\P^n, \CO(1))$ over $\valring$ gives a semipositive $(1,1)$-form on $(\P^n)^{\an}$. We denote by $\omega_{\NA}$ its restriction on $\varX^{\an}$.  Let $\mu_{\mathrm{CY}, \NA}$ be the normalized non-Archimedean Calabi--Yau measure supported on the essential skeleton $\Sk(\varX) \subset \varX^{\an}$ defined in \cites{KS06, BJ17}, and let $\varphi_{\mathrm{CY}, \NA} \in \CPSH_{\omega, \NA}(\varX)$ be the Calabi--Yau potential given by the non-Archimedean Calabi--Yau theorem (\cite{BFJ15}, see also Theorem \ref{thm:NA_Calabi_Yau}).  

We view $\P^n$ as a toric variety with  the open dense torus $\varT$. Let $\varY \coloneqq \varX \cap \varT$, and let
\[
\Trop \colon \varX^{\an} \to \TP^n
\]
be the (extended) tropicalization map. We denote by $X \coloneqq \Trop(\varX, \iota)$ and  $Y\coloneqq \Trop(\varY, \iota)$ the tropicalizations of $\iota\colon \varX \hookrightarrow \P^n$ and $\iota\colon \varY \rightarrow \varT$, respectively. Then $Y$ is a tropical variety embedded in a real vector space of dimension $n$. However, $X$ is not necessarily a tropical variety, since  $X$ does not necessarily satisfy the regularity along the boundary of $\TP^n$ 
(\textit{i.\,e.} it might develop ``singularities'' at infinity).

Therefore, instead of considering tropical \MongeA equations on $X$, we consider the canonical compactification $\comp{Y}$ of $Y$ with respect to a rational simplicial polyhedral structure on $Y$, such that $\comp Y$ behaves nicely at infinity and moreover admits a closed  semipositive $(1,1)$-form $\omega$ induced by the Fubini--Study potential $h=\sup \{0, -x_1, \cdots, -x_n\}$ in $\R^n$. (For these to happen, it is enough to ensure that the recession cones of the polyhedra in the polyhedral structure form a fan that refines the fan of the projective space; this can be always achieved, see Corollary~\ref{cor:refinement_of_open_tropical_varieties_are_projective}.)

\smallskip

Recall that $\mu_{\mathrm{CY}, \NA}$ is a measure supported on the dual complex $\Delta_{\fX}$, where $\fX$ is an SNC model of $\varX$. We say that the tropicalization map is \emph{strongly faithful} (with respect to $\Delta_{\fX}$) if $\Trop\rest{\Delta_{\fX}}$ is injective, and  any maximal dimensional face contained in $\Trop(\Delta_{\fX})$ is of weight $1$. 
This is a refinement of the existing notion of faithful tropicalization in tropical geometry~\cite{BPR13,GRW16}. 
Refining~\cite[Thm.\,9.5]{GRW16}, we show that there always exists a projective embedding of $\varX$ given by a basis of $H^0(\varX, \varL^{\otimes k})$ with $k$ large enough such that the associated tropicalization to this embedding is strongly faithful (with respect to $\Delta_{\fX}$).

\smallskip

We prove the following result.

\begin{theoremE} [Theorem \ref{thm:faithful_tropicalization_partial_retraction_invariance}] \phantomsection\label{mainthm:comparison}
  Suppose that $Y$ is a strongly faithful tropicalization with respect to $\Delta_{\fX}$ for some SNC model $\fX$. Choose a polyhedral structure on $Y$ and let $\comp Y$ be the corresponding canonical compactification of $Y$. 
  Suppose that $\mu_{\mathrm{CY}, \Trop}$ admits a tropical \MongeA solution $\varphi_{\mathrm{CY}, \Trop} \in \CPSH_{\omega}(\comp{Y})$.  Then,  $(\varphi_{\mathrm{CY}, \Trop})|_{Y} \circ \Trop$ coincides with $(\varphi_{\mathrm{CY}, \NA})|_{\varY^{\an}}$ after adding a constant. 
\end{theoremE}

Therefore, an alternative approach to the metric SYZ conjecture is to try to solve the corresponding tropical \MongeA equations on some  tropicalization of $\varX$ which is strongly faithful with respect to $\Delta_{\fX}$. Also, the structure of the non-Archimedean Calabi--Yau potential would be easier to understand in this setting. 

At last, we sketch the strategy to prove Theorem \ref{thm:faithful_tropicalization_partial_retraction_invariance}. We first prove the theorem in the case when $X$ is the canonical compactification of $Y$ with respect to some polyhedral structure.
More precisely, applying the techniques of tropical compactifications and tropical models, we show that for any $\varphi \in \PSH_{\omega} \cap \PL(X)$, one has  
\[
\varphi \circ \Trop \in \PSH_{\omega, \NA} \cap \PL(\varX^{\an}), \quad \text{and} \quad \Trop_{*}(\MA_{\omega, \NA})(\varphi \circ \Trop)=\MA_{\omega, \Trop}(\varphi).
\]
This together with the regularization theorem for tropical psh functions, as well as the Bedford--Taylor theory in both tropical and non-Archimedean setting, indicates that 
\[
\varphi_{\mathrm{CY}, \Trop} \circ \Trop \in \CPSH_{\omega, \NA}(\varX^{\an}),  \quad \text{and} \quad \Trop_{*}(\MA_{\omega, \NA})(\varphi_{\mathrm{CY}, \Trop} \circ \Trop)=\mu_{\mathrm{CY}, \Trop}. 
\]
On the other hand, the strongly faithful tropicalization condition implies that 
\[
\mathrm{Supp}(\MA_{\omega, \NA} (\varphi_{\mathrm{CY}, \Trop} \circ \Trop)) \subset \Sk(\varX).
\] 
Therefore, we have $\varphi_{\mathrm{CY}, \Trop} \circ \Trop=\varphi_{\mathrm{CY}, \NA}$ after adding a constant, which proves the theorem. We also note that our proof of the compatibility results for PL functions gives a new proof of similar results in the setting of \cite{CLD} and \cite{GK17}. 

In the general case when $X$ is not isomorphic to a canonical compactification of $Y$, we take a birational modification $\widetilde{\varX}$ of $\varX$ such that the extended tropicalization of $\widetilde{\varX}$ equals the canonical compactification $\comp{Y}$. By the special case treated before, we get $\varphi_{\mathrm{CY}, \Trop} \circ \Trop \in \CPSH_{\omega, \NA}(\widetilde{\varX}^{\an})$. Then, applying Proposition \ref{prop:NA_envelope_properties_implies_birational_descend}, we can descend $\varphi_{\mathrm{CY}, \Trop} \circ \Trop$ to $\CPSH_{\omega, \NA}(\varX^{\an})$. 
A similar argument based on strongly faithful tropicalization concludes.

\subsection{Related work and further directions} 
Potential theory on finite and compact metric graphs (also known as abstract tropical curves and quantum graphs) 
has a long history, with applications ranging across multiple disciplines, see e.g.~\cite{Post, KN23, Kurasov}, and~\cite{Zha93, Thu05, BR10} for connections to arithmetic geometry.

Tropical Monge--Ampère equation on tropical toric varieties and tropical hypersurfaces have been studied in~\cite{BGJK25} and \cite{HJMM22, Li23Fano, AH23}, respectively, where also connection to optimal transport appears. 

A work~\cite{BMPS26} parallel to the work of the authors contains results overlaping partially with the material presented in this article.
For tropical varieties embedded in projective tropical spaces, 
the authors formulate a tropical Monge--Ampère equation by associating \MongeA measures to PL functions via intersection products of tropical cycles, and extend it via Bedford--Taylor theory.
Derivability of the energy functional is established in \emph{loc.\,cit.} 
assuming the regularity property (Theorem \hyperref[mainthm:regularization]{A}).  
Theorem~\hyperref[mainthm:existence_regular]{D} is established there when $\dim X=1$.  
Theorem~\hyperref[mainthm:comparison]{E} is established in \emph{loc.~cit.} in the case of toric degenerations under the additional assumption of compatibility between tropical and non-Archimedean psh functions. We also note that a special case of Theorem~\hyperref[mainthm:comparison]{E} appears in a recent work~\cite{Lin26}.

The most natural problem arising from our work is the question of 
classifying the measures $\mu$ for which Equation~\eqref{equation:tropical_Monge_Ampere} has a solution. Extending the framework of optimal transport theory to the tropical setting is also an interesting direction of research. Such an extension should allow to address the classification problem, in particular, for discrete measures (\textit{i.e.} finite sums of dirac masses).

Finally, the formalism of mixed differential calculus and tropical pluripotential theory allows to formulate and study a tropical analogue of the Kähler--Einstein equation, an interesting direction for future research.

\subsection{Organization of the paper}
In Section \ref{preliminaries-paper-2}, we review background on tropical and toric varieties, as well as the formalism of mixed differential calculus developed in \cite{APW1}. Section \ref{sec:tropical_convexity} introduces the notion of locally convex functions on tropical varieties, and provides some of their basic properties. In Section \ref{sec:psh_functions}, we define $\omega$-psh functions for a semipositive form $\omega$, and prove the compactness of the space $\PSH_{\omega}(X)/\BR$.
In Section \ref{sec:regularization_theorem}, we introduce the notion of \kahler forms and \prekahler tropical varieties, and prove the regularization theorem (Theorem~\hyperref[mainthm:regularization]{A}).
We establish the Bedford-Taylor theory on tropical varieties   in Section \ref{sec:tropical_Monge_Ampere_measures_paper_2}.  In Section \ref{sec:energy_functional}, we define and study the tropical energy functional. 
In Section \ref{sec:uniqueness}, we formulate the tropical \MongeA equations, and prove the uniqueness of solutions (Theorem~\hyperref[mainthm:uniqueness]{B}). 
The equivalence between the existence of solutions and the orthogonality property is established in Section \ref{sec:orthogonality_properties}. 
In Section \ref{sec:Orthogonality_property_in_codimension_one}, we prove the orthogonality property in codimension $1$ and existence of solutions for regular-in-codimension-one tropical varieties (Theorem~\hyperref[mainthm:existence_regular]{D}). The failure of existence of the solutions in general (Theorem~\hyperref[mainthm:failure_of_existence]{C}) is proved in Section~\ref{sec:failure_of_Calabi_Yau}. 
In Section \ref{sec:relation_to_SYZ}, we prove a comparison theorem between tropical and non-Archimedean \MongeA solutions (Theorem~\hyperref[mainthm:comparison]{E}). In Appendix~\ref{appendix:uniform_lipschitz}, we prove a uniform Lipschitz property for tropical psh functions. Appendix~\ref{appendix:polyhedral_geometric_lemmas} contains four lemmas in polyedral geometry used in the proof of Theorem~\hyperref[mainthm:failure_of_existence]{C}.

\subsection*{Acknowledgment} We thank Sébastien Boucksom for discussions and comments on earlier drafts of this paper. 
We thank Ana Botero, Enrica Mazzon, and L\'eonard Pille-Schneider for agreeing to exchange the drafts of our papers before their release. 
We thank all three of them, as well as José Ignacio Burgos Gil, Antoine Ducros, and Charles Favre for their helpful suggestions and comments on this work. 
This paper is also part of the PhD thesis of the third author, who would like to thank his advisor, Sébastien Boucksom, for his support and encouragement.

The first author was partially supported by ANR (AdAnAr project, ANR-24-CE40-6184). The first and third authors are supported by the SpringBoard package of Université Paris--Saclay. The second author received funding from the European Research
Council under the European Union’s Horizon 2020 research and innovation program
(grant N. 101001995).

\subsection*{AI statement} The mathematical ideas, results, and proofs presented in this paper are all work of the authors. LLM was used occasionally for English grammatical assistance.

\subsection{Notation and Convention}
\begin{itemize}[leftmargin=*]
  \item Unless otherwise specified, all the fans and polyhedral structures (complexes) in this paper are assumed to be rational and simplicial. 
  \item For $\sigma$ a rational polytope or a polyhedron, we denote by $\relint(\sigma)$ its relative interior. 
  \item For a fan $\Sigma$, we use $|\Sigma|$ to denote its support, and use $\comp{\Sigma}$ to denote its \emph{canonical compactification} (see Section \ref{subsubsec:canonical-compactifications}). A fanfold is the support of a rational fan. 
  \item We use $X$ for tropical varieties, and  $\xo$ for their open parts (the minimum sedentarity stratum). We use bold letters (\textit{e.g.} $\varX, \varY, \cdots$) to denote algebraic varieties.
  \item For a tropical variety $X$, we denote by $\Aff$ the sheaf of affine functions on $X$ (see Section \ref{subsubsec:affine-functions}), and denote by $\LConv$ the sheaf of locally convex functions on $X$ (see Section \ref{sec:tropical_convexity}).
  \item A simplicial rational polyhedral complex (\emph{resp.} structure) on a tropical variety is defined in Section \ref{subsubsec:polyhedral-complex-structures}. Throughout this paper, unless stated otherwise, the term polyhedral complexes (\emph{resp.} structures) refers to simplicial rational polyhedral complexes (\emph{resp.} structures).
  \item We use $\olPi$ to denote a polyhedral structure on $X$, and $\Pi$ to denote the restriction of $\ol{\Pi}$ on $\xo$ (see Section \ref{subsubsec:polyhedral-complex-structures}). We denote by 
  $\Pi_{\f}$ the polyhedral subcomplex of $\Pi$ formed by compact faces in $\Pi$. 
  \item For a natural number $k$, $\ol{\Pi}_k$ (\textit{resp.} $\Pi_k$) denote the set of $k$-dimensional polytopes in $\ol{\Pi}$ (\textit{resp.} $\Pi$).   
  \item For $\sigma \in \Pi$, we denote by $T_{\sigma}$ the tangent space at $\sigma$ (see Section \ref{subsec:tropical-varieties}), and denote by $\Pi^{\sigma}$  the star fan of $\Pi$ at $\sigma$.  The integer lattice in $T_\sigma$ is denoted by $\ssN_\sigma$.- 
  \item For a tropical variety $X$ with open part $\xo$, we denote by $\PL$ the sheaf of piecewise linear functions on $X$, and denote by $\PSm$ the sheaf of piecewise smooth functions on $\xo$ (see Section~\ref{subsubsec:PL_functions}). 
  \item For a fan $\Sigma$, we use $\PL_{\hom}(\Sigma)$ to denote the set of  homogeneous PL functions on $\supp{\Sigma}$ (\textit{i.e.} PL functions on $\supp{\Sigma}$ satisfying $f(tx)=tf(x)$). 
  \item The mixed differential operators are denoted by $\dm$ and $\dmb$ (see Section \ref{subsec:dd-p2}). They are generalizations to the piecewise smooth setting 
  of the tropical differential operators $d', d''$ for smooth forms, introduced in \cites{Lag12, CLD}. 
  \item Let $E$ be a finite dimensional Euclidean space. For a subset $S \subset E$, we denote by $\conv(S)$ the convex hull of $S$ in $E$, and we denote by $\mathrm{Cone}(S)$ the convex cone generated by $S$ in $E$. 
  \item We use the standard abbreviations \emph{usc} for ``upper semicontinuous'', \emph{lsc} for ``lower semicontinuous''. We also use \emph{SNC} for ``simple normal crossing'', \emph{PL} for ``piecewise linear'', and \emph{psh} for ``plurisubharmonic''.
  \item  Let $W$ be a locally compact Hausdorff space.  The notation $\Subset$ indicates being compactly contained, \textit{i.e.}  $V \Subset W$ if the closure of $V$ is compact and contained  in $W$. 
  \item If $W$ is a Hausdorff topological space, and $\varphi \colon W \to \R \cup \{\pm \infty\}$ is any function, then, the \emph{usc envelope} $\varphi^{\star}$ of $\varphi$ is the smallest usc function with $\varphi^{\star} \ge \varphi$. Concretely, $\varphi^{\star}(x) = \limsup_{y \to x} \varphi(y)$. 
\end{itemize}

%%%%%%%%%%%%%%%%%
%%Preliminaries
%%%%%%%%%%%%%%%%%

\section{Preliminaries} \label{preliminaries-paper-2}
We recall basic definitions and results on tropical varieties, closely following \cites{JSS19, AP-tht, ACM}, to which we refer for further details. We also provide a brief account of the formalism of mixed differential forms from~\cite{APW1}.

\subsection{Tropical varieties}\label{subsec:tropical-varieties}

 Let $N$ be a free abelian group of finite rank, with dual $M = \ssN^\dual = \Hom(N,\Z)$.   Denote by $\ssN_{\Q}, \ssN_{\R}, \ssM_{\Q}, \ssM_{\R}$ the associated rational and real vector spaces, so that $\ssM_{\Q} = \ssN_{\Q}^\dual$ and $\ssM_{\R} = \ssN_{\R}^\dual$.  
The pairing between $M$ and $N$ is denoted by $\langle\cdot\,,\cdot\rangle$.

An affine half-space in $\ssN_\R$ is a set of the form 
 $\{ v \in \ssN_{\R} \mid \langle u,v\rangle \leq b \}$ with $u \in \ssM_{\R}$ and $b \in \R$.  It is called rational if $u\in \ssM$. It is linear if $b=0$.

 A finite intersection of affine (\resp linear) half-spaces is called a polyhedron (\resp polyhedral cone). It is called rational if the half-spaces can be chosen rational. A compact polyhedron is called a polytope. 
 
 The tangent space $T_Q$ of a polyhedron $Q$ 
 is the linear span of the differences of points in $Q$.  
It is rational if $N \cap T_Q$ is a full-rank lattice in $T_Q$.  The dimension of $Q$ denoted $\dims{Q}$ is the dimension of $T_Q$.

Each polyhedron $Q$ can be written as a Minkowski sum $Q' +\tau$ of a polytope $Q'$ and a unique polyhedral cone $\tau$ called the recession cone of $Q$ and denoted by $\rec(Q)$. 
A polyhedral cone $\tau$ with exactly $\dims{\tau}$ rays is called 
simplicial. A polyhedron $Q$ is called simplicial if it is the sum of a simplex $\delta$ and a simplicial cone $\tau$ such that the intersection of the two tangent spaces $T_\delta$ and $T_\tau$ is  $\conezero$. In this case, $\delta$ is also unique with this property.

\subsubsection{Tropical fans} For a polyhedral cone $\sigma \subset \ssN_{\R}$, let $\ssN_{\sigma,\R}$ be the real span of $\sigma$ and set  
\[
  \ssN_{\R}^{\sigma} \coloneqq \rquot{\ssN_{\R}}{\ssN_{\sigma,\R}}.
\]
We have $\dims \sigma = \dim \ssN_{\sigma,\R}$. If $\sigma$ is rational, this induces full-rank lattices $\ssN_{\sigma} \subset \ssN_{\sigma,\R}$ and $\ssN^{\sigma} \subset \ssN_{\R}^{\sigma}$.  

A polyhedral cone that contains no line is called strongly convex.  
A fan $\Sigma$ in $\ssN_{\R}$ is a finite, nonempty collection of strongly convex polyhedral cones with the property that every face of a cone $\sigma \in \Sigma$ lies in $\Sigma$, and moreover, the intersection of any two cones in $\Sigma$ is a common face. We refer to elements of $\Sigma$ as its faces. 

A fan is called rational if all cones are rational. It is called simplicial if all its cones are. 
All the fans appearing in this paper will be rational and simplicial.

For cones $\sigma,\tau$, we write $\tau \subfaceq \sigma$ if $\tau$ is a face of $\sigma$, and $\tau \ssubface \sigma$ if it has codimension one. This induces a partial order on $\Sigma$.  We denote by $\ssSigma_k$ the set of $k$-dimensional cones. The cone $\{0\}$ is written $\conezero$ and one-dimensional elements of $\Sigma$ are called rays.  Maximal faces of $\Sigma$ are called facets.
A fan is pure-dimensional if all its facets have the same dimension.

The support of $\Sigma$, written $\supp \Sigma$, is the union of its cones in $\ssN_\R$. For $\Sigma$ of dimension one, there is a unique fan structure on its support, so we simply use $\Sigma$ in place of $\supp{\Sigma}$.  
A fanfold $X \subset \ssN_{\R}$ is a closed subset that is the support of some rational fan. A point $x$ in a fanfold of dimension $d$ is regular if $x$ admits a neighborhood in $X$ homeomorphic to an open subset of $\R^d$.

For a cone $\sigma \in \Sigma$ and a codimension-one face $\tau \ssubface \sigma$, the subspace $\ssN_{\tau,\R}$ divides $\ssN_{\sigma,\R}$ into two closed half-spaces, one of which contains $\sigma$.  
A primitive normal vector to $\tau$ in $\sigma$ is any $\nu \in \ssN_\sigma$ lying in this half-space and satisfying $\ssN_\tau + \Z \nu = \ssN_\sigma$; we denote such a choice by $\nvect_{\sigma/\tau}$.  Its image in the quotient lattice $\ssN_{\sigma}^{\tau} \coloneqq \rquot{\ssN_{\sigma}}{\ssN_{\tau}}$ is then well-defined and denoted by $\e^\tau_\sigma$. If $\tau=\conezero$, and $\sigma$ is a ray of $\Sigma$, we simply write $\e_{\sigma}$.
We denote by $\nvect^\dual_{\sigma/\tau}\in \ssN_\sigma^\dual$  the unique integral linear form that is zero on $\ssN_{\tau}$ and takes value one on $\nvect_{\sigma/\tau}$.

A tropical fan is a pair $(\Sigma,\weight)$ consisting of a rational fan $\Sigma$ in $\ssN_{\R}$ of pure dimension $d \in \Z_{\geq 0}$ and  a function 
\[
  \weight \colon \ssSigma_d \to \N
\]
called tropical weight that satisfies the balancing condition: for every cone $\tau \in \ssSigma_{d-1}$,  
\begin{align}\label{eq:codimension_one_balancing} 
  \sum_{\sigma \ssupface \tau} \weight(\sigma)\nvect_{\sigma/\tau} \in \ssN_\tau,
  \qquad\text{equivalently}\qquad
  \sum_{\sigma \ssupface \tau} \weight(\sigma)\e^\tau_{\sigma} = 0 \quad \text{in $\ssN^\tau$}.
\end{align}  
 The tropical weight function $\weight$ defines a weight function on the set of regular points of $\supp{\Sigma}$. The support of a tropical fan endowed with this weight function is called a tropical fanfold.

For a cone $\sigma \in \Sigma$, the star fan of $\sigma$ in $\Sigma$ denoted $\Sigma^\sigma$ is a fan in the quotient space $\ssN_{\R}^\sigma$ defined as follows. 
Consider the projection map $\pi_\sigma\colon \ssN_{\R} \to \ssN_{\R}^\sigma$. The star fan $\Sigma^\sigma$ consists of the projected cones $\eta^\sigma \coloneqq \pi_\sigma(\eta)$ for $\eta\supfaceq \sigma$ in $\Sigma$.
  For $\Sigma$ a rational fan of pure dimension $d$, the star fan $\Sigma^\sigma$ is of pure dimension $d - \dims \sigma$ and rational with respect to the quotient lattice $\ssN^\sigma$. 

A tropical weight function $\weight$ on $\Sigma$ naturally induces a tropical weight on $\Sigma^\sigma$ that we denote by $\weight^\sigma \colon \ssSigma^\sigma_{d - \dims \sigma} \to \N$. 
It is defined by setting $\weight^\sigma(\eta^\sigma) \coloneqq \weight(\eta)$ for each cone $\eta \in \ssSigma_d$ with $\eta \supfaceq \sigma$. 
The balancing condition around a face $\eta^\sigma$ of codimension one in $\Sigma^\sigma$ reduces to 
 the balancing condition around $\eta$ in $\Sigma$. If there is no risk of confusion, we drop $\sigma$ in $\weight^\sigma$, and use $\weight$ for the induced weight function on the star fan.

A \emph{generalized tropical line} is a tropical fan $\CL$ of dimension one in some $\ssN_\R$ such that the space of linear relations between the primitive vectors $\e_\varrho$, $\varrho\in \CL_1$, is 
of dimension one, generated by the balancing condition $\sum_\varrho \weight(\varrho)\e_\varrho=0$. A tropical curve locally modeled by generalized tropical lines is called \emph{locally matroidal}.
\subsubsection{Tropical toric varieties and canonical compactifications of fans} \label{subsubsec:canonical-compactifications}

A rational fan $\Sigma$ in $\ssN_{\R}$ defines the tropical toric variety $\TP_\Sigma$, a partial compactification of $\ssN_{\R}$ that coincides with the tropicalization of the toric variety $\P_\Sigma$~\cites{BGJK20, Kaj08, OR11, Pay09, Thu07}.  
The canonical compactification of $\Sigma$ denoted $\comp \Sigma$ is obtained by taking the closure of cones of $\Sigma$ in $\TP_\Sigma$, whose support also denoted $\comp\Sigma$, is called the canonical compactification induced by $\Sigma$ of the fanfold $\supp \Sigma$. We give a brief account of the constructions.

Let $\eR \coloneqq \R \cup \{+\infty\}$ endowed with the topology generated by intervals $(a,+\infty]$, $a\in \R$, and 
denote $\eR_+ = \R_+ \cup \{+\infty\}$. The operation of addition and multiplication by positive scalars extend continuously to $\eR$ and $\eR_+$. If there is no risk of confusion, we  write $\infty\in \eR$ for $+\infty$.

Let $\RpMod$ be the category of $\R_+$-modules. Objects of $\RpMod$ are commutative monoids $(M,+)$ endowed with an action of $(\R_+,\times)$ that is distributive under addition of scalars. In particular, $\RpMod$ contains $\eR$, $\eR_+$, and all 
polyhedral cones.

For a cone $\sigma \subset \ssN_{\R}$, define its dual $\sssigma^\vee$ and orthogonal $\sssigma^\perp$ by
\[
\sssigma^\vee = \{u \in \ssM_{\R} \st \langle u,v \rangle \ge 0\ \forall v\in \sigma\}, \quad 
\sssigma^\perp = \{u \in \ssM_{\R} \st \langle u,v \rangle = 0\ \forall v \in \sigma\},
\]
and set $\bbU_\sigma = \Hom_{\RpMod}(\sssigma^\vee, \eR)$. The tropical toric variety
 $\TP_\Sigma$ is obtained by gluing the $\bbU_\sigma$ for $\sigma\in \Sigma$, as in the classical setting.  The tropical torus orbit of a cone $\tau \in \Sigma$ is 
\[
\bbO(\tau) = \{v \in \bbU_\tau \st \langle u, v\rangle = \infty \text{ exactly for all } u \in \sstau^\vee \setminus \sstau^\perp\} \simeq \ssN_{\R}^{\tau}.
\]  
Then, $\TP_\Sigma$ is the union of the tropical torus orbits of its cones, and 
for each $\sigma$, we have $\bbU_\sigma = \bigsqcup_{\tau \subfaceq \sigma} \bbO(\tau)$. 

The canonical compactification $\comp{\sigma}$ of a cone $\sigma\in \Sigma$ is $\comp{\sigma} = \Hom_{\RpMod}(\sssigma^\vee, \eR_+)$, and coincides with the closure of $\sigma$ in $\bbU_\sigma$. We have $\comp{\Sigma} = \bigcup_{\sigma\in \Sigma} \comp{\sigma}$.

The polyhedral face structure on $\comp{\Sigma}$ is described as follows.  For each $\tau \in \Sigma$, let $\infty_\tau$ be the elemet of $\bbO(\tau)$ that vanishes on $\sstau^\perp$.
 For $\sigma\supfaceq \tau$, set
\begin{equation*}
  \sssigma_{\infty}^{\tau} \coloneqq \infty_{\tau} + \sigma
  = \left\{v \in \comp{\sigma} \st \langle u,v \rangle = \infty 
     \,\, \text{iff } u \in \tau^{\vee} \setminus \tau^{\perp}\right\} 
  \subset \bbO(\tau).
\end{equation*}
Via the identification $\bbO(\tau) \simeq N_{\R}^{\tau}$, the cone $\sssigma_{\infty}^{\tau}$ is the image of $\sigma$ under the projection map $\pi \colon N_{\R} \to N_{\R}^{\tau}$.  
The family $\{\sssigma_{\infty}^{\tau} \mid \sigma \supfaceq \tau\}$ forms a fan $\ssSigma_{\infty}^{\tau}$ in $\bbO(\tau)$ with origin $\infty_{\tau}$. Under $\ssN_{\R}^{\tau} \simeq \bbO(\tau)$, this fan can be identified with the star fan $\Sigma^{\tau}$ of $\tau$ in $\Sigma$.  
For $\conezero \in \Sigma$, we have $\Sigma_{\infty}^{\conezero} = \Sigma$.
The closure in $\comp\Sigma$ of each cone  $\sssigma^\tau_\infty$ coincides with the canonical compactification $\comp{\sssigma^{\tau}_\infty}$. A face of $\comp\Sigma$ is a compactified cone $\comp{\sssigma_\infty^\tau}$ for a pair $\tau\subfaceq\sigma$ of elements of $\Sigma$. 

For $x \in \TP_\Sigma$, the sedentarity $\sed(x) \in \Sigma$ is the unique cone $\tau$ such that $x \in \bbO(\tau)$. The sedentarity fan of a subset $U \subset \comp \Sigma$ denoted $\ssSigma_U$ is the subfan of $\Sigma$ consisting of the sedentarities $\sed(x)$ for $x\in U$. (We omit the proof that this is indeed a subfan of $\Sigma$).

Let $\delta$ be a simplicial cone with rays $\varrho_1, \dots, \varrho_k$, $k=\dims{\delta}$. 
Then, in terms of the primitive vectors $\e_{\varrho_i}$, we have 
 $\comp \delta = \eR_{\ge 0}\e_{\varrho_1}+\dots + \eR_{\ge 0}\e_{\varrho_k}$. 
 Under this identification, each $x\in \comp\delta$ gets a set of $k$ coordinates, 
 and the sedentarity of $x \in \comp{\delta}$ corresponds to the subset of $[k]$ given by the $\infty$-coordinates of $x$.

Suppose $\Sigma$ is simplicial.  For $x \in \comp\Sigma$, let $\delta = \sed(x)$ and let $x_{0}$ be its image under the natural isomorphism $\ssSigma_{\infty}^{\delta} \simeq \ssSigma^{\delta}$. 
The local structure of $\comp \Sigma$ around $x$ is modeled by an open subset of the point $(\infty_{\delta}, x_0)\in \comp\delta\times \ssSigma^{\delta}$, given by a 
set of $k=\dims{\delta}$ coordinates in $\eR_{\geq 0}$ 
corresponding to the rays $\varrho_1, \dots, \varrho_k$ of $\delta$, and an open neighborhood $V$ of $x_0$ in $\ssSigma^{\delta}$. 
Under this identification, a basic open neighborhood  of $x$ is an open neighborhood $U$ of $(\infty_{\delta}, x_0)$ of the form 
\[
\Bigl((t,\infty]\cdot\e_{\varrho_1}+\dots +(t,\infty]\cdot\e_{\varrho_k}\Bigr)\times V.
\]
 A basic open subset of $\comp\Sigma$ is a basic open neighborhood of some point $x\in \comp{\Sigma}$. Basic open sets form a basis of open sets 
 for the topology of $\comp\Sigma$.

Let $U$ be a basic open subset of $\comp\Sigma$ as above. 
An affine function on $U$ is a function $u \colon U \to \R$ such that 
$u = l \circ \pi$ for an affine function $l$ on $V$, with $\pi \colon U \to V$ the projection. 
We denote by $\Aff(U)$ the set of affine functions on $U$. More generally, given an open subset $W$ of $\comp\Sigma$, an affine function on $W$ is a continuous function whose restrictions to basic 
open subsets $U \subset W$ is affine. They form the set $\Aff(W)$, and together define the 
sheaf of affine functions on $\comp\Sigma$.

In this paper $\comp\Sigma$ will always refer to the canonical compactification of $\Sigma$.

\subsubsection{Tropical varieties}  \label{subsubsec:tropical-varieties}
Let $X$ be a connected Hausdorff space.  A polyhedral chart is a pair $(W,\phi)$ with $W \subset X$ a nonempty open set and $\phi \colon W \to U$ a homeomorphism onto a basic open neighborhood $U$ of a point in the canonical compactification of a simplicial rational fan $\Sigma$, as in the previous section. We suppose without loss of generality that the linear span of $\Sigma$ coincides with the ambient space in which it is embedded.
We refer to the lattice of $\Sigma$ as the ambient lattice of the chart $(W, \phi)$.

Two charts $(W_i,\phi_i \colon W_i \to U_i)$ and $(W_j,\phi_j\colon W_j \to U_j)$, with ambient fans $\Sigma_i$ and $\Sigma_j$ and 
lattices $\ssN_i$ and $\ssN_j$, are compatible if either $W_i \cap W_j = \emptyset$ or the transition map
\[
\phi_i \circ \phi_j^{-1} \colon \phi_j(W_i \cap W_j) \to \phi_i(W_i \cap W_j)
\]
is induced by a global affine $\Z$-linear map (\textit{i.\,e.} affine map with a $\Z$-linear part) $\psi_{ij}\colon \ssN_j \to \ssN_i$ that restricts to a $\Z$-linear isomorphism between the sedentarity fans of $\phi_j(W_i \cap W_j)$ and $\phi_i(W_i \cap W_j)$ in $\Sigma_j$ and $\Sigma_i$, respectively.

An extended polyhedral space is a connected Hausdorff space endowed with an atlas of pairwise compatible polyhedral charts.  

Let $(W_i,\phi_i \colon W_i \to U_i)$ be a chart. We say that two points $x,y \in W_i$ have the same sedentarity if the images $\phi_i(x)$ and $\phi_i(y)$ do in $U_i$. 
The transitive closure in $X$ of the property of having the same sedentarity in a chart yields a stratification of $X$ into sedentarity types, denoted $\SedT(X)$. 
(The sedentarity types play the role of generic points of the strata, analogous to stratified algebraic varieties.) 
For $s \in \SedT(X)$, denote by $X_s$ the stratum of type $s$. 

The set $\SedT(X)$ carries a natural partial order: $s_1 \prec s_2$ if each point of type $s_2$ is a limit of points of type $s_1$.  Using the connectivity of $X$, it is easy to show 
that $\SedT(X)$ has a minimum element (see~\cite[Prop.\,2.2]{ACM}) that we denote by $o$, and refer to $X_o$ as the open part of $X$.

A point $x \in X$ is regular if it has a neighborhood homeomorphic to an open subset of $\R^d$ for some $d$.  
We denote by $\ssub X_{\reg}$ the locus of regular points.

Let $X$ be an extended polyhedral space.  A tropical weight function is a map $\weight \colon X_{\reg} \to \N$ such that, on each chart $(W_i,\phi_i)$ with image $U_i$, the restriction $\omega|_{W_{i,\reg}}$ is the pullback of a tropical weight function on $\Sigma_i$ (restricted to $U_{i,\reg}$).

A tropical variety is a pair $(X,\weight)$ consisting of an extended polyhedral space and a tropical weight function.  
For each sedentarity type $s$ of $X$, the stratum $\comp{X_s}$ inherits the structure of a tropical variety. 

There is a natural notion of morphism between tropical varieties. 
A continuous map $\varphi \colon X\to X'$ between two tropical varieties $X$ and $X'$  is a tropical morphism if the image $\varphi(X_o)$ lies in a single sedentarity stratum  $X_{s}'$ of $X'$, and there exist compatible atlases  $\{(\ssW_i,\phi_i)\}_{i\in I}$ on $X$ and $\{(\ssW_j',\phi_j')\}_{j\in J}$ on $\comp{X_{s}'}$, with lattices $N_i$ and $N'_j$
 such that, for each $i\in I$, there is some $j\in J$ satisfying 
 $\varphi(\ssW_i) \subset \ssW_{j}'$, and  $\phi_j'\circ \varphi \circ \phi_{i}^{-1}$ is induced by a  $\Z$-linear map of lattices $\ssN_{i} \to \ssN_{j}'$.

\subsubsection{The sheaf of affine functions} \label{subsubsec:affine-functions}
Let $W$ be an open subset of a tropical variety $X$.  
A function $u \colon W \to \R$ is called affine for every point $x \in W$, there exists a basic open neighborhood $U$ of $x$ (in a local polyhedral chart) such that the restriction $u\rest{U}$ is affine. 
 We denote by $\Aff(W)$ the set of affine functions on $W$. 
  They form a sheaf of abelian groups on $X$.

\subsubsection{The sheaf of PL-functions} \label{subsubsec:PL_functions}
Let $U \subset X$ be an open subset. A function $f\colon U \to \R$ is called piecewise linear (PL) if $f \in C^0(U)$ and there exists a polyhedral structure $\olPi$ for $X$ such that the restriction of $f|_{U\cap P}$ is affine  for  each face $P$ of $\Pi$. 
We denote by $\PL$ the sheaf of PL functions on $X$.

\subsubsection{Polyhedral  structures}\label{subsubsec:polyhedral-complex-structures} We now discuss polyhedral structures on extended polyhedral spaces.

Let $\sigma$ be a rational simplicial cone, and $Q\subseteq \sigma$ be a polyhedron. 
We say that $Q$ is compatible with $\sigma$ if 
the recession cone $\rec(Q)$ is a face of $\sigma$. In this case, writing $Q=Q'+\tau$, for a polytope $Q'$ and $\tau=\rec(Q)$, 
the closure of $Q$ in $\comp\sigma$ is $Q'+\comp{\tau} \subset \comp\sigma$.
A (simplicial) rational polytope $P$ in $\comp\sigma$ is by definition the closure of a compatible (simplicial) rational polytope $Q$ in $\sigma$. A face of $P$ is a nonempety intersection of $P$ with a face of $\comp\sigma$ (for the face structure defined in \S\,\ref{subsubsec:canonical-compactifications})

Let $\Sigma$ be a rational simplicial fan in $\ssN_\R$ with canonical compactification $\comp\Sigma$. 
A rational polytope in $\comp\Sigma$ is by definition a rational polytope in a face of $\comp\Sigma$ (that is a rational polytope in the canonical compactification $\comp{\sssigma^\tau_\infty}$ of a face $\sssigma^\tau_\infty \in \ssSigma^\tau_\infty$, for a pair of elements $\tau\subface \sigma$ in $\Sigma$).

A rational polytope in an extended polyhedral space $X$ is a closed set $P \subset X$ such that, for each chart $(W,\phi)$ with $P \subset W$, the image $\phi(P)$ is a rational polytope.  

A (simplicial) rational polyhedral structure on $X$ is a collection $\olPi$ of (simplicial) rational polytopes satisfying: 
\begin{itemize} 
  \item For each $P\in \olPi$, there exists a chart $(W,\phi)$ with $P\subset W$. Moreover, for any such chart, if $F$ is a face of $\phi(P)$, then $\phi^{-1}(F) \in \olPi$.
  \item If $P_1, P_2 \in \olPi$ and $(W_1,\phi_1), (W_2,\phi_2)$ are charts containing them, then $P_{12} = P_1 \cap P_2 \in \Pi$, with $\phi_1(P_{12})$ a face of $\phi_1(P_1)$ and $\phi_2(P_{12})$ a face of $\phi_2(P_2)$.
\end{itemize}
We refer to $\olPi$ as a (simplicial) rational polyhedral complex. A (simplicial) rational polyhedral structure $\olPi$ on $X$ induces a (simplicial) rational polyhedral structure on $\xo$ denoted by $\Pi$.

 \medskip

 All tropical varieties admit simplicial polyhedral structures. For simplicity, 
 we assume that the tropical varieties studied in this paper admit a rational simplicial polyhedral structure. In particular, this includes tropical varieties coming from tropicalization of algebraic varieties over a discretely valued field. Throughout the remainder of this paper, unless stated otherwise, polyhedral complexes (\emph{resp.} structures) are assumed to be rational and simplicial.

  \smallskip

Given a polyhedral structure $\olPi$ on $X$, with induced polyhedral structure $\Pi$ on $\xo$, for each point 
$x\in Z$, we set
\[
\olPi^{(x)}\coloneqq \bigcup_{x \in P, P \in \olPi} \relint(P).
\] 
Note that $\olPi^{(x)}\subset X$ is open. When $x\in \xo$, we have $\Pi^{(x)}=\olPi^{(x)}\subset \xo$. An affine function defined on a neighborhood of $x\in X$ can be 
extended facewise linearly to $\olPi^{(x)}$.

  \medskip

  For each $\sigma \in \Pi$, we denote by $\Pi^\sigma$ the star fan of $\sigma$ in $\Pi$, defined as follows. Let $\ssN_\R$ be the linear span of all the faces $\sigma'\supfaceq\sigma$, with underlying lattice $\ssN$,
    and $\ssN_{\sigma,\R} \subset \ssN_\R$ be the linear span of $\sigma$ with underlying lattice $\ssN_\sigma$.
     The star fan $\Pi^\sigma$ is the fan in $\ssN^\sigma_\R=\rquot{\ssN_\R}{\ssN_{\sigma,\R}}$ 
  defined as the collection of all the cones generated by the projection of a face $\sigma'\supfaceq \sigma$ into $\ssN^\sigma_\R$. In local polyhedral chart $(W,\phi)$ that contains $\sigma$, let $\comp\Sigma$ be the canonically compactified fan with where the image of $\phi$ lands, i.e., 
  $\phi(W) \subset \comp\Sigma$.   Let $\eta$ be the smallest face of $\Sigma$ that contains $\phi(\sigma)$. Then, $\Pi^\sigma = \Sigma^\eta$.  
  (Note that, more generally, for each sedentarity stratum $X_s$, 
  we get an induced polyhedral structure $\Pi_s$ from $\olPi$, 
  and for $\sigma\in \Pi_s$, we can define its star fan $\Pi_s^\sigma$ in a similar way.)

 \medskip

Let $\olPi$ be a polyhedral structure on an exteneded polyhedral space $X$. A polyhedral subcomplex $\olPi'$ of $\olPi$ is a collection of polytopes $P \in \olPi$ such that $\olPi'$ is a polyhedral structure. This means that the union $Y$ 
of polytopes in $\olPi'$ is itself an extended 
 polyhedral space admitting $\olPi'$ as a polyhedral structure. In this case we say that $Y$ is a polyhedral subspace of $X$.
 
Let $\olPi$ and $\olPi'$ be two polyhedral structures on $X$. We say that $\olPi'$ is a refinement of $\olPi$ if each $P'\in \olPi'$ is contained in an element $P\in \olPi$.

\subsubsection{Pushforward of weight functions} \label{subsubsec:pushforward_of_weight_functions}
Suppose $\varphi \colon X \to X'$ is a proper tropical morphism,  and let $n = \dim X$.  
The pushforward $\phi_*(X)$ is defined as follows. 

Replacing $X'$ by the closure of the sedentarity stratum that contains $\varphi(X_o)$, 
we may assume $\varphi(X_o) \subset X_o'$.

Let $\Pi$ be a rational polyhedral structure on $X_o$. After taking a refinement of $\Pi$ if necessary, we may assume that $\varphi(P)$ is contained in a rational polytope of $X_o'$ for each full-dimensional polytope $P \in \Pi$. 
Then the set $Y_o$ underlying $\varphi_*(X)_o$ is defined as the union of all $\varphi(P)$ for $P\in \Pi$ such that $\varphi(P)$ is $n$-dimensional. Then, $\varphi_*(X)$ is the closure $Y$ of $Y_o$ in $X'$. 
We define a weight function on $Y_o$. 
Let $Z \subset Y_{\reg}$ be the subset of points $z$ such that $\varphi^{-1}(z)$ is a finite subset of $X_{\reg}$ 
(with $Y_{\reg}$ the set of points in $Y$ that admit an open neighborhood homeomorphic to a ball in $\R^n$).
  
 \noindent Define $\weight_Y \colon Z \to \N$ 
by
\begin{equation} \label{equation:definition_weight_projection}
  \weight_Y(z) = \sum_{w \in \varphi^{-1}(z)} 
  \left| N_z / d\varphi(N_w) \right| \, \weight_X(w),
\end{equation}
where $d\varphi$ denotes the linear part of $\varphi$ at $w$. 
By \cite[Prop.~6.2.1]{MikhalkinRau}, the function $\weight_Y$ extends to a weight function on $Y_{\mathrm{reg}}$ such that $(Y_o, \weight_Y)$ is a tropical variety. 
The pushforward $\varphi_*(X)$ is then defined as the closure of $Y$ in $X'$ endowed with this weight function on $Y_o$. In general, this is not a tropical variety in the sense defined in this section, 
as it might have ``singularities at infinity''. We view $\varphi_*(X)$ as a tropical variety with singularities at infinity.

\subsubsection{Projective tropical toric varieties} \label{subsubsec:projective_tropical_variety-p2}
Let $\Xi$ be a simplicial rational fan on $\ssN_\R$. 
We say the tropical toric variety $\BT \BP_{\Xi}$ is projective if the corresponding toric variety $\BP_{\Xi}$ is projective. That is to say, the rational fan $\Xi $ satisfies the following two properties.
\begin{enumerate}
  \item $\Xi$ is complete.
  \item There exists a strictly convex conewise linear function $g$ on $\Xi$, \textit{i.e.} for each $\delta \in \Xi$, there exists a linear function $l$ which agrees with $g$ on $\delta$ and satisfies $g>l$ on $E \setminus \delta$.
\end{enumerate}

We recall the following basic result in toric geometry.
\begin{thm} \label{thm:toric_resolution} \textnormal{[Theorem 6.1.18 and Theorem 11.1.9 in \cite{CLS11}]}
For any complete simplicial rational fan $\Xi$ in $\ssN_\R$, there exists a successive stellar subdivision $\Xi'$ of $\Xi$ such that $\Xi'$ is simplicial and rational, and $\BT \BP_{\Xi'}$ is projective. 
\end{thm}

Let $X$ be a tropical variety. We say that $X$ is a tropical subvariety of $\BT \BP_{\Xi}$ if $X$ is a closed polyhedral subspace of $\BT \BP_{\Xi}$ and $X \cap (\BT \BP_{\Xi})_{o}$ is nonempty. Let $Y$ be a tropical variety embedded in $\ssN_\R$. 
Suppose that there exists a simplicial rational polyhedral structure $\Pi$ on $Y$ such that $\Sigma \coloneqq \rec(\Pi)$ is a subfan of $\Xi$. Then the closure of $Y$ in $\BT \BP_{\Xi}$ is isomorphic to the canonical compactification of $Y$ with respect to $\Sigma$, and it is a tropical subvariety of $\BT \BP_{\Xi}$.

We say that a tropical variety $X$ is projective if it is isomorphic to a tropical subvariety of a projective tropical toric variety.

\subsection{Chow rings of fans} \label{subsec:chow-ring-p2}

Let $\Sigma$ be a simplicial rational fan of dimension $d$. By definition, the Chow ring of $\Sigma$ is the Chow ring of the corresponding toric variety $\P_\Sigma$.  
With real coefficients, it is the graded algebra defined by quotienting the polyomial ring in variables $\x_\zeta$, $\zeta\in \Sigma_1$,
\[
  \ssA^\bul \coloneqq 
  \rquot{\R[\x_\zeta]_{\zeta\in \Sigma_1}}{(I + J)},\quad \text{where}
\]
\begin{itemize}
  \item $I$ is the homogeneous ideal generated by the products $\x_{\zeta_1}\cdots \x_{\zeta_k}$, $k\in \N$, for rays $\zeta_1,\dots,\zeta_k$ that do not span a cone in $\Sigma$, and
  \item $J$ is the homogeneous ideal generated by the elements
  \[
    \sum_{\zeta \in \Sigma_1} m(\e_\zeta)\, \x_\zeta, \qquad m \in \ssM_{\R}.
  \]
\end{itemize}

For each face $\sigma \in \Sigma$, set $\x_\sigma \coloneqq \prod_{\zeta\in\Sigma_1 \\ \zeta\subface\sigma} \x_{\zeta}$. It can be shown that the homogeneous part $\ssA^k$ is spanned by elements of the form $\x_\sigma$ with $\sigma \in \Sigma_k$.  
In particular, $\ssA^k(\Sigma) = 0$ unless $0 \leq k \leq d$.

Let now $(\Sigma, \weight_\Sigma)$ be a simplicial tropical fan. The tropical weight  $\weight_\Sigma$ induces a map 
\[
\deg\colon \ssA^d \to \R.
\]
Explicitly, for each face $\sigma$ of $\Sigma_d$, we have 
\[
\deg(\x_\sigma) = \frac1{\abs{\det\left(\e_{\rho_1}\dots \e_{\rho_d}\right)}}\weight_\Sigma(\sigma),
\]
where $\left(\e_{\rho_1}\dots \e_{\rho_d}\right)$ is the $d\times d$ matrix with columns the primitive vectors of the rays $\rho_1, \dots, \rho_d$ of $\sigma$, written in a basis of the lattice $\ssN_\sigma$.  

\smallskip

Throughout we use $\x, \y,\z$ for elements of the Chow rings.

\subsection{Mixed differential forms relative to a polyhedral structure}\label{subsec:trop-diff-calculus-2}

We give a brief account of the formalism of mixed differential forms introduced in~\cite{APW1}. This yields tropical semipositive and K\"ahler forms.

\smallskip

Let $X$ be a tropical variety of dimension $d$ with tropical weight function $\weight$. Let 
$\olPi$ be a simplicial polyhedral structure on $X$, with $\Pi$ the induced structure on $\xo$, and let $\Pi_\f \subset \Pi$ be the subcomplex defined by compact faces (simplices).

The differential calculus is defined first relative to $\Pi$. It behaves well under refinements and yields a differential calculus depending only on the support of $X$.
 
For $\delta \in \Pi$, let $\iota_\delta\colon \delta \hookrightarrow X$ denote the inclusion.

\subsubsection{Coefficients} 

For each $\tau\in \Pi$, the star fan $\Pi^\tau$ has dimension $d-\dims\tau$, and is a simplicial tropical fan with the tropical weight function induced from $\weight$. 
Denote by $\ssA^\bul(\tau)$ the Chow ring of the star fan $\Pi^\tau$. This is a graded $\R$-algebra
\[
  \ssA^\bul(\tau) = \ssA^0(\tau) \oplus \dots \oplus \ssA^{d-\dims\tau}(\tau).
\]
In the presentation given in Section~\ref{subsec:chow-ring-p2}, the generator given by a ray $\eta^\tau$ of $\Pi^\tau$ (for $\eta \ssupface \tau$ in $\Pi$), is denoted by $\x_{\eta^\tau}$. For any cone $\sigma^\tau$ in $\Pi^\tau$, we denote by $\x_{\sigma^\tau}$ the corresponding element in $A^{\dims\sigma-\dims\tau}(\tau)$, defined as the product of $\x_{\eta^\tau}$ for the rays $\eta^\tau$ of $\sigma^\tau$.

\smallskip

The map  $\ssA^\bul$, $\tau \mapsto \ssA^\bul(\tau)$, from $\Pi$ to the category of finite-dimensional graded $\R$-algebras,  comes with the following additional structure:

\begin{itemize}
  \item For each $\tau\in \Pi$, we have the degree map
  \[
  \deg_\tau \colon \ssA^{d-\dims\tau}(\tau) \to \R.
  \]

  \item For elements $\tau \subface \sigma$ of $\Pi$, we have the restriction map 
  \[
    \res_{\tau \subface\sigma} \colon \ssA^\bul(\tau) \to \ssA^\bul(\sigma).
  \]
  \item For elements $\tau \subface \sigma$ of $\Pi$, we have the Gysin map 
  \[
    \gys_{\sigma \supface \tau} \colon \ssA^\bul(\sigma) \to \ssA^{\bul+\dims\sigma-\dims\tau}(\tau).
  \]
  \end{itemize}
  We refer to~\cite[\S\,3.1]{APW1} for a list of compatibility properties between these maps.

\noindent Mixed differentials restrict to smooth forms on each $\tau$ taking values in the Chow ring $A^\bul(\tau)$.  

\subsubsection{Tropical conjugation} We first recall the formalism of Lagerberg~\cite{Lag12}. Let $\delta$ be a face of $\Pi$ with tangent space 
 $\ssT_\delta$, and lattice $\ssN_\delta \subset \ssT_\delta$.  Take a second copy $\ssbarT_\delta$ and $\ssbarN_{\delta}$ of $\ssT_\delta$ and $\ssN_{\delta}$, respectively. Set
\[
\ssW_{\delta} \coloneqq \ssT_{\delta} \oplus \ssbarT_{\delta}.
\]

For each positive integer $k$, we have a natural decomposition
 \[ 
 \bigwedge^k\ssW_{\delta} = \bigoplus_{p+q=k} \bigwedge^{p,q}\ssW_{\delta} \quad \textrm{with} \quad \bigwedge^{p,q}\ssW_{\delta} \coloneqq \bigwedge^p \ssT_{\delta} \otimes \bigwedge^q \ssbarT_{\delta}.
 \]

The isomorphism $\ssT_{\delta} \simeq \ssbarT_{\delta}$ given by the identity map induces an involution $\a \mapsto \bar\a$ on $\ssW_{\delta}$
\[
  \begin{pmatrix}
   0 & \id\\
   \id & 0
    \end{pmatrix} 
    \colon \ssT_{\delta} \oplus \ssbarT_{\delta} \rightarrow \ssT_{\delta} \oplus \ssbarT_{\delta}
\]
that we call the \emph{tropical conjugation}. It extends naturally to $\bigwedge^k\ssW_{\delta}$, for all integers $k$, giving the identification $\overline{\bigwedge^{p,q}\ssW_{\delta}} = \bigwedge^{q,p}\ssW_{\delta}$, for all integers $p,q$ with $p+q=k$.

\smallskip
We denote by $\ssW_\delta^{^*} = \ssT_\delta^{^*}\oplus \overline{\ssT_\delta^{^*}}$ the dual of $\ssW_\delta$, and by $\bigwedge^k\ssW_\delta^{^*}$ and $\bigwedge^{p,q}\ssW_\delta^{^*}$ the duals of $\bigwedge^k\ssW_{\delta}$ and of $\bigwedge^{p,q}\ssW_{\delta}$, respectively. We have a decomposition
\[
\bigwedge^{k}\ssW_{\delta}^{^*} = \bigoplus_{p+q=k} \bigwedge^{p.q}\ssW_{\delta}^{^*},\]
and an induced tropical conjugation on $\bigwedge^k\ssW_\delta$ that verifies $\overline{\bigwedge^{p,q}\ssW_{\delta}^{^*}} = \bigwedge^{q,p}\ssW_{\delta}^{^*}$.

If $\v \in \bigwedge^k\ssW_{\delta}$ and $\alpha \in \bigwedge^k\ssW_\delta^{^*}$, we call $k$ the \emph{degree} of $\v$ and $\alpha$, and write $\deg(\v) = \deg(\alpha)=k$.

\medskip

In the following, if $D$ is an endomorphism of a vector space $\ssW$ endowed with the tropical conjugation as above, we denote by $\bar D$ the \emph{tropical conjugate} of $D$ defined as follows: 
\[\bar D(\alpha) \coloneqq \bar{D(\bar \alpha)} \qquad \text{for any element }\alpha \in \ssW.\]

\subsubsection{Dolbeault forms on faces}

For each face $\delta$ in $\Pi$, let $\nC^{^{\,\infty}}_\delta$ be the sheaf of smooth functions on $\delta$. The sheaf of \emph{Dolbeault forms of bidegree $(r,s)$} on $\delta$ is the sheaf
\[
\nC^{^{\,\infty}}_\delta \otimes \bigwedge^{r,s}\ssW_\delta^{^*}.
\]

For an inclusion of faces $\tau \subface \delta$, we have the restriction map of sheaves 
\[ 
\cdot\rest{\tau}\colon \nC^{^{\,\infty}}_\delta \otimes \bigwedge^{r,s}\ssW_\delta^{^*} \longrightarrow \nC^{^{\,\infty}}_\tau \otimes \bigwedge^{r,s}\ssW_\tau^{^*}.
\]

\subsubsection{The sheaf $\ssOmega^{p,q}$ of mixed $(p,q)$-forms relative to $\Pi$} 

Let $\delta$ be an element of $\Pi$. 
Consider the sheaf 
\[
\ssA^\bul(\delta)\otimes \nC^{\,\infty}_\delta \otimes \bigwedge^{r,s}\ssW_\delta^{^*}
\]
that we view as the sheaf of differential $(r,s)$-forms on $\delta$ taking values in the ring $\ssA^\bul(\delta)$. We will glue these sheaves and define the sheaf of mixed differential forms, relative to $\Pi$.

For a triple of nonnegative integers $p,q, k$, consider the sheaf
\[
\ssOmega^{p,q,k}_{\delta} \coloneqq  \ssA^k(\delta) \otimes \nC^{^{\,\infty}}_\delta \otimes \bigwedge^{p-k,q-k}\ssW_\delta^{^*}.
\]
Sections of $\ssOmega^{p,q,k}_{\delta}$ are $(p-k,q-k)$-differential forms which take values in $\ssA^k(\delta)$. Here, we view elements of the Chow ring $A^k$ as contributing by bidegree $(k,k)$, so that the bidegree of the elements in the right hand side is $(p,q)$. In the triple $(p,q,k)$ on the left hand side, $(p,q)$ is bidegree, and $k$ is the part that comes from the Chow ring (cohomology) of the star fan.

The sheaf $\ssOmega^{p,q,0}_{\delta}$ coincides with the sheaf of Dolbeault forms $\nC^{^{\,\infty}}_\delta \otimes \bigwedge^{p,q}\ssW_\delta^{^*}$. For $k> \min(p,q, d-\dims{\delta})$, we get $\ssOmega^{p,q,k}_{\delta} =0$. 

The sheaf $\ssOmega^{p,q}_{\delta}$ of \emph{mixed $(p,q)$-forms on $\delta$} is 
\[
  \ssOmega^{p,q}_{\delta} \coloneqq \bigoplus_{k\geq 0} \ssOmega^{p,q,k}_{\delta}.
\]

By an abuse of notation, we simply write $\ssOmega^{p,q}_{\delta}$ and $\ssOmega^{p,q,k}_{\delta}$ for the pushout of these sheaves to $X$ via the inclusion map $\iota_\delta \colon \delta \hookrightarrow X$.

For an inclusion of faces $\tau \subface \delta$ of $\Pi$, 
the two restriction maps 
\[ 
\res_{\tau \subface \delta} \colon \ssA^k(\tau) \longrightarrow \ssA^k(\delta)\quad \textrm{and} \quad \cdot\rest{\tau}\colon \nC^{^{\,\infty}}_\delta \otimes \bigwedge^{p-k,q-k}\ssW_\delta^{^*} \longrightarrow C^\infty_\tau \otimes \bigwedge^{p-k,q-k}\ssW_\tau^{^*}
\]
yield the following maps
\begin{equation}\label{eq:gluing}
\begin{tikzcd}
  \ssOmega^{p,q,k}_{\delta} = \ssA^k(\delta) \otimes \nC^{^{\,\infty}}_\delta \otimes \bigwedge^{p-k,q-k}\ssW_\delta^{^*}\arrow[r, "\,\id\otimes-\rest\tau\,"] & \ssA^k(\delta) \otimes \CC^\infty_\tau \otimes \bigwedge^{p-k,q-k}\ssW_\tau^{^*}\\
& \ssOmega^{p,q,k}_{\tau} = \ssA^k(\tau) \otimes \nC^{^{\,\infty}}_\tau \otimes \bigwedge^{p-k,q-k}\ssW_\tau^{^*}  \arrow[u, "\res_{\tau\subface\delta}\otimes\id"']
\end{tikzcd}
\end{equation}
We define the sheaf $\ssOmega^{p,q,k}$ on $X$ as the gluing of the sheaves $\ssOmega^{p,q,k}_{\delta}$, $\delta \in \Pi$, 
via the diagrams in~\eqref{eq:gluing}. 
More precisely, the sheaf $\ssOmega^{p,q,k}$ of \emph{mixed $(p-k,q-k)$-forms with values in $A^k$} on $X$, relative to $\Pi$, is 
 the kernel of the map of sheaves on $X$
\begin{equation}\label{eq:gluing2}
\begin{tikzcd}[column sep=6cm]
\bigoplus\limits_{\delta\in \Pi} \ssOmega^{p,q,k}_{\delta} \arrow[r, "\bigoplus_{\tau\subface \delta}\Bigl(\,\bigl(\id\otimes-\rest\tau\bigr) - \bigl(\res_{\tau\subface\delta}\otimes\id\,\bigr)\Bigr)"] & \bigoplus\limits_{\substack{\tau,\delta\in \Pi \\ \tau\subface\delta}} \ssA^k(\delta) \otimes \CC^\infty_\tau \otimes \bigwedge^{p-k,q-k}\ssW_\tau^{^*}.
\end{tikzcd}
\end{equation}

Explicitly, a section of $\ssOmega^{p,q,k}$ is given as follows.  
On an open subset $U \subseteq X$,  a section $\a$ in $\ssOmega^{p,q,k}(U)$ is a collection of mixed forms 
\[
\a_\delta \in \ssA^k(\delta) \otimes C^\infty_\delta(U\cap\delta) \otimes \bigwedge^{p-k,q-k}\ssW_\delta^{^*}, \qquad \delta \in \Pi,
\]
 with compatibility relations
\[ 
\a_\delta\rest\tau = \res_{\tau \subface\delta}(\a_\tau) \quad \text{in}\quad  \ssA^k(\delta) \otimes C^\infty_\tau(U\cap\tau) \otimes \bigwedge^{p-k,q-k}\ssW_\tau^{^*} \quad \text{ for each }\tau \subface \delta.
\]

We define \emph{the sheaf of mixed $(p,q)$-forms on $X$, relative to $\Pi$,} as the direct sum
\begin{equation} \label{eq:decomp_toric_weight}
\ssOmega^{p,q} = \bigoplus_{k\geq 0}\ssOmega^{p,q,k}.
\end{equation}
We naturally define a wedge product on the sheaves $\ssOmega^{\bul,\bul}$, see~\cite[\S\,3]{APW1}.
%cross-ref to APW1: \ref{sec:wedge-product-p1}

Note that $\ssOmega^{0,0}$ is the sheaf of continuous functions on $\xo$ that are smooth on each face $\delta\in\Pi$. When necessary, we write $\ssOmega_\Pi$ to emphasize the dependency on $\Pi$.

\subsubsection{Tropical conjugation on mixed forms} 

We extend the tropical conjugation to $\ssOmega_\delta^{\bul,\bul}$, $\delta\in \Pi$, and  $\ssOmega^{\bul,\bul}$ by setting
\[
\bar{\z\otimes\alpha} \coloneqq (-1)^k \z\otimes\bar{\alpha}
\]
for each $\delta\in \Pi$, $\z\in \ssA^k(\delta)$, and $\alpha$ a section of $ \nC^{^{\,\infty}}_\delta \otimes \bigwedge^{p-k,q-k}\ssW_\delta^{^*}$, for nonnegative integers $k,r,s$. Note that the tropical conjugation only acts on $\bigwedge^{\bul,\bul}\ssW_\delta^{^*}$ (elements of the Chow ring are real).

\subsection{Mixed differential calculus} \label{subsec:dd-p2} 
We explain the construction of the two operators $\dm$ and $\dmb$ of bidegrees $(1,0)$ and $(0,1)$, respectively, acting on mixed differential forms.

\smallskip

Let $\delta$ be a face of $\Pi$.
 We have the two operators $\d_\delta$ and $\dhb_\delta$ acting on Dolbeault forms $\nC^{^{\,\infty}}_\delta \otimes \bigwedge^{\bul,\bul}\ssW_\delta^{^*}$. Namely,  
\begin{equation*}
_\delta \colon \nC^{^{\,\infty}}_\delta \otimes \bigwedge^{p,0}\ssW_\delta^{^*} = \nC^{^{\,\infty}}_\delta \otimes \bigwedge^{p}\ssT_\delta^* \longrightarrow \nC^{^{\,\infty}}_\delta \otimes \bigwedge^{p+1,0}\ssW_\delta^{^*} = \nC^{^{\,\infty}}_\delta \otimes \bigwedge^{p+1}\ssT_\delta^{^*}
\end{equation*}
is the usual differential on smooth forms on $\delta$. Tensoring by $\bigwedge^q \overline{\ssT_\delta^{^*}}$ on both sides, this yields 
\[
\dh_\delta\colon \nC^{^{\,\infty}}_\delta \otimes \bigwedge^{p,q}\ssW_\delta^{^*} \longrightarrow \nC^{^{\,\infty}}_\delta \otimes \bigwedge^{p+1,q}\ssW_\delta^{^*}.
\]
 We similarly define 
 \[
 \dhb_\delta\colon \nC^{^{\,\infty}}_\delta \otimes \bigwedge^{p,q}\ssW_\delta^{^*} \longrightarrow \nC^{^{\,\infty}}_\delta \otimes \bigwedge^{p,q+1}\ssW_\delta^{^*}.
 \]
 Note that $\dhb_\delta$ is the tropical conjugation of $\dh_\delta$, that is, we have 
\[ \dhb_\delta\alpha = \bar{(\dh_\delta\bar\alpha)}. \]

If $\v\in \bigwedge^{p,q}\ssW_\delta$ is a $(p,q)$-multivector, we denote by $\imath_\v$ 
the contraction of Dolbeault forms on $\delta$ by $\v$, that is,  
\[
\imath_\v\colon \nC^{^{\,\infty}}_\delta \otimes \bigwedge^{\bul,\bul}\ssW_\delta^{^*} \to \nC^{^{\,\infty}}_\delta \otimes \bigwedge^{\bul-p,\bul-q}\ssW_\delta^{^*}, \qquad \alpha \longmapsto \imath_\v(\alpha) = \alpha(\v\wedge{}\cdot\,).
\]
By our convention on the tropical conjugation of endomorphisms, we have $\bar\imath_\v = \imath_{\bar\v}$.

\smallskip

We naturally extend the operators $\d_\delta$, $\bar \dh_\delta$, and $\imath_\v(\alpha)$, defined above, to the mixed sheaves $\ssOmega^{\bul,\bul,\bul}_{\delta}$ and $\ssOmega^{\bul,\bul}_\delta$. In this context, we have a Leibniz and a commutator rule, see~\cite[\S\,3.7]{APW1}. 
%cross-ref to APW1: \ref{sec:dd} and {subsec:mixed-differential-operators}

\subsubsection{Definition of $\dm$ and $\dmb$} We define the mixed differential operators
\[
\dm \colon \ssOmega^{\bul, \bul} \to \ssOmega^{\bul+1, \bul} \quad \textrm{and} \quad \dmb \colon \ssOmega^{\bul, \bul} \to \ssOmega^{\bul, \bul+1}
\]
by taking into account the contributions of $\d_\bul$ and $\dhb_\bul$ on boundary faces, as follows.

The definition is local, so taking a chart, we can assume without loss of generality
 that $\Pi$ is embedded in  $\R^n$.
For each $\delta \in \Pi$, choose an arbitrary complement $\NN^\delta$ to $\ssN_{\delta}$ in $\Z^n$ so that
\[
\NN^\delta \oplus \ssN_{\delta} =\Z^n.
\]
For each pair $\eta\ssupface\delta$, we pick $\nvect_{\eta/\delta}$ to be the primitive vector in $\ssN_{\eta}\cap\NN^\delta$ pointing to $\eta$.

By definition, $\ssOmega^{\bul, \bul}$ is a subsheaf of the sheaf $\bigoplus_{\delta\in \Pi}\ssOmega_\delta^{\bul,\bul}$ on $X$.
We define first 
\[
\dm\colon \bigoplus_{\delta\in \Pi}\ssOmega_\delta^{\bul,\bul} \longrightarrow \bigoplus_{\delta\in \Pi}\ssOmega_\delta^{\bul+1,\bul}
\]
as the sum 
\[
\dm \coloneqq \dm_{\parr} + \dm_{\perp}
\]
of \emph{tangential} $\dm_{\parr}$ and \emph{transversal} $\dm_{\perp}$ 
parts. The tangential part $\dm_{\parr}$ is given simply by the $\d_\delta$
\[
  \dm_{\parr}\coloneqq \sum_{\delta\in\Pi}\d_\delta \colon \bigoplus_{\delta\in \Pi}\ssOmega_\delta^{\bul,\bul} \longrightarrow \bigoplus_{\delta\in \Pi}\ssOmega_\delta^{\bul+1,\bul}.
\]
The transversal part 
\[
 \dm_{\perp}\colon \bigoplus_{\delta\in \Pi} \ssOmega_\delta^{\bul,\bul} \longrightarrow \bigoplus_{\delta\in \Pi} \ssOmega_\delta^{\bul+1,\bul}
\]
absorbs the boundary contribution (in the distributional sense) of $\dm_{\parr}$ at each face of codimension one by transferring the contraction along the corresponding primitive vector into the Chow ring via the Gysin map.

Formally, it is defined by sending a section $\z\otimes\alpha$ of $\ssOmega_{\delta}^{\bul,\bul}$ to the section $\dm_{\perp}(\z\otimes\alpha)$ of $\bigoplus_{\gamma \ssubface \delta}\ssOmega_\gamma^{\bul+1,\bul}$ given by 
 \[ 
 \begin{array}{rrcl}
    \dm_{\perp}(\z\otimes\alpha) \coloneqq   \sum_{\gamma\ssubface\delta} \gys_{\delta\ssupface\gamma}(\z)\otimes\bar\imath_{\nvect_{\delta/\gamma}}(\alpha)\rest\gamma.
 \end{array} \qedhere 
 \]
 Note that for $\z$ of degree $k$, $\gys_{\delta\ssupface\gamma}(\z)$ has degree $k+1$. Also, if $\alpha$ has bidegree $(p-k, q-k)$, then $\bar\imath_{\nvect_{\delta/\gamma}}(\alpha)\rest\gamma$ has bidegree $(p-k, q-k-1)$. Therefore, 
 \[
 \dm_{\perp} \colon \bigoplus_{\delta\in \Pi}\ssOmega_\delta^{p,q,k} \longrightarrow \bigoplus_{\gamma\in \Pi}\ssOmega_\gamma^{p+1,q,k+1}.
 \]
We define $\dmb$ as the tropical conjugate of $\dm$. We proved in~\cite[Thm.\,3.4]{APW1} that two operators are independent of the choices of the subspace complements $\NN^\delta$ for $\delta\in \Pi$. Moreover, they preserve the subsheaves of mixed differentials, that is, they induce differential operators 
%cross-ref to APW1: \ref{thm:diff-well-defined-p1}
\[
  \dm_{\parr}  \colon \ssOmega^{\bul,\bul} \to \ssOmega^{\bul+1,\bul} \quad \textrm{,} \quad \dm_{\perp}  \colon \ssOmega^{\bul,\bul} \to \ssOmega^{\bul+1,\bul} \quad \textrm{and} \quad \dm = \dm_\parr+\dm_\perp \colon \ssOmega^{\bul,\bul} \to \ssOmega^{\bul+1,\bul}.
  \]
We also proved in~\cite{APW1} that 
\[
  \dm\circ\dm = 0, \quad \dmb\circ\dmb = 0, \quad \dm\circ\dmb = \dmb\circ\dm.
\]

\subsection{The degree map}  \label{subsec:degree_map}  
We can define a \emph{degree map} by integrating top-degree mixed forms. 
Namely, the integration map $\int$ on $\ssOmega^{d,d}$ 
is the map 
\[
\int\colon \quad \bigoplus_{\delta\in \Pi}\ssOmega_\delta^{d,d} \to \R\cup\{\infty\}
\]
defined for sections on bounded open sets, or globally for nonnegative $(d,d)$ forms, as follows.

Let $U\subset X$ be an open set. For a face $\delta\in \Pi$ of dimension $k$, a mixed $(d,d)$ form on $U\cap\delta$ is a linear combination of elements of the form $\z \otimes \alpha$ for  $\z \in A^{d-k}(\delta)$ and $\alpha$ a section in $\nC^{^{\,\infty}}_\delta(U\cap \delta) \otimes \bigwedge^{d-k,d-k}\ssW_\delta^{^*}$. This is because $\ssOmega_\delta^{d,d,j} =0$ if $j\neq \dims{\delta}$.
 
  Then, we set 
\[ 
\int_U \z \otimes \alpha \coloneqq (-1)^{\binom{k+1}{2}}\deg_{\delta}(\z) \cdot \int_{U\cap \delta} \bar\imath_{\nu_\delta}(\alpha),
\]
and extend it by linearity to all of $\ssOmega_{\delta}^{d,d}$. Here, $\nu_\delta$ refers to a canonical multivector of $\delta$ defined as $\nvect_1\wedge \dots \nvect_{k}$ for a basis $\nvect_1, \dots, \nvect_k$ of the lattice $\ssN_\delta$. Note that there are two choices for $\nu_\delta$ differing by a sign. This choice defines an orientation of $\delta$, and the integration of the $d$-form $\bar\imath_{\nu_\delta}(\alpha)$ on $\delta$ is defined relative to this orientation. Therefore, the definition of $\int$ is independent of the choice of canonical multivector.

The chosen of sign in the above expression guarantees that the degree of the $(d,d)$ form $\omega^d$ for a semipositive mixed form $\omega\in\Omega^{1,1}$ is nonnegative.

The degree map is by definition the integration of top-degree mixed forms on $X$ (for those that are either semipositive or compactly supported). We use both the notation $\deg$ and $\int_X$ for the degree map, so $\deg(\omega^d)=\int_X\omega^d$.
\subsection{The sheaf of mixed differential forms $\ssOmega_{X}^{\bul, \bul}$} 

For a refinement $\olPi'$ of $\olPi$, with $\Pi'$ the corresponding refinement of $\Pi$, we get a natural injection 
\[
\Phi\colon \ssOmega^{\bul,\bul}_{\Pi} \hookrightarrow \ssOmega_{\Pi'}^{\bul,\bul},
\]
where $\ssOmega^{\bul,\bul}_{\Pi}$ and $\ssOmega^{\bul,\bul}_{\Pi}$ are mixed differential forms relative to $\Pi$ and $\Pi'$, respectively. 

The inclusion $\Phi$ commutes with the inclusions of $\ssOmega^{0,0}_{\Pi}$ and $\ssOmega_{\Pi'}^{0,0}$ into $\nC^0(X)$, with the tropical conjugation, with the wedge product, with the differentials $\dm$ and $\dmb$, and the degree map.

For each pair of nonnegative integers $p,q$, we define 
\[\ssOmega^{p,q}_{X} \coloneqq \varinjlim_{\Pi} \ssOmega^{p,q}_{\Pi}\]
as $\Pi$ varies over all simplicial polyedhral structures on $\xo$ induced by a simplicial polyhedral structrues $\olPi$ on $X$. We infer from the abvoe discussion that the operations  of tropical conjugation, wedge product, $\dm$ and $\dmb$, and the degree map all 
extend to $\ssOmega_X^{p,q}$. 

The sheaf $\ssOmega_X^{0,0}$ coincides with the sheaf $\PSm$  of piecewise smooth functions on $\xo$. Here, a continuous function $f\colon \xo\to \R$ 
is called piecewise smooth if there exists a polyhedral structure $\Pi$ for $\xo$ 
such that the restriction of $f$ to each face $P$ of $\Pi$ is smooth. 

\smallskip

As we showed in~\cite[\S\,3]{APW1}, the sheaf of tropical Dolbeault forms on $X$, defined in \cite{JSS19, GK17, CLD}, as well as $\delta$-forms defined in~\cite{GK17,Mih24,BGK25} inject into $\ssOmega_X^{\bul,\bul}$.

\subsection{Integration-by-parts formula} For future use, we note the following consequence of Stokes' theorem.
We say that a subset $S\subset \xo$ has polyhedral boundary if there exists a polyhedral structure $\Pi$ on $\xo$ such that the boundary of $S$ is a polyhedral subcomplex of $\Pi$. We recall the integration-by-parts formula established in \cite{APW1}, which is used in establishing Bedford--Taylor theory.
\begin{prop} [{\cite[Corollary 3.19]{APW1}}] \label{Green_formula}
Let $U\subset \xo$ be an open subset. Let $\zeta \in \ssOmega^{d-1,d-1} (U)$ be a mixed form such that $\dm\zeta = \bar\dm\zeta = 0$. Let $\varphi,\psi \in \PSm(U)$, such that $\Supp(\varphi) \Subset U$. Then, 
  \[ \int_U(\varphi \cdot \dm\bar\dm\psi \wedge \zeta) = -\int_U(\dm\varphi \wedge \bar\dm\psi \wedge \zeta) = \int_U(\psi \cdot \dm\bar\dm\varphi \wedge \zeta). 
  \]
\end{prop}

\smallskip

In the following, $\dm\dmb$ plays a role analogous to the operator $dd^c$ in complex geometry.

%%%%%%%%%Locally convex functions
%%%%%%%%%%%%%%%%%%%%%%%%%%%%%%%%%
\section{Locally convex functions} \label{sec:tropical_convexity}
Throughout this section, we fix a tropical variety $X$ and 
denote by $\xo$ its open part. 
Let $U \subset X$ be an open subset. We begin by defining the class of locally 
convex functions on $U$, which will be the central object of study in this paper.
\begin{definition} \label{def:locally_convex_paper_2}
  Let $u \colon U \to \BR$.
  \begin{enumerate}
    \item We say that $u$ \emph{admits a local affine support} at a point $x \in U$ if there exist a neighborhood $U'$ of $x$ contained in $U$, and an affine function $l \in \Aff(U')$ such that
      \[
      u(x) = l(x) \quad \text{and} \quad u(y) \ge l(y) \quad \text{for all } y \in U'.
      \]
    In this case, we say that $l$ is an affine support function (of $u$ at $x$).
    \item We say that $u$ is \emph{locally convex at} $x$ if $u$ is continuous at $x$, and admits a local affine support at $x$. We say that $u$ is \emph{locally convex on} $U$ if $u$ is locally convex at each point $x \in U$.  We denote by $\LConv(U)$ the set of locally convex functions on $U$. 
  \end{enumerate}
\end{definition}

\begin{remark}
  It is not hard to see that $\LConv(U)$ is a convex cone: For $u_1, u_2 \in \LConv(U)$ and $\lambda_1, \lambda_2\geq 0$, $\lambda_1 u_1+\lambda_2 u_2 \in \LConv(U)$. Moreover, $\LConv$ defines a sheaf on $X$.
\end{remark}

The rest of this section is devoted to fundamental properties of locally convex functions.

\subsection{Locally convex functions on the open part $\xo$} \label{subsec:locally_convex_functions_open_part} 
In this subsection, we assume that $U$ is an open subset of $\xo$, and unless otherwise stated, $u \in \LConv(U)$. 
Recall that for $x\in \xo$ and  a polyhedral structure $\olPi$ on $X$, with induced polyhedral structure $\Pi$ on $\xo$, we denote $\Pi^{(x)}=\bigcup_{x \in \sigma, \sigma \in \Pi} \relint(\sigma)$. Note that $\Pi^{(x)}\subset \xo$ is open, and an affine support function $l$ of $u$ at $x$ can be extended facewise linearly to $\Pi^{(x)}$. 

Let $x, y \in \xo$ such that there exists a polyhedral structure $\Pi$ on $\xo$ and $\sigma \in \Pi$ with $x, y \in \sigma$. Denote by $L_{x,y} \coloneqq \{tx + (1-t)y\ |\ t \in [0,1]\} \subset \sigma  \subset \xo$ be the line segment connecting $x$ and $y$. We remark that the definition of $L_{x,y}$ as a subset of $\xo$ does not depend on the choices of $\Pi$ and $\sigma$. We also remark that if $y \in \Pi^{(x)}$, then $L_{x,y} \subset \Pi^{(x)}$. 
\begin{prop} \label{prop:properties_LConv_functions}
  \begin{enumerate}
    \item \label{prop:conv1} Suppose that $\Omega$ is a convex open subset in $\R^n$. Then $u$ is locally convex on $\Omega$ if and only if $u$ is convex in the usual sense.    
    \item \label{prop:conv2} Let $U \subset \xo$ be an open subset, and let $\Pi$ be a polyhedral structure on $\xo$. Let $x\in U$ and $l$ be an affine support function of $u$ at $x$.  Assume that the line segment $L_{x,y} \subset U \cap \Pi^{(x)}$ for each $y \in U \cap \Pi^{(x)}$. Then, we have $u \ge l$ on $U \cap \Pi^{(x)}$. In particular,   $u \ge l$ on $ \Pi^{(x)}$ when $\Pi^{(x)} \subset U$.
  \end{enumerate}
\end{prop}
\begin{proof}
  For statement \ref{prop:conv1}, it is a standard result that $u$ is convex implies $u$ is locally convex. 
  For the inverse direction, we claim that for any $g \in C^0([0,1])$, if there exists an affine function $l$ and $\epsilon >0$ such that $g(0)=l(0)$ and $g \ge l$ on $[0, \epsilon)$, then $g \ge l$ on $[0,1]$.  Since a function is convex on $\Omega$ if and only if its restriction to every line segment in $\Omega$ is convex, it is not hard to see that this claim concludes statement \ref{prop:conv1}. 

  Now we prove this claim. By adding an affine function to $g$, we may assume that $g(0)=0$ and $l=0$.  For the sake of a contradiction, suppose there exists $t \in (0,1]$ such that $g(t)<0$. Let $M=\sup \{g(s)\ |\ 0 \le s \le t\}$. Since $g(0)=0$, we have $M \ge 0$. 
  Let $t_0 \coloneqq \sup \{s  \in [0,t] \ |\ g(s) =M \}$. Since $g\rest{[0,\epsilon)}\geq 0$, we have $t_0>0$. 
  Since $g$ is continuous, we have $g(t_0)=M$, and hence $t_0 <t$. 
  By assumption, there exists an affine function $l_0(s) \coloneqq a(s-t_0)+M$ such that $l_0 \le g$ for $s\in (t_0-c, t_0+c)$ for some $c>0$. If $a <0$, we have $l_0 > M \ge g$ on $(t_0-c,t_0)$ yielding a contradiction. 
  If $a \ge 0$, then $l_0 \ge M$ on $(t_0, t_0+c)$. At the same time, $g < M$ on $(t_0, t_0+c)$ by assumption. This is again a contradiction. Therefore, we infer that $g \ge l$ on $[0,1]$, and statement \ref{prop:conv1} is proved.

  For statement \ref{prop:conv2}, let $y \in \Pi^{(x)}$. By assumption, the line segment $L_{x,y}$ connecting $x$ and $y$ are contained in $U$. Then, statement \ref{prop:conv2} is concluded by applying the claim to $u|_{L_{x,y}} \in C^0(L_{x,y})$.
\end{proof}

\begin{remark}
  In Definition~\ref{def:locally_convex_paper_2} of locally convex functions, the assumption that $u$ is continuous is necessary. 
  Let $u \colon (-1,1) \to \R$ be the function defined by $u(0)=0$, and $u(x)=1$ for $x \neq 0$. Then, $u$ admits a local affine support at each point in $(-1,1)$, but $u$ is not a convex function in the usual sense.
\end{remark}

In~\cite{BBS22}, the authors introduced a class of convex functions on tropical varieties without infinity boundaries, that satisfy $\xo=X$. Their definition is based on the notion of \emph{polyhedral convex combination}.  
We next show in Proposition~\ref{prop:equivalent_def_LConv} that their class of convex functions is equivalent to our notion of locally convex function in $\xo$.  

Let $U \subset \ssN_\R$ by a local polyhedral chart of $\xo$. A \emph{convex combination} in $U$ is a triple
\[
\Bigl(x, (x_i)_{i \in I}, (\nu_i)_{i \in I}\Bigr)
\]
consisting of points $x$ and $x_i$ in $U$ for $i$ in a finite set $I$, and nonnegative real numbers 
$\nu_i$, $i\in I$, satisfying 
\[
\sum_{i \in I} \nu_i = 1 \qquad \text{and} \qquad \sum_{i \in I} \nu_i\, x_i = x \quad \text{in $\ssN_\R$.}
\]

A convex combination in $U$, as above, is called \emph{polyhedral} if there exist a polyhedral structure $\Pi$ on $\xo$ and polyhedra $\tau, \sigma_i \in \Pi$
for $i \in I$ with $\sigma_i \supfaceq \tau$ such that $x \in \tau, x_i \in \sigma_i$, $i \in I$.

We need the following lemma.  For a convex subset $\Omega \subset \R^n$, denote by $\Conv(\Omega)$ the set of all convex functions $u \colon \Omega \to \R$. 

\begin{lemma} \label{characterization_convex_envelope}
  Let $\Omega \subseteq \R^n$ be a convex domain, and let $f \colon \Omega \to \R \cup \{+\infty\}$ be an arbitrary function.  Define the function $P(f) \colon \Omega \to \R\cup\{+\infty\}$ by
\[
P(f)(x) \coloneqq \inf \left\{\, \sum_{i \in I} \nu_i f(x_i) \, \st \,
\Bigl(x, (x_i)_{i \in I}, (\nu_i)_{i \in I}\Bigr) \text{ is a convex combination in } \Omega \right\}, 
\quad \forall\, x \in \Omega.
\]
Then, we have 
\[
P(f) = \sup \Bigl\{\, u \in \Conv(\Omega) \,\st\, u \le f \Bigr\}.
\]
In particular, $P(f) \in \Conv(\Omega)$.
\end{lemma}
\begin{proof}
  It is clear that $P(f) \le f$ by taking convex the combination $\bigl(x,(x), (1)\bigr)$ (for $I$ a singleton).  For any $u \in \Conv(\Omega)$ with $u \le f$, since $u(x) \le \sum_{i \in I} \nu_i u(x_i)$ for any convex combination  $\bigl(x, (x_i)_{i \in I}, (\nu_i)_{i \in I}\bigr)$ in $\Omega$, we have $u \le P(f)$.  Therefore, it suffices to prove that $P(f) \in \Conv(\Omega)$.  

   Let $y_1, y_2 \in \Omega$ and $0 \le t \le 1$. For any $\epsilon>0$, by definition we may choose convex combinations $\bigl(y_1, (x_{1,i})_{i \in I}, (\nu_{1,i})_{i \in I}\bigr)$ and $\bigl(y_2, (x_{2,i})_{i \in I}, (\nu_{2,i})_{i \in I}\bigr)$, such that 
  \[
  \sum_{i \in I} \nu_{j,i} f(x_{j,i}) \le P(f)(y_j) +\frac{\epsilon}{2} ,\ j\in \{ 1,2\}
  \]
  Then,
  \[
   P(ty_1+(1-t)y_2)  \le  \sum_{i \in I} t\nu_{1,i} f(x_{1,i}) + \sum_{i \in I} (1-t)\nu_{2,i} f(x_{2,i})
    \le  tP(f)(y_1)+(1-t)P(f)(y_2) + \epsilon
  \]
  Here the first inequality follows from the observation that $\bigl(x, (x_{1,i}, x_{2,i})_{i \in I}, (t\nu_{1,i}, (1-t)\nu_{2, i})_{i\in I}\bigr)$ is a convex combination.  We conclude the proof of the lemma by letting $\epsilon \to 0$.
\end{proof}

\begin{prop} \label{prop:equivalent_def_LConv}
  Let $U\subset N_{\R}$ be a local polyhedral chart of $\xo$, and let $u \colon U \to \R$. Suppose that there exists a convex open subset $\widetilde{\Omega} \subset N_{\R}$ such that $U=\widetilde{\Omega} \cap \xo$.  Then $u \in \LConv(U)$ if and only if, for any polyhedral convex combination $\bigl(x, (x_i)_{i \in I}, (\nu_i)_{i \in I}\bigr)$ in $U$, we have
\begin{equation} \label{inequality_polyhedral_convex_combination}
u(x) \le \sum_{i \in I} \nu_i \, u(x_i).
\end{equation}
\end{prop}
\begin{proof}
  Suppose that $u \in \LConv(U)$.  Let 
  $\bigl(x, (x_i)_{i \in I}, (\nu_i)_{i \in I}\bigr)$ be a polyhedral convex combination 
  for a polyhedral structure $\Pi$ on $X$, with $x \in \tau$, $x_i \in \sigma_i$, 
  and $\tau, \sigma_i \in \Pi$ for $i \in I$ with $\sigma_i \supfaceq \tau$.  Let $l$ be a local affine support function of $u$ at $x$. Note that $\sigma \cap U=\sigma \cap \widetilde{\Omega}$ is a convex subset of $N_{\R}$ for every $\sigma \in \Pi$. Thus, for any $y \in \Pi^{(x)}$, the line segment $L_{x, y}$ connecting $x$ and $y$ is contained in $\sigma \cap U$, where $\sigma \in \Pi$ satisfies $x, y \in \sigma$.   Then,  
  Proposition~\ref{prop:properties_LConv_functions}~\ref{prop:conv2} implies that 
  $u \ge l$ on $\Pi^{(x)}$. In particular, we have $u(x_i) \ge l(x_i)$ for each $i \in I$, and  inequality~\eqref{inequality_polyhedral_convex_combination} follows.

  Conversely, assume that the inequality~\eqref{inequality_polyhedral_convex_combination} 
  holds for any polyhedral convex combination.  We show that $u \in \LConv(U)$. By \cite[Thm.\,6.2]{BBS22}, we have $u \in C^0(U)$. It remains to prove 
  that $u$ admits a local affine support at each point $x$. Since this is a local statement, we may reduce to the case where $x$ is the origin 
  of a tropical fanfold $Z$, and $U$ is an open neighborhood of $0$ in $Z$.  We need to prove that $u$ admits a local affine support at $0$.

  By restricting to a linear subspace, we may assume that $\mathrm{span}_{\R}(U) =\ssN_{\R}$. Take a convex domain $\Omega_0$ in $\ssN_{\R}$ such that $\Omega_0 \cap Z \subset U$. 
  The balancing condition of $Z$ implies that the convex hull $\Omega$ of $\Omega_0 \cap Z$ is an open neighborhood of $0$. 
  We extend $u$ to a function $\widetilde{u}$ on $\Omega$ by taking $+\infty$ on $\Omega \setminus U$. Take the convex envelop $P(\widetilde{u}) \in \Conv(\Omega)$ as in Lemma \ref{characterization_convex_envelope}. By definition, $P(\widetilde{u})(0)=u(0)$. Since $u$ takes finite value on $\Omega_0 \cap Z$, we have $P(\widetilde{u}) < +\infty$ on $\Omega$. Therefore, we may take an affine support function $l \in \Aff(\Omega)$ of $P(\widetilde{u})$ at $0$. Then, we have $l \le P(\widetilde{u}) \le u$ on $\Omega \cap Z$, and $l(0)=P(\widetilde{u})(0)=u(0)$. Hence, $l$ is also an affine support function of $u$ on $\Omega \cap Z$, as required. 
\end{proof}

We get the following consequences on pointwise limits and supremums in $\LConv(U)$.
\begin{prop} \label{prop:limit_of_lconv_is_lconv}
    Let $U$ be an open subset in $\xo$, and let $\{ u_{i} \}_{i\ge 1}$ be a sequence of locally convex functions on $U$. Suppose that $u_i$ converges pointwise to a function $u\colon U \to \R$. Then $u \in \LConv(U)$.   
\end{prop}
\begin{proof}
  Since $u \in \LConv(U)$ is a local condition, we may assume that $U$ is a local polyhedral chart of $\xo$ satisfying the condition in Proposition \ref{prop:equivalent_def_LConv}. It is easy to see that  inequality  \eqref{inequality_polyhedral_convex_combination} for $u$ holds for any polyhedral convex combination on $U$. Therefore by Proposition \ref{prop:equivalent_def_LConv}, $u \in \LConv(U)$.  
\end{proof}

\begin{prop} \label{prop:sup_of_lconv_is_lconv}
  Let $U$ be an open subset in $\xo$, and let $\{ u_{i} \}_{i \in I} $ be a collection of locally convex functions on $U$. Suppose that $\sup_{i \in I} u_i(x)< +\infty$ for every $x \in U$. Then $\sup_{i \in I} u_i$ is a locally convex function on $U$.  
\end{prop}
\begin{proof}
  Similar to the proof of Proposition \ref{prop:limit_of_lconv_is_lconv}.
\end{proof}

\medskip

Now, we discuss the Lipschitz property of locally convex functions. Let $(Z,d)$ be a metric space and $f \colon Z \to \mathbb{R}$ a function on it. 
Following the convention in \cite{BBS22}, we define the \emph{Lipschitz constant} of $f$ on $Z$ with respect to $d$ as
\[
    \Lip_{(Z,d)}(f) = \sup_{\substack{x,y \in Z \\ x \neq y}} \frac{|f(x) - f(y)|}{d(x,y)} \in \mathbb{R}_{\geq 0} \cup \{+\infty\},
\]
with $\Lip_{(Z,d)}(f) = 0$ if $Z$ consists of a single point. A function $f$ is called \emph{Lipschitz continuous} on $Z$ with respect to $d$ if $\Lip_{(Z,d)}(f) < +\infty$.

Combining Proposition \ref{prop:equivalent_def_LConv} with \cite[Thm.\,6.18]{BBS22}\footnote{The initial idea of the proof of this theorem comes from \cite{BFJ14}.}, we get the following.

\begin{prop} \label{prop:Lipschitz_bounds_in_terms_of_C0_bounds}
  Let $U \subset N_{\R}$ be a local polyhedral chart of $\xo$.  Let $d$ the restriction to $U$ of an Euclidean metric on $N_{\R}$.  Let $K$ be a compact subset of $U$. Then, there exists a constant $C>0$ depending only on $U$ and $K$ such that, for any $u \in \LConv(U)$, we have $\Lip_{(K,d)}(u) \le C \left\|u\right\|_{C^0(U)}$. In particular, for any $u \in \LConv(U)$, $u$ is Lipschitz on $K$.
\end{prop}

\subsection{Locally convex functions around the boundary}
In this subsection, we assume that $U$ is an open subset of $X$, which may intersect the boundary strata of $X$. The main result of this subsection is the following.
\begin{prop} \label{prop:local_affine_support_function_boundary}
  Let $U$ be an open subset of $X$.
  \begin{enumerate}
    \item Let $u_o \colon U \cap \xo \to \R$ be a bounded locally convex function. Then, $u_o$ can be extended uniquely to a continuous function $u \in C^0(U)$. Furthermore, we have $u \in \LConv(U)$. 
    \item Suppose that $u_o\colon U \cap \xo \to \R$ is a locally convex function which is locally bounded from above. Then, $u_o$ can be extended uniquely to a continuous map $u\colon U \to \R \cup \{-\infty\}$. 
  \end{enumerate}
\end{prop}
Here, $\R \cup \{-\infty\}$ is given the usual extended topology.
 We need the following lemma. 
\begin{lemma} \label{lem:extension_bounded_convex_function_polyhedron}
 Let $k$ be a positive integer. Suppose that $u \colon \R_{+}^k \to \R$ is a convex function bounded from above. Then, 
    \begin{enumerate}
      \item If $x$ and $y$ are two points in $\R_+^k$ satisfying $x \ge y$, then $u(x) \le u(y)$.
      \item There exists a constant $c \in \R \cup\{-\infty\}$ such that for every sequence of points $\{x_i\}_{i=1}^{\infty}$  in $\R_+^k$ converging to  $\underline{\infty}=(\infty, \cdots , \infty) \in \BT_+^k=(0,\infty]^k$, we have 
      \[\lim_{i \to +\infty} u(x_i) \to c
      \]
    \end{enumerate}
  \end{lemma}
  \begin{proof} We treat first the case $k=1$. We claim that $u(x) \le u(y)$ whenever $x > y$. In fact, otherwise, the convexity of $u$ would imply that $u(t) \ge \frac{(t-y)}{(x-y)} (u(x)-u(y))$ for any $t \ge x$, in which case, we would obtain $u(t ) \to +\infty$ as $ t\to \infty$. This would contradict the assumption that $u$ is bounded from above. Therefore, the claim is proved, and we have 
    \[
      \lim_{t \to +\infty} u(t) =  \inf_{x \in \R_{\ge 0}} u(x),
    \]
  Thus, both statements in (1) and (2) hold when $k=1$.

    Now assume that $k \ge 2$. Let $x$ and $y \in \R_{+}^k$ be distinct points satisfying $x\geq y$.   The ray $\rho \coloneqq \{\, y + t(x - y) \mid t \ge 0 \,\}$ is contained in $\R_{+}^k$. The restriction $u\rest{\rho}$ yields a convex function $\varphi \colon \R_{\geq 0} \to \R$, which is bounded from above. We thus obtain $u(x) = \varphi(1) \le u(y) = \varphi(0)$, proving part~(1) of the lemma.
    
    It remains to prove part~(2). Set $c\coloneqq \inf_{x \in \R_+^k} u(x)$. Let $\{x_i\}_{i \ge 1} \subset \R_{+}^k$ be a sequence of points converging to $\underline\infty=(\infty, \dots, \infty)$. 
    Suppose instead that $\lim_{i \to +\infty} u(x_i) = c' > c$.  Then, there exists $y \in \R_{\ge 0}^k$ such that $u(y) < c'$. Observe that $x_i \geq y$ for large enough $i$.  By part~(1), however, this implies $u(x_i) \le u(y)$ for large enough $i$, and so $c' = \lim_i u(x_i) \leq  u(y)$, a contradiction.  Therefore, the lemma is proved.
  \end{proof}

  \begin{proof}[Proof of Proposition \ref{prop:local_affine_support_function_boundary}] \label{proof:prop:local_affine_support_function_boundary}
  (1) We claim that for any $x \in U$, there exists $c_x \in \R$ such that for each sequence 
  $\{x_i\}_{i=1}^\infty$ in $U \cap \xo$ converging to $x$, we have $u_{o}(x_i) \to c_x$.   Granting this claim, we define $u(x) \coloneqq c_x$ for any $x \in U$.   Since $U$ is second countable and $U \cap \xo$ is dense in $U$, it follows that $u \in C^0(U)$.

  To prove the claim, we may assume that $U$ is a basic open set. By the definition of basic open sets, we have an isomorphism  $U \simeq \BT_+^k \times V \subset \BT_+^k \times \supp{\Sigma}$ induced by an affine map 
  for an open subset $V$ in the support of a tropical fan $\Sigma$, and that 
  $x = (\underline{\infty}, z) \in \BT_+^k \times V$.  For each $y\in \R_+^k$, we define $u_y \colon V \to \R$ by $u_y(w)\coloneqq u_{o}(y,w)$.  
  
  Let $\{x_i=(y_i, w_i)\}_{i=1}^\infty$ be a sequence of points in $\R_+^k \times V$ such that $x_i \to x$, equivalently, $y_i \to \underline{\infty}$ and $w_i \to z$, as $i$ tends to $+\infty$. Passing to a subsequence, we can assume that $y_1\leq y_2 \leq \dots $, that is, $y_i$ form an increasing sequence. Lemma \ref{lem:extension_bounded_convex_function_polyhedron} implies that $\{u_{y_i}\}_{i \ge 1}$  is a decreasing function sequence converging pointwise to some function $u_{\uinfty} \colon V \to \R$. Therefore, Proposition \ref{prop:limit_of_lconv_is_lconv} implies that $u_{\uinfty} \in \LConv (V)$. In particular, $u_{\uinfty} \in C^0(V)$. Hence, by Dini's theorem, $\{u_{y_i}\}_{i \ge 1}$ converges uniformly to $u_{\uinfty}$. This implies that $\lim_{i \to +\infty} u(y_i , w_i) =u_{\uinfty} (z) \in \R$. 

  Now suppose that $\{x'_i=(y'_i, w'_i)\}_{i=1}^\infty$ is another sequence of points in $\R^k_+ \times V$ converging to $x$. Then, Lemma \ref{lem:extension_bounded_convex_function_polyhedron} implies that $\{u_{y'_i}\}_{i \ge 1}$ also converges pointwise to $u_{\infty} $, and hence converges uniformly to $u_{\infty} $ after the same argument as in the previous paragraph.  Therefore, we have  $\lim_{i \to +\infty} u(y_i , w_i)=\lim_{i \to +\infty} u(y'_i , w'_i)=u_{\uinfty}(z)$. Hence, the claim is proven.
  
  \smallskip

  It remains to show that $u$ admits a local affine support at $x$. Let $l_0$ be an affine support function of $u_{\uinfty}$ at a point $z \in V$ defined on a neighborhood $V'\subset V$ of $z$. 
  We define $l\coloneqq l_0 \circ \pi$, where $\pi \colon \BT_{+}^k \times V \to V$ is the projection map. 
  Since for any $y \in \R_{+}^k$, $u_{y} \ge u_{\uinfty}$, we have $u(y,w ) \ge u(\uinfty, w) \ge  l_0(w)$ for any $w \in V'$. Therefore, $u \ge l$ on $U' =\BT_{+}^k \times V'$, and hence $l$ is an affine support function of $u$ at $x=(\underline \infty, z)$.

  \smallskip

  (2) We keep the notation as in the proof of (1). Let $U \cong \BT_{+}^k \times V$ be a basic open neighborhood of $x=(\underline{\infty}, z)$.  Without loss of generality, we may assume that $u_{o}$ is bounded from above on $U \cap \xo$. We may assume further that $V$  satisfies the assumptions  of Proposition \ref{prop:equivalent_def_LConv},     $V=\Pi^{(z)} =\bigcup_{z \in \sigma, \sigma \in \Pi} \relint(\sigma)$ for some polyhedral structure $\Pi$ on $V$. 
  
  Let $y_1 \le y_2 \le \cdots$ be an increasing sequence in $\R_+^k$, and $u_{\underline{\infty}} \colon V \to \R \cup \{-\infty\}$ be the pointwise limit of $\{u_{y_i}\}_{i \ge 1}$. Then, by assumption, we have $\{u_{y_i}\}_{i \ge 1}$ is uniformly bounded from above on $V$. By the same arguments as in (1), it suffices to show that $\{u_{y_i}\}_{i \ge 1}$ converges uniformly to $u_{\underline{\infty}}$ on an open neighborhood $V' \subset V$ of $z$. We argue separately in each of the following two cases.

  \begin{itemize}[leftmargin=*]
    \item \textbf{Case A}, $u_{\underline{\infty}}(z)> -\infty$.  Take $z_1 \in V$. Then, the balancing condition of $\Pi$ implies that there exist a finite subset $\{z_j\}_{2 \le j \le k} \subset V$ and positive real numbers $\{\nu_j\}_{1 \le j \le k}$ such that $\bigl(z, (z_j)_{1 \le j \le k}, (\nu_j)_{1 \le j \le k}\bigr)$ is a convex combination. Proposition \ref{prop:equivalent_def_LConv} implies that 
    \[
      u_{y_i}(z_1) \ge \frac{1}{\nu_1} \left( \sum_{2 \le j \le k} \nu_j u_{y_i}(z_j)- u_{y_i}(z) \right)
    \]
    for each $i \ge 1$. On the other hand, the mononicity of $\{u_{y_i}\}_{i \ge 1}$ implies that $u_{y_i}(z_j) \le u_{y_1}(z_j)$ and $u_{y_i}(z) \ge u_{\underline{\infty}}(z)$ for  each $i \ge 1$ and $2 \le j \le k$. Therefore, $\{u_{y_i}(z_1)\}_{i \ge 1}$ is uniformly bounded from above, and $u_{\underline{\infty}}(z_1) =\lim_{i \to +\infty} u_{y_i}(z_1)  > -\infty$. This shows that $u_{\underline{\infty}}$ is a finite-valued function on $V$, and hence $u_{\underline{\infty}} \in \LConv(V)$ by Proposition \ref{prop:limit_of_lconv_is_lconv}. In particular, $u_{\underline{\infty}} \in C^0(V)$. Thus, Dini's theorem implies that $\{u_{y_i}\}_{i \ge 1}$ converges uniformly to $u_{\underline{\infty}}$ on any compact open neighborhood $V' \Subset V$. 
    \item \textbf{Case B}, $u_{\underline{\infty}}(z)= -\infty$. Let $V' \Subset V$ be a compact open neighborhood in $V$. We claim that $\{u_{y_i}\}_{i \ge 1}$ converges uniformly to $-\infty$ in this case.  
     Proceeding by contradiction, suppose that there exist a constant $C_0 \in \R$ and a sequence of points $\{w_i\}_{i=1}^\infty \subset V'$ such that $u_{y_i}(w_i) \ge C_0$. After passing to a subsequence, we may assume that $w_i \to w$ with $w \in V$. Moreover, we may assume that there exists $\sigma \in \Sigma$ such that $w_i \in \sigma$ and $2w_i-w \in V \cap \sigma$ for each $i \ge 1$,  Then, the convexity of $(u_{y_i})|_{\sigma}$ implies that $u_{y_i}(2w_i-w) \ge (2C_0 - u_{y_i}(w))\to +\infty $. This contradicts the assumption that $\{u_{y_i}\}_{i \ge 1}$ is uniformly bounded from above on $V$. Therefore, the claim is proved. \qedhere
  \end{itemize} 
\end{proof}

The following result is the generalization of Proposition \ref{prop:properties_LConv_functions} to the boundary.
\begin{cor} \label{cor:boundary_properties_LConv_functions}
  Let $U$ be any open subset of $X$ with $x \in U$, and let $u \in \LConv(U)$. Suppose that $l$ is a local affine support function of $u$ at $x$, and that $\Pi^{(x)} \subset U$. Then $u \ge l$ on $\Pi^{(x)}$.
\end{cor}
\begin{proof}
  We keep the notation in the proof of Proposition \ref{prop:local_affine_support_function_boundary}\,(1). Recall that $U \cong \BT_{+}^k \times V \subset \BT_{+}^k \times |\Sigma|$, and $x=(\underline{\infty}, z) \in \{\underline{\infty} \} \times V$. We may assume that  $\Pi|_{U} \cong \BT_{+}^{k} \times \Pi'$, where $\Pi'$ is a polyhedral structure on $V$. Hence, $\Pi^{(x)}=\BT_{+}^{k} \times \Pi'^{(z)}$. By assumption, $l=l' \circ \pi$, where $l'$ is a local affine support function of $u|_{\{\underline{\infty} \} \times V }$ at $z$. Therefore, by Proposition \ref{prop:properties_LConv_functions}, we have $u|_{\{\underline{\infty}\} \times V} \ge l'$ on $ \Pi'^{(z)}$.  The proof of Proposition \ref{prop:local_affine_support_function_boundary}\,(1) implies that $u(y) \ge u((\underline{\infty},  \pi(y)))$ for any $y \in U$. Therefore, $u(y) \ge l'((\underline{\infty}, \pi(y)))=l'(y)$ for any $y \in \Pi^{(x)}$. 
\end{proof}

Motivated by Proposition \ref{lem:extension_bounded_convex_function_polyhedron}\,(2), we introduce the following definition. 
\begin{definition} \label{def:extended-real-valued_locally_convex_functions}
  Let $U$ be any open subset of $X$. 
  \begin{enumerate}
    \item An \emph{extended-real-valued function} on $U$ is a map $u \colon U \to \R \cup \{- \infty \}$. Such a function $u$ is called \emph{continuous} if it is continuous when we equip $\R \cup \{- \infty \}$ with the usual extended topology.   We use $\nC^{0, \ext}(U)$ to denote the set of continuous extended-real-valued functions on $U$.
    \item A \emph{plurisubharmonic (psh) function} on $U$ is a function $u \colon U \to \R \cup \{ -\infty \}$ satisfying $u \in \nC^{0,\ext}(U)$  and $ u \rest{U \cap \xo}\in \LConv(U \cap \xo)$. We use $\extLConv(U)$ to denote the set of psh functions on $U$.
  \end{enumerate}
\end{definition}
\begin{remark} \label{rmk:real_valued_extended_locally_convex_function_is_locally_convex}
  \begin{enumerate}
    \item Let $u \in \extLConv(U)$. Note that  $u(x) \in \R$ for all $x \in U \cap \xo$, since a locally convex function is always real-valued by definition. In particular, $\{U=-\infty\} \subset U \cap \partial X$, where $\partial X=X\setminus \xo$.
    \item If $u \in \extLConv(U)$ satisfies $u(x) \in \R$ for all $x \in U$, then $u \in \LConv(U)$. To see this, fix $x \in U$ and an open neighborhood $V$ of $x$ with $V \Subset U$. It suffices to show that $u \rest{V} \in \LConv(V)$. Since $\ol{V}$ is compact and $u \colon U \to \R \cup \{-\infty\}$ is continuous, $u \rest{V \cap \xo}$ is a  bounded locally convex function. Therefore, by Proposition \ref{prop:local_affine_support_function_boundary}\,(1), there exists $\widetilde{u} \in \LConv(V)$ such that $\widetilde{u} \rest{V \cap \xo}=u \rest{V \cap \xo}$. Since both $\widetilde{u}, u|_V \in C^0(V)$, and $V \cap \xo$ is dense in $V$, we conclude that $u \rest{V}=\widetilde{u} \in \LConv(V)$. \qedhere
  \end{enumerate}
\end{remark}
Recall that $(\cdot)^\star$ denotes the upper semicontinuous (usc) envelope of a function.
\begin{prop} \label{prop:limit_of_lconv_is_lconv_boundary}
  Let $U$ be any open subset of $X$.
  \begin{enumerate}
    \item Let $\{ u_{i} \}_{i \ge 1} \subset \extLConv(U)$. Suppose that $u_i$ converge pointwise to a function $u: U \to \R \cup \{-\infty\}$ with $u\rest{\xo} > -\infty$. Then, $u^{\star} \in \extLConv(U)$. 
    \item Let $\{ u_{i} \}_{i \in I}$ be a class of psh functions on $U$. Suppose that $\sup_{i \in I} u_i(x)<+\infty$ for any $x \in U$.  Then, $(\sup_{i \in I} u_i)^{\star} \in  \extLConv(U)$. If in addition $u_j \in \LConv(U)$ for some $j \ge 1$, then $(\sup_{i \in I} u_i)^{\star} \in \LConv(U)$.
    \item Let $\{u_i\}_{i \in I} \subset \LConv(U)$ be a finite set of locally convex functions on $U$. Then, $\sup_{i \in I} u_i \in \LConv(U)$.
  \end{enumerate}
\end{prop}
\begin{proof}
  For (1), we claim that $u|_{U \cap \xo}$ is bounded from above. Fix $x \in U$, and take a basic open neighborhood $W$ of $x$ satisfying $W \Subset U$. 
  By definition, $W \cong \BT_{+}^k \times V \subset \BT_{+}^k \times \supp{\Sigma}$, where $\Sigma$ is a tropical fan. For each $y\in \BT_+^k$, we define $u_{i,y} \colon V \to \R$ (\resp $u_y$) by $u_{i,y}(z)\coloneqq u_{i}(y,z)$ (\resp $u_{y}(z)\coloneqq u(y,z)$). Set $\mathsf{y}_0  \coloneqq (1, \cdots, 1) \in \R_+^k$, and $W' \coloneqq (1,\infty]^k \times V$. As $u_{\mathsf{y}_0} \in \LConv(\ol{V})$ and $\ol{V}$ is compact, there exists $M \in \R$ such that $u_{\mathsf{y}_0} \le M$. Moreover, $u_{i,\mathsf{y}_0} 
 \ge u_{i,y}$ for each $i \ge 1$ and each $y \ge \mathsf{y}_0$. 
Consequently, $u \le M$ on $W' \cap \xo=(1,\infty)^k \times V$, and the claim follows. 

  By Proposition \ref{prop:limit_of_lconv_is_lconv}, we have $u \rest{U \cap \xo} \in \LConv(U\cap \xo)$. Therefore, Proposition \ref{prop:local_affine_support_function_boundary}\,(2) implies that there exists $\widetilde{u} \in \extLConv(U)$ such that $\widetilde{u} \rest{U \cap \xo}=u \rest{U \cap \xo}$. To complete the proof of (1), it remains to show that $\widetilde{u}=u^{\star}$. Since $\widetilde{u}$ is continuous and $U \cap \xo$ is dense in $U$, it is enough to show that $u  \le \widetilde{u}$ on $U$. Take $W'$ as in the previous paragraph, and fix $(y', z') \in W'= (1,\infty]^k \times V$. Choose an increasing sequence $\{y_j\}_{j \ge 1} \subset (1, \infty)^k$ converging to $y'$. For any $i \ge 1$ and $j \ge 1$, $u_i(y_j,z') \ge u_i(y', z')$. Therefore, $u(y_j,z') =\lim_{i \to +\infty}u_i(y_j,z') \ge \lim_{i \to +\infty} u_i (y',z') = u(y', z')$ for all $j \ge 1$. On the other hand, $\widetilde{u}(y',z')=\lim_{j \to +\infty} u(y_j,z')$ since $\widetilde{u}$ is continuous.  We conclude that $u(y', z') \le \widetilde{u}(y',z')$. Hence, statement (1) is proven.
  
  For  (2), the proof of the first part of the statement is similar to the proof of Statement (1), which we omit. Now suppose that $u_j \in \LConv(U)$ for some $j \ge 1$. Then, $(\sup_{i \in I} u_i)^{\star}(x) \ge u_j(x) >-\infty$ for all $x \in U$. Hence, $(\sup_{i \in I} u_i)^{\star} \in \LConv(U)$ by Remark \ref{rmk:real_valued_extended_locally_convex_function_is_locally_convex}\,(2). 

  For (3), since the supremum of finitely many continuous functions is continuous, we have $(\sup_{i \in I} u_i)=(\sup_{i \in I} u_i)^{\star} \in \LConv(U)$.
\end{proof}

%%%%%%%%%%%%%%%%%
%%%Psh functions
%%%%%%%%%%%%%%%%%

\section{Tropical $\omega$-Psh functions} \label{sec:psh_functions}
In this section, we always assume that $X$ is a compact tropical variety.
\subsection{Semipositive forms and $\omega$-psh functions}
We recall the following definitions from \cite{APW1}.

\begin{definition} \label{def:semipositive-form_paper_2}  
  \begin{enumerate}
    \item A \textit{real closed simplicial $(1,1)$-form on $X$} is a mixed $(1,1)$-form $\omega$ on $\xo$ satisfying the following. For any $x \in X$, there exist a neighborhood $U$ of $x$ and $h \in \PL(U)$ such that $\omega=\dm \dmb h$ on $U \cap \xo$. When there is no ambiguity, we simply call such $\omega$ a \emph{simplicial form} on $X$. 
    \item A \emph{real closed semipositive $(1,1)$-form on $X$} is a mixed $(1,1)$-form $\omega$ on $\xo$ satisfying the following. For any $x \in X$, there exist a neighborhood $U$ of $x$ and $h \in \LConv(U)$ such that $h\rest{U \cap \xo} \in \PSm(U \cap \xo)$, and $\omega=\dm \dmb h$ in $U \cap \xo$. If this is satisfied, $h$ is called a \emph{local potential} of $\omega$ on $U$.  When there is no ambiguity, we simply call such $\omega$ a \emph{semipositive form} on $X$. 
    \item We use $\CK^{\geq 0}(X)$ to denote the set of semipositive forms on $X$, and use $\CK^{\geq 0}_{\PL}(X)$ for its subset consisting of semipositive simplicial forms on $X$. 
  \end{enumerate}
\end{definition}

\begin{remark}
  \begin{enumerate}
    \item Let $\omega$ be a semipositive form on $X$. By definition, a local potential $h$ on an open subset $U \subset X$ of $\omega$ is \emph{real-valued} and  \emph{continuous} (since a locally convex function is by definition real-valued and continuous). However, we require $h$ to be piecewise smooth only on $U \cap \xo$. 
    \item Let $X=\TP^1$ and $\omega=0$. Then, the function $h$  defined by $h(t) \coloneqq t$ cannot be extended to a local potential of $\omega$ on an open neighborhood of $+\infty$, since it admits no real-valued continuous extension to such an open neighborhood. \qedhere
  \end{enumerate}
\end{remark}

We give the definition of $\omega$-psh functions.
\begin{defi} \label{def:psh}
  Let $\omega$ be a semipositive form on $X$. We say that a function $\varphi\colon X \to \R \cup \{-\infty\}$ is an \emph{$\omega$-psh function} if 
    \begin{enumerate}
      \item $\varphi \in \nC^{0, \ext}(X)$.  (See Definition \ref{def:extended-real-valued_locally_convex_functions}.)
      \item For any $x \in X$, there exists a neighborhood $U$ of $x$ and a local potential $h \in \LConv(U)$ of $\omega$ on $U$ such that $h+\varphi \in \extLConv(U)$.
    \end{enumerate}
  We use $\PSH_{\omega}(X)$ to denote the set of $\omega$-psh functions on $X$. Furthermore, we use $\CPSH_{\omega}(X)$ to denote the set of bounded $\omega$-psh functions on $X$. 
\end{defi}

\begin{remark}
  Let $\varphi \in \nC^{0, \ext}(X)$. Let $U \subset X$ be an open subset and let $h \in \LConv(U)$ be a local potential of $\omega$. Assume that  $h+\varphi \in \PSH(U)$, then \cite[Proposition 3.10]{APW1} implies that $h'+\varphi \in \PSH(U)$ for any local potential $h' \in \LConv(U)$ on $U$.  %Proposition \ref{prop:difference_local_potential_is_affine}
\end{remark}

\begin{remark} 
  \begin{enumerate}
    \item By definition, each $\varphi \in \CPSH_{\omega}(X)$ is a continuous real-valued function, \textit{i.e.} $\CPSH_{\omega}(X) \subset \nC^0(X)$. 
    \item For each $\varphi \in \CPSH_{\omega}(X)$ satisfying $\varphi \rest{\xo} \in \PSm(\xo)$, $\omega_{\varphi} \coloneqq \omega+\dm \dmb \varphi$ defines a semipositive form on $X$.
    \item Observe that by Proposition \ref{prop:limit_of_lconv_is_lconv_boundary}\,(3), for any $\varphi \in \PSH_{\omega}(X)$ and any constant $\lambda \in \R$, $\sup\{\varphi, \lambda \} \in \CPSH_{\omega}(X)$.
    \item Since $\LConv(U)$ is a convex space for any open subset $U$, $\PSH_{\omega}(X)$ is also a convex space. \qedhere
  \end{enumerate}
\end{remark}

We now introduce the definition of numerical class for semipositive forms. 

\begin{definition}
  Let $\omega, \omega' \in \CK^{\ge 0}(X)$. We say that $\omega$ and $\omega'$ are in the \emph{same numerical class} if there exists $\varphi \in C^0(X)$ satisfying $\varphi \rest{\xo} \in \PSm(\xo)$ such that $\omega' \coloneqq \omega+\dm \dmb \varphi$. We use $[\omega]$ to denote the numerical class of $\omega$. \qedhere
\end{definition}

\begin{remark}
  If $\omega, \omega' \in \CK^{\ge 0}(X)$ are in the same numerical class, then the space $\PSH_{\omega}(X)$ is canonically identified with $\PSH_{\omega'}(X)$. 
\end{remark}

The following theorem is the tropical analogue of Theorem I.5.26 in \cite{Dem}.
\begin{thm} \label{thm:equivalent_def_psh}
  Let $\varphi_o \colon \xo \to \R$ be a function bounded from above which satisfies Property (2) in Definition \ref{def:psh}. Then, $\varphi_o$ can be extended uniquely to an $\omega$-psh function $\varphi$ on $X$. 
\end{thm}
\begin{proof}
  Take a finite open cover $X=\bigcup_{i \in I} W_i$ by tropical charts $(W_i, \phi_i, W_i \to U_i)$. Let $h_i \in \LConv(W_i)$ such that $\omega=\dm \dmb h_i$ on $W_i$. By Proposition \ref{prop:local_affine_support_function_boundary}, $\varphi_o \rest{W_i\cap \xo}$ can be extended to $\varphi_i$ over $W_i$ such that $\varphi_i+h_i \in \extLConv(W_i)$. Since $\varphi_i$ is continuous and $\xo$ is dense in $X$, we obtain that  $\varphi_i \rest{W_i \cap W_j}=\varphi_j \rest{W_i \cap W_j}$ for all $i,j$. Therefore, the theorem is concluded.  
\end{proof}

\begin{example} \label{example_kahler_form_embedded_tropical_varieties} 
  Let $\BT \BP_{\Xi}$ be a projective tropical toric variety, and $X \subset \TP_{\Xi}$ be a tropical subvariety. Suppose that $g$ be a convex conewise linear function on $\Xi$, and $h \coloneqq g \rest{\Xi}$. 
  In \cite[Example 4.5]{APW1}, we showed that $\omega \coloneqq \dm \dmb h$ extends to a semipositive simplicial form on $X$. 
  %cross-ref APW1: \ref{example:simplicial_kahler_form_tropical_toric_variety}
  Therefore, by Theorem \ref{thm:equivalent_def_psh}, there are bijective maps
  \begin{align*}
  \PSH_{\omega}(X) &\to \left\{u \in \LConv(\xo)\ \st \ u-h \le O(1)\right\}, \qquad \varphi \mapsto \varphi+h, \text{ and }  \\
  \CPSH_{\omega}(X) &\to \left\{u \in \LConv(\xo)\ \st \ |u-h|=O(1)\right\}, \qquad \varphi \mapsto \varphi+h. 
  \end{align*}
  Moreover, for any $h' \in \LConv \cap \PSm(\xo)$ satisfying $|h'-h|=O(1)$, $\omega' \coloneqq \dm \dmb h'$ extends to a  semipositive form on $X$.
\end{example}

\begin{prop} \label{prop:limit_of_psh_is_psh}
  \begin{enumerate}
    \item Let $\{ \varphi_{i} \}_{i \ge 1}$ be a sequence of $\omega$-psh functions on $X$. Suppose that $\varphi_i$ converge pointwise to a function $\varphi\colon X \to \R \cup \{-\infty\}$ with $\varphi \rest{\xo}>  -\infty$. Then, $\varphi^{\star} \in \PSH_{\omega}(X)$, where $(\cdot )^{\star}$ denotes the usc envelope.  
    \item Let $\{ \varphi_{i} \}_{i \in I}$ be a class of $\omega$-psh functions on $X$. Suppose that for each $x \in X$, we have $\sup_{i \in I} \varphi_i(x)<+\infty$.  Then, $(\sup_{i \in I} \varphi_i)^{\star} \in \PSH_{\omega}(X)$. 
    \item Let $\{\varphi_i\}_{i \in I} \subset \CPSH_{\omega}(X)$ be a finite set of bounded $\omega$-psh functions. Then, $\sup_{i \in I} \varphi_i \in \CPSH_{\omega}(X)$.
  \end{enumerate}
\end{prop}
\begin{proof}
  These are immediate corollaries of Proposition \ref{prop:limit_of_lconv_is_lconv_boundary}. 
\end{proof}

\begin{cor} \label{cor:decreasing_limit_of_psh_is_psh}
  Let $\{\varphi_i\}_{i \ge 1}$ be a decreasing sequence of $\omega$-psh functions such that 
  $\lim\limits_{i \to +\infty} \varphi_i(x)> - \infty$ for any $x \in \xo$. Then, the pointwise limit $\varphi$ of $\{\varphi_i\}_{i \ge 1}$ is an $\omega$-psh function on $
  X$. Furthermore, suppose that there exists $C$ such that $\varphi_i \ge C$ on $X$ for any $i$.   Then $\varphi_i$ converges uniformly to $\varphi \in \CPSH_{\omega}(X)$.
\end{cor}
\begin{proof}
  Since $\{\varphi_i\}_{i \ge 1}$ is a decreasing sequence of  continuous function, $\varphi$ is usc. Therefore, by Proposition \ref{prop:limit_of_psh_is_psh}\,(1), $\varphi=\varphi^{\star} \in \PSH_{\omega}(X)$. The second part of the corollary follows directly from Dini's theorem. 
\end{proof}

We will apply the following construction several times in our paper.

\begin{construction} \label{construction:sup_omega_psh_functions}
  Let $\omega$ be a semipositive form on $X$, and $\varphi \in \CPSH_{\omega}(X)$. Let $U \subset X$ is an open subset, and let $u \in \LConv(U)$. Let $h \in \LConv(U)$ satisfy $\omega=\dm \dmb h$ on $U$. Suppose that $W \coloneqq \{\varphi +h < u\} \Subset U$. 
  We define a function denoted $\sup \{\varphi, u-h\} \colon X \to \R$  by setting 
  \begin{equation}
    \sup \{\varphi, u-h\}(x) \coloneqq \begin{cases*}
      \sup \{\varphi(x), u(x)-h(x)\}, \qquad x \in U, \\
      \varphi(x), \qquad x \in X \setminus U\,.
    \end{cases*}
  \end{equation}
\end{construction}

\begin{lemma} \label{lem:construction_sup_omega_psh_functions} 
  In the setting of Construction \ref{construction:sup_omega_psh_functions}, we have $\sup \{\varphi, u-h\} \in \CPSH_{\omega}(X)$.
\end{lemma}
\begin{proof}
  Since $W \Subset U$, $\{X \setminus \ol{W} ,U\}$ forms an open cover of $X$. Denote $\varphi' \coloneqq \sup \{\varphi, u-h\}$. Remark that $\varphi'\rest{X \setminus \ol{W}} =\varphi\rest{X \setminus \ol{W}} \in \nC^0(X) $, and $\varphi'|_U = \sup\{\varphi \rest{U}, u-h \} \in C^0(U)$. Hence, $\varphi' \in C^0(X)$. It remains to verify Property (2) in Definition \ref{def:psh}. When $x \in X \setminus \ol{W}$, this is satisfied by the assumption that $\varphi \in \CPSH_{\omega}(X)$. When $x \in U $, by Proposition \ref{prop:sup_of_lconv_is_lconv}, $\varphi'+h=\sup\{\varphi \rest{U}+h, u\} \in \LConv(U)$. This completes the proof. 
\end{proof}

\subsection{Uniform Lipschitz property for $\omega$-psh functions}

Let $\omega$ be a semipositive form on $X$. Fix a finite open cover $\xo=\bigcup_{i \in I} W_i$ by tropical charts $(W_i, \phi_i, W_i \to U_i)$. For each $i\in I$, fix an Euclidean metric on $U_i$, 
and let $d_i$ denote its pullback to $W_i$ via $\phi_i$.

Let $S\subset \xo$ and $\varphi\colon S \to \R$. We say that $\varphi$ is $C$-Lipschitz on $S$, for  $C>0$, 
if for any $i \in I$, $\varphi\rest{S\cap W_i}$ is $C$-Lipschitz with respect to $d_i$. 

\smallskip

We have the following theorem.

\begin{thm} \label{thm:lipschitz} 
  Let $K$ be a compact subset of $\xo$. There exists a positive constant $C_K$ such that for any $\varphi \in \PSH_{\omega}(X)$, $\varphi|_{K}$ is $C_K$-Lipschitz.
\end{thm}
\begin{proof}
This is Theorem~\ref{thm:appendix_lipschitz} proved in Appendix \ref{appendix:uniform_lipschitz}. 
\end{proof}

\subsection{Envelope property of $\omega$-psh functions}\label{subsec:envelope}

\begin{defi}[Envelope] \label{def:psh_envelope}
Let $f \colon X \to \R \cup \{ + \infty \}$ be a function bounded from below. The \emph{envelope} of $f$ denoted $\env_{\omega}(f)$ is defined to be $-\infty$ if there is no $\omega$-psh function $\varphi \leq f$, otherwise, 
\[
\env_{\omega}(f) \coloneqq \Bigl(\sup\Bigl\{\varphi \st \varphi \in \PSH_{\omega}(X) \textrm{ and } \varphi \leq f\Bigr\}\Bigr)^{\star }.\]
Here, as before, $(\cdot)^{\star}$ denotes the usc envelope.
\end{defi}

\begin{prop} \label{prop:uniform_upper_bound}
  Let $x \in \xo$. There exists $C>0$ such that for any $\varphi \in \PSH_{\omega}(X)$ satisfying $\varphi (x) \le 0$, we have $\varphi \le C$ on $X$. 
\end{prop}
\begin{proof}
  Fix a polyhedral structure $\olPi$ on $X$ such that $x \in \Pi_{\f}$. Since $\Pi_{\f}$ is a compact subset of $\xo$, Theorem \ref{thm:lipschitz} implies that there exists $C>0$ such that $|\varphi| \le C$ on $\Pi_{\f}$ for any $\varphi \in \PSH_{\omega}(X)$ satisfying $\varphi (x) \le 0$. Let $\varphi$ be such a function. We now show that $\varphi \le C$ on $X$. 
  Since $\varphi \in \nC^{0, \ext}(X)$ by definition, we have $\sup_{x \in \xo} \varphi(x)=\sup_{x \in X} \varphi(x)$, and $\varphi$ is bounded from above on $X$. Let $\sigma \in \Pi$ be an unbounded polyhedron in $\xo$. We can write $\sigma=\tau +\delta$ for a polytope $\tau \in \Pi_{\f}$ and a polyhedral cone $\delta$ (the recession cone of $\sigma$). Let $y=y_1+y_2 \in \sigma$ with  $y_1 \in \tau$ and $y_2 \in \delta$. Then, Lemma \ref{lem:extension_bounded_convex_function_polyhedron} implies that $\varphi(y) \le \varphi(y_1) \le C$, as required.
\end{proof}

\begin{cor} \label{cor:psh_envelope}
  Suppose that $f: X \to \R \cup \{+\infty  \}$ is a function such that $f$ is bounded from below, and there exists $x \in \xo$ with $f(x)< +\infty$. Then, $\envomega(f) \in \CPSH_{\omega}(X)$.  
\end{cor}
\begin{proof}
  Consider the family of $\omega$-psh functions $\CF \coloneqq \Bigl\{\varphi \st \varphi\in \wpsh_{\omega}(X) \textrm{ and } \phi \leq f\Bigr\}$.  That $f$ is bounded from below implies that $-M \in \CF$ for $M \gg 0$. Hence, $\CF$ is a nonempty subset of $\PSH_{\omega}(X)$.  Proposition \ref{prop:uniform_upper_bound} implies that there exists $C \in \R$ such that  $\varphi<C$ on $X$ for each $\varphi \in \CF$. Therefore, $\envomega(f) \in \PSH_{\omega}(X)$ by Proposition \ref{prop:limit_of_psh_is_psh}\,(2). On the other hand, $\envomega(f) \ge -C$, so $\envomega(f) \in \CPSH_{\omega}(X)$. 
\end{proof}

\begin{remark} \label{rmk:operator_psh_envelope_monotone_concave}
  It follows from the definition that, as an operator, the psh envelope is monotone and concave. 
\end{remark}

\begin{example} \label{example_functions_with_values_at_finite_points}
  We call $\CS=\{ (x_i, \lambda_i) \in \xo \times \R \}_{1 \le i \le k}$ a  \textit{finite collection of pairs} if $x_i$ are pairwise distinct points in $\xo$. We use $S$ to denote the finite subset $\{ x_i \}_{1 \le i \le k}$ of $\xo$.  We associate to $\CS$  a lower semicontinuous (lsc) function $f_{\CS}$ by setting
  \bes
  f_{\CS}(x)=
  \begin{cases}
    \lambda_i &\text{ if } x=x_i \text{ for some } x_i \in S, \\
    +\infty & \text{ if } x \in X \setminus S.
  \end{cases}
  \ees
  Then, by Corollary \ref{cor:psh_envelope}, we have $\envomega(f_{\CS}) \in \CPSH_{\omega}(X)$. 
\end{example}

We finish this section with the following Lemma.
\begin{definition}[$\omega$-PL functions on $X$] \label{def:omega-PL}
  Let $\omega$ be a semipositive form on $X$, and let $\varphi$ be a function on $X$. 
  \begin{enumerate}
    \item We say that $\varphi$ is $\omega$-PL if for any open subset $U \subset X$ and for any local potential $h \in \LConv(U)$ of $\omega$ on $U$, we have $h+\varphi \in \PL(U)$. We use $\PL_{\omega}(X)$ to denote the space of $\omega$-PL functions on $X$. In particular, when $\omega$ is simplicial, then $\PL_{\omega}(X)=\PL(X)$. 
    \item Suppose that $\olPi$ is a polyhedral structure on $X$. We say that $\psi$ is $\omega$-facewise affine on $\olPi$ if for any open subset $U \subset X$ and for any local potential $h \in \LConv(U)$ of $\omega$ on $U$, $h+\varphi$ is facewise affine on $\olPi \cap U$. We say that $\omega$ is facewise affine on $\olPi$ if $0$ is an $\omega$-facewise affine function on $\olPi$. 
  \end{enumerate}
\end{definition}

\begin{lemma} \label{lem:envelope_PL_functions_vertices}
  Let $\omega$ be a semipositive form on $X$. Let $\olPi$ be a polyhedral structure on $X$ with induced polyhedral structure $\Pi$ on $\xo$. Let $\psi$ be an $\omega$-facewise affine function on $\olPi$. Then, for any $\varphi \in \PSH_{\omega}(X)$, $\varphi \le \psi$ on $X$ iff $\varphi \le \psi$ on $\Pi_0$, the set of vertices of $\Pi$. Equivalently, we have $\psi=\envomega(f_{\CS})$ for the finite collection of pairs $\CS=\{ (x, \psi(x)) \ |\ x \in \Pi_0  \}$. 
\end{lemma}
\begin{proof}
  Fix $\varphi \in \PSH_{\omega}(X)$ and $\sigma \in \Pi$. Observe that $(\varphi-\psi)|_{\sigma}$ is convex. Write $\sigma =\tau+\delta$, where $\tau$ is a bounded polytope and $\delta$ is a polyhedral cone. By Lemma \ref{lem:extension_bounded_convex_function_polyhedron} applied to each cone $x+\delta$, $x\in \tau$, we have $\sup_{\sigma} (\varphi-\psi) = \sup_{\tau} (\varphi-\psi)$. Furthermore, since the set of vertices $\sigma_{0}$ of $\sigma$ coincides with the set of vertices $\tau_{0}$ of $\tau$, we have $\sup_{\tau} (\varphi-\psi) = \sup_{\tau_{0}} (\varphi-\psi) = \sup_{\sigma_{0}} (\varphi-\psi)$. Hence, $\sup_{X}(\varphi-\psi)=\sup_{\Pi_0}(\varphi-\psi)$.  
\end{proof}

\subsection{Topology on $\PSH_{\omega}(X)$}

 We endow the space $\PSH_{\omega}(X)$ with the topology of pointwise convergence on $\xo$.

For each $x\in \xo$, let 
\[
\PSH_{\omega,x}(X) \coloneqq \left\{ \varphi \in \PSH_{\omega}(X) \st \varphi(x)=0 \right\} \subset \PSH_{\omega}(X).
\]
\begin{prop} \label{prop:compactness_of_psh}
  For each point $x \in \xo$, the subspace $\PSH_{\omega, x}(X)$ of $\PSH_{\omega}(X)$ is compact.
\end{prop}
\begin{proof}
  We follow the proof of Theorem 5.11 in \cite{BJ22}. Pick a sequence $\{\varphi_i\}_{i \ge 1}$ in $\PSH_{\omega, x}(X)$. 
  By Theorem \ref{thm:lipschitz}, $\{\varphi_i(y)\}_{i \ge 1}$ is a bounded sequence in $\mathbb{R}$ for each $y \in \xo$. By Tychonoff's theorem, after passing to a subsequence, we may thus assume that $\{\varphi_i\}_{i \ge 1}$ converges pointwise on $\xo$ to a function 
  $\varphi_{o} \colon \xo \to \mathbb{R}$.

  By assumption, $\varphi_i+h \in \LConv (W)$ for any local tropical chart $W$ of $X$ and for any local potential $h$ of $\omega$ on $W$. Hence, Proposition \ref{prop:limit_of_lconv_is_lconv} implies that $\varphi_{o}+h \in \LConv (W \cap \xo)$. On the other hand, Proposition \ref{prop:uniform_upper_bound} implies that there exists $C \in \R$ such that $\varphi_i < C$ on $X$ for each $i \ge 1$,  which implies further that that $\varphi_{o}<C$ on $\xo$. Therefore, by Theorem \ref{thm:equivalent_def_psh},  $\varphi_{o}$ extends to a function $\varphi$ in $\PSH_{\omega,x}(X)$.
\end{proof}

%%%%%%%%%%%%%%%%%%%%%%%%%%%%%
%%%%%%%Regularization theorem
%%%%%%%%%%%%%%%%%%%%%%%%%%%%%

\section{Regularization theorem on compact \prekahler tropical varieties} \label{sec:regularization_theorem}
In this section we prove our regularization theorem on \prekahler tropical varieties.
\subsection{K\"ahler forms} We first introduce \kahler forms and \prekahler tropical varieties.
\begin{definition} 
  Let $U$ be an open subset of $X$ and $u \colon U \to \R$.
  \begin{enumerate}
    \item  We say that $u$ \emph{admits a strict local affine support at $x$} if there exist an open neighborhood $U'$ of $x$ and   $l \in \Aff(U')$ such that $u(x)=l(x)$, and $u(y)> l(y)$  for any $y \in U' \setminus \{x \}$. 
    \item We say that $u$ is \emph{strictly locally convex at x} if $u$ is continuous at $x$, and admits a strict local affine support at $x$. We say that $u$  is \emph{strictly locally convex on $U$} if $u$ is strictly locally convex at $x$ for each $x \in U$. 
  \end{enumerate}
\end{definition}

\begin{definition} [K\"ahler forms] \label{def:kahler_forms_and_strictly_convex_functions}
  Let $\omega$ be a semipositive form on $X$.
  \begin{enumerate}
    \item Suppose $\varphi \in \CPSH_{\omega}(X)$. We say that $\varphi$ is \emph{strictly $\omega$-psh} if for any $x \in X$, there exist an open neighborhood $W$ of $x$ and  a local potential $h$ of $\omega$ on $W$ such that $\varphi+h$ is strictly locally convex at $x$. We use $\pCPSH_{\omega}(X)$ to denote the set of strictly $\omega$-psh functions on $X$. 
    \item  We say that $\omega$ is a \emph{\kahler form on $X$} if $0 \in \pCPSH_{\omega}(X)$.  We use $\CK^{>0}(X)$  to denote the set of \kahler forms on $X$.
    \item A numerical class $[\eta]$ is called a \emph{\kahler class} if $[\eta]$ can be represented by a \kahler form.
  \end{enumerate}
\end{definition}

\begin{remark} \label{rmk:definition_kahler_forms_independent_choice_local_potential}
Similar to~\cite[Rem.\,4.3]{APW1}, the definition of strictly $\omega$-psh functions and  \kahler forms does not depend on the choice of the local potentials. This follows from~\cite[Prop.\,3.10]{APW1} which implies that if $\dm\dmb h=\dm\dmb h'$ for two potentials on an open set, 
then $h-h'$ is affine.
%cross-ref to APW1: \ref{rmk:definition_simplicial_forms_independent_choice_local_potential} and \ref{prop:difference_local_potential_is_affine}
\end{remark}

\begin{remark} \label{rmk:zero_is_strictly_psh_when_theta_is_kahler}
  For a semipositive form $\omega$,  $\pCPSH_{\omega}(X)$ is nonempty if $[\omega]$ is a \kahler class. 
\end{remark}

\begin{definition} [Compact \prekahler tropical varieties]
  A compact tropical variety $X$ is called \emph{\prekahler}if $X$ admits a \kahler form $\theta$. 
\end{definition}

\begin{remark}
  In tropical geometry, a tropical variety is called \kahler if each of its star fans satisfies the \kahler package (see \cites{AP-tht, AP20-hc, AP-hodge-fan}), and it admits a ``simplicial Kähler form'' in the sense of~\cite{AP-tht}. (It is shown in \emph{loc.\.cit.} that the variety itself then verifies the \kahler package.) Therefore, we choose the terminology ``\prekahler tropical varieties'' in our paper, as we do not impose any cohomological requirement on the variety. 
  Also, we note that our definition of \kahler forms is different from the simplicial \kahler forms  in \cite{AP20-hc}. In the setting of \emph{loc.\,cit.}, a simplicial \kahler form is a semipositive form whose local potential induces a strictly convex conewise linear function on each star fan (in the sense of Section \ref{subsubsec:projective_tropical_variety-p2}). 
  For a comparison between these two definitions, see the discussion below. 
\end{remark}

Let $\Xi$ be a complete simplicial fan, and let $g$ be a strictly convex conewise linear function on $\Xi$ (see Section \ref{subsubsec:projective_tropical_variety-p2}). Notice that the semipositive form $\dm \dmb g$ is not a \kahler form. This is because $g$ is linear on $\relint(\sigma)$ for each $\sigma \in \Xi$, which implies that $g$ is strictly locally convex at $x$ iff $x=0$.  
However, we have the following result. 

\begin{prop} \label{prop:kahler_form_on_tropical_toric_varieties}
  Any projective tropical toric variety $\TP_{\Xi}$, for a complete simplicial fan $\Xi$, is a compact \prekahler tropical variety. More precisely, suppose that $g$ is a strictly convex conewise linear  function on $\Xi$. Then $[\dm \dmb g]$ is a \kahler class on $\TP_{\Xi}$. 
\end{prop}

To prove this proposition, we need the following lemma.

\begin{lemma} \label{lem:strictly_convex_for_convolution}
  Let $E=\R^n$ be an Euclidean space. Consider the Gaussian function
  \[
  \rho(x) = \frac{1}{(2\pi)^{n/2}} \exp\!\left(-\frac{\|x\|^2}{2}\right), x \in E.
  \]
  Let $g\colon E \to \R$ be a convex function. Then the convolution 
  \[
  (g \star \rho) (x) = \int_E g(x - x')\, \rho(x')\, dx'.
  \]
  is also a convex function on $E$. Moreover, let $x_1, x_2 \in E$ and $\lambda_1, \lambda_2 \in \R_{\ge 0}$ with $\lambda_1+\lambda_2=1$. We have 
  \[
  (g \star \rho) (\lambda_1 x_1+\lambda_2 x_2) < \lambda_1 \cdot  (g \star \rho)(x_1) +\lambda_2 \cdot (g \star \rho) (x_2)
  \]
  if there exists $x_0 \in E$ such that 
  \[
  g (\lambda_1 x_1+\lambda_2 x_2-x_0) < \lambda_1 \cdot g(x_1-x_0) +\lambda_2 \cdot g (x_2-x_0)
  \]
\end{lemma}
\begin{proof}
  The first part of the statement is proven in~\cite[Lem.~2.7.7]{BPS14}. The second part comes from the fact that the function  $x' \mapsto \lambda_1 \cdot g(x_1-x') +\lambda_2 \cdot g (x_2-x') -g (\lambda_1 x_1+\lambda_2 x_2-x')$ is a nonnegative continuous function on $E$ which takes positive value at $x_0$. 
\end{proof}

\begin{proof} [Proof of Proposition \ref{prop:kahler_form_on_tropical_toric_varieties}]
  We use $E$ to denote the support of the fan $\Xi$. Set $\widetilde{g} \coloneqq (g \star \rho)$. Then, Lemma \ref{lem:strictly_convex_for_convolution} implies that $\widetilde{g}$ is convex on $E$. That $g$ is Lipschitz on $E$ implies that $|\widetilde{g} -g|=O(1)$. Therefore, by Example \ref{example_kahler_form_embedded_tropical_varieties}, $\widetilde{\omega} \coloneqq \dm \dmb (\widetilde{g})$ extends to a semipositive form on $\TP_{\Xi}$ with  $[\widetilde{\omega}]=[\dm \dmb g]$. Our goal is to show that $\widetilde{\omega}$ is a \kahler form on $\TP_{\Xi}$. 

  Since $g$ is strictly convex on $\Xi$, for any $x_1, x_2 \in E$ and $\lambda_1, \lambda_2 \in \R_{> 0}$ with $\lambda_1+\lambda_2=1$, there exists  $x' \in E$ such that 
  $g (\lambda_1 x_1+\lambda_2 x_2-x') < \lambda_1 \cdot g(x_1-x') +\lambda_2 \cdot g (x_2-x')$. Otherwise, $g$ would be linear on the line $x''+\R[x_2-x_1]$ for any $x'' \in E$, which would contractict with the assumption that $g$ is strictly convex. Hence, by Lemma \ref{lem:strictly_convex_for_convolution}, $\widetilde{g}$ is a strictly locally convex function on $E=(\TP_{\Xi})_o \subset \TP_{\Xi}$. 

  Now, fix $x \in \TP_{\Xi}$ with $\sed(x)=\delta$. By definition, there exists a linear function $l$ on $E$ such that $g\rest{\delta}=l\rest{\delta}$, and $g>l$ on $ \eta \setminus \delta$ for each $\eta \supface \delta$. Since $g$ is convex on $E$, this implies that $g>l$ over $E \setminus \delta$. Consider the tangenet space $T_{\delta}$, and let $p \colon E \to E^{\delta} \coloneqq E/T_{\delta} \simeq \BO(\delta)$ be the quotient map.  Set $g_{\delta} \coloneqq g-l$, and $\widetilde{g_{\delta}} \coloneqq  (g_{\delta} \star \rho)$. Since $g_\delta$ vanishes on $\delta$, it induces a strictly convex linear function on the star fan $\Xi^{\delta}$ that we denote by $g^\delta$.

  We identify $E$ with $T_{\delta} \times E^{\delta}$, and identify $\delta$ with $\R_{ \ge 0}^k \subset T_{\delta}$ for  $k=\dim \delta$. We identify further  $x$ with $(\underline{\infty},z)$ for $z \in E^{\delta}$. Then, $\BT_{+}^k \times E^{\delta}$ is a basic open neighborhood of $x$ in $E$.  Notice that for any bounded open subset $V \subset E^{\delta}$, $g_{\delta}$ is bounded above on $\BT_{+}^k \times V$. And since $|g_{\delta}-\widetilde{g_{\delta}}|=O(1)$ on $E$, $\widetilde{g_{\delta}}$ is also bounded above on $\BT_{+}^k \times V$. 
  Therefore, Proposition \ref{prop:local_affine_support_function_boundary} implies that both $g_{\delta}$ and $\widetilde{g_{\delta}}$ can be extended to locally convex functions on $\BT_{+}^k \times V$. Slightly abusing of notation, we use the same symbol to denote the extensions. Notice that $\widetilde{g_{\delta}}-\widetilde{g}=l$ on $E$. Hence, $\widetilde{g_{\delta}}$ is a local potential of $\widetilde{\omega}$ on $\BT_{+}^k \times E^{\delta}$. 
  
  It is not hard to see $(g_{\delta})\rest{\{\underline{\infty}\} \times E^{\delta}}=g^{\delta}$, the induced conewise linear function on $\Xi^\delta$ by $g^{\delta}$.   Now we seek to determine $\widetilde{g_{\delta}}|_{\{\underline{\infty}\} \times E^{\delta}}$. 
  Let $\{y_i\}_{i=1}^{\infty} \subset \BT_{+}^k$ be a sequence of points  such that $y_i \to \underline{\infty} \in \BT_{+}^{k}$, and let  $z \in E^{\delta}$.
  Observe that for any $x' \in E$, $\{g_{\delta}((y_i, z)-x')\}_{i \ge 1}$ is a decreasing sequence with $\lim_{i \to +\infty} g_{\delta}((y_i, z)-x')= g^{\delta}(z-p(x'))$. Therefore, the monotone convergence theorem implies that 
  \[
  \widetilde{g_{\delta}}(\underline{\infty}, z)=\lim_{i \to +\infty} \widetilde{g_{\delta}} ((y_i, z))= \lim_{i \to +\infty} \int_E g^{\delta}(z-p(x')) \rho(x') dx' =((g^{\delta} \circ p) \star \rho) ((0,z)).
  \]

  Since $g^{\delta}$ is strictly convex on $\Xi^{\delta}$, for any $z_1, z_2 \in E^{\delta}$ and $\lambda_1, \lambda_2 \in \R_{> 0}$ with $\lambda_1+\lambda_2=1$, there exists  $z' \in E^{\delta}$ such that 
  $g^{\delta} (\lambda_1 z_1+\lambda_2 z_2-z') < \lambda_1 \cdot g^{\delta}(z_1-z') +\lambda_2 \cdot g^{\delta} (z_2-z')$. Therefore, by Lemma \ref{lem:strictly_convex_for_convolution}, $((g^{\delta} \circ p) \star \rho)$ is strictly convex on $\{0\} \times E^{\delta}$, which implies further that $\widetilde{g_{\delta}}|_{\{\underline{\infty}\} \times E^{\delta}}$ is strictly convex on $\{\underline{\infty}\} \times E^{\delta}$. 

  Let $l \in \Aff(E^{\delta})$ such that $l$ is a strict affine support of $\widetilde{g_{\delta}}|_{\{\underline{\infty}\} \times E^{\delta}}$ at $z$. We aim to show that $l \circ p$ is a strict affine support function of $\widetilde{g_{\delta}}$ on $\BT_{+}^k \times E^{\delta}$, \textit{i.e.} $\widetilde{g_{\delta}}(y',z') > l \circ p(y',z')$ for any $(y',z' ) \in \BT_{+}^k \times E^{\delta}$ with $(y',z') \neq (\underline{\infty}, z')$.  
  By assumption, this inequality holds when $y'=\underline{\infty}$, so we may henceforth assume that $y' \neq \underline{\infty}$. Take $\{y_i\}_{i \ge 1} \subset \BT_{+}^{k}$ such that $y_1=y'$, $y_{i+1} \ge y_i$ for any $i  \ge 1$, and $\lim_{i \to \infty}y_i=\underline{\infty}$.  Since $\widetilde{g_{\delta}}$ is strictly convex on $E$ by Lemma \ref{lem:strictly_convex_for_convolution} and $\widetilde{g_{\delta}}$ is bounded above on $\BT_{+}^k \times E^{\delta}$, Lemma \ref{lem:extension_bounded_convex_function_polyhedron} implies that $\widetilde{g_{\delta}}(y_{i+1},z') > \widetilde{g_{\delta}}(y_{i},z')$ for any $i \ge 1$. Therefore, we have $\widetilde{g_{\delta}}(y',z')> \widetilde{g_{\delta}}(\underline{\infty}, z') \ge l(z')=l(p((y',z')))$.
   This proves that $\widetilde{g_{\delta}}$ is locally strictly convex at $(\underline{\infty}, z)$, and hence $\widetilde{\omega}$ is a \kahler form. 
\end{proof}

The next corollary shows that one can always compactify an open tropical variety in $\R^n$ to a compact \prekahler tropical variety. 

\begin{cor} \label{cor:refinement_of_open_tropical_varieties_are_projective}
  Let $Y$ be a tropical variety embedded in $E=\R^n$. Then, 
  \begin{enumerate}
    \item There exists a complete simplicial rational  fan $\Xi$ on $E$, and a simplicial rational polyhedral structure $\Pi$ on $Y$ such that 
    \begin{itemize}
      \item $\TP_{\Xi}$ is projective.
      \item $\rec(\Pi)$ is a subfan of $\Xi$.
      \item Let $\ol{Y}$ denote the closure of $Y$ in $\TP_{\Xi}$. Then, $\ol{Y}$ is a tropical subvariety of $\TP_{\Xi}$.
    \end{itemize}
    In particular, $\ol{Y}$ is a \prekahler tropical variety. 
    \item Furthermore, there exists a convex conewise linear function $g$ on $\Xi$ with $g>0$ on $E \setminus \{0\}$  such that  $[\dm \dmb h]$ is a \kahler class on $\ol{Y}$, where $h \coloneqq g \rest{Y}$. 
    \item If $Y$ is a fanfold, then $\Pi$ can be taken as a fan structure on $Y$, and $h \in \LConv \cap \PL_{\hom}(Y)$.
    \end{enumerate}
\end{cor}
\begin{proof} 
  We first prove statement (1). This is a folklore result, for example, there is a variant in~\cite{KKMS,EI06}. For the sake of completeness, we give the argument.  Let $\Pi'$ be any simplicial rational polyhedral structure on $Y$. Denote by $\Sigma' \coloneqq \rec(\Pi')$ the recession cone of $\Pi'$. By Theorem \ref{thm:toric_resolution}, there exist a complete simplicial rational fan $\Xi$ supported on $E$ and a subfan $\Sigma$ of $\Xi$ such that $\TP_{\Xi}$ is projective and $|\Sigma|=|\Sigma'|$. 
  
  Recall that any simplicial polyhedron $\sigma$ can be written uniquely as a Minkowski sum of a simplex $\tau$ and a simplicial cone $\delta$ such that their tangent spaces satisfy $T_{\tau} \cap T_{\delta}=\{0\}$.  Any face of $\sigma$ is also a simplicial polyhedron which can be written as $\tau'+\delta'$ for $\tau' \subface \tau$ and $\delta' \subface \delta$. Moreover, for any subsets $\mathsf{T}, \mathsf{T}'  \subset \tau$ and $\mathsf{D}, \mathsf{D}'  \subset \delta$, we have $(\mathsf{T} \cap \mathsf{T}') + (\mathsf{D} \cap \mathsf{D}')=(\mathsf{T} +\mathsf{D} ) \cap (\mathsf{T}' +\mathsf{D}' ) $. 

  Now, we define
  \[
  \Pi \coloneqq \{\tau+\delta \,|\, \tau \in \Pi'_{\f}, \delta \in \Sigma, \text{ such that there exists } \delta' \in \Sigma' \text{ satisfying } \delta \subset \delta', \tau+\delta' \in \Pi' \}
  \]
  Any $\tau +\delta \in \Pi$ is a simplicial polyhedron because $T_{\tau} \cap T_{\delta} \subset T_{\tau} \cap T_{\delta'}=\{0\}$. Let $(\tau_1+\delta_1), (\tau_2 + \delta_2) \in \Pi$, and let $\delta'_1, \delta'_2 \in \Sigma'$ such that $\delta_i \subset \delta'_i$ and $(\tau_i+\delta_i') \in \Pi'$ for each $i \in \{1, 2\}$. By assumption, $(\tau_1+\delta'_1) \cap (\tau_2 +\delta'_2)$ is a face of $\tau_i+\delta_i'$. Hence, $(\tau_1+\delta'_1) \cap (\tau_2 +\delta'_2)=\tau+\delta' \in \Pi'$, where  $\tau=(\tau_1 \cap \tau_2) \in \Pi$ and $\delta'=(\delta_1' \cap \delta_2')$. Therefore, 
  \bens
  &&(\tau_1+\delta_1) \cap (\tau_2 +\delta_2)=\bigl( (\tau_1+\delta_1) \cap (\tau+\delta') ) \cap ((\tau_2+\delta_2) \cap (\tau+\delta') \bigr) \\
  &=& \bigl( \tau+(\delta_1 \cap \delta') \bigr) \cap \bigl( \tau+(\delta_2 \cap \delta') \bigr) = \tau+(\delta_1 \cap \delta_2 ) \in \Pi
  \eens
  This implies that $\Pi$ is a simplicial polyhedral complex. Furthermore, it is a refinement of $\Pi'$ and satisfies $\rec(\Pi)=\Sigma$. Therefore, $\ol{Y}$ coincides with the canonical compactification of $Y$ with respect to $\Pi$, and hence $\ol{Y}$ is a tropical subvariety of  $\TP_{\Xi}$. 

  For statement (2), choose a strictly convex conewise linear function $g'$ on $\Xi$. By definition, there exists an affine function $l$ on $E$ such that $g'>l$ on $E \setminus \{0\}$. Then, $g\coloneqq g'-l$ satisfies $g>0$ on $E \setminus \{0\}$, and Proposition \ref{prop:kahler_form_on_tropical_toric_varieties} implies $[\dm \dmb h]$ is a \kahler class on $\comp{Y}$. Statement (3) is clear from the proof of Statement (1). 
\end{proof}

\subsection{Regularization Theorem} 
The main theorem of this section is Theorem \ref{thm:LConv_regularization_PL}. 
We begin with a result that allows us to construct an $\omega$-psh function $\psi$ on $X$ satisfying a prescribed local property near a fixed point $x \in X$, together with a $C^0$-estimate relative to a given strictly $\omega$-psh function $\varphi$.

\begin{prop} \label{prop:construction_theta_psh_functions_with_given_local_properties}
  Let $\omega$ be a semipositive form. Suppose that  $[\omega]$ is a \kahler class, and let $\varphi \in \pCPSH_{\omega}(X)$.  Then, for every open set $W \subset X$, every $f \in \LConv(W)$, every $x \in W$, every local potential $h$ of $\omega$ on $W$, and every $\epsilon >0$, there exist $C > 0$, $l \in \Aff(V)$, $\psi \in \CPSH_{\omega}(X)$ and an open neighborhood $V \Subset W$ of $x$ such that 
  \bens
  \psi+h=Cf+l \text{ on } V, \qquad  \psi=\varphi \text{ on } X \setminus V, \quad \text{ and   } \varphi \le \psi \le \varphi+\epsilon \text{ on } X.
  \eens
\end{prop}

We recall the following lemma.
\begin{lemma} \label{lem:nested_sequence_of_convex_subset} 
  Let $T$ be a compact topological space, and $x \in T$. Suppose that $K_1 \supset K_2 \supset \dots$ be a nested sequence of closed subset of $T$ satisfying $\bigcap_{i=1}^{+\infty} K_i=\{x \}$. Then for any open neighborhood $U$ of $x$ in $T$, there exists $j \ge 1$ such that $K_{i} \subset U$ for any $i \ge j$.
\end{lemma}
\begin{proof}
  For each $i$, consider the open set $U_i \coloneqq T \setminus K_i$  of $T$. Observe that $\{U, U_i \,|\, i \ge 1 \}$ is an open cover of $T$. Since $T$ is compact, and $U_i \subset U_{i'} $ for any $i \le i'$, there exists $j \ge 1$ such that $T=U_j \cup U$. Therefore, for any $i \ge j$, $K_i \subset K_j=T \setminus U_j \subset U$. 
\end{proof}

\begin{proof} [Proof of Proposition \ref{prop:construction_theta_psh_functions_with_given_local_properties}]
 
  Fix $W \subset X$, $x \in W$,  a local potential $h$ on $W$ and $\epsilon >0$. We may assume without loss of generality that $\overline{W}$ is compact. 
  Denote $g \coloneqq \varphi +h \in \LConv(W)$. By assumption, $g$ is strictly locally convex on $W$. Let $l_g$ be a strict local affine support function of $g$ at $x$, and $l_f$ be a local affine support function of $f$ at $x$.

  For any $C, \delta>0$, define the open subset 
  \[ 
  V(C, \delta)\coloneqq \left\{ y \in \ol{W} \ \st \ g(y)-l_g(y) < C(f(y) -l_f(y)) +\delta \right\} \subset \ol{W}  \]

  Let $\{(C_i, \delta_i)\}_{i \ge 1}$ be a decreasing sequence of pairs of positive numbers converging  to $(0,0)$. Then, by Dini's theorem, $\{C_i(f -l_f) +\delta_i\}_{i \ge 1}$ converges uniformly to $0$ on $\ol{W}$. Therefore, by definition,  
  \[
  \bigcap_i \ol{V(C_i, \delta_i)}=\left\{ y \in \ol{W} \ \st\ f(y)-l_f(y) \le 0 \right\}=\{ x \}.
  \]
  Here $\ol{V(C_i, \delta_i)}$ denotes the closure of $V(C_i, \delta_i)$ in $\ol{W}$. Now,  take an open neighborhood $V$ of $x$ such that $V \Subset W$.  Then, applying Lemma \ref{lem:nested_sequence_of_convex_subset} to $T=\ol{W}$ and $K_i=\ol{V(C_i, \delta_i)} \cap \ol{W}$ and open neighborhood $V$ of $x$, there exists $k \ge 1$ such that $V(C_k, \delta_k) \Subset V$, and $C_k(f-l_f)+\delta_k \le \epsilon$ on $W$. Set $u \coloneqq C_{k}(f-l_f)+\delta_k+l_g \in \LConv(V)$, and set 
  \[
  \psi(y)\coloneqq \begin{cases*}
    \sup \{ \varphi(y), u(y)-h(y)  \}, \qquad y \in V \\
    \varphi(y), \qquad y \in X \setminus V
  \end{cases*} 
  \]
  Remark that $\varphi(y) +h(y) < u(y)$ iff $y \in V(C_k, \delta_k)$. Therefore, $\psi+h=C_{k}f+(\delta_k+l_g-C_k l_f)$ on $V(C_k, \delta_k)$, and $\psi =\varphi$ on $X \setminus V(C_k, \delta_k)$. Furthermore, 
  \[
  0 \le \psi-\varphi=C_{k}(f-l_f)+\delta_k+l_g-g \le C_{k}(f-l_f)+\delta_k \le \epsilon \text{ on } V(C_k, \delta_k). 
  \]
  On the other hand, applying Lemma \ref{lem:construction_sup_omega_psh_functions} to $\varphi \in \CPSH_{\omega}(X)$ and $u \in \LConv(V)$, we have $\psi \in \CPSH_{\omega} (X)$.  Therefore, $\psi$ satisfies all the conditions in the statement, and the proposition is proved. 
  \end{proof}

Now, we prove the following simplicial version of Proposition \ref{prop:construction_theta_psh_functions_with_given_local_properties}.
\begin{prop} \label{prop:construction_theta_psh_PL_functions}
  Let $\omega$ be a semipositive  form. Suppose that $[\omega]$ is a \kahler class, and let $\varphi \in \pCPSH_{\omega}(X)$. Then, for every $\epsilon>0$, there exists $\psi \in \CPSH_{\omega} \cap \PL_{\omega}(X)$ such that $\varphi \le \psi \le \varphi+\epsilon$ on $X$. Moreover, for every open set $W \subset X$, every $f \in \LConv \cap \PL(W)$, every $x \in W$, every local potential $h$ of $\omega$ on $W$, and every $\epsilon >0$, there exist  $C > 0$, $l \in \Aff(V)$, $\psi \in \CPSH_{\omega} \cap \PL_{\omega}(X)$ and an open neighborhood $V \Subset W$ of $x$ such that 
  \bens
  \psi+h=Cf+l \text{ on } V,  \quad \text{ and   } \varphi \le \psi \le \varphi+\epsilon \text{ on } X.
  \eens
\end{prop}
\begin{proof}
  Fix $W \subset X$ and $x \in W$. Fix $f \in \LConv \cap \PL(W)$ and a local potential $h$ of $\omega$ on $W$. Fix $\epsilon>0$. Applying Proposition \ref{prop:construction_theta_psh_functions_with_given_local_properties} to $(\varphi, W, x, f, h, \epsilon)$, we obtain an open neighborhood $V_x \Subset W$ of $x$, and $\psi_x \in \CPSH_{\omega}(X)$ such that $\varphi \le \psi_x \le \varphi+\epsilon$ on $X$, and $\psi_x+h=C_xf+l_x$ on $V_x$ for some constant $C >0$ and $l \in \Aff(V_x)$. Take an open neighborhood $V$ of $x$ such that $V \Subset V_x$.  
  
  For any $y \in X \setminus \ol{V}$, we apply Proposition \ref{prop:construction_theta_psh_functions_with_given_local_properties} to $(\varphi, X \setminus \ol{V}, y, 0, h, \epsilon)$, where $h$ is any local potential of $\omega$ around $y$. Then, we obtain  an open neighborhood $V_{y} \Subset X \setminus \ol{V}$ of $y$, $\psi_y \in \CPSH_{\omega}(X)$, and $l_y \in \Aff(V_y)$,  such that  $\psi_y+h=l_y$ on $V_y$, $\psi_y=\varphi$ on $V \subset X \setminus V_y$, and $\varphi \le \psi_y \le \varphi+\epsilon$ on $X$.  

  Notice that $X$ is compact, and $\{V_x\} \cup  \{V_y \,|\, y \in X \setminus \ol{V}\}$ is an open cover of $X$. Then, there exists a finite subset $\{y_j\}_{j \in J} \subset X \setminus \ol{V}$ such that $X=V_x \cup (\bigcup_{j \in J} V_{y_j})$, where $J$ is a finite index set. We define $\psi \coloneqq  \sup \{ \psi_x, \psi_{y_j} \,|\, j \in J \}$. Then by construction, we have $\varphi \le \psi \le \varphi+\epsilon$ on $X$, and we have $\psi+h=\sup\{C f+l, \varphi+h \}=Cf+l$ on $V$. Furthermore, by Proposition \ref{prop:limit_of_psh_is_psh}\,(3), $\psi \in \CPSH_{\omega}(X)$. 

  To finish the proof of the proposition, it remains to show that $\psi \in \PL_{\omega}(X)$. Let $X=\bigcup_{i \in I} W_i$, where $I$ is a finite index set. Let  $(W_i, \phi_i, W_i \to U_i)$ be  a tropical chart of $X$ for any $i \in I$. Let $h_i$ be a local potential of $\omega$ on $W_i$. We now show that $\psi+h_i \in \PL(W_i)$ for any $i \in I$. To simplify the notation, we denote $V_0 \coloneqq V_x$, $V_j \coloneqq V_{y_j}$ for each $j \in J$, and $J' \coloneqq \{0\} \cup J$. Let $f_0 \in \PL(W_i)$ be an extension of $Cf+l \in \PL(V_0 \cap W_i)$, and $f_j \in \PL(W_i)$ be an extension of $l_{y_j} \in \PL(V_j \cap W_i)$ for each $j \in J$. 
  
  For any $z \in W_i$, by construction, $(\psi+h_i)(z)=\sup \{ (h_i+\varphi)(z), f_j(z)\ |\  j \in J', z \in V_j \}$. Since $h_i+\varphi \le f_j $ on $V_j$ for each $j \in J'$, we have $(\psi+h_i)(z)=\sup \{  f_j(z)\ |\  j \in J', z \in V_j \} \in  \{ f_j(z)\}_{j \in J'}$. Since $(\psi+h_i) \in \nC^{0}(W_i)$ and $\{f_j\}_{j \in J'}$ is a finite set of PL function on $W_i$, by \cite[Def.\,7.10 \& Prop.\,7.13]{AT07}, $(\psi+h_i) \in \PL(W_i)$. The proof is complete.
  \end{proof}

We are now ready to state and prove the following regularization theorem for $\omega$-psh functions.

\begin{thm} \label{thm:LConv_regularization_PL}
  Let $X$ be a compact \prekahler tropical variety. Let $\omega$ be any semipositive form on $X$ and $\theta$ be any \kahler form on $X$.  
  \begin{enumerate} 
    \item \label{item:thm:LConv_regularization_PL_1} For any $\varphi \in \CPSH_{\theta}(X) $, there exists a sequence of functions $\{\varphi_i\}_{i \ge 1} \subset \PSH_{\theta} \cap \PL_{\theta}(X)$ such that $\varphi_i$ converges uniformly to $\varphi$. 
    \item \label{item:thm:LConv_regularization_PL_2} For each $\varphi \in \PSH_{\omega}(X)$, there exist a decreasing sequence of positive numbers $\{\alpha_i\}_{i \ge 1}$ with $\alpha_i \to 0$, and a decreasing sequence of functions $\{ \varphi_i \}_{i \ge 1}$ with $\varphi_i \in \PSH_{\omega + \alpha_i\theta} \cap \PL_{\omega + \alpha_i\theta} (X)$, for each $i \ge 1$, such that $\{ \varphi_i \}_{i \ge 1}$ converges pointwisely to $\varphi$.
  \end{enumerate}
\end{thm}
\begin{remark}
  When $\varphi \in \CPSH_{\omega}(X)$, the sequence of functions $\varphi_i$ in \ref{item:thm:LConv_regularization_PL_2} converges uniformly to $\varphi$ by Dini's theorem.  
\end{remark}
\begin{proof}[Proof of Theorem~\ref{thm:LConv_regularization_PL}]  Set $\theta_i \coloneqq \omega+ \alpha_i \theta$. 
  We first prove statement \ref{item:thm:LConv_regularization_PL_1}. When $\varphi \in \pCPSH_{\theta}(X)$, this is implied by Proposition \ref{prop:construction_theta_psh_PL_functions}. In general, since $0 \in \pCPSH_{\theta}(X)$, we have $(1-\delta)\varphi \in \pCPSH_{\theta}(X)$ for any $\varphi \in \CPSH_{\theta}(X)$ and $\delta>0$. Since $(1-\delta) \varphi$ converges uniformly to  $\varphi$ when $\delta \to 0+$, this concludes statement \ref{item:thm:LConv_regularization_PL_1}.

  Now, we prove statement \ref{item:thm:LConv_regularization_PL_2}. Fix $\varphi \in \PSH_{\omega}(X)$.
  Let $\{\alpha_i\}_{i \ge 1}$ be an arbitrary decreasing sequence of positive numbers such that $\alpha_i \to 0$. 
  Take a sequence of positive numbers $\{C_i\}_{i \ge 1}$  such that $C_i \to +\infty$. Set $\psi'_i\coloneqq \sup \{ \varphi, -C_i\} \in \CPSH_{\omega}(X)$. Remark that $\{ \psi'_i\}_{i \ge 1}$ is a  decreasing sequence converging pointwise to $\varphi$. Observe that for any $i \ge 1$,  $\theta_i=\omega+\alpha_i \theta$ is a \kahler form, and $\psi'_i \in \CPSH_{\theta_i}(X)$. Therefore, statement \ref{item:thm:LConv_regularization_PL_1} implies that for each $i$, one can take $ \psi_i \in \PSH_{\theta_i} \cap \PL_{\theta_i}(X)$ that $  0 \le \psi_i-\psi'_i  \le 1/2^i$ on $X$.  Set $\varphi_i\coloneqq \psi_i+1/2^i$. Then, $\{ \varphi_i \}_{i \ge 1}$ is also a decreasing sequence converging pointwise to $\varphi$. Hence, statement \ref{item:thm:LConv_regularization_PL_2}  is proved.  
  \end{proof}

We end this section with another application of Proposition \ref{prop:construction_theta_psh_PL_functions}.

\begin{cor} \label{cor:dpsh_span_PL_function}
  Let $\theta$ be a \kahler form. For any $f \in \PL(X)$, there exist $C >0$ and $\psi_1, \psi_2 \in \PSH_{\theta} \cap \PL_{\theta}(X)$ such that $f=C(\psi_1-\psi_2)$. 
\end{cor}
\begin{proof}
    For any $x \in X$, there exist an open neighborhood $W_x$ of $x$ and $g_x \in \LConv \cap \PL (W_x)$ such that $f+g_x \in \LConv(W_x)$. Let $h_x \in \LConv(W_x)$ be a local potential of $\theta$. Remark that $0 \in \pCPSH_{\theta}(X)$. Applying Proposition \ref{prop:construction_theta_psh_PL_functions} to $(\varphi=0, W_x, x, h_x, g_x, \epsilon=1)$, there exist $C_x>0$, $\psi_x \in \PSH_{\theta} \cap \PL_{\theta}(X)$, and an open neighborhood $V_x$ of $x$ such that $C_x(\psi_x+h_x)=g_x+l_x$ on $V_x$ for some constant $C_x >0$ and $l_x \in \Aff(V_x)$. Therefore, we have  $C_x(\psi_x+h_x)+f \in \LConv(V_x)$. 
  
  Since $X=\bigcup_{x \in X}V_x$ and $X$ is compact, there exists a finite subset $\{x_j\}_{j \in J} \subset X$ such that $X=\bigcup_{j \in J} V_{x_j}$,  where $J$ is a finite index set. 
  Now, we set 
  \[
  C\coloneqq \sum_{j \in J} C_{x_j}, \quad \text{ and } \quad \psi_2\coloneqq \frac{1}{C} \sum_{j \in J} C_{x_j} \psi_{x_j} \in \PSH_{\theta} \cap \PL_{\theta}(X).
  \]
  Then, for each $j \in J$,  we have 
  \[
  C (\psi_2+h_{x_j}) +f=C_{x_j}(\psi_{x_j}+h_{x_j})+f + \sum_{k \in J, k \neq j} C_{x_k}(\psi_{x_k}+h_{x_k})  \,\,\in \LConv \cap \PL(V_{x_j}). 
  \] 
  Therefore $\psi_1  \coloneqq \psi_2 +\frac{1}{C} f \in \PSH_{\theta} \cap \PL_{\theta}(X)$. 
\end{proof}

\subsection{Strongly semipositive forms}
\begin{definition} \label{def:strongly_semipositive_form}
  Let $\omega$ be a semipositive form. We say that $\omega$ is \emph{strongly semipositive} if  it satisfies the following conditions. For every open set $W \subset \xo$, every $f \in \LConv(W)$, every $x \in W$, and every local potential $h$ of $\omega$ on $W$, there exist $C > 0$, $l \in \Aff(V)$, $\varphi \in \CPSH_{\omega}(X)$ and an open neighborhood $V \Subset W$ of $x$ such that $\varphi+h=Cf+l$ on $V$. Moreover, if $f \in \PL(W)$, then $\varphi$ can be chosen in $\PSH_{\omega}(X) \cap \PL_{\omega}(X)$.
\end{definition}

\begin{remark} \label{rmk:kahler_is_strong_semipositive}
  Remark \ref{rmk:zero_is_strictly_psh_when_theta_is_kahler} together with Propositions \ref{prop:construction_theta_psh_functions_with_given_local_properties} and \ref{prop:construction_theta_psh_PL_functions} imply that any \kahler form is strongly semipositive. We prove in Proposition~\ref{prop:example_weak_kahler_form_on_projective_tropical_variety} below that the same holds for simplicial \kahler forms. We remind the reader that in the statement of  Proposition \ref{prop:construction_theta_psh_functions_with_given_local_properties} and Proposition \ref{prop:construction_theta_psh_PL_functions}, we allow $W$ to be any open subset of $X$; while in the definition of strongly semipositive forms, we require $W$ to be an open subset of $\xo$. 
\end{remark}

\begin{remark}
  The terminology  ``strongly semipositive" is justified by the following. For any $x \in \xo$, there exist $\varphi \in \CPSH_{\omega}(X)$ and an open neighborhood $V \subset \xo$ of $x$ such that $\varphi+h$ is strictly locally convex on $V$, where $h$ is a local potential of $\omega$ on $V$. This is achieved by taking a strictly convex function $f$ on an open neighborhood $W$ of $x$, and then applying the conditions of Definition \ref{def:strongly_semipositive_form} to $W$ and $f$.  
\end{remark}

The next proposition provides examples of strongly semipositive forms that are not Kähler.

\begin{prop} \label{prop:example_weak_kahler_form_on_projective_tropical_variety}
  Let $X \subset \TP_{\Xi}$ be a projective tropical variety, and $g$ be a convex conewise linear function on $\Xi$ such that $g>0$ on $ E \setminus \{0\}$, where $E=(\TP_{\Xi})_{o}=\R^n$. We use  $\omega\coloneqq \dm \dmb g$ to denote the semipositive simplical form on $X$ defined as in Example \ref{example_kahler_form_embedded_tropical_varieties}. Then, $\omega$ is strongly semipositive. 
  
  More precisely, let $U \subset \xo$ be an open subset. Then, for any $x \in U$ and any $f \in \LConv(U)$, there exist an open neighborhood $V \Subset U$ of $x$, a constant $C>0$, and $\varphi \in \CPSH_{\omega}(X)$ such that $\varphi+g=Cf$ on $V$, and $\varphi+g \ge Cf$ on $U$. Furthermore, if $f \in \PL(U)$,  $\varphi$ can be chosen in $\CPSH_{\omega} \cap \PL_{\omega}(X)$. 
\end{prop}
We recall the following basic result before we go to the proof. 
\begin{lemma} \label{lem:growth_rate_of_PL_hom_functions}
  Let $E=\R^n$ with its Eucliedean norm $\left\|\cdot\right\|$, and let $g \in \PL_{\hom}(E)$ satisfy $g> 0$ on $E \setminus \{0\}$. Then, there exists a constant $C>0$ such that $g(y) \ge C \left\| y \right\|$ for each $y \in E$. 
\end{lemma}

\begin{proof}[Proof of Proposition \ref{prop:example_weak_kahler_form_on_projective_tropical_variety}]
  Let $f \in \LConv(U)$. Take a compact neighborhood $U' \Subset U$ of $x$. We may assume $0 \le f \le M$ on $U'$ for some $M>0$. Define $g_x(y) \coloneqq g(y-x)$. Set $h = g|_{\xo}$ and $h_x = g_x|_{\xo}$, where $\xo = X \cap E$. Then, by Example \ref{example_kahler_form_embedded_tropical_varieties}, $h_x-h \in \PSH_{\omega} \cap \PL(X)$.

  By Lemma \ref{lem:growth_rate_of_PL_hom_functions}, there exists $C \gg 0$ such that $\{y \in E \,|\, C h_x (y)-1 \le M \} \subset U'$. Therefore, the open subset $V \coloneqq \{y \in U \,|\, Ch_{x}(y)-1 < f\}$ is contained in $U'$, and hence $V \Subset U$. Since $Ch_x (x)-1=-1$ and $f(x) \ge 0$, we have $x \in V$. 

  Let $\varphi \coloneqq \frac{1}{C}\sup\{C(h_x-h)-1, f-Ch \}$ as defined in Construction \ref{construction:sup_omega_psh_functions}. Then, Lemma \ref{lem:construction_sup_omega_psh_functions} implies that $\varphi \in \CPSH_{\omega}(X)$. By definition, $\varphi+h=\frac{1}{C}f$ on $V$ and $\varphi+h \ge \frac{1}{C}f$ on $U$. Furthermore, it's not hard to see that  $\varphi \in \PSH_{\omega} \cap \PL(X)$ when $f \in \LConv \cap \PL(U)$. Hence, the proposition is concluded. 
\end{proof}

We end this section with an application of Proposition \ref{prop:example_weak_kahler_form_on_projective_tropical_variety}.

  Let $W$ be any open subset of $\xo$. We say that a locally convex function $u$ on $W$ is \emph{approximable}, and write $u \in \appLConv(W)$ if for any $x \in W$, there exist an open neighborhood $V$ of $x$ in $W$ and a function sequence $\{u_i\}_{i \ge 1} \subset \LConv \cap \PL(V)$ converging uniformly to $u$ on $V$. 

\begin{prop} \label{prop:locally_convex_approximable_functions}
  Let $W$ be any open subset of $\xo$. Then $\appLConv(W)=\LConv(W)$, that is, any locally convex function on $W$ is approximable. 
\end{prop}
\begin{proof}
  The inclusion $\appLConv(W) \subset \LConv(W)$ follows directly from Proposition \ref{prop:limit_of_lconv_is_lconv}. We only need to show that $\LConv(W) \subset \appLConv(W)$.
  
  Let $u \in \LConv(W)$ and $x \in W$. Since this is a local statement, one can assume that $W$ is a bounded open subset of a tropical fanfold $\supp\Sigma$, where $\Sigma$ is  embedded in an Euclidean space $E$. 
  By Corollary \ref{cor:refinement_of_open_tropical_varieties_are_projective}, replacing $\Sigma$ with another fan with the same support, we can suppose without loss of generality that the canonical compactification $\comp{\Sigma}$ is a projective tropical variety. This gives $h \in \LConv \cap  \PL_{\hom} (\Sigma)$ such that $h>0$ on $\Sigma \setminus \{0\}$. By Proposition \ref{prop:kahler_form_on_tropical_toric_varieties}, $[\theta]$ is a \kahler class for $\theta \coloneqq \dm \dmb h$. 
 
  Applying Proposition \ref{prop:example_weak_kahler_form_on_projective_tropical_variety} to $\comp{\Sigma}$, $h$, and $u  \in \LConv(W)$, we obtain an open neighborhood $V$ of $x$ in $W$, a constant $C>0$, and $\varphi \in \CPSH_{\theta}(\comp{\Sigma})$ such that $\varphi +h=Cu$ on $V$.

  By Theorem \ref{thm:LConv_regularization_PL} (using the equality $\PL_\theta(\comp\Sigma)=\PL(\comp\Sigma))$, there exists a sequence $\{\varphi_i \}_{i \ge 1} \subset \PSH_{\theta} \cap \PL(\comp{\Sigma})$ such that $\varphi_i$ converges uniformly to $\varphi$. Setting $u_i \coloneqq \frac{1}{C}(\varphi_i+h)$, we have $u_i \in \LConv \cap \PL(V)$, and $\{u_i\}_{i \ge 1}$ converges uniformly to $u$ on $V$. This completes the proof.
\end{proof} 

%%%%%%%%%%%%%%%%%%%%%%%%%%%%%
%%%%%Monge--Ampere measures
%%%%%%%%%%%%%%%%%%%%%%%%%%%%%
\section{\MongeA measures} \label{sec:tropical_Monge_Ampere_measures_paper_2}
\subsection{Bedford--Taylor theory}
In this section, we assume that $X$ is a compact \prekahler tropical variety of dimension $d$, 
and $\omega$ is a semipositive form on $X$. 
We fix a \kahler form $\theta$ on $X$. 
Denote by $\Radonp(X)$ the set of nonnegative Radon measures on $X$. 

Recall the following definition from \cite[\S\,5]{APW1}.

\begin{definition} \label{def:MA_measure_PL_functions_paper_2}
 Let $\omega_1, \dots, \omega_d$ be semipositive forms on $X$, and  let $\varphi_i \in \PL_{\omega_i}(X)$. The \emph{mixed \MongeA measure for PL functions} $\MA_{\omega_1, \cdots, \omega_d}(\varphi_1, \cdots, \varphi_d)$ is the Radon measure on $X$ associated to  
 $\omega_{\varphi_1} \wedge \dots \wedge \omega_{\varphi_d}$. Here, $\omega_{\varphi_i} \coloneqq \omega_i+ \dm \dmb \varphi_i$. 
When no confusion arises, we also write $\omega_{\varphi_1} \wedge \dots \wedge \omega_{\varphi_d}$ for this measure on $X$. 
\end{definition}
 The Radon measure associated to $\omega_{\varphi_1} \wedge \dots \wedge \omega_{\varphi_d}$ is given by the integration of compactly supported top-degree mixed differential forms on $X$, see Section~\ref{subsec:degree_map} and~\cite[\S\,5]{APW1}.

\begin{thm} \label{thm:global_Bedford-Taylor_theory}
  For each $(\omega_1, \dots, \omega_n) \in (\CK^{\ge 0}(X))^d $, there is a unique operator  
  \begin{align*}
  \MA_{\omega_1, \dots, \omega_d} \colon \CPSH_{\omega_1}(X) \times \dots \times \CPSH_{\omega_d}(X)  &\longrightarrow \Radonp(X)\\
  (\varphi_1, \dots, \varphi_d) &\lms \MA_{\omega_1, \dots, \omega_d} (\varphi_1, \dots, \varphi_d)
  \end{align*}
  that extends the \MongeA operator on $\left( \PSH_{\omega_1} \cap \PL_{\omega_1}(X) \right) \times \cdots \times \left( \PSH_{\omega_d} \cap \PL_{\omega_d}(X) \right)$ defined in Definition \ref{def:MA_measure_PL_functions_paper_2}, and satisfies the following properties:
\begin{enumerate}
  \item Let $\{\epsilon^k_i\}_{k \ge 1}$ be a sequence of positive real numbers such that $\epsilon_i^k \to 0$ for each $1 \le i \le d$. Let $\{\varphi_i^k\}_{k \ge 0} \subset \nC^0(X)$ be a function sequence for any $1 \le i \le d$ such that $\{\varphi_i^k\}_{k \ge 0}$ converges uniformly to   $\varphi_i$, and $\varphi_i^k \in \CPSH_{\theta_i^k}(X)$ for each $k \ge 0$ with $\theta_i^k\coloneqq \omega_i+\epsilon^k_i \theta$. Then, \[ \MA_{\theta_1^k, \cdots, \theta_d^k}(\varphi_1^k, \cdots, \varphi_d^k) \rightharpoonup \MA_{\omega_1, \cdots, \omega_d}(\varphi_1, \cdots, \varphi_d),  \quad  (k \to +\infty)\]   
  
  In particular, the operator $\MA_{\omega_1, \cdots, \omega_d}$ is continuous.
  \item Suppose that $(W, \phi: W \to U)$ is a tropical chart of $\xo$, and for any $1 \le i \le d$, $\omega_i=\dm \dmb h_i$ for some $h_i \in \LConv (W)$ on $W$. If $\varphi_i+h_i \in \PSm \cap \LConv(W)$ for each $1 \le i \le d$, then 
  \[\MA_{\omega_1, \cdots, \omega_d} (\varphi_1, \dots, \varphi_d) = \dm \dmb (h_1+\varphi_1) \wedge \dots  \wedge \dm \dmb (h_d+\varphi_d) \]
  \item $\MA_{\omega_1, \cdots, \omega_d}$ is multilinear. 
\end{enumerate}
\end{thm}
\begin{proof}
  The proof is exactly the same as the proof of Theorem 5.13 in \cite{APW1} for the Bedford--Taylor theory of the projective case, except that here we use Theorem \ref{thm:LConv_regularization_PL} which implies that
  $\CPSH_{\omega}(X)=\appCPSH_{\omega}(X)$, and Lemma 5.14 in \textit{loc.cit.} needs to be replaced by Corollary \ref{cor:dpsh_span_PL_function} in this paper. 
%cross-ref in APW1: \ref{thm:projective_Bedford-Taylor_theory} and \ref{lem:dpsh_span_PL_function_projective_tropical_varieties}
\end{proof}

In the remaining part of this subsection, we assume that  $\omega_1, \dots, \omega_{d-1}$ are semipositive  forms on $X$, and $ (\varphi_{1}, \dots, \varphi_{d-1}) \in \CPSH_{\omega_1}(X) \times \dots \times \CPSH_{\omega_{d-1}}(X)$. We denote 
\[
T\coloneqq (\omega_1+\dm \dmb \varphi_1) \wedge \dots \wedge (\omega_{d-1}+\dm \dmb \varphi_{d-1}).
\]
\begin{definition} \label{def:dpsh}
    We say that $f \in \dCPSH(X)$ if there exist $\omega \in \CK^{\ge 0}(X)$ and  $\psi_1, \psi_2 \in \CPSH_{\omega}(X)$ such that $f=\psi_1-\psi_2$. We define $\dm \dmb \varphi \wedge T \coloneqq (\omega+\dm \dmb \psi_1) \wedge T- (\omega+\dm \dmb \psi_2) \wedge T$.
\end{definition}

\begin{cor} [Cauchy--Schwartz inequality] \label{cor:global_Cauchy_Schwartz_inequality}
  For any $f,g \in \dCPSH(X)$, we have
  \[\int_{X} f \dm\dmb f\wedge T \le 0 \]
  Furthermore, 
  \[ \left| \int_{X} f\dm\dmb g \wedge T \right| \leq \left(\int_{X} -f \dm\dmb f\wedge T\right)^{\frac 12}\left(\int_{X} -g \dm\dmb g\wedge T \right)^{\frac 12}\]
\end{cor}
\begin{proof}
  The proof is exactly the same as the proof of Corollary 5.16 in \cite{APW1}. 
  %cross-ref in APW1: \ref{cor:projective_Cauchy_Schwartz_inequality} 
\end{proof}

\begin{cor} \label{cor:global_integration-by-parts_formula_and_volume_omega_psh}
  \begin{enumerate}
    \item For any $\varphi, \psi \in \dCPSH(X)$, we have
    \[\int_{X} \psi\ \dm \dmb \varphi \wedge T =\int_{X} \varphi\ \dm \dmb \psi \wedge T \]
   \item If $(\varphi_1, \dots, \varphi_d) \in \CPSH_{\omega_1}(X) \times \cdots \times \CPSH_{\omega_{d}}(X)$, then, we have 
  \[\MA_{\omega_1, \cdots, \omega_d}(\varphi_1, \dots, \varphi_d) = \deg(\omega_1 \wedge \cdots \wedge \omega_d).\]
  \end{enumerate}
\end{cor}
\begin{proof}
  The proof is exactly the same as the proof of Corollary 5.18
  %\ref{cor:projective_integration-by-parts_formula_and_volume_omega_psh} 
  in \cite{APW1}.
\end{proof}

\begin{thm} \label{thm:Monge-Ampere-measure-takes-no-mass-boundary}
  Let $\varphi \in \CPSH_{\omega}(X)$. Then $\MAo(\varphi)(\partial X)=0$. 
\end{thm}

\begin{remark}
This result was proven for projective tropical varieties in \cite[Thm.\,5.19]{APW1}. The proof presented below for  arbitrary compact \prekahler tropical varieties reduces to the projective case.
\end{remark}

\begin{proof} [Proof of Theorem \ref{thm:Monge-Ampere-measure-takes-no-mass-boundary}]
  Let $x \in \partial X$. We take a basic open neighborhood $U$ of $x$ such that there exists an isomorphism $\phi\colon U \to W \coloneqq \BT_{+}^{k} \times V  \subset  Z_1 \times \BT_{+}^{k}$, and $\phi(x)=(0,\uinfty) \in V\times \BT_{+}^{k}$, where $Z_1$ is a tropical fanfold embedded in $E_1=\R^n$, and $V$ is a bounded open subset of $Z_1$. Set $E_2 =\R^k$ and $E=E_1 \times E_2$. 
  Set $Z=Z_1 \times E_2$. 
  
  We claim that there exists a fan structure $\Sigma$ on $Z$,  $\eta \in \CK^{\ge 0}_{\PL}(\comp{\Sigma})$ (for $\comp\Sigma$ the canonical compactification of $\Sigma$), and  $\psi \in \CPSH_{\eta}(\comp{\Sigma})$ such that $\comp{\Sigma}$ is \prekahler and $(\MA_{\omega}(\varphi))=\phi^{*}(\MA_{\eta}(\psi))$ on an open neighborhood $U' \subset U$ of $x$. Granted this claim, since $\MA_{\eta}(\psi)(\partial \comp{\Sigma}) =0$ by 
  \cite[Thm.\,5.19]{APW1},
  %cross-ref in APW1: \ref{thm:projective_Monge-Ampere-measure-takes-no-mass-boundary}
  and $\partial X \cap U'= \phi^{-1}(\partial \comp{\Sigma} \cap \phi(U'))$, the theorem follows.

  \smallskip

  It remains to prove the claim. We first give the construction of the fan structure $\Sigma$. 
  By Corollary \ref{cor:refinement_of_open_tropical_varieties_are_projective}, there exists a fan structure $\Sigma_1$ on $Z_1$ such that 
  the canonical compactification $\comp{\Sigma}_1$ is projective. Furthermore, there exists a convex conewise linear function $g_1$ on $\Sigma_1$ such that $g_1>0$ on $\Z_1 \setminus \{0\}$, and the numerical class $[\dm \dmb g_1]$ is a \kahler class. On the other hand, let $\Sigma_2$ be the standard complete fan on $E_2$ of the projective space, and let $g_2\coloneqq \max\{0, -x_1, -x_2, \cdots, -x_k\}$. 
  Then, $g_2$ is a strictly convex conewise linear function on $\Sigma_2$, and hence $[\dm \dmb g_2]$ is a \kahler class on $\comp{\Sigma}_2$. Let $\Sigma=\Sigma_1 \times \Sigma_2$, it gives a fan structure on $Z$.  Notice that $\pi_1^{*} [\dm\dmb g_1]+ \pi_2^{*} [\dm\dmb g_2]$ is a \kahler class on $\comp\Sigma$, 
  where $\pi_1 $ and $\pi_2$ are projection maps to $\Sigma_1$ and $\Sigma_2$ respectively. Therefore, $\comp{\Sigma}$ is \prekahler.   

  \smallskip
  
  We proceed to construct $\eta$ and $\psi \in \CPSH_{\eta}(\comp{\Sigma})$. By Theorem \ref{thm:LConv_regularization_PL}, there exists a sequence of positive numbers $\{\alpha_i\}_{i \ge 1}$ satisfying $\alpha_i \to 0$, and there exists a sequence of functions
   $\{\varphi_i\}_{i \ge 1}$ such that $\varphi_i \in \PSH_{\theta_i} \cap \PL_{\theta_i}(X)$ for all $i \ge 1$, and $\varphi_i \rightrightarrows \varphi$, where 
   $\theta_i \coloneqq \omega+\alpha_i \theta$.  Let $h \in \LConv(U)$ and $h_i \in \LConv(U)$ be local potentials of $\omega$ of $\theta_i$ on $U$ respectively such that $h_i \rightrightarrows h$ on $U$. Set $u_i \coloneqq \varphi_i+h_i$ and $u \coloneqq \varphi+h$. Then, we have $u_i \in \LConv \cap \PL(U)$, $u \in \LConv(U)$, and $u_i \rightrightarrows u$ on $U$. Notice that $\ol{U}$ is compact, and $u_i$ are continuous. So after adding constants to $h_i$, we may assume that  $0 \le u_i \le M$ for all $i \ge 1$, where $M>0$ is a constant.
  
  Let $V_0$ be an open neighborhood of $0$ in $Z_1 =\supp{\Sigma_1}$ such that $V_0 \Subset V$. Define $W_0 \coloneqq V_0 \times (1, +\infty]^{k} \Subset W$. Let $g_2'  \coloneqq \max\{ 0, 2-x_1, \cdots, 2-x_k \}$.  Fix $C \gg 0$. Observe that we have $Cg_2'-1 >M$ on $[1, +\infty]^{k} \setminus (1, +\infty]^{k}$. On the other hand, Lemma \ref{lem:growth_rate_of_PL_hom_functions} implies that $Cg_1-1 >M$ on $\partial V_0$ when $C \gg 0$. We denote $g' \coloneqq C(g_1 \circ \pi_1)+C(g'_2 \circ \pi_2)-2$. Therefore, $g'>M$ on $\partial W_0$, and hence $\{ g' < u_i \circ \phi^{-1}\}, \{ g' < u \circ \phi^{-1}\} \Subset W_0$.  
  
  Now, we define $g \coloneqq C(g_1 \circ \pi_1 + g_2 \circ \pi_2)$, and $\eta \coloneqq \dm \dmb g= C(\pi_{1}^{*}\dm\dmb g_1+\pi_{2}^{*}\dm\dmb g_2) \in \CK_{\PL}^{\ge 0}(\comp{\Sigma})$.
  Since $|g'-g|=O(1)$ on $\Sigma$, Example \ref{example_kahler_form_embedded_tropical_varieties} implies that $g'-g$ extends to $\CPSH_{\eta}(\comp{\Sigma})$. Then, 
    \[
  \psi_i \coloneqq \sup\{g'-g, u_i \circ \phi^{-1} \} \in \PSH_{\eta} \cap \PL(\comp{\Sigma}),\,\,i=1,2\,\, \text{and} \,\, 
  \psi \coloneqq \sup\{g'-g, u \circ \phi^{-1} \} \in \CPSH_{\eta}(\comp{\Sigma}),
  \] 
 by Lemma \ref{lem:construction_sup_omega_psh_functions}. Set 
 \[
 v_i \coloneqq \sup\{ g',  u_i \circ \phi^{-1} \} \in \LConv \cap \PL(W_0),\,\, \text{and}\,\, 
 v \coloneqq \sup \{g', u \circ \phi^{-1} \} \in \LConv(W_0).
 \]
 Then, $(\psi_i+g)|_{W_0}=v_i$, and $(\psi+g)|_{W_0}=v$.

  \smallskip

  Set $W' \coloneqq (\{g' < 0 \} \cap W) \subset W$, and $U' \coloneqq \phi^{-1}(W')$. It remains to prove that $\MA_{\omega}(\varphi)=\phi^{*}(\MA_{\eta}(\psi))$ on $U'$. Since  $u_i \circ \phi^{-1} \ge 0$ and $u \circ \phi^{-1} \ge 0$ on $W$ by assumption, we have $v_i=u_i \circ \phi^{-1}$, $v=u  \circ \phi^{-1}$ on $W'$.  By definition (see Definition \ref{def:MA_measure_PL_functions_paper_2}), for any $i \ge 1$, we have $\MA_{\omega}(\varphi_i)=(\dm \dmb (u_i) )^{\wedge n}$ on $U' \cap \xo$,  $\MA_{\eta}(\psi_i)=(\dm \dmb (v_i) )^{\wedge n}$ on $W' \cap \Sigma$, and $\MA_{\omega}(\varphi_i)(U' \setminus \xo)=\MA_{\eta}(\psi_i)(W' \setminus \Sigma)$. Therefore, we conclude that $\MA_{\omega}(\varphi_i)=\phi^{*}(\MA_{\eta}(\psi_i))$ on $U'$.  Now, Theorem \ref{thm:global_Bedford-Taylor_theory}\,(1) implies that $\MA_{\omega}(\varphi) \rightharpoonup  \MA_{\omega}(\varphi)$, and $(\MA_{\eta}(\psi)) \rightharpoonup  (\MA_{\eta}(\psi))$. 
  This proves the claim, and hence finishes the proof of the theorem.
\end{proof}

\subsection{Comparison Principle} We end this section with the following.

\begin{thm}[Comparison Principle]  \label{thm:comparison_principle}
Let $X$ be a compact \prekahler tropical variety, and let $\omega$ be a semipositive form on $X$. Let  $\varphi, \psi \in \CPSH_{\omega}(X)$. 
\begin{equation*}
    \int_{\{ \varphi < \psi \} } \MA_{\omega} (\psi) \le \int_{\{ \varphi < \psi \}} \MA_{\omega} (\varphi).
\end{equation*}
\end{thm}
\begin{proof}
     By Proposition \ref{prop:limit_of_psh_is_psh}\,(3), for any $\epsilon >0$, $\max \{ \varphi, \psi-\epsilon \} \in \CPSH_{\omega}(X)$. Then, for $d=\dim X$, 
    \bens
        && \deg (\omega^{d}) 
        = \int_{X } \MA_{\omega} (\max \{ \varphi, \psi-\epsilon \}) \\
        &\ge& \int\limits_{\{ \varphi < \psi-\epsilon \} } \MA_{\omega} (\max \{ \varphi, \psi-\epsilon \}) + \int\limits_{\{ \varphi > \psi-\epsilon \} } \MA_{\omega} (\max \{ \varphi, \psi-\epsilon \})  \\
        &=& \int\limits_{\{ \varphi < \psi-\epsilon \} } \MA_{\omega} (\psi) + \int\limits_{\{ \varphi > \psi-\epsilon \} }  \MA_{\omega} (\varphi) 
        = \int\limits_{\{ \varphi < \psi-\epsilon \} } \MA_{\omega} (\psi) + \deg(\omega^{d}) - \int\limits_{\{ \varphi \le \psi-\epsilon \} }  \MA_{\omega} (\varphi)
    \eens
    where the first equality comes from Corollary \ref{cor:global_integration-by-parts_formula_and_volume_omega_psh}\,(2). Therefore, we have $ \int_{\{ \varphi < \psi-\epsilon \} } \MA_{\omega} (\psi) \le \int_{\{ \varphi \le \psi-\epsilon \} }  \MA_{\omega} (\varphi)$. 
    Letting $\epsilon \to 0$, we conclude.
\end{proof}

%%%%%%%%%%%%%%%%%%%%%%%%%%%%%
%%%%%%Energy functional
%%%%%%%%%%%%%%%%%%%%%%%%%%%%%
\section{Energy functional} \label{sec:energy_functional}

Let $X$ be a compact \prekahler tropical variety of dimension $d$, and $\omega$ a semipositive form on $X$. We define a tropical analogue of the energy functional on $\PSH_{\omega}(X)$, following \cite[\S\,6]{BFJ15}. 

\subsection{Tropical energy functional $\E$ on $\CPSH_{\omega}(X)$}
For $\varphi \in \CPSH_{\omega}(X)$, we define the value $\E(\varphi)$ of the \emph{tropical energy functional} $\E$ as
\[ 
\E(\varphi) \coloneqq \frac{1}{d+1} \sum_{j=0}^d \int_X \varphi (\omega + \dm \dmb \varphi)^j \wedge \omega^{d-j}.
\]

\begin{prop} \label{prop:energy_concavity_PL}
  $\E$ is non-decreasing, concave on $\CPSH_{\omega}(X)$, and satisfies $\E(\varphi +\lambda) = \E(\varphi) + \lambda \deg(\omega^d)$ for any $\varphi \in \CPSH_{\omega}(X)$ and $\lambda \in \R$.
\end{prop}
\begin{proof}
Let $\varphi, \psi \in \CPSH_{\omega} (X)$. It follows formally 
from an integration-by-parts argument (see e.g. \cite[\S\,1.6]{BJ25} or \cite[Lem.\,6.2]{Tian2012}) that if $\varphi, \psi$ are two  functions in $\CPSH_{\omega}(X)$, then
\begin{equation}
\E(\psi) - \E(\varphi) 
= \frac{1}{d+1} \sum_{j=0}^d \int (\psi - \varphi)(\omega + \dm \dmb \varphi)^j 
\wedge (\omega + \dm \dmb \psi)^{d-j}.
\tag{6.2}
\end{equation}
That $\E$ is non-decreasing then follows from the non-negativity of $(\omega + \dm \dmb \varphi)^j 
\wedge (\omega + \dm \dmb \psi)^{d-j}$ (Theorem \ref{thm:global_Bedford-Taylor_theory}). And $\E(\varphi +\lambda) = \E(\varphi) + \lambda \deg(\omega^d)$ follows from Corollary \ref{cor:global_integration-by-parts_formula_and_volume_omega_psh}\,(2). 

Set $\varphi_t = (1-t)\varphi + t\psi$ for $t\in[0,1]$. Expanding 
$\E(\varphi_t) - \E(\varphi)$ in $t$ 
yields the following formulas for first and second derivatives of $\E$: 

\begin{equation} \label{equation:differential_energy_functional}
  \left.\frac{d}{dt}\right|_{t=0+} \E(\varphi_t) 
= \int (\psi - \varphi)\, \MAo(\varphi), \text{ and }
\tag{6.3}
\end{equation}

\begin{equation}
\left.\frac{d^2}{dt^2}\right|_{t=0+} \E(\varphi_t) 
= d \int ( \psi - \varphi)\ \dm \dmb(\psi - \varphi) \wedge (\omega+\dm \dmb \varphi)^{d-1}.
\tag{6.4}
\end{equation}
The Cauchy--Schwartz inequality (Corollary \ref{cor:global_Cauchy_Schwartz_inequality}) implies that $\left.\frac{d^2}{dt^2}\right|_{t=0+} \E(\varphi_t) \le 0$. Therefore, the functional $\E$ is concave. 
\end{proof}

\begin{cor} \label{cor:estimation_difference_energy_functional}
  For each pair $\psi, \varphi \in \CPSH_{\omega}(X)$, we have 
  \bens
  \int_X (\psi-\varphi) \MAo(\psi) \le \E(\psi)-\E(\varphi) \le \int_X (\psi-\varphi) \MAo(\varphi)
  \eens
\end{cor}
\begin{proof}
  For each $t\in[0,1]$, set $\varphi_{t}\coloneqq(1-t)\varphi+t\psi$ and $\psi_t\coloneqq(1-t)\psi+t\varphi$. The concavity of the energy functional implies that \[ \left.\frac{d}{dt}\right|_{t=1-}\E(\varphi_t) \le \E(\varphi_1)-\E(\varphi_0) \le \left.\frac{d}{dt}\right|_{t=0+} \E(\varphi_t). \]
  Applying Equation \ref{equation:differential_energy_functional}, we get 
  \begin{align*}
    \int_X (\psi-\varphi) \MAo(\psi) &= - \left.\frac{d}{dt}\right|_{t=0+}\E(\psi_t) = \left.\frac{d}{dt}\right|_{t=1-}\E(\varphi_t) \\  &\le   \E(\psi)-\E(\varphi) \le \left.\frac{d}{dt}\right|_{t=0+} \E(\varphi)= \int_X (\psi-\varphi) \MAo(\varphi).  \qedhere
  \end{align*}
\end{proof}

\subsection{Energy functional for $\omega$-psh functions}
For $\varphi \in \PSH_{\omega}(X)$ and $C \in \R$, we denote $\varphi_C \coloneqq \max \{\varphi, C\} \in \CPSH_{\omega}(X)$. We extend the tropical energy functional $\E \colon \PSH_{\omega}(X) \to \R \cup \{- \infty \}$ by  
\[\E(\varphi) \coloneqq \lim_{C \to -\infty} \E(\varphi_C).\]

\begin{prop} \label{prop:energy_concavity}
  The energy functional $\E$ is well-defined on $\PSH_{\omega}(X)$. It is non-decreasing, concave, and satisfies $\E(\varphi +\lambda) = \E(\varphi) + \lambda \deg(\omega^d)$ for any $\lambda \in \R$. It is also upper semicontinuous, and
  continuous along decreasing sequences on $\PSH_{\omega}(X)$.
\end{prop}
\begin{proof} 
  Proposition \ref{prop:energy_concavity_PL} implies that $\E(\varphi_C)$ is decreasing in $C$. Therefore, $\lim\limits_{C\to -\infty}\E(\varphi_C)$ exists.
  That $\E$ is non-decreasing, concave, and $\E(\varphi +\lambda) = \E(\varphi) + \lambda \deg(\omega^d)$ then follows formally from Proposition \ref{prop:energy_concavity_PL}. That $\E$ is continuous along decreasing sequences follows from the upper semicontinuity and monotonicity of $\E$. It remains to prove the upper semicontinuity for $\E$.
  
  Let $ \varphi  \in \PSH_{\omega}(X)$ and $M \in \R$ such that $\E(\varphi) <M$. We need to find an open neighborhood $U$ of $\varphi$ inside $\PSH_{\omega}(X)$ such that $\E(\varphi')<M$ for any $\varphi' \in U$.  By definition, we may choose $C \in \R$ and $\lambda >0$ such that $\E(\varphi_{C}+\lambda)<M$. Since $\PL_{\omega}(X)$ is dense in $\nC^0(X)$, we may further choose $\psi \in \PL_{\omega}(X)$ such that $\varphi_C < \psi < \varphi_C +\lambda$. 
  
  Consider the functional on $\PSH_{\omega}(X)$ given by $\varphi'  \mapsto \sup_{X} (\varphi'- \psi )$. 
  We choose a polyhedral structure $\olPi$ of $X$ with induced polyhedral structure on $\xo$ such that $\psi$ is $\omega$-facewise affine on $\olPi$. By Lemma \ref{lem:envelope_PL_functions_vertices}, $\sup_X (\varphi'- \psi )=\sup(\varphi'-\psi)\rest{\Pi_0}$. Since $\Pi_0$ is a finite subset of $\xo$, the functional $\varphi'  \mapsto \sup_{X} (\varphi'- \psi )$ is continuous on $\PSH_{\omega}$.

  Suppose $\sup_X(\varphi-\psi)=-\epsilon$ with $\epsilon>0$. Then, 
  \[
  U \coloneqq \left\{ \varphi' \in \PSH_{\omega}(X) \ |\ -2 \epsilon<\sup_X(\varphi'-\psi) < 0 \right\}
  \]
  is an open neighborhood of $\varphi$. For any $\varphi' \in U$, we have $\varphi' \le \envomega(\psi) \le \psi < \varphi_{C}+\lambda$. Hence, the monotonocity of $\E$ implies that $\E(\varphi') \le \E(\varphi_{C}+\lambda) <M$. This concludes the upper semicontinuity of $\E$. 
\end{proof}

\subsection{The Lagrangian functional} Given a nonnegative Randon measure $\mu$ with compact support in $\xo$ with $\vol(\mu)=\deg(\omega^{d})$, define the Lagrangian functional 
\bens
\FE_\mu \colon \PSH_{\omega}(X) &\longrightarrow& \R \cup\{-\infty\} \\ 
\varphi &\lms& \E(\varphi)-\int_X \varphi\ d\mu.
\eens

Corollary \ref{cor:global_integration-by-parts_formula_and_volume_omega_psh}\,(2) and Proposition \ref{prop:energy_concavity} imply that $\FE_{\mu}(\varphi+\lambda)=\FE_{\mu}(\varphi)$ for any constant $\lambda \in \R$. Recall  that $\PSH_{\omega, x}(X) \coloneqq \{\varphi\in \PSH_{\omega}(X) \,|\, \varphi(x)=0\}$, and $\CPSH_{\omega, x}(X) \coloneqq \{\varphi\in \CPSH_{\omega}(X) \,|\, \varphi(x)=0\}$. 
\begin{prop} \label{prop:maximizer_Lagrangian}
  The functional $\FE_{\mu}$ is upper semicontinuous on $\PSH_{\omega, x}(X)$, and it admits a maximizer $\varphi_{\mu} \in \mathrm{CPSH}_{\omega,x}(X)$. In particular, $\E(\varphi_{\mu})$ is finite.
\end{prop}
\begin{proof}  Let $K$ be a compact subset of $\xo$ containing the support of $\mu$.

  We first prove the upper semicontinuity of $\FE_{\mu}$. The upper semicontinuity of the functional $\E$ is already established in Proposition \ref{prop:energy_concavity}. It remains to prove that the functional $\varphi \mapsto \int_{X} \varphi d\mu$ is continuous. Let $\{ \varphi_i \}_{i \ge 1} \subset \PSH_{\omega, x}(X)$ be a sequence of functions converging to $\varphi$ in $\PSH_{\omega, x}(X)$, that is, $\{\varphi_i\}_{i \ge 1}$ converge pointwise to $\varphi$ on $\xo$. Then by Theorem \ref{thm:lipschitz}, $\varphi_i$ are $C_K$-Lipschitz on $K$ for a uniform constant $C_K>0$. Therefore, $\{ \varphi_i \}_i$ converges uniformly to $\varphi$ on $K$, and the continuity of the above functional follows.

  Since $\PSH_{\omega,x}(X)$ is compact (Proposition \ref{prop:compactness_of_psh}), $\FE_{\mu}$ admits a maximizer $\psi$ on $\PSH_{\omega,x}(X)$. Note that $\psi\rest{K}$ is continuous. Hence, there exists $C \in \R$ such that $\psi \ge C$ on $K$. Set $\varphi_{\mu}=\sup \{\psi, C \}$. 
  By monotonicity property of $\E$ (Proposition \ref{prop:energy_concavity}), $\E(\varphi_{\mu}) \ge \E(\psi)$. At the same time, we have $\psi\rest{K} = \varphi_{\mu}\rest{K}$. Therefore, $\int_X \varphi_{\mu}\, d\mu=\int_X \psi\, d\mu$. Hence, $\varphi_{\mu}$ is a maximizer of $\FE_{\mu}$ satisfying $\varphi_{\mu} \in \CPSH_{\omega,x}(X)$.
\end{proof}
 
\begin{prop}\label{prop:maximizer} 
  Let $\mu$ be a nonnegative Radon measure with compact support in $\xo$. Assume that $\varphi \in \CPSH_{\omega,x}(X)$ satisfies $\MA_{\omega}(\varphi)=\mu$. Then, $\varphi$ is a maximizer of $\FE_{\mu}$. 
\end{prop}
\begin{proof}
For any $\psi \in \PSH_{\omega,x}(X)$, on the segment joining $\varphi$ and $\psi$ in $\PSH_{\omega,x}(X)$, we have
\[ \left.\frac{d}{dt}\right|_{t=0+} \FE_{\mu}(t\varphi +(1-t)\psi) = \int_X (\psi-\varphi) (\MA_{\omega}(\varphi) -\mu) =0.\]
Since $\FE_{\mu}$ is concave, this implies that $\FE_{\mu}(\varphi) \geq \FE_{\mu}(\psi)$, and the proposition follows.
\end{proof}

\section{Tropical \MongeA equations} \label{sec:uniqueness}

 Let $X$ be a compact \prekahler tropical variety of dimension $d$, and let 
 $\omega$ be a strongly semipositive  form on $X$.  We say that the \emph{pair $(X, \omega)$ verifies the tropical Calabi--Yau theorem} if for any nonnegative Radon measure $\mu$ compactly supported  in $\xo$ of total volume $\deg(\omega^d)$, there exists a unique $\varphi \in \CPSH_{\omega}(X)$ up to a constant to the tropical Monge--Ampère equation:
\begin{align}\label{eq:MA}
 \MA_\omega(\varphi) = \mu.
\end{align}

In this section, we prove the uniqueness part of the tropical Calabi--Yau statement.

\begin{thm}\label{thm:uniqueness} Let $X$ be a compact \prekahler tropical variety of dimension $d$, and let $\omega$ be a strongly semipositive form on $X$. Suppose that  $\varphi, \psi \in \CPSH_{\omega}(X)$ satisfy $\MA_{\omega}(\varphi)=\MA_{\omega}(\psi)$. Then, $\varphi -\psi$ is constant on $X$.
\end{thm}

\begin{remark}
  Note that $\deg(\omega^d)>0$ when $\omega$ is strongly semipositive.  To see this, take a simplicial polyhedral structure $\olPi$ on $X$ with induced polyhedral structure $\Pi$ on $\xo$. Take $\tau \in \Pi_0$ with $\dim \tau =\dim X$, and take $x \in \relint(\tau)$. Let $(t_1, \cdots, t_d)$ be affine coordiates on the $d$-dimensional simplicial polyhedron $\tau$, and set $u \coloneqq t_1^2+ \cdots + t_d^2 \in \LConv(\relint(\tau))$.  Then, by definition (Definition \ref{def:strongly_semipositive_form}), there exists an open neighborhood $U \Subset \relint(\tau)$ of $x$, $C>0$, $l \in \Aff(U)$, a potential $h$ of $\omega$ on $V$ and $\varphi \in \CPSH_{\omega}(X)$ such that $\varphi+h=Cu+l$. Then, we have $\MA_{\omega}(\varphi)(U)= C^d \cdot \weight(\tau)\cdot \MA_{\R}(u)(U)>0$. Thus, by Corollary \ref{cor:global_integration-by-parts_formula_and_volume_omega_psh} (2),  we obtain that $\deg(\omega^d)=\MA_{\omega}(\varphi)(X) \ge \MA_{\omega}(\varphi)(U)>0$.
\end{remark}

\begin{remark}   
  One can formulate the tropical Calabi--Yau theorem in exactly the same way for any semipositive form $\omega$ with $\deg(\omega^d)>0$.  We are however only able to prove the uniqueness part assuming strong semipositivity, and do not know whether this holds in general for any semipositive form. Note also that a semipositive form $\omega$ with $\deg(\omega^d)>0$ is not necessarily strongly semipositive (see Remark \ref{rmk:example_semipositive_form_with_positive_degree_but_not_strongly_semipositive}).
\end{remark}

In preparation for the proof, we have the following two lemmas. 
\begin{lemma} \label{lem:holder_continuity_J-functional}
  Under the hypotheses of Theorem \ref{thm:uniqueness}, for any $f \in \PL(X)$, and for any $\widetilde{\varphi}\in \PSH_{\omega} \cap \PL_{\omega}(X)$, we have \[\int_{X} (\psi-\varphi) \dm\dmb(f) \wedge (\omega+\dm \dmb \widetilde{\varphi})^{d-1} =0.\]
\end{lemma}
\begin{proof}
  The proof proceeds as in the corresponding step in the proof of \cite[Prop.\,8.2]{BFJ15}. The only inputs are the Cauchy--Schwartz inequality and the integration-by-parts formula, whose tropical versions are proven in Corollary \ref{cor:global_Cauchy_Schwartz_inequality} and Corollary \ref{cor:global_integration-by-parts_formula_and_volume_omega_psh}\,(1), respectively.
\end{proof}

\begin{lemma} \label{lem:PL_functions_with_strict_positive_MA_measures}
  For any tropical fan $\Sigma$ in $E =\R^n$, there exists $g \in \LConv \cap \PL_{\hom}(|\Sigma|)$ such that $\MA(g)\bigl(\{0\}\bigr)>0$. 
\end{lemma}
Here $\MA$ denotes the local tropical \MongeA operator on $\Sigma$, see~\cite[\S\,5]{APW1}.

\begin{proof}[Proof of Lemma~\ref{lem:PL_functions_with_strict_positive_MA_measures}]
Take $h \in \Conv \cap \PL_{\hom}(E)$ such that $h(0)=0$, and $h>0$ on $E \setminus \{0\}$. Let $g\coloneqq h\rest{\Sigma}$. It is not hard to show that $\MA(g)$ is supported at $\{0\}$. We prove that $\MA(g)\bigl(\{0\}\bigr)>0$. Let $\sigma$ be a facet of $\Sigma$, and let $\{x_i\}_{i} \subset \relint(\sigma)$ such that $x_i \to 0$. We define $g_i \in \LConv(|\Sigma|)$ by $g_i(x)\coloneqq h(x-x_i)$ for any $x \in \supp{\Sigma}$. Then, $g_i$ converge uniformly to $g$ in a neighborhood $U$ of $0$, so $\MA(g_i) \rightharpoonup \MA(g)$ on $U$. On the other hand, there exists a positive constant $c$ depending only on $\sigma$ such that $\MA(g_i)\bigl(\{x_i\}\bigr)=c\MA_{\R}(g_i\rest{\sigma})\bigl(\{x_i\}\bigr)=c\MA_{\R}(h\rest{\ssT_{\sigma}})(\{0\})$, where $\ssT_{\sigma}$ is the tangent space to $\sigma$ (the linear span of $\sigma$), and $\MA_{\R}$  refers to the  real Monge--Ampère operator on $\sigma$ with respect to the lattice structure on $\sigma$.   Notice that $h>0$ on $\ssT_{\sigma} \setminus \{0\}$, so $\MA_{\R}(h\rest{\ssT_{\sigma}})\bigl(\{0\}\bigr)>0$. Therefore, $\MA(g)\bigl(\{0\}\bigr) \ge c\MA_{\R}(h\rest{\ssT_{\sigma}})\bigl(\{0\}\bigr)>0$.
\end{proof}

  %cross-ref in APW1: Proposition in section 5

We are ready to prove Theorem \ref{thm:uniqueness}.

\begin{proof}[Proof of Theorem~\ref{thm:uniqueness}]  \label{proof:thm:uniqueness}
  We show that for any simpicial polyhedral structure $\olPi$ on $X$ with induced polyhedral structure $\Pi$ on $\xo$, the restriction of $\varphi-\psi$ on the set of vertices $\Pi_0$ is constant. Taking refinements of the polyhedral structure and using the continuity of $\varphi$ and $\psi$, we infer that $\varphi-\psi$ is constant on $X$. 

  Consider the set $\Pi_{\f,1}$ of finite edges in $\Pi$. Let $\Gamma$ be the metric graph with vertex set $\Pi_0$, edge set $\Pi_{\f,1}$, with the primitive vector of each edge having length one. So, $\Gamma$ has the same support as the one-skeleon $\Pi_{\f,(1)} = \Pi_0\cup\Pi_{\f,1}$, and the length of each edge is given in terms of the primitive vector of that edge. By connectivity of $X$, $\Gamma$ is connected. We will construct below a semipositive simplicial form $\omega'$ defined on some simplicial refinement $\olPi'$ of $\olPi$ such that $\Gamma$ is included in the support of $(\omega')^{d-1}$, which is a nonnegative Minkowski weight of dimension one on $\Pi'$ and belongs to $\MWp_{1}(\Pi')$.  
  
  \smallskip

  Granted the existence of such $\omega'$, we explain how to finish the proof.
   The weight on each edge $e'\in \Pi'_{\f,1}$ is given by $\deg_{e'}\bigl((\omega')^{d-1}\bigr)$. 
    Denote by $\Gamma'$ the support of $(\omega')^{d-1}$. Endowed with the weight function $\deg_{e'}\bigl((\omega')^{d-1}\bigr)>0$ for each $e'$ in the support, this is a tropical subvariety of dimension one in $\xo$. 
   We associate edge lengths to finite edges in $\Pi'_1$ as in the definition of edge lengths in 
   $\Gamma$ above, by declaring the primitive vector of the edge to be of norm one. For each edge $e'$ in $\Pi'_{\f,1}$, denote by $l_{e'}$ its length. By the assumption, $\Gamma'$ contains $\Gamma$.  
  
  Let now $f$ be a facewise affine function on $\Pi'$ such that $f(v)=\varphi(v)-\psi(v)$ for any $v \in \Pi'_0$. Applying Lemma \ref{lem:holder_continuity_J-functional}, combined with Proposition 4.12 and Lemma 5.10 in~\cite{APW1}, yields
  %cross-ref in APW1: \ref{lem:tropical_curve_differential_calculations}, Proposition 4.12 in~\cite{APW1}
  %cross-ref in APW1: \ref{prop:compatibility_integral_divisors} 
  \bens
   -\sum_{e' =uv \in \Pi'_{\f,1} } l_{e'}^{-1} \left(f(u)-f(v)\right)^2  =\int_{
  \Gamma'} (\varphi-\psi)\dm\dmb f =\int_{X} (\varphi-\psi)\dm\dmb f \wedge (\omega')^{d-1 }=0.
  \eens
  Since $l_{e'} \in \R_+$, this implies that $f(u)=f(v)$ for any finite edge $e'=uv$ in $\Pi'_{\f,1}$. 
  Since $\Gamma$ is included in $\Gamma'$ and connected, $f$ is constant on the set of vertices of $\Gamma$, and hence $\varphi-\psi$ is constant on $\Pi_0$, as desired. 

  \smallskip

  It remains to construct such $\omega'$. We may assume that for any $e \in \Pi_{\f,1}$, 
  there exists a tropical chart $(W, \phi\colon W \to U)$ with $U \subset \supp{\Sigma} \subset E=\R^n$ for $\Sigma$ a tropical fan of dimension $d$ such that $W$ contains $e$, and such that, viewed in $U$, $e$ is a segment on a ray $\rho$ of $\Sigma$ with one extremity the origin of $E$. 
  Slightly abusing of notation, we identify $U$ and $W$. We use $L_e$ to denote the $1$-dimensional subspace of $E$ generated by $e$. Let $\pi\colon E \to E/L_{e}$ be the projection map and $\Sigma'\coloneqq \pi(\Sigma)$ be the image of $\Sigma$ with its structure of tropical fan, see Section~\ref{subsec:tropical-varieties}. By Lemma \ref{lem:PL_functions_with_strict_positive_MA_measures}, we can take $g' \in \LConv\cap \PL(\Sigma')$ such that $(\dm \dmb g')^{d-1} (\{0\})>0$. Set $g\coloneqq g' \circ \pi$.

  Let $\div(g)$ denote the real Minkowski weight of dimension $d-1$ associated to $g$. The support of $\div(g)$ coincides with the support of $\dm\dmb g$, and the weight on each $\tau \in \Pi_{\f,d-1}$ in the support is given by $\deg_\tau(\dm\dmb g)$. Since $\Sigma'$ is isomorphic to the star fan $\Sigma^{\rho}$, we have $\bigl(\div(g)\bigr)^{d-1}(\rho)=(\dm \dmb g')^{d-1} (\{0\})>0$.
  
  By definition of the strongly semipositive forms (Definition \ref{def:strongly_semipositive_form}), for any $x \in e$, there exist $\varphi_x \in \PSH_{\omega} \cap \PL_{\omega}(X)$,  an open neighborhood $U_x$ of $x$, a constant $C>0$,  local affine function $l$ on $U_x$, and $h \in \LConv \cap \PL_{\omega}(U_x)$ verifying $\omega=\dm \dmb h$ such that  $\varphi_x+h=Cg+l$ on $U_x$. Let $\Pi_x$ be any refinement of $\Pi$ such that $\varphi_x$ is facewise affine on $\Pi_x$. 
  Therefore, for any $e' \in (\Pi_{x})_{\f,1}$ such that $e' \subseteq e$ and $ U_x \cap e' \neq \emptyset$, we have $\deg_{e'}(\omega_{\varphi_x}^{d-1})>0$. Since $e$ is compact, there exists a finite collection of points $x_i \in e$, $i\in I_e$ such that $e \subset \bigcup_{i \in I_e} U_{x_i}$. Let $I$ be the disjoint union of $I_e$, for $e\in \Pi_{\f,1}$.
   
 Set
  \[
    \omega'\coloneqq \omega+\frac 1{|I|} \sum_{e\in \Pi_{\f, 1}}\sum_{i \in I_e} \dm\dmb\varphi_{x_i},
  \] 
and let $\olPi'$ be a simplicial refinement of $\olPi$ such that $\omega'$ is simplicial on the induced polyhedral structure $\Pi'$ on $\xo$. 
Then, for any $e' \in \Pi'_{\f,1}$ such that $e' \subseteq e$, for $e\in \Pi_{\f,1}$, 
we have $\deg_{e'}((\omega')^{d-1}) \ge \frac{1}{|I|} \deg_{e'}(\omega_{\varphi_{x_i}})^{d-1}(e')>0$, where $i \in I_e$ satisfies  $e' \cap U_{x_i} \neq \emptyset$. This implies that each edge $e\in \Pi_{\f,1}$ is in the support of $(\omega')^{d-1}$, as desired. 
  \end{proof}

%%%%%%%%%%%%%%%%%%%%%%%%%%%%%
%%%%Orthogonality property
%%%%%%%%%%%%%%%%%%%%%%%%%%%%%
\section{The orthogonality property and the existence of solutions}  \label{sec:orthogonality_properties}

In this section, $X$ is a compact \prekahler tropical variety and $\omega$ is a strongly semipositive (1,1)-form on $X$. The main result of this section is Theorem~\ref{thm:equivalence_orthogonality_existence_of_solutions} which relates the validity of the tropical Calabi--Yau theorem and the following orthogonality property.

\begin{defi}[Orthogonality property] Suppose that $f\colon X \to \R \cup \{+ \infty \}$ is an lsc function which is not constantly $+\infty$ in $\xo$. 
  We say that $f$ \emph{verifies the orthogonality property} if we have 
    \[
    \int_{X} (f-\env_{\omega}(f)) \MA_{\omega}({\env_{\omega}(f)}) =0. \qedhere
    \]
\end{defi}

\begin{thm} \label{thm:equivalence_orthogonality_existence_of_solutions}
  Let $(X, \omega)$ be a compact \prekahler tropical variety equipped with a strongly semipositive (1,1)-form. The followings are equivalent:
  \begin{enumerate}
    \item \label{orth1} For each finite collection of  pairs $\CS=\{ (x_j, \lambda_j) \in \xo \times \R \}_{1 \le j \le k}$, we have 
    \[
    \Supp(\MAo(\env_\omega(f_{\CS}))) \subset S. 
    \]
    Here, $S \coloneqq \{x_j\}_{1 \le j \le k} \subset \xo$, and 
    \bes
    f_{\CS}(x) \coloneqq
    \begin{cases}
      \lambda_i &\text{ if } x=x_i \text{ for some } x_i \in S, \\
      +\infty & \text{ if } x \in X \setminus S.
    \end{cases}
    \ees
    is the lsc function on $X$ as defined in Example \ref{example_functions_with_values_at_finite_points}. 
    \item \label{orth2} Each lsc function $f\colon X \to \R \cup \{+\infty \}$ which is  not constantly $+\infty$ on $\xo$ verifies the orthogonality property.  
    \item \label{orth3} Each function $f\in \nC^{0}(X)$ verifies the orthogonality property. 
    \item \label{orth4} $(X, \omega)$ verifies the tropical Calabi--Yau theorem. 
  \end{enumerate}
  \end{thm}
The rest of this section is devoted to the proof of this theorem.
\subsection{Orthogonality implies derivability} 
We start by proving that the orthogonality property implies derivability of the energy functional. 
\begin{thm}\label{thm:derivation} Let $f \in \CPSH_{\omega}(X)$ and $g \in \nC^{0}(X)$. Let $\epsilon >0$.  Suppose that for all $t \in (-\epsilon, \epsilon)$, $f+tg$ verifies the orthogonality property. Then, 
\be \label{aim_equation_thm:derivation}
\left.\frac{d}{dt}\right|_{t=0} \E(\env_{\omega}(f+tg)) =\int_X g\ \MA_{\omega}(f).
\ee
\end{thm}
\begin{proof}
  Let $t \in (-\epsilon, \epsilon)$. We apply Corollary \ref{cor:estimation_difference_energy_functional} to $\envomega(f)$ and $\envomega(f+tg)$:
  \bens
  \int_X (\envomega(f+tg)-\envomega(f))\ \MAo(\envomega(f+tg)) & \le & \E(\envomega(f+tg))-\E(\envomega(f)) \\ & \le & \int_X (\envomega(f+tg)-\envomega(f))\ \MAo(\envomega(f)).
  \eens
Using the orthogonality property for $f+tg$ and the inequality $\envomega(f) \le f$, we obtain
\bes
t \int_X g\ \MAo(\envomega(f+tg))\le \int_X (\envomega(f+tg)-\envomega(f))\ \MAo(\envomega(f+tg)).
\ees
Similarly, we have 
\bes
\int_X (\envomega(f+tg)-\envomega(f))\ \MAo(\envomega(f)) \le t \int_X g\ \MAo(\envomega(f)).
\ees
On the other hand, the Bedford--Taylor theory (Theorem \ref{thm:global_Bedford-Taylor_theory}) implies that 
\[ 
\lim_{t \to 0} \int_X g\ \MAo(\envomega(f+tg)) = \int_X g\ \MAo(\envomega(f)).
\]
Combining the above results, Equation \eqref{aim_equation_thm:derivation} follows. 
\end{proof}

\subsection{Proof of Theorem~\ref{thm:equivalence_orthogonality_existence_of_solutions}}

$\ref{orth1} \Longrightarrow  \ref{orth2}$. Let $f\colon X \to \R \cup \{ + \infty \} $ be an lsc  function which is not constantly $+\infty$ on $\xo$. Since $X$ is a second countable space, one can take  a countable dense subset  $A_c$ of $ \{x \in \xo \,\st\, f(x) \le c \}$ for each $c \in \BQ$. Set 
\[
A=\bigcup_{c \in \BQ} A_c.
\]
Take a nested sequence of  finite subsets $S_1 \subset S_2 \subset S_3 \subset \dots$ of $A$ such that 
$A=\bigcup_{0}^{\infty} S_i$.  For each $i$, take the finite collection of pairs $\CS_i=\{ (x, f(x)) \}_{x \in S_i}$ and set $f_i\coloneqq f_{\CS_i}$. Note that $\{f_i\}_{i \ge 1}$ is a decreasing sequence of  functions converging pointwise to $f$. By Remark \ref{rmk:operator_psh_envelope_monotone_concave}, $\{\envomega(f_i)\}_{i \ge 1}$ is a decreasing sequence of $\omega$-psh functions. Since $X$ is compact, there exists $C \in \R$ such that $f \ge C$ on $X$. This implies that for each $i \ge 1$, we have $f_i \ge C$, and hence $\envomega(f_i) \ge C$. Thus, by Corollary \ref{cor:decreasing_limit_of_psh_is_psh}, $\envomega(f_i)$ converges uniformly to some function $\varphi \in \CPSH_{\omega}(X)$. 

We claim that $\varphi=\envomega(f)$. In fact, for each $i$, we have $\envomega(f_i)  \ge \envomega(f)$, hence $\varphi \ge \envomega(f) $. On the other hand, $\varphi \le f$ on $S_i$ for each $i \ge 1$. This implies that $\varphi \le f$ on $A$.  So $\varphi \le c$ on $A_c\subset \{f \le c \}$ for any $c \in \BQ$. Since $A_c$ is dense in $\{ f \le c \}$ and $\varphi \in \nC^{0}(X)$, it follows that  $\varphi \le c$ on $\{ f \le c \}$. Therefore, $\varphi \le f$ on $X$ and hence $\varphi \le \envomega(f)$, proving the claim.

\smallskip

Let $i \ge 1$. Define the finite subset $S'_i \coloneqq \{y \in S_i \ |\ \envomega(f_i)(y)=f(y) \} \subset X$, and the finite collection of pairs $\CS'_i \coloneqq \{(y, f(y))\ |\ y \in S'_i \}$. By assumption, we have $\envomega(f_{\CS'_i})=\envomega(f_i)$, and  $\Supp(\MAo(\envomega(f_{\CS'_i}))) \subset S'_i$. Therefore, 
\[
\int_{X} (f-\envomega(f_i))\, \MAo(\envomega(f_i))=\int_{S'_i} (f-\envomega(f_{\CS'_i}))\, \MAo(\envomega(f_{\CS'_i}))=0.
\]
Then, we have 
\bens
0 &\le& \int_{X} (f-\envomega(f))\, \MAo(\envomega(f)) \\
&=& \lim_{i \to +\infty} \left( \int_{X} (f-\envomega(f))\, \MAo(\envomega(f))-\int_{X} (f-\envomega(f_i))\, \MAo(\envomega(f_i)) \right) \\
&=& \lim_{i \to +\infty} (R_{1,i}+R_{2,i} ),
\eens
with 
\begin{align*}
R_{1,i} &\coloneqq \int_X (\envomega(f_i)-\envomega(f))\, \MAo(\envomega(f_i)), \,\, \text{and} \\
R_{2,i} &\coloneqq \int_{X} (f-\envomega(f))\, (\MAo(\envomega(f)) - \MAo(\envomega(f_i))). 
\end{align*}
On the one hand,
\bens
\lim_{i \to +\infty} R_{1,i}
\le \deg(\omega^d) \cdot  \lim_{i \to +\infty} \left\| \envomega(f) - \envomega(f_i) \right\|_{\nC^0(X)}  =0
\eens
On the other hand, Theorem \ref{thm:global_Bedford-Taylor_theory} implies that  $\MAo(\envomega(f_i)) \rightharpoonup  \MAo (\envomega(f))$. Since $(f-\envomega(f))$ is a nonnegative lsc function, Portmanteau theorem~\cite[Cor.\,2.2.6]{Bog18} implies that
\[
\lim_{i \to +\infty} R_{2,i} \le \liminf_{i \to +\infty} \left( \int_{X} (f-\envomega(f))\, \MAo(\envomega(f)) -  \int_{X} (f-\envomega(f))\, \MAo(\envomega(f_i)) \right) \le 0.
\]
Combining the above statements, we conclude that $f$ verifies the orthogonality property.

\medskip

$\ref{orth2} \Longrightarrow  \ref{orth3}$. This is obvious.

\medskip

$\ref{orth3} \Longrightarrow \ref{orth4}$. Let $\mu$ be a nonnegative Radon measure with compact support in $\xo$ of total mass $\deg(\omega^d)$. Let $\varphi \in \CPSH_{\omega}(X)$ be a maximizer of $\FE_{\mu}$ on $\PSH_{\omega}(X)$ as in Proposition \ref{prop:maximizer_Lagrangian}. We claim that the equality $\MA_{\omega}(\varphi) =\mu$ holds. 

Let $g$ be a continuous function on $X$. Consider the function 
  \[
  \GE(t) \coloneqq \E(\env_{\omega}(\varphi+tg)) - \int_{X}(\varphi+tg) d\mu.
  \]
Observe that 
\[
\GE(t)=\FE_{\mu}(\env_{\omega}(\varphi+tg))+\int_X (\env_{\omega}(\varphi+tg)-(\varphi+tg)) d\mu. 
\]
For any $t \ne 0$, we have $\GE(t) \le \GE(0)$, since $\int_X (\env_{\omega}(\varphi+tg)-(\varphi+tg)) d\mu \le 0$ and $\FE_{\mu}(\env_{\omega}(\varphi+tg)) \le \FE_{\mu}(\varphi)$. 

Applying Theorem \ref{thm:derivation}, we get 
\[
\left.\frac{d}{dt}\right|_{t=0} \GE(t) =\int_X g\ \MA_{\omega}(\varphi) - \int_{X} g \ d\mu  =0.
\]
Since this holds for all continuous functions $g$, we get $\MA_{\omega}(\varphi)=\mu$. 

\medskip

$\ref{orth4} \Longrightarrow \ref{orth1}$. We proceed by contradiction. Suppose that there exists a finite collection of pairs $\CS=\{ (x_j, \lambda_j) \in \xo \times \R \}_{1 \le j \le k}$ such that $\Supp( \MAo (\envomega(f_{\CS}))) \nsubseteq S$ where $S=\{ x_j \}_{1 \le j \le k}$. Consequently, $\MAo (\envomega(f_{\CS}))(S) < \deg (\omega^d)$. So we may choose a measure $\mu$ on $\xo$ such that $\mathrm{Supp}(\mu)=S$, $\mathrm{Vol}(\mu)=\mathrm{deg}(\omega^{d})$, and  $\mu(\{ x_j \}) > \envomega(f_{\CS}) (\{ x_j \})$ for each $j \in S$.

By assumption, $(X, \omega)$ satisfies the tropical Calabi--Yau theorem. Therefore, there exists $\varphi \in \PSH_{\omega}(X)$ such that $\MAo(\varphi)=\mu$. Adding a constant to $\varphi$, we may assume that $\max_{x \in S} \{\,\varphi(x)-\envomega(f_{\CS})(x) \,\}=0$. 

Let $x\in S$ such that $\varphi(x)=\envomega(f_{\CS})(x)$. Choose a neighborhood $U$ of $x$ such that $\omega=\dm \dmb h$ for some $h\in \LConv (U)$. Hence, $h+\varphi, h+\envomega(f_{\CS}) \in \LConv(U)$. By the definition of the envelope, we have $\varphi \le \envomega(f_{\CS})$ on $X$, which implies that $h+\varphi \le h+\envomega(f_{\CS})$  on $U$. At the same time, $(h+\varphi)(x)=(h+\envomega(f_{\CS}))(x)$, and $\MA((h+\varphi))(\{ x \}) > \MA((h+\envomega(f_{\CS})))(\{ x \})$ by constructions. But these contradict the following lemma.

  \begin{lemma} \label{lem:a_local_comparison}
      Let $U$ be an open subset in $\xo$, and let $x\in U$. Suppose that $u, v\in \LConv (U)$ such that $u(x)=v(x)$ and $u \le v$  on $U$. Then we have $\MA(u)(\{x\}) \le \MA(v)(\{x\})$. 
  \end{lemma}
  \begin{proof}
    This is a local statement, so we may assume that $(\xo, x)=(\supp{\Sigma}, 0)$, where $\Sigma$ is a tropical fan embedded in a linear space $E=N_{\R}$, and $0$ is the origin in $E$. After adding a local affine support function of $u$ at $0$, we assume further that $u(0)=v(0)=0$, and $0 \le u \le v$.

    By Corollary \ref{cor:refinement_of_open_tropical_varieties_are_projective}, possibly refining $\Sigma$ if necessary, we can suppose that the canonical compactification 
    $\comp{\Sigma}$ is \prekahler. Furthermore, there exists $h \in \LConv \cap \PL_{\hom}(\Sigma)$ such that $h>0$ on $\supp{\Sigma} \setminus \{0\}$, and $\omega \coloneqq \dm \dmb h$ is a semipositive simplicial form on $\comp{\Sigma}$.

    For $\epsilon >0$, we consider $v_{\epsilon}=v+\epsilon h \in \LConv(U)$. By local Bedford--Taylor theory \cite[Prop.\,5.6]{APW1},
    %cross-ref in APW1: \ref{prop:Local_Bedford-Taylor_theory} (Proposition \ref{prop:Local_Bedford-Taylor_theory}), 
    $\MA(v_{\epsilon}) \rightharpoonup \MA(v)$ when $\epsilon \to 0$. Moreover, $\MA(v_{\epsilon}) \ge \MA(v) \ge 0$ for any $\epsilon >0$, and $\MA(v_{\epsilon})(U)$ are uniformly bounded. Therefore, $\MA(v_{\epsilon}) \{0 \} \to \MA(v) \{0 \}$. It is thus enough to prove the lemma for $v_{\epsilon}$ for any $\epsilon>0$. 

    From now on, we fix $\epsilon>0$, and a constant $C \gg 1$. We choose an open neighborhood $V$ of $0$ such that $V  \Subset U$. By Proposition \ref{prop:Lipschitz_bounds_in_terms_of_C0_bounds}, $u$ and $v_{\epsilon}$ are Lipschitz on $V$. Therefore, there exists $A \gg 0$ such that $V_1\coloneqq \{(Ah-C) < u\} \Subset V$, and $V_2\coloneqq \{ Ah-(C-1) < v_{\epsilon} \} \Subset V$. We define functions $\varphi, \psi\colon\comp{\Sigma} \to \R $ by
    \[
    \varphi(y)=\begin{cases}
      u-Ah &  \quad \text{ if }y \in V_1 \\
      -C &  \quad \text{ if } y \in \comp{\Sigma} \setminus V_1
    \end{cases}\,, \qquad \text{and} \qquad
    \psi(y)=\begin{cases}
      v_{\epsilon}-Ah & \quad \text{ if } y \in V_2 \\
      -(C-1) & \quad \text{ if } y \in \comp{\Sigma} \setminus V_2
    \end{cases}\,.
    \]
    By Lemma \ref{lem:construction_sup_omega_psh_functions}, we have $\varphi, \psi \in \CPSH_{A\omega}(\ol\Sigma)$.

\smallskip

    For $\delta>0$, we define $W_{\delta}=\{\epsilon h-\delta <0 \} \subset \Sigma$. Since $h$ is homogeneous, and $h>0$ on $\Sigma \setminus \{0\}$,  we have $W_{\delta}=\delta \cdot W_1$, and $W_{\delta}$ is a bounded subset of $E$. Therefore, when $\delta \ll 1$, $W_{\delta} \subset V_{1} \cap V_2$. We also define the subset $U_{\delta}\coloneqq \{ \psi - \delta < \varphi \} \subset \comp{\Sigma}$. 
    
 \begin{claim}\label{claim:4to1} When $0< \delta \ll 1$, we have $U_{\delta} \subset W_{\delta}$. 
 \end{claim}
    \begin{proof}[Proof of the claim]
    Take $y \in U_{\delta}$. By definition, $\psi(y) - \varphi(y) < \delta$. We need to verify that $\epsilon h(y) < \delta$, which proves $y \in W_{\delta}$. 
    
    If $y \in V_1 \cap V_2$, $\psi(y)-\varphi(y)=v(y)+\epsilon h(y)-u(y) < \delta$. By assumption $v(y)-u(y) \ge 0$. Thus, $\epsilon h(y)<\delta$, as required.  
    
    Proceeding by a case by case analysis, we show that $y \in V_1 \cap V_2$ is the only possible case.

    \begin{enumerate}
      \item If $y \in \comp{\Sigma} \setminus (V_1 \cup V_2)$, then $\psi(y)-\varphi(y)=1 > \delta$, which is impossible.
      
      \item If $y \in V_1 \setminus V_2$, then by assumption, 
      \[ 
      v_{\epsilon}(y) \le  Ah(y)-C+1, \quad \text{and}
      \]  
      \[ 
      \psi(y)-\varphi(y)=Ah(y)-C+1-u(y)< \delta. 
      \] 
      This implies that $v(y)+\epsilon h(y)-u(y) \le Ah(y)-C+1-u(y)< \delta$. Since $v(y)-u(y) \ge 0$, it follows that $\epsilon h(y) < \delta$. Therefore, $y \in W_{\delta} \subset V_1 \cap V_2$, which is again impossible by our assumption on $y$.

      \item In the remaining case,  $y \in V_2 \setminus V_1$. Then by assumption, 
      \[
      Ah(y)-C+1 \le v_{\epsilon}(y), \quad \text{and}     
      \] 
      \[ 
      \psi(y)-\varphi(y)=v_{\epsilon}(y)-Ah(y)+C. 
      \] 
      This implies that $\psi(y)-\varphi(y) \ge 1$, which is impossible.  
    \end{enumerate}
    Therefore, the claim is proved. 
  \end{proof}

    For any $0 < \delta \ll 1$, Claim~\ref{claim:4to1} implies that $U_{\delta} \subset V_1 \cap V_2$. Therefore, by constructions,  $\MA_{A\omega}(\psi) (U_{\delta})=\MA(v_{\epsilon})(U_{\delta})$, and $\MA_{A\omega}(\varphi) (U_{\delta})=\MA(u)(U_{\delta})$.
    On the other hand, the comparison priciple (Theorem \ref{thm:comparison_principle}) implies that $\MA_{A\omega}(\psi) (U_{\delta}) \ge \MA_{A\omega}(\varphi) (U_{\delta})$. Therefore, $\MA(v_{\epsilon})(U_{\delta}) \ge \MA(u)(U_{\delta})$.

    Claim~\ref{claim:4to1} also implies that $\bigcap_{\delta>0} U_{\delta} \subset \bigcap_{\delta>0} W_{\delta}=\{0\}$. So when $\delta \to 0+$, we have $\MA(v_{\epsilon})(U_{\delta}) \to \MA(v_{\epsilon})(\{0 \})$, and $\MA(u)(U_{\delta}) \to \MA(u)(\{0 \})$. This finishes the proof of the lemma.
  \end{proof} 
  
  At this point, the proof of Theorem~\ref{thm:equivalence_orthogonality_existence_of_solutions} is complete.

\section{Orthogonality property in codimension one} \label{sec:Orthogonality_property_in_codimension_one}

Let $X$ be a compact \prekahler tropical  variety of dimension $d$ equipped with a strongly semipositive (1,1)-form $\omega$.

Recall that a point $x\in X$ is regular if it admits a tropical chart $\phi\colon W \to U \subset \R^d$, $x\in W$. In particular, this implies that $x\in \xo$. We introduce the following definition. 
\begin{definition} \label{def:regular_in_codim_one}
  We say that $x\in X$ is \emph{regular in codimension one} if there exists a tropical chart $U \to W \subset  \mathcal L \times \R^{d-1}$, where $\mathcal L$ is a generalized tropical line (see Section~\ref{subsec:tropical-varieties}) and $W$ is an open subset in $\mathcal L \times \R^{d-1}$. Denote by $\CR$ the open set of points of $\xo$ that are regular in codimension one. 
\end{definition}
We remark that being regular in codimension one is an open property. We also remark that $x \in X$ is regular in codimension one implies automatically that $x \in \xo$. 

\smallskip

In this section, we prove the following theorem.

\begin{thm} \label{thm:orthogonality_codimension_1}
Let $X$ be a compact \prekahler tropical  variety of dimension $d$ equipped with a strongly semipositive form $\omega$. 
Then, the orthogonality property holds on $\CR$ in the following sense: for any lsc function $f\colon X \to \R\cup\{+\infty\}$ which is not constantly $+\infty$ on $\xo$, we have 
\begin{equation}\label{equation:orthogonality_codimension_1} 
\int_{\CR} (f-\env_{\omega}(f)) \MA_{\omega}({\env_{\omega}(f)}) = 0.
\end{equation}
\end{thm}
We have the following consequences.
\begin{cor}\label{cor:domination} Let $\mu$ be a nonnegative Radon measure which is compactly supported in $\xo$ and has total volume $\deg(\omega^d)$. 
  Let $\varphi \in \CPSH_{\omega}(X)$  be a maximizer of $\FE_\mu$. Then $\MAo(\varphi) \le \mu$ on $\CR$, the open subset consisting of points regular in codimension one. 
\end{cor}
The proof is given in Section~\ref{subsec:orthogonality_codimension_one_proof}, along with the proof of the following.
\begin{thm} \label{thm:existence_of_solutions_nice_tropical_varieties}
  Let $X$ be a compact \prekahler variety of dimension $d$ equipped with a strongly semipositive form $\omega$.  Assume that $\CR=\xo$, \textit{i.e.} every point $x \in \xo$ is regular in codimension one. Then,  the tropical Calabi--Yau theorem holds for $(X, \omega)$.  
\end{thm}

\begin{remark}
  In particular, Theorem \ref{thm:existence_of_solutions_nice_tropical_varieties} applies to a tropical variety $X$ which is a product of the form $\Gamma\times Z$ where $Z$ is a compact integral affine manifold and $\Gamma$ is a compact metric graph. Here, the tropical structure on $\Gamma$ is locally matroidal~\cite[\S\,5]{ACM}.  
\end{remark}

\subsection{A local result} \label{subsec:local_orthogonality_codim_1}
 Let $E_1=\R^{n_1}$ and $E_2=\R^{n_2}$, and set $E=E_1\times E_2 =\R^n$, $n=n_1+n_2$. Denote by $\pi_1, \pi_2$ the projection maps $E \to E_1$ and $E \to E_2$ respectively. Suppose either that $\mathcal L \subset E_1$ is a generalized tropical line or $\CL=E_1=\{0\}$. 
 In the first case, without loss of generality, assume further that $\CL$ generates  $E_1$ linearly. Consider the fanfold $\mathcal L \times E_2$ in $E$. 

 We fix a basis $\{e_1, \cdots, e_{n_1} \}$ of $E_1$. Without loss of generality, we can assume that $\rho\coloneqq \R_{\ge 0} e, \rho_1\coloneqq \R_{\ge 0} [e_1], \dots, \rho_{n_1}\coloneqq \BR_{\ge 0} [e_{n_1}]$ are all the $(n_1+1)$ rays of $\mathcal L$, where $e \coloneqq -(e_1 + \cdots + e_{n_1})$. The fanfold $\CL$ is equipped with a canonical fan structure $\Sigma_1$. We take $g_1 \coloneqq \max\{0, -x_1, -x_2, \cdots, -x_{n_1}\}$ and $h_1 \coloneqq g_1 \rest{\Sigma_1}$.

  Let $\Sigma_2$ be the standard complete fan on $E_2$ of the projective space. Take 
  \[
  h_2 \coloneqq \max\{0, -x_1, -x_2, \cdots, -x_{n_2}\},
  \]
   which is a strictly convex conewise linear function on $\Sigma_2$. 
 
 Therefore, $\Sigma \coloneqq \Sigma_1 \times \Sigma_2$ is a fan supported on $\CL \times E_2$. Define $h \coloneqq h_1 \circ \pi_1+h_2 \circ \pi_2 \in \LConv \cap \PL_{\hom}(\Sigma)$. Then, $\omega \coloneqq \dm \dmb h$ is a semipositive simplicial form on the canonical compatification 
 $\comp{\Sigma}$. Proposition \ref{prop:kahler_form_on_tropical_toric_varieties} implies that $[\omega]$ is a \kahler class on $\comp{\Sigma}$.

 \begin{thm} \label{thm:orthogonality_product_of_curve_and_Rn} 
 Notation as above, $(\comp{\Sigma}, \omega)$ verifies the equivalent properties~\ref{orth1}-~\ref{orth4} stated in Theorem \ref{thm:equivalence_orthogonality_existence_of_solutions}.
 \end{thm} 
 The remaining of this subsection is devoted to prove Theorem \ref{thm:orthogonality_product_of_curve_and_Rn}.  
 We will show that $(\comp{\Sigma}, \omega)$ satisfies Property \ref{orth1} in Theorem \ref{thm:equivalence_orthogonality_existence_of_solutions}. 
 
 We start by defining a map
 \[
 \DL\colon \LConv(\Sigma) \to \Conv(E), 
 \]
 as follows. For each $u \in \LConv(\Sigma)$, let $u'\colon E \to \R \cup \{+ \infty \}$ be the extension of $u$ obtained by taking value $+\infty$ on $E \setminus \supp{\Sigma}$. Let $P$ denote the envelope operator given in Lemma \ref{characterization_convex_envelope}, and set $\DL(u) \coloneqq P(u')$, that is, 
  \bes
  \DL(u)(x) \coloneqq \inf \left\{\,\sum_{i\in I} \nu_i u'(x_i) \, \st \, (x, (x_i)_{i \in I}, (\nu_i)_{i \in I}) \text{ is a convex combination in } E \,\right\}, \forall\,\, x \in E.
  \ees
  Since the linear span of $\Sigma$ equals to $E$, we have $\DL(u)(x)<+\infty$ for any $x \in E$. 

  \smallskip
  Now we prove two auxiliary lemmas.

 \begin{lemma} \label{lem:equivalence_LConv_GConv_codimension_1} 
  For each $u \in \LConv(\Sigma)$, we have $\DL(u)\rest{\Sigma}=u$. In particular, the map $\DL\colon \LConv(\Sigma) \to \Conv(E)$ is an injection. Moreover, for each $v\in \Conv(E)$, we have $\DL(v\rest{\Sigma}) \geq v$ and $\DL(v\rest{\Sigma})\rest{\Sigma}=v\rest{\Sigma}$.
 \end{lemma}
 \begin{proof}
  When $\mathcal L =\conezero$, we have $\Sigma=\Xi$.  Therefore, there is nothing to prove. In the remaining, we assume that $\mathcal L$ is a generalized tropical line.
 
  Let $u\colon \Sigma \to \R$ be an element of $\LConv(\Sigma)$, and set $\widetilde u =\DL(u)$. We need to show that $\widetilde u\rest{\Sigma} =u$, that is, to prove that for any convex combination $(x, (x_i)_{i \in I}, (\nu_i)_{i \in I})$ in $E$, we have 
  \be \label{aim_equation:lem:equivalence_LConv_GConv_codimension_1}
  u(x) \le \sum_{i \in I} \nu_i u(x_i).
  \ee
  Identify $E_2$ with $\{0\} \times E_2 \hookrightarrow E$.  Two cases occur depending on whether $x \in E_2 \cap U$ or not.
  
  \medskip

\noindent  (1) If $x \in E_2$, after applying an affine translation, we may assume that $x=0 \in E$. 
  Then, $(x, (x_i)_{i \in I}, (\nu_i)_{i \in I})$ is a polyhedral convex combination on $\Sigma$. Since $u \in \LConv(U)$, by Proposition \ref{prop:equivalent_def_LConv} the inequality  \eqref{aim_equation:lem:equivalence_LConv_GConv_codimension_1} holds.
  
  \smallskip

 \noindent (2) If $x \notin E_2$, there exists a ray $\rho \in \mathcal L$ such that $x \in \relint(\rho) \times E_2$.  Let $H_{+}$ denote $\relint(\rho) \times E_2$, and $W$ denote the linear subspace of $E$ generated by $\rho \times E_2$. 
  
  Partition the set $\{x_i\}_{i \in I}$ into $\{x_j\}_{j \in I_+} \subset H_+$ and $\{x_j\}_{j \in J} \subset \Sigma \setminus H_+$, $I = I_+ \sqcup J$ according to whether the points lie in $H_+$ or not. Since $x$ lies in $H_+$ and it is a convex combination of $x_i$, the set $I_+$ is nonempty. 
  
  Note that $u\rest{H_{+}} \in \Conv(H_{+})$. Therefore, if $J=\emptyset$, then the inequality in \eqref{aim_equation:lem:equivalence_LConv_GConv_codimension_1} follows automatically. Consequently, in the rest of the proof, we assume that $J$ is nonempty. Set 
  \[ \mu' \coloneqq \sum_{j \in I_+} \nu_j, \quad \text{and} \quad \mu'' \coloneqq\sum_{j \in J} \nu_{j}. \] 
  We do a normalization for the coefficients by
  \[ \nu'_j\coloneqq \nu_j/\mu',\ j \in I_+ \qquad \text{and} \qquad \nu''_{j}\coloneqq \nu_{j}/ \mu'', \ j \in J.\]
  Consider the points $x' \in H_+$ and $x'' \in E$ defined by
  \[ x'\coloneqq \sum_{j \in I_+} \nu'_j x_j \quad \text{and} \quad x'' \coloneqq \sum_{j \in J} \nu''_{j} x_{j}. \]
  Then, we have the convex combination 
  \[ x=\mu' x' + \mu'' x''. \]
  
  \begin{claim}  \label{claim:lem:characterization_polyhedral_structure_envelope}
    We have $x'' \in W \setminus H_{+}$.
  \end{claim} 
  \begin{proof}[Proof of the claim]
  The fact that $x, x' \in W$ implies that $x'' \in W$. It remains to show that $x'' \notin H_{+}$.

  Without loss of generality, we may assume that $\rho$ is generated by $e$. For any $j \in J$, the projection of $x_{j}$ onto $E_1$ has nonnegative coefficients in the basis $\{e_1, \dots, e_{n_1}\}$. Since $x''$ is a convex combination of the $x_{j}$, its $E_1$ component also has nonnegative coefficients. In contrast, the $E_1$ component of every element of $H_{+}$ has strictly negative coefficients under the same basis. Hence $x'' \notin H_{+}$.
  \end{proof} 

  Back to the proof of the lemma, note that $E_2$ is a hyperplane of $W$, and $H_+$ is one of the two half-spaces bounded by $E_2$. 
  Thus, Claim \ref{claim:lem:characterization_polyhedral_structure_envelope}  
  implies that there exists $0 < t < \mu'$ such that $y\coloneqq  tx'+(1-t)x'' \in E_2$. 
  After applying an affine translation, we may assume that $y=0 \in E_2$.  
  Therefore, $\left( y, \left( \{x_j\}_{j \in I_+} , \{x_{j}\}_{j \in J} \right), \left( \{t \nu'_j\}_{j \in I_+} , \{(1-t) \nu''_{j}\}_{j \in J} \right) \right)$ 
  is a polyhedral convex combination on $\Sigma$. 
  Hence, Proposition \ref{prop:equivalent_def_LConv} implies that 
  \be \label{ref_equation:1}
  u(y) \le \sum_{j\in I_+} t\nu'_{j} u(x_j)+ \sum_{j' \in J} (1-t)\nu''_{j} u(x_{j}).
  \ee
  At the same time, we have the relation $x=\frac{\mu'(1-t)-t\mu''}{1-t}x'+\frac{\mu''}{1-t}y$. Note that $\frac{\mu'(1-t)-t\mu''}{1-t}+\frac{\mu''}{1-t}=1$ since $\mu'+\mu''=1$. Moreover, $0< \frac{\mu''}{1-t} < 1 $ since $0<\mu''<1$ and $t<\mu'$.  Therefore, $x$ is a convex combination of $x'$ and $y$ on $\comp{H_+}$.  Since $u|_{\comp{H_+}}$ is convex,  it follows that 
  \be \label{ref_equation:2}
  u(x) \le \frac{\mu'(1-t)-t\mu''}{1-t} u(x')+ \frac{\mu''}{1-t} u(y).
  \ee
  Also, note that 
  \be \label{ref_equation:3}
  u(x') \le \sum_{j \in I_+} \nu'_{j} u(x_j).
  \ee
  Combining Equations \ref{ref_equation:1}, \ref{ref_equation:2}, and \ref{ref_equation:3}, we get Equation \ref{aim_equation:lem:equivalence_LConv_GConv_codimension_1}. 
 
  \smallskip

 Therefore, the proof is complete. 
 \end{proof}
 \begin{lemma} \label{lem:characterization_polyhedral_structure_envelope}
   Let $\Xi$ be a complete simplicial fan on $E=\R^n$. Let $h$ be a convex conwise linear function on $\Xi$. Suppose that $\CS= \{(x_i, \lambda_i) \in E \times \R \}_{1 \le i \le k}$ is a finite collection of pairs in the sense of Example~\ref{example_functions_with_values_at_finite_points}. We consider the function 
  \[ u_{\CS}\coloneqq \sup \left\{\, u \in \Conv(E) \ |\ |u-h|=O(1), u \le f_{\CS} \,\right\}. \] 
  Then $u_{\CS} \in \PL(E)$. Furthermore, there exists a complete simplicial polyhderal complex $\Pi$ on $E$ satisfying the following properties:
  \begin{enumerate}
  \item $\Pi_0 \subset \{ x_i \}_{1 \le i \le k}$. 
  \item The recession fan of $\Pi$ is $\Xi$, that is, for each $\sigma \in \Pi$, $\rec(\sigma) \in \Xi$. 
  \item $u_{\CS}$ is facewise linear on $\Pi$. 
  \end{enumerate}
 \end{lemma}
 \begin{proof} The first assertion that $u_{\CS} \in \PL(E)$ is shown in \cite[Prop.\,9.2]{BFJ15}. For any affine function $l$ on $E$, we denote $\sigma_{l}\coloneqq \left\{ x \in E\ \st \ u(x)=l(x) \right\}$. Since $u_{\CS}$ is PL, there exists a finite subset $\{l_i\}_{i \in I} \subset \Aff(E)$ such that  $E=\bigcup_{i \in I}\sigma_{l_i}$, and $\dim \sigma_{l_i}=n$ for any $i \in I$. Since $u_{\CS}$ is convex, $\sigma_{l_i}$ is a polyhedron for any $i \in I$. Therefore, $\{ \sigma_{l_i} \}_{i \in I}$ generates a polyhedral structure $\Pi'$ with support $E$. Then, $u_{\CS}$ is facewise linear on $\Pi'$. 
  Moreover, $|u_{\CS}-h|=O(1)$ implies Property (2) holds for $\Pi'$. 
 
  We now show that $\Pi'_{0} \subset \{ x_i \}_{1 \le i \le k}$. Let $x \in E \setminus \{ x_i \}_{1 \le i \le k}$. Then, the domination principle (see Lemma 8.5 in \cite{BFJ15})  implies that  $\MA_{\R}(u_{\CS})(\{x \})=0$. Let $u \in \Conv \cap \PL_{\hom}(E)$ such that there exist an affine transform $A \colon E \to E$ and an open neighborhood $U$ of $x$ satisfying $A(x)=0$ and $u_{\CS}=u \circ A$ on $U$. 
  Therefore, $\MA_{\R}(u)(\{x \})=\mu_{\mathrm{leb}}(\dom(u^*))=0$. Here, $u^* \in \Conv(E^*)$ is the Legendre transform of $u$, and $\dom(u^*) \coloneqq \{u^{*}< +\infty\} \subset E^*$. This implies that $\dom(u^*)$ is a polytope of dimension strictly less than $n$. 
  Let $x' \in E$ such that $\langle x', \dom(u^*) \rangle=0$. Then we have $u=v \circ \pi_{x'}$ for  $v \in \Conv \cap \PL_{\mathrm{hom}}(\pi_{x'}(E))$, where $\pi_{x'} \colon E  \to E/\R{x'}$ is the quotient map, and $\R{x'}$ is linear space of $x'$. 
  Therefore, for any $\sigma_{l_i}$ containing $x$, $x$ is not an extremal point of $\sigma_{l_i}$. This implies that $x \notin \Pi'_{0}$, and hence $\Pi'_{0} \subset \{ x_i \}_{1 \le i \le k}$. 
 
  We take a refinement  $\Pi$ of $\Pi'$ keeping $\Pi_0=\Pi'_0$, breaking only the compact faces into simplices, to obtain a simplicial polyhedral structure $\Pi$. Then $\Pi$ satisfies (1). Moreover, since we did not touch the recession cones, 
  the recession fan of $\Pi$ is $\Xi$ and (2) is satisfied. Being a refinement of $\Pi'$, and $u_{\CS}$ being facewise linear on $\Pi'$, Property (3) is also verified by $\Pi$.
 \end{proof}

 \smallskip

 We are now ready to prove Theorem \ref{thm:orthogonality_product_of_curve_and_Rn}.
 
 \begin{proof} [Proof of Theorem \ref{thm:orthogonality_product_of_curve_and_Rn}] \label{proof:prop:orthogonality_product_of_curve_and_Rn}
  Let $\CS=\{ (x_i, \lambda_i) \in \supp{\Sigma} \times \R \}_{1 \le i \le k},$ be a finite collection of pairs and set $\varphi = \env_{\omega}(f_{\CS})= \sup \left\{\psi \in \CPSH_{\omega}\ \st \ \psi \le f_{\CS} \right\}$. Let $S$ denote the finite subset $\{ x_i \}_{1 \le i \le k} \subset \supp{\Sigma}$. We will prove that $\Supp(\MAo(\varphi)) \subset S$. 
 By Theorem \ref{thm:equivalent_def_psh}, the map
  \[
  \CPSH_{\omega}(\comp{\Sigma})  \longrightarrow \LConv_{h}(\Sigma) \quad 
  \psi \lms \psi+h
  \]
  is a bijection, where $ \LConv_{h}(\Sigma)\coloneqq \{u \in \LConv(\Sigma) \st u-h=O(1) \}$. Hence, denoting  $\CS'=\{ (x_i, \lambda_i+h(x_i)) \in \supp{\Sigma} \times \R \}_{1 \le i \le k}$, and 
  \[
  u_{0} \coloneqq \sup \left\{ u \in \LConv_{h}(\Sigma) \ \st \ u \le f_{\CS'} \right\},
  \]
  we have $\varphi+h=u_0$.
 
  Consider $\DL(h) \in \Conv(E)$. It is not hard to verify that $\DL(h)=g$, where $g \coloneqq g_1 \circ \pi_1+h_2 \circ \pi_2 \in \Conv \cap \PL_{\hom}(E)$. We define $\Conv_{g}(E) \coloneqq \{v \in \Conv(E) \ |\ v-\DL(h)=O(1)\}$, and 
  \[
  v_{0}\coloneqq \sup \left\{ v \in \Conv_{g}(E) \ \st \ v \le f_{\CS'} \right\}.
  \]

  We claim that $v_{0}\rest{\Sigma}=u_{0}$ on $\Sigma$. It is clear that $v_{0}\rest{\Sigma} \le f_{\CS'}$. Since  $\DL(h) \rest{\Sigma}=g \rest{\Sigma}=h$, we have $v_{0}\rest{\Sigma} \in \LConv_{h}(\Sigma)$. Therefore,  by definition $v_{0}\rest{\Sigma}\le u_0$.    On the other hand, by definition it is not hard to check that $\DL(u_0)-\DL(h)=O(1)$. Thus, $\DL(u_0) \in \Conv_{g}(E)$. By Lemma \ref{lem:equivalence_LConv_GConv_codimension_1}, $\DL(u_0) \rest{\Sigma}=u_0$. Hence, $ \DL(u_0)\le f_{\CS'}$, which implies that $\DL(u_0) \le v_0$ on $E$. We conclude that $u_0=\DL(u_0) \rest{\Sigma} \le v_0$ on $\Sigma$. Therefore, the claim is proved.

  In the remaining part of the proof, we show that $\Supp(\MA(u_0)) \subset S$, where $\MA$ denotes the local tropical \MongeA operator on $\Sigma$.  Let $\Xi_1$ be the complete fan on $E_1$ defining the tropical projective space, and set $\Xi \coloneqq \Xi_1 \times \Sigma_2$. Notice that $g$ is a convex conewise linear function on $\Xi$. We apply Lemma \ref{lem:characterization_polyhedral_structure_envelope} to $g$ and $\CS'$, and let $\Pi$ be the complete simplicial polyhedral structure on $E$ satisfying the properties stated in the lemma. Due to Property (1) in the lemma, it is enough to prove that $\MA(u_0)(E \setminus \Pi_{0})=0$. Let $\sigma \in \Pi$ such that $\dim \sigma \ge 1$, and let $x \in \interior(\sigma) \cap \Sigma$. Let $E_{\sigma}$ be the linear subspace of $E$ parallel to the affine span of $\sigma$, and let $\pi_{\sigma} : E \to E/E_{\sigma}$ denote the quotient map.

  Locally around $x$, we have $u_0=v \circ \pi_{\sigma}$ for some $v \in \PL(\pi_{\sigma}(\Sigma))$. By assumption, $\dim \Sigma=n_2+1$. By \cite[Prop.\,5.6]{APW1},
  %cross-ref in APW1: \ref{prop:compatibility_Monge-Ampere_projections} , 
  we have $(\pi_{\sigma})_{*}((\dm \dmb u_0)^{ \wedge (n_2+1)})=(\dm \dmb v)^{\wedge (n_2+1)}$ around $\pi_{\sigma}(x)$. Therefore, it suffices to show that $\dim \pi_{\sigma}(\Sigma) \le n_2$ around $x$, as this will imply then $(\dm \dmb v)^{\wedge (n_2+1)}=0$.  
  
In the case $\dim \sigma \ge n_1$, it is clear that $\dim \pi_{\sigma}(\Sigma) \le \dim E/E_{\sigma} \le n_2$. On the other hand, we have the following claim.
  \begin{claim} \label{claim:proof:prop:orthogonality_product_of_curve_and_Rn} 
    Suppose that $1 \le \dim \sigma \le n_1-1$. Then, either $\sigma \subset \Sigma$, or $\relint(\sigma) \cap \supp\Sigma =\emptyset$.
  \end{claim} 
  Granted this claim, we explain how to finish the proof of the theorem. Indeed, when $\sigma \subset \Sigma$, we have $\dim \pi_{\sigma}(\Sigma)=n_2+1-\dim \sigma \le n_2$ around $x$, which concludes the proof. While in the other case, the condition $\relint(\sigma) \cap \supp\Sigma =\emptyset$  contradicts the assumption that $x \in \relint(\sigma) \cap \Sigma$. 
 
  It remains to prove Claim \ref{claim:proof:prop:orthogonality_product_of_curve_and_Rn}. Observe that that for any $\tau \in \Xi$, either $\tau \in \Sigma$ or $\relint(\tau) \cap \Sigma=\emptyset$. Therefore, due to Property (2) in Lemma \ref{lem:characterization_polyhedral_structure_envelope}, it is enough to prove the claim when $\sigma$ is bounded. 
  Since $\Pi$ is simplicial,  $\sigma$ is the convex hull of $\{z_j \}_{j\in J} \subset \Sigma$ with $\card{J} \le n_1$. Therefore, without loss of generality, we may assume that $\pi_{1}(\{z_j \}_{j\in J}) \subset \bigcup_{1 \le i \le n_1} \rho_i$, where $\pi_1 \colon E \to E_1$ is the projection map, and $\rho_i$ are the rays of $\CL$ defined in the beginning of this subsection. Let $\tau$ be the cone generated by $\{\rho_i\}_{1 \le i \le n_1}$. We observe that $\pi_{1}(\sigma) \subset \tau$, $\tau \cap C=\bigcup_{1\le i \le n_1} \rho_i$, and $\{\rho_i\}_{1 \le i \le n_1}$ are extremal rays of $\tau$. 
   
  Two cases occur depending on whether all the $z_j$s belong to the same ray in $\mathcal L$ or not. In the former case, we have $\pi_{1}(\sigma) \subset \mathcal L$. While in the latter case, $\pi_1(\relint(\sigma)) \subset \relint(\pi_{1}(\sigma)) \subset E_1 \setminus \mathcal L$. Hence, the claim is proved.
 \end{proof}
 
 \smallskip

 Now, we prove the following corollary of Theorem \ref{thm:orthogonality_product_of_curve_and_Rn}. 

 \begin{cor} \label{cor:local_trimming_codimension_1}
  Let $U$ be a bounded open subset of $\supp\Sigma$. Suppose that $u \in \LConv(U)$, and it satisfies $\MA(u)(U) >0$. Then, there exists $v \in  \LConv(U)$ such that $v \ge u$ on $U$, and $\{u < v\}$ is a nonempty open subset compactly contained in $U$. 
 \end{cor}
   Although this is a purely local statement, we will prove it by a global argument on $(\comp{\Sigma}, \omega)$.
 \begin{proof}[Proof of Corollary \ref{cor:local_trimming_codimension_1}]
  We observe first that there exists $x \in U$ such that for any open neighbourhood $U'$ of $x$, we have $ \MA_{\Trop, \comp{\Sigma}}(u)(U')>0$. Indeed, otherwise, for each $x \in U$, take an open neighborhood $U_x$ of $x$ in $U$ such that $\MA_{\Trop, \comp{\Sigma}}(u)(U_x)=0$. Then, $U=\bigcup_{x \in U}U_x$. Since $U$ is second-countable, there exists a countable subset $A$ of $U$ such that $U=\bigcup_{x \in A}U_x$. Therefore, $\MA_{\Trop, \comp{\Sigma}}(U) \le \sum_{x \in A} \MA_{\Trop, \comp{\Sigma}}(U_x)=0$, which contracdicts with the assumption.  

  Take an affine function $l$ on $E$ such that $h' \coloneqq h+l$ satisfies $h'(y)>0$ for any $y \in \supp\Sigma \setminus \{0\}$. Then, $\omega=\dm \dmb h =\dm \dmb h'$. Applying Proposition \ref{prop:example_weak_kahler_form_on_projective_tropical_variety} to $u \in \LConv(U)$, there exist an open neighborhood $W \Subset U$ of $x$ and $\varphi \in \CPSH_{C \omega}(\comp{\Sigma})$, for some $C>0$, such that $\varphi+Ch'=u$ on $W$ and $\varphi+Ch' \ge u$ on $U \setminus W$. 
 
  Take an open neighborhood $V \Subset W$ of $x$, and define the lsc function $f: \comp{\Sigma} \to \R \cup \{+\infty \}$ by 
  \[
  f(y)= \begin{cases}
  +\infty &\quad \text{ if }y \in V \\
  \varphi(y)& \quad \text{ if } y \in \comp{\Sigma} \setminus V
  \end{cases}\,.
  \]
  Since $\env_{C\omega}(f)=\sup\{\env_{C\omega}(f), \varphi\}$, we have $\env_{C\omega}(f)=\varphi$ on $\comp{\Sigma} \setminus V$. Also, $\env_{C\omega}(f)+Ch' \ge u$ on $W$. On the other hand, the orthogonality property of $(\comp{\Sigma}, C\omega)$ (Theorem \ref{thm:orthogonality_product_of_curve_and_Rn}) yields that $\MA_{C\omega}(\env_{C\omega}(f))=\MA_{\Trop, \comp{\Sigma}}(\env_{C\omega}(f)+Ch')(V)=0$. This implies that $\env_{C\omega}(f)+Ch' \ne u$ on $V$. Hence, $\{u < \env_{C\omega}(f)+Ch'\}$ is nonempty. 

  Now, we define 
  \[
  f(y)= \begin{cases}
  +\infty &\quad \text{ if } y \in V \\
  \varphi(y)  &\quad \text{ if } y \in \comp{\Sigma} \setminus V
  \end{cases}\,.
  \]

  We take $v=(\env_{C\omega}(f)+Ch')|_{\ol{U}} \in \LConv(U)$, we have   
  \[
  v(y)= \begin{cases}
  (\env_{C\omega}(f)+Ch')(y) & \quad \text{ if }y \in W \\
  u(y) & \quad \text{ if } y \in U \setminus W
  \end{cases}\,.
  \]
  Then, Lemma \ref{lem:construction_sup_omega_psh_functions} implies that $v \in \LConv(U)$. By construction, $v \ge u$ on $U$, $\{v < u\}$ is nonempty, and $\{v < u\} \subset W \Subset U$. 
\end{proof}

\subsection{Proof of Theorem~\ref{thm:orthogonality_codimension_1} and its consequences} \label{subsec:orthogonality_codimension_one_proof} We start with the main theorem.
\begin{proof}[Proof of Theorem~\ref{thm:orthogonality_codimension_1}]
  Arguing similar to the proof of \ref{orth1} $\Longrightarrow$ \ref{orth2} in Theorem \ref{thm:equivalence_orthogonality_existence_of_solutions}, to prove Theorem \ref{thm:orthogonality_codimension_1}, it suffices to show that, for any finite collection of pairs $\CS=\{ (x_i, \lambda_i) \in \xo \times \R \}_{1 \le i \le k}$, the corresponding function $f_{\CS}$ satisfies the orthogonality property on $\CR$.
  
  We prove this by contradiction. Suppose that there exists $f_{\CS}$ such that the equality \eqref{equation:orthogonality_codimension_1} is not satisfied. Then, there exists an open subset $U \Subset \left(\CR \setminus \{ x_i \}_{1 \le i \le k}\right)$ such that  $\MAo(\env_{\omega}(f))(U)>0$. Write $\omega=\dm \dmb h$ for some $h \in \LConv(U)$. Set $\varphi\coloneqq \env_{\omega}(f_{\CS})$, and let $u\coloneqq \varphi+h \in \LConv(U)$.
  
  \smallskip

  We may assume that there is a compatible chart $\phi\colon U \to W$ where $W$ is a bounded open subset of $ \mathcal L\times\R^{d-1}$ for a generalized  tropical line $\mathcal L \subset \R^k$.
  Applying Corollary \ref{cor:local_trimming_codimension_1} to $u' \coloneqq u\circ\phi^{-1} \in  \LConv(W)$, and  get $v' \in \LConv(W)$ satisfying the conditions in  Corollary \ref{cor:local_trimming_codimension_1}. 

  Since $\{\varphi+h < v' \circ \phi \} \Subset U$, we can define a function $\widetilde \varphi \coloneqq \sup\{\varphi, v' \circ \phi-h\}$ on $X$ as in Construction \ref{construction:sup_omega_psh_functions}. Then, Lemma \ref{lem:construction_sup_omega_psh_functions} implies that $\widetilde{\varphi} \in \CPSH_{\omega}(X)$. By construction, $\widetilde{\varphi}=\varphi$ on $\{ x_i\}_{1 \le i \le k}$, which indicates that $\widetilde{\varphi} \le f_{\CS}$. On the other hand, Corollary \ref{cor:local_trimming_codimension_1} implies that $\widetilde{\varphi} \ge \varphi$ and $\widetilde{\varphi}(x) > \varphi(x)$ for some $x \in U$.  This contradicts with the assumption that $\varphi =\envomega(f_{\CS})$. 
\end{proof}
\begin{proof}[Proof of Corollary~\ref{cor:domination}]
  We only need to prove that for any $g \in \nC^0(X)$ such that $g \ge 0$ and $\mathrm{Supp}(g) \Subset \CR$, we have 
  \bes
    \int_X g\ \MA_{\omega}(\varphi) \leq \int_X g d\mu.
  \ees
  Note that for any  $t>0$ $\varphi \le \varphi +tg$, $\varphi+tg=\varphi$ on $X \setminus \mathrm{Supp}(g)$. On the other hand, we have $\env_{\omega}  (\varphi+tg)=\sup \{\varphi+ \env_{\omega}(\varphi+tg)\}$. Therefore, $\env_{\omega}  (\varphi+tg)=\varphi$ on $X \setminus \CR \subset X \setminus \mathrm{Supp}(g)$. It follows that
  \bens
  \int_X(\varphi+tg-\env_{\omega}(\varphi+tg))\MA_{\omega}(\env_{\omega}(\varphi+tg)) = \int_{\CR}(\varphi+tg-\env_{\omega}(\varphi+tg))\MA_{\omega}(\env_{\omega}(\varphi+tg))=0
  \eens
  where the second equality is induced by Theorem \ref{thm:orthogonality_codimension_1}. Hence, $\varphi+tg$ verifies the orthogonality property for any $t>0$. 

  The remaining part of the proof goes similar to the proof of \ref{orth3} $\Longrightarrow$ \ref{orth4} in Theorem \ref{thm:equivalence_orthogonality_existence_of_solutions}. Consider the function $\GE(t) \coloneqq \FE_{\mu}(\varphi+tg)$.
  Then, $G(t) \le G(0)$ for any $t >0$ since $\varphi$ is a maximizer of $\FE_\mu$. Applying Theorem \ref{thm:derivation}, we get \[\left.\frac{d}{dt}\right|_{t=0+} \GE(t) =\int_X g\ \MA_{\omega}(\varphi) - \int_{X} gd\mu \le 0.\]
  The corollary is proved. 
\end{proof}

\begin{proof}[Proof of Theorem~\ref{thm:existence_of_solutions_nice_tropical_varieties}] 
We have proven the uniqueness part of tropical Calabi--Yau theorem in Theorem~\ref{thm:uniqueness}. It remains to show the existence part. Let $\mu$ be a nonnegative Radon measure of volume $\deg(\omega^d)$ compactly supported in $\xo$. Let $\varphi$ be a maximizer of $\FE_{\mu}$. Since $\CR=\xo$, by Corollary~\ref{cor:domination}, we have $\MA(\varphi)\leq \mu$ on $\xo$. Note that $\MA(\varphi)(\partial X)=0$ (Theorem \ref{thm:Monge-Ampere-measure-takes-no-mass-boundary}), and $\MA(\varphi)(X)=\mu(X)=\deg(\omega^d)$ (Corollary \ref{cor:global_integration-by-parts_formula_and_volume_omega_psh}~(2)). Thus, we obtain that $\MA(\varphi)= \mu$ on $X$, and the existence of solutions follows. 
\end{proof}

%%%%%%%%%%%%%%%%%%%%%%%%%%%%%
%%%%Failure CY theorem
%%%%%%%%%%%%%%%%%%%%%%%%%%%%%

\section{Failure of the tropical Calabi--Yau theorem} \label{sec:failure_of_Calabi_Yau}

We define the following genericity condition.

\begin{definition} [\textbf{Condition (G) for a 2-dimensional tropical fanfold}] \label{condition:genericity}
  Let $Z$ be a 2-dimensional tropical fanfold embedded in a real vector space, and denote by $E$ the linear subspace generated by $Z$. Let $n$ denote the dimension of $E$. We say $Z$ verifies Condition (G) if the following holds: we have $n\geq 3$ and for any $(n-2)$-dimensional subspace $F \subset E$, we have $\dim \pi(Z) =2$, where $\pi$ denotes the projection map $E \to E/F$. We say that a tropical fan $\Sigma$ verifies (G) if the corresponding fanfold $\supp{\Sigma}$ does.
\end{definition}
\begin{remark}
 Condition $G$ is an intrinsic property of the tropical fanfold, 
it does not depend on the real vector space into which the fanfold embeds.
 \end{remark}

\begin{remark} A 2-dimensional tropical fanfold $Z$ in $E=\R^3$ fails to satisfy Condition (G) if and only if $Z$ is the product of $\R$ and a tropical fanfold of dimenson one, that is, if and only if $E$ admit a decomposition of the form $E=E_1\times E_2$ for $E_1\simeq \R$ and $E_2\simeq \R^2$, 
 and there exists a $1$-dimensional tropical fanfold $\Gamma \subset E_2$ such that $Z = E_1\times \Gamma$.
\end{remark}

The main purpose of this section is to prove the following theorem.

\begin{thm} \label{thm:global_failure_of_orthogonality}
    Let $X$ be a compact \prekahler tropical variety of dimension $d$. Suppose that there exist a polyhedral structure $\olPi$ on $X$ with induced polyhedral structure $\Pi$ on $\xo$, and $\sigma \in \Pi_{d-2}$
     such that the star fan $\Pi^{\sigma}$ satisfies the genericity condition (G). Then, for any strongly semipositive form $\omega$ on $X$, the pair $(X, \omega)$ does not verify the tropical Calabi--Yau theorem.  
\end{thm}
An explicit example is treated in Section~\ref{subsec:example_failure_of_tropical_Calabi-Yau}.

\subsection{A combinatorial result} We start by the following definition.
\begin{defi} \label{def:nonnegatively_depedent}
   A finite collection  $\{e_1, \dots, e_k\}$ of vectors in a real vector space $E$ is \emph{nonnegatively} (\resp \emph{positively}) dependent if there exist reals $a_i \ge 0$ (\resp $a_i >0$), not all equal to 0, such that $\sum_{i=1}^{k} a_i e_i=0$, . 
   Otherwise, we say that the the collection is nonnegatively (\resp positively) independent.

   A finite collection of rays $\{\rho_1, \dots, \rho_k\}$ in $E$ is nonnegatively (\resp positively) independent if a (equivalently, any) collection of vectors $\{e_1, \dots, e_k\}$ with $e_j$ generator of $\rho_j$ is nonnegatively (\resp positively) independent.
  \end{defi}

\begin{prop} \label{prop:existence_of_balanced_rays_in_2d_fan}
    Suppose that $Z$ is a 2-dimensional tropical fan in $E=\R^n$ satisfying
    \begin{enumerate}
      \item $Z$ satisfies the genericity condition (G).
      \item $\mathrm{span}(Z) =E$.
      \item $Z$ is connected through codimension 1.
    \end{enumerate}
    Then, there exist an $(n-2)$-dimensional subspace $F \subset E$ and a collection $\{\rho_1, \cdots \rho_{n-1}\}$ of $(n-1)$ nonnegatively dependent rays inside $F \cap Z$. 
\end{prop}
We postpone the proof of this proposition to Section~\ref{subsec:proof-existence_of_balanced_rays_in_2d_fan}, and first explain how deduce the main result of the section.

\subsection{Proof of Theorem~\ref{thm:global_failure_of_orthogonality}}
We start by proving a local result.
\begin{prop} \label{prop:local_failure_of_orthogonality}
  Let $E_1$ and $E_2$ be real vector spaces of dimension $d_1$ and $d_2$ respectively, and $E=E_1\times E_2$. 
  Let $Z \subset E_1$ be a 2-dimensional tropical fanfold satisfying the genericity  condition (G), and let $Z' \coloneqq Z \times E_2 \subset E$.
    Then, there exists $u \in \LConv \cap \PL(Z')$ such that 
  \begin{enumerate}
    \item $\MA_{Z}(u)(\{0\})>0$. Here, $\MA_Z$ denotes  the local tropical Monge--Ampère operator on $Z$.
    \item $u$ is not strictly convex at $0$. More precisely, for any open neighborhood $U$ of $0$ in $Z'$, there exists a finite subset $\{x_i\}_{i \in I} \subset U \setminus \{0\}$ such that for every $v \in \LConv(U)$ with $v(x_i) \le u(x_i)$ for all $i \in I$, we have $v(0) \le u(0)$. 
  \end{enumerate}
\end{prop}
\begin{proof}
  We first show that it is enough to prove the statement when $Z'=Z$, that is, when $E_2=(0)$. Suppose that $u_1 \in \LConv \cap \PL_{\hom}(Z)$ is a function satisfying the properties (1) and (2) in the proposition for $Z'=Z$, and $E_2 =(0)$. In particular, $\MA_Z(u_1)(\{0\})>0$.  Take any $u_2 \in \Conv \cap \PL_{\hom}(E_2)$ such that $\MA_{E_2}(u_2)(\{0\})>0$.  Then, we define $u \coloneqq u_1 \circ \pi_1 +u_2 \circ \pi_2$, where $\pi_1$ and $\pi_2$ are the projection maps from $Z'$ to $Z$ and $E_2$, respectively. 
   We claim that $u$ satisfies the two properties in the proposition. Fix simplicial fan structures $\Sigma$ and $\Xi$ on $Z$ and $E_2$, respectively, such that $u_1$ and $u_2$ are conewise linear on $\Sigma$ and $\Xi$, respectively. The product $\Sigma'=\Sigma\times \Xi$ is a fan structure on $Z'$, and $u$ is conewise linear on $\Sigma'$.
  The Künneth formula for the Chow rings of fans says that $\ssA^\bul(\Sigma')=\ssA^\bul(\Sigma) \otimes \ssA^\bul(\Xi)$, and the degree map on the left is the product of the two degree maps on the right under this decomposition (see \cite[Prop.\,3.5]{AP-hodge-fan}). 
  Therefore, $\MA_{Z'}(u)(\{0\})=\MA_{Z}(u_1)(\{0\}) \cdot  \MA_{E_2}(u_2)(\{0\})>0$. At the same time, for any open neighborhood $U=U_1 \times U_2$ of $0$ in $Z'$, where $U_1$, $U_2$ are open neighborhood of $0$ in $Z$ and $E_2$, respectively, take a finite subset $\{x_i\}_{i \in I} \subset U_1 \setminus \{0\}$ which satisfies property (2) for $u_1 \in \LConv \cap \PL(\Sigma)$. Notice that for any $v \in \LConv(U)$, we have $v \rest{U_1\times \{0\}} \in \LConv(U_1 \times \{0\})$.  Hence, by assumption, $v((x_i,0)) \le u((x_i,0))$ for each $i$ would imply $v(0) \le u(0)$. Therefore, the finite subset $\{(x_i, 0)\}_{i \in I} \subset U \setminus \{0\}$ satisfies the property (2) for $u$, and hence the reduction step is completed.

  \medskip

  In the remainder of the proof, we assume that $Z'=Z$, and $E=E_1$. We prove the proposition by dividing into two cases.

  \smallskip

 \noindent \textbf{Case 1:} $Z$ is not connected through codimension~1.  
  In this case, since $\dim(Z)=2$, there exist tropical fanfolds $Z_1$ and $Z_2$ in $E$ such that  $Z=Z_1 \cup Z_2$, and $Z_1 \cap Z_2=\{0\}$. By Lemma \ref{lem:PL_functions_with_strict_positive_MA_measures}, there exists $u_1 \in \LConv \cap \PL_{\hom}(Z_1)$ 
  such that $\MA_{Z_1}(u_1)>0$. Take an affine support function $l$ of $u_1$ at $0$, and define $u \in \LConv \cap \PL_{\hom}(Z)$ by $u \coloneqq u_1$ on $Z_1$ and $u \coloneqq l$ on $Z_2$. It remains to check that $u$ satisfies the two properties stated in the proposition. Observe that $\MA_{Z}(u)(\{0\})= \MA_{Z_1}(u_1)(\{0\})+ \MA_{Z_2}(l)(\{0\})=\MA_{Z_1}(u_1)(\{0\})>0$.  On the other hand, for any open neighborhood $U$ of $0$ in $Z$, 
  take a finite subset $\{x_i\}_{i \in I} \subset Z_2 \cap (U \setminus \{0\})$ such that $0 \in \conv(\{x_i\}_{i \in I}) \subset U \cap Z$, where $\conv(\{x_i\}_{i \in I})$ denotes  the convex hull of $\{x_i\}_{i \in I}$ in $E$. Then, for any $v \in \LConv(U)$ satisfying $v(x_i) \le u(x_i)=l(x_i)$ for each $i \in I$, Proposition \ref{prop:equivalent_def_LConv} implies that $v(0) \le u(0)=0$. Therefore, the proposition is concluded in this case.
\smallskip

 \noindent \textbf{Case 2:} $Z$ is  connected through 
 codimension~1.   Without loss of generality,
  may assume that  $E=\mathrm{span}(Z)$ and $\dim E=n$. Then,  $Z$ verifies all the three conditions in Proposition \ref{prop:existence_of_balanced_rays_in_2d_fan}. Let $F \subset E$ be the $(n-2)$-dimensional subspace and $\{\rho_1, \cdots, \rho_{n-1}\}$ the $(n-1)$ rays in $F \cap Z$ given by that proposition.
  Since $Z$ satisfies the genericity condition (G), and is a tropical fanfold, 
  we have $\pi(Z)=E/F$, where $\pi$ denote the quotient map $E \to E/F$. We equip $E/F$ with a tropical fanfold structure as the pushforward of $Z$, as in Section \ref{subsubsec:pushforward_of_weight_functions} (the corresponding tropical weight function is a positive constant). 
  Let $u' \in \Conv \cap \PL_{\hom}(E/F)$ such that $\MA_{E/F}(u')(\{0\})>0$. Then, $u \coloneqq u' \circ \pi \in \LConv \cap \PL_{\hom}(Z)$. 
  Since $\MA_{Z}(u)$ is a Dirac mass at $0$, 
  %\ref{prop:compatibility_Monge-Ampere_projections} 
  \cite[Prop.\,5.20]{APW1} implies that $\MA_{Z}(u)(\{0\})=\MA_{E/F}(u')(\{0\}) >0$. On the other hand, let $U$  be an open neighborhood of $0$ in $Z$. 
  Take $x_i \in \relint(\rho_i)$ for each $1 \le i \le n-1$ such that $\conv(\{x_i\}_{1 \le i \le n-1}) \subset U \cap Z$, where $\conv(\{x_i\}_{1 \le i \le n-1})$ denotes the convex hull of $\{x_i\}_{1 \le i \le n-1}$ in $E$. Since $\{\rho_i\}_{1 \le i\le n-1}$ is nonnegatively dependent in $F$, $(0, (x_i)_{1 \le i\le n-1}, (\nu_i)_{1 \le i\le n-1})$ is a polyhderal convex combination for some real numbers $0 \le \nu_i \le 1$. Notice that $u \rest{F \cap \Sigma}=0$. 
  Thus, by Proposition \ref{prop:equivalent_def_LConv},  for any $v \in \LConv(U)$ satisfying $v(x_i) \le u(x_i)=0$ for each $1 \le i\le n-1$, we have $v(0) \le 0=u(0)$. Therefore, $u$ satisfies the two properties in the proposition, and the proposition is proven in this case.

This finishes the proof.  
\end{proof}

We are now ready to prove the main theorem of this section. 
\begin{proof}[Proof of Theorem \ref{thm:global_failure_of_orthogonality}] \label{proof:thm:global_failure_of_orthogonality}
  Notation as in the statement of the theorem, let $\sigma \in \Pi_{d-2}$ such that $\Pi^{\sigma}$ satisfies Condition (G). We take a point $x \in \relint(\sigma)$. Writing $\Pi^{x}=\Pi^{\sigma} \times \R^{d-2}$, and 
  applying Proposition \ref{prop:local_failure_of_orthogonality} to $Z=\supp{\Pi^x}$, we obtain $u \in \LConv \cap \PL(\supp{\Pi^{x}})$ that satisfies (1) and (2) in that proposition.  

  Take an open neighborhood $W$ of $x$ such that $W$ is isomorphic to an open neighborhood of $0$ in $\supp{\Pi^x}$. 
  Since $\omega$ is a strongly semipositive form on $X$, by definition (see Definition \ref{def:strongly_semipositive_form}) there exists $\psi \in \PSH_{\omega} \cap \PL_{\omega}(X)$ such that $\psi+h=Cu+l$ on an open neighborhood $U$ of $x$ in $W$,  where $h$ is the local potential of $\omega$ on $U$, $C$ is a positive constant,  and $l$ is an affine function on $U$.
  Therefore, $\MA_{\omega}(\psi)(\{x\})=\MA_{Z}(Cu)(\{x\})>0$. Furthermore, there exists a finite subset $\{x_i\}_{i \in I} \subset U$ such that for every $v \in \LConv(U)$ with $v(x_i) \le (\psi+h)(x_i)$ for all $i \in I$, we have $v(x) \le (\psi+h)(x)$. 

  We choose a refinement $\olPi'$  of $\olPi$ on $X$, inducing a refinement $\Pi'$ of $\Pi$ on $\xo$, such that $\{x\} \cup \{x_i\}_{i \in I} \subset \Pi'_0$, and $\psi$ is $\omega$-facewise affine on $\olPi'$. 
  Set the finite collection of pairs $\CS\coloneqq \{ (x, \psi(x))\ |\ x \in \Pi'_0 \}$. Using Lemma \ref{lem:envelope_PL_functions_vertices}, we have $\psi=\envomega(f_{\CS})$. 
  Then, by assumption, $\psi=\envomega(f_{\CS'})$ for $\CS'\coloneqq \left\{ (x, \psi(x))\ \st \ x \in \Pi'_0 \setminus \{x\} \right\}$. Therefore, $\Supp(\MAo(f_{\CS'})) \not\subset \Pi'_0 \setminus \{x\}$. Hence, by Theorem \ref{thm:equivalence_orthogonality_existence_of_solutions}, $(X, \omega)$ does not verify the tropical Calabi--Yau theorem.    
\end{proof}

\subsection{An explicit example for the failure of the 
tropical Calabi--Yau theorem}\label{subsec:example_failure_of_tropical_Calabi-Yau} 
We present an explicit example that elucidates  the proof of Theorem \ref{thm:global_failure_of_orthogonality}. 

\begin{example} Let $Z$ be the fanfold of dimension $2$ in $E=\R^3$ associated to the uniform rank 3 matroid $U_{3,4}$ on 4 elements. 
It is given as follows. Let $(\e_1,\e_2,\e_3)$ be the standard basis of $\R^3$ 
and $\e_0 = -\e_1-\e_2-\e_3$. Then, denoting $\rho_i \coloneqq \R_{\geq0} \e_i$, we have 
$ Z= \bigcup_{0\leq i<j\leq 3}\rho_i+\rho_j$.

\begin{figure}[ht]
  \begin{tikzpicture}[scale = 1]
    \let\i\relax

    \BFcoordinates

    \foreach \i/\j in {1/1,2/3,3/2,0/0} {
      \draw (I) -- (I\i);
      \draw ($(I)!1.1!(I\i)$) node {\j};
    }
    \foreach \i/\j/\cc in {1/3/blue, 2/3/green, 1/0/black, 3/0/violet, 1/2/red, 2/0/yellow} {
      % \draw (I) -- (I\i\j);
      % \draw ($(I)!1.1!(I\i\j)$) node {\ifnumequal{\j}{0}{\j\i}{\i\j}};
      \roundface\i\j[\cc, opacity=.1];
    }
    \foreach \k/\l/\r in {20/03/1,13/12/1} {
      \draw (I) -- ($(I)!\r!(I\k)$) node[pos=1.1] {\l};
    }
  \end{tikzpicture}
  \caption{The fan $\Sigma$ with support the Bergman fanfold $Z$ of the uniform matroid $U_{3,4}$. \label{fig:counterexample_fan}}
\end{figure}
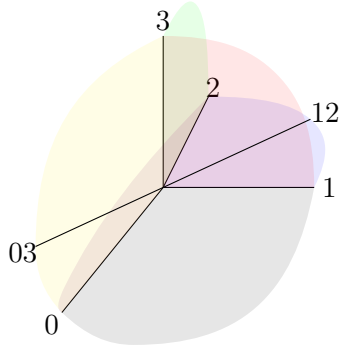

We define a unimodular fan $\Sigma$ with support $Z$ whose rays are the $\rho!_J$ for $J$ labeled as in Figure \ref{fig:counterexample_fan}. 
For instance, for $J=03$, the corresponding ray is $\rho!_{03} = \R_{\geq 0}\e_{03}$ where $\e_{03} \coloneqq \e_0 + \e_3$. Let $X=\comp{\Sigma}$ be the canonical compactification of $\Sigma$, which is a compact tropical variety with one single chart given by $\comp{\Sigma}$. 
Note that we have $\xo=Z$.

Let $h \in \PL_{\hom}(Z)$ defined by the following table on the generators of the rays of $\Sigma$
\[ \begin{array}{r|cccccccccccccccc}
  J         & 0 & 1 & 2 & 3 & 03 & 12  \\
  h(\e_J)   & 1 & 1 & 1 & 1 & 2  & 2   \\
\end{array} 
\]
extended linearly on each cone of $\Sigma$. Notice that $h$ is a convex conewise linear  function on $\Sigma$. Moreover, Proposition \ref{prop:example_weak_kahler_form_on_projective_tropical_variety} shows that $\omega=\dm \dmb h$ is a strongly semipositive simplicial form on $X$. A direct calculation shows that $\omega^2=16 \delta_{0}$.

Observe that $\e_{03}=-\e_{12}$. Let $F$ be the line containing $\e_{03}$ and $\e_{12}$. 
Then $F \subset E$ is a line satisfying the conditions in Proposition \ref{prop:existence_of_balanced_rays_in_2d_fan}. Let $\pi\colon E \to E/F$ be the quotient map. It is clear that $\dim(\pi(Z))=2$. 

We define $g \in \PL_{\hom}(Z)$ by the following table on the generators of the rays of $\Sigma$,
\[ \begin{array}{r|cccccccccccccccc}
  J           & 0 & 1 & 2 & 3 & 03 & 12  \\
  g(\e_J)   & 1 & 1 & 1 & 1 & 0  & 0   \\
\end{array} \]
 extended linearly to each cone of $\Sigma$. Then it is 
  not hard to see that $g=g' \circ \pi$ for some strictly convex homogeneneous piecewise linear function $g'$ on $E/F$. Therefore, by \cite[Prop.\,5.2]{APW1} 
  %cross-ref in APW1: \ref{prop:compatibility_Monge-Ampere_projections} 
  we have $\MA(g)\{0\}>0$. In fact one can check that $\MA(g)\{0\}=8$ via a direct calculation in the Chow ring of $\Sigma$.
  Furthermore, since $g\rest{F}=0$, $g \in \LConv \cap \PL(\Sigma)$ satisfies the conditions in Proposition \ref{prop:local_failure_of_orthogonality}.

We now make use of this function $g$ to construct an explicit counterexample to tropical Calabi--Yau theorem. 
Observe that 
\[\max\{g, h-1\} \in \LConv(\xo), \qquad  |\max\{g, h-1\}-h|=O(1).\]
Therefore, by Theorem \ref{thm:equivalent_def_psh}, $\max\{g, h-1\}-h$ can be extended to $\varphi \in \PSH_{\omega} \cap \PL(X)$.
The restriction of $\varphi+h$ on the cones $\rho_0+\rho_1$, $\rho_0+\rho_2$, $\rho_1+\rho_3$, $\rho_2+\rho_3$ equals to $g$. The restriction of $\varphi+h$ on the cone $\rho_1+\rho_2$ is shown in Figure \ref{fig:counterexample_wing}, so one can observe that $\varphi+h=g$ in the blue and yellow region, and $\varphi+h=h$ in the green region.  The figure for $\rho_0+\rho_3$ is symmetric. Therefore, $\varphi+h=g$ in a neighborhood of $\bigcup_{0 \le i \le 3} \rho_i$, and $\MA_{\omega}(\varphi)$ is supported on $\{0, a, b\}$, where $a$ and $b$ denote $\frac{1}{2}\e_{12}$ and $\frac{1}{2}\e_{03}$ respectively. 
A direct calulation shows that $\MA_{\omega}(\varphi)=8\delta_0+4\delta_{a}+4\delta_{b}$.

We consider the finite collection of pairs $\CS=\{(a,0), (b, 0)\} \subset \Sigma_{o} \times \R$. Then $\envomega(f_{\CS})=\varphi$, and therefore
 $f_{\CS}$ does not verify the orthogonality property. Following the proof of Theorem \ref{thm:equivalence_orthogonality_existence_of_solutions}, one can show that the measure \[\mu=8 \delta_{a}+8 \delta_{b}\] 
does not admit any solution in $\PSH_{\omega}(X)$. 
\end{example}
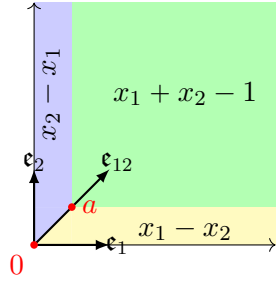
\begin{figure}
  \begin{tikzpicture}[scale=1]

    \fill[blue!20] (0,0.5) -- (0,3.2) -- (0.5,3.2) -- (0.5,0.5) -- (0,0.5);
    \fill[blue!20] (0,0.5) -- (0.5,0.5) -- (0,0) -- cycle;

    \fill[yellow!30] (0.5,0) -- (3.2,0) -- (3.2,0.5) -- (0.5,0.5) -- (0.5,0);
    \fill[yellow!30] (0,0) -- (0.5,0) -- (0.5,0.5) -- cycle;

    \fill[green!30] (0.5,0.5) -- (3.2,0.5) -- (3.2,3.2) -- (0.5,3.2) -- cycle;

    \draw[->] (0,0)--(3.2,0);
    \draw[->] (0,0)--(0,3.2);

    \draw[-latex, thick] (0,0)--(1,0) node[pos=1.1] {$\e_1$};
    \draw[-latex, thick] (0,0)--(0,1) node[pos=1.1] {$\e_2$};
    \draw[-latex, thick] (0,0)--(1,1) node[pos=1.1] {$\e_{12}$};

    \draw
    (2,2) node {$x_1+x_2-1$}
    (2,.2) node {$x_1-x_2$}
    (.2,2) node[rotate=90] {$x_2-x_1$};

    \fill[red] (0,0) circle (.05) node[below left] {$0$};
    \fill[red] (0.5,0.5) circle (.05) node[right] {$a$};

  \end{tikzpicture}
  \caption{The values of the function $\phi+h$ on the cone $\rho_1+\rho_2$.} \label{fig:counterexample_wing}
\end{figure}

\begin{remark} \label{rmk:example_semipositive_form_with_positive_degree_but_not_strongly_semipositive}
  Take $\omega' \coloneqq \dm \dmb g$ in the above example. Then, $\omega$ is a semipositive form on $Z$ with $\deg(\omega^2)=8>0$. However, $\omega'$ is not strongly semipositive. To see this,   observe that $\omega'|_F=0$, where $F$ is the line in $E$ containing $\fe_{03}$ and $\fe_{12}$. This implies that $\varphi|_F$ is constant for any $\varphi \in \CPSH_{\omega}(Z)$. On the other hand, one can easily find an open neighborhood $W$ of $0$ in $Z$ and  $f \in \LConv(W)$  such that $f|_{U \cap F}$ is non-constant. Therefore, the condition in Definition \ref{def:strongly_semipositive_form} cannot be satisfied for $\omega'$ and  $(W,f, 0, g)$. 
\end{remark}

\begin{remark}
  Notice that if a $2$-dimensional fanfold $Z$ can be written as a product $Z=\Gamma_1 \times \Gamma_2$ 
  for two $1$-dimensional tropical fanfolds $\Gamma_1, \Gamma_2$, then $Z$ does not verify the genericity condition (G). However, if neither $\Gamma_1$ nor $\Gamma_2$ is equal to $\R^1$, we can still choose an
   $(n-2)$-dimensional subspace $F \subset E$ satisfying the conditions in Proposition \ref{prop:existence_of_balanced_rays_in_2d_fan}, and moreover verifying that $\dim(\pi(Z))=2$, for the quotient map $\pi \colon E \to E/F$. We illustrate this in Example \ref{example_products_of_tropical_curves_not_satisfying_orthogonality_property} below. 
   In this situation, the arguments in Proposition \ref{prop:local_failure_of_orthogonality} and Theorem \ref{thm:global_failure_of_orthogonality} apply verbatim, so Theorem \ref{thm:global_failure_of_orthogonality} still holds for any projective compactification of $Z$.
\end{remark} 
  
\begin{example} \label{example_products_of_tropical_curves_not_satisfying_orthogonality_property}
  Suppose that $\Sigma_1$ is a tropical fan in $\R^2$ with 3 rays $\R_{\geq 0} \e_1$, $\R_{\geq 0} \e_2$,
   $\R_{\geq 0}  \e $, where $\{\e_1,\e_2\}$ is a base of $\R_2$, and $\e \coloneqq -\e_1-\e_2$. Suppose further that $\Sigma_2 \cong \Sigma_1$, where we denote  the rays of $\Sigma_2$ by $\R_{\geq 0} \e'_1$, $\R_{\geq 0}  \e'_2$, $\R_{\geq 0} \e$ with $\e' \coloneqq -\e'_1-\e'_2$. Let $\Sigma=\Sigma_1 \times \Sigma_2 \subset E=\R^4$. Then, one can take nonnegatively dependent rays $\rho_1 \coloneqq \R_{\geq 0}(\e_1+\e_1')$, $\rho_2 \coloneqq \R_{\geq 0}(\e_2+\e_2')$, $\rho_3 \coloneqq \R_{\geq 0}(\e+\e')$ inside $\Sigma$, and let $F$ be the subspace generated by them. Then, it is not difficult to check that $\dim F=2$ and $\dim \pi(\Sigma)=2$, for the quotient map $\pi\colon E \to E/F$.
\end{example}

Actually, we have the following conjecture. 
\begin{conj} \label{conj:exact_condition_for_orthogonality_properties}
   Let $(X, \omega)$ be a compact \prekahler tropical variety equipped with a strongly semipositive form $\omega$. Then, $(X, \omega)$ verifies the tropical Calabi--Yau theorem if and only if $\CR=\xo$, \textit{i.e.} every point $x \in \xo$ is regular in codimension one (see Definition \ref{def:regular_in_codim_one}). 
\end{conj}

\subsection{Proof of Proposition~\ref{prop:existence_of_balanced_rays_in_2d_fan}}\label{subsec:proof-existence_of_balanced_rays_in_2d_fan}
It remains to prove the combinatorial statement in Proposition~\ref{prop:existence_of_balanced_rays_in_2d_fan}. We will use several polyhedral geometric results in the proof. In order to increase the readability, we gather them all in Appendix \ref{appendix:polyhedral_geometric_lemmas}, to which we refer in the course of the proof.
We also recall the following result in topology.
\begin{lemma} \label{lem:application_Borsuk_Ulam_theorem}
  Let $D^{n+1}$ be the $(n+1)$-dimensional disk, and $S^n \coloneqq \partial D^{n+1}$ the $n$-dimensional sphere. Define the open northern and southern hemisphere by 
    \[
    S^n_+ \coloneqq \{x \in S^n \st x_{n+1} > 0\}, \quad
    S^n_- \coloneqq \{x \in S^n \st x_{n+1} < 0\}.
    \]
  Let $f \colon D^n \to D^{n+1}$ be a continuous map whose restriction $f \rest{\partial D^n}$ induces a homeomorphism from $\partial D^n$ to the equator $S^{n-1} \subset S^n$. Let $\gamma \colon [0,1] \to D^{n+1}$ be a curve such that $\gamma(0)=x_+  \in S^{n}_{+}$ and $\gamma(1)=x_- \in S^{n}_{-}$. Then, there exist $t_0 \in [0,1]$ and $x_0 \in D^n$ such that $\gamma(t_0)=f(x_0)$. 
\end{lemma}
\begin{proof}
  This is a standard application of the Borsuk-Ulam Theorem~\cite[Thm.\,2.1.1]{Mat03}. Since we were not able to find a reference for this precise statement in the literature, we sketch a proof here. 
  Proceeding by contradiction, suppose the statement is not correct. Define the continuous map $g \colon D^n \times [0,1] \to S^n$ by $(x, t) \mapsto (f(x)-\gamma(t))/\left\| f(x)-\gamma(t) \right\|$. Notice that $D^n \times [0,1]$ is homeomorphic to $D^{n+1}$. It is not hard to show that the map $g \rest{\partial (D^n \times [0,1])} \colon \partial (D^n \times [0,1]) \to S^n$ has degree $1$. But at the same time, $g$ is nullhomotopic since $D^n \times [0,1]$ is retractable. This yields a contradiction. 
\end{proof}

We are prepared to prove Proposition \ref{prop:existence_of_balanced_rays_in_2d_fan}.  We proceed in several steps.
  
  \medskip
  
  \noindent $\bullet$ \textbf{Step 1. Reduction to tropical curves.} \label{Step_1}

  \smallskip

  We fix a fan structure $\Sigma$ on $Z$. 
  Also, we fix an Euclidean metric on $E$.  
  
  Consider the dual space $E^\vee$, and take a generic rational element $l \in E^{\vee}$ such that the $(n-1)$-dimensional vector subspace $A \coloneqq\{l=0\} \subset E$ does not contain any nonzero cone $\sigma$ in $\Sigma$.
 Then, $\{l=1\} \cap Z$ and $\{l=-1\} \cap Z$ are both tropical curves. 
 Let $\pi \colon E \to A$ be the orthogonal projection (with respect to the fixed Euclidean metric). Denote the image of these two tropical curves in $A$ by 
 \[
 \Gamma \coloneqq \pi \left(\, \{l=1\} \cap Z \,\right) \text{ and }\Gamma' \coloneqq \pi \left(\, \{l=-1\} \cap Z \,\right), \text{ and set } -\Gamma' \coloneqq -\pi \left(\, \{l=-1\} \cap Z \,\right).
 \]

  Let $\fC$ be the subset of $A$ defined by 
  \[
  \fC  \coloneqq \left\{\, x=\sum_{i=1}^{n-2} \lambda_i x_i \in A \ \st \  x_i \in \Gamma, \,\,\lambda_i\geq 0\text{ for each } 1 \le i \le n-2, \text{ and } \sum_{i=1}^{n-2} \lambda_i=1 \,\right\}.
  \] 
  In the rest of the proof, we show that $\fC \cap (-\Gamma')\neq \emptyset$. This implies the proposition, by taking the $(n-2)$ rays  generated by $x_i\in \Gamma$, and the one generated by $x$ with $x=\sum_i\lambda_i x_i\in \fC\cap(-\Gamma')$.

  \medskip

 \noindent $\bullet$ \textbf{Step 2. Selection of rays.} \label{Step_2}

 \smallskip

  Set $\Upsilon = \rec(\Gamma)$, the recession cone of $\Gamma$, and note that $\Upsilon=\rec(\Gamma')$. Notice that $\Upsilon$ is balanceable (that is, the rays of $\Upsilon$ are positively dependent), and $\mathrm{span}(\Upsilon)=A$.
  Applying Lemma \ref{lem:completion_of_fans_with_prescribed_rays} to the balanceable $1$-dimension fan $\Upsilon$  and the simplicial fan $\{ \conezero \}$ supported on $A$, we obtain a complete simplicial fan $\Xi''$ supported on $A$ such that the set of rays $\Xi''$ and that of $\Upsilon$ coincide, that is, 
  $\Xi''_1=\Upsilon_{1}$. 

  For $\rho \in \Upsilon_1$, we use $\e_{\rho}$ for the primitive vector on $\rho$. Fix $\rho_0 \in \Upsilon_{1}$. Let $\sigma \in \Xi''$ such that  $-\e_{\rho_0} \in \relint(\sigma)$.  Let $\rho_1, \cdots, \rho_{m} \in \Upsilon_1$ be the rays generating $\sigma$, where $m=\dim \sigma \le n-1$. Then, we have $-\e_{\rho_0}=\sum_{1 \le j \le m} a_j \e_{\rho_j}$ for positive $a_j >0$, $j=1, \dots, m$. Since $\sigma$ is simplicial, $\e_{\rho_1}, \cdots, \e_{\rho_m}$ are
   linearly indepedent, and generate an $m$-dimensional vector subspace $A_0 \subset A$.  

  Let $l \coloneqq \R\e_{\rho_0} \subset A$ be the linear span of $\rho_0$. 
  Consider the projection map  $p \colon A \to B \coloneqq A/l$.  Denote $\Upsilon'=p(\Upsilon)$. Let $B_0$ be the image of $A_0$ under the quotient map $p$. 
  We denote $\rho'_j \coloneqq p(\rho_j) \in \Upsilon'_1$, and $\e_{\rho'_j} \coloneqq p(\e_{\rho_j})$.   Then, $B_0$ is an $(m-1)$-dimensional vector subspace in $B$, and we have 
  \[ 
  \mathrm{span}_{1 \le j \le m}(\{\e_{\rho'_j}\})=B_0, \quad \text{and} 
  \quad  
  0=\sum_{1 \le j \le m} a_j \e_{\rho'_j} \qquad \text{in }B_0.
  \] 
   Therefore, the rays $\rho'_1, \dots, \rho'_m$ generate a complete simplicial fan supported on $B_0$ that we denote by $\Xi'$. Each cone of $\Xi'$ is generated by a proper subsets of $\{\rho'_1, \dots, \rho'_m\}$. 

  Notice that $\Upsilon'$ is a balanceable $1$-dimensional fan supported on $B$. Then, the subfan of $\Upsilon'$ defined by $\CL \coloneqq \{\conezero \} \cup \{ \rho'_j  \,\st \, 1 \le j \le m  \} \cup \{\xi \,|\, \xi \in \Upsilon'_1, \text{ and } \xi \nsubseteq B_0 \}$ is also balanceable. Applying Lemma \ref{lem:completion_of_fans_with_prescribed_rays} to $\CF'= \Xi'$ and $\CL$ defined above, we obtain a complete simplicial fan $\Xi$ supported on $B$ such that $\Xi'$ is a subfan of $\Xi$, and $\Xi_1 \subset \Upsilon'_1$.

  We choose a map $s_{(1)} \colon \Xi_1 \to \Upsilon_{1}$ such that $p_{(1)}(s_{(1)}(\xi))=\xi$ for each $\xi \in \Xi_1$, and $s_{(1)}(\rho'_j)=\rho_j$ for each $1 \le j \le m$. Here $p_{(1)}$ denote the map induced by $p$, which sends the set of rays in $A$ to the set of rays in $B$. This induces a continuous map $s \colon B \to A$ by extending $s_{(1)}$ linearly on each cone of $\Xi$. Therefore, $s$ is a  conewise linear map on $\Xi$, and $p \circ s=\mathrm{id}_B$. 
  
  Define the open subsets
  \[
  U_+ \coloneqq \{x \in A \ |\ x = s(p(x))+ a \cdot \e_{\rho_0} \text{ for some } a >0 \} ,
  \]
  \[
  U_- \coloneqq \{x \in A \ |\ x = s(p(x))+ a \cdot \e_{\rho_0} \text{ for some } a <0 \} 
  \]

  Since $s(p(-\e_{\rho_0}))=0$, we have $-\e_{\rho_0} \in U_-$. On the other hand, $-\e_{\rho'_1}=\sum_{2 \le j \le m} \frac{a_j}{a_1} \e_{\rho'_j}$ with $\frac{a_j}{a_1} >0$. Hence, $-\e_{\rho'_1} \in \mathrm{Cone}(\rho'_2, \cdots, \rho'_m) \in \Xi'$. Since $\Xi'$ is a subfan of $\Xi$ and $s$ is conewise linear on $\Xi$,  we have 
  \[
  s(p(-\e_{\rho_1}))= s(-\e_{\rho'_1}) =\sum_{2 \le j \le m} \frac{a_j}{a_1} s( \e_{\rho'_j})= \sum_{2 \le j \le m} \frac{a_j}{a_1}  \e_{\rho_j} =-\e_{\rho_1}-\frac{1}{a_1} \e_{\rho_0} .
  \]
  This implies that $-\e_{\rho_1} \in U_+$.

  \medskip

  \noindent $\bullet$ \textbf{Step 3. Constructions of $P \subset B$ and $\widetilde{P} \subset B \times \R$.}  \label{Step_3}

  \smallskip

  Recall from the previous step that $\dim A=n-1$, $\dim B=n-2$, and $\Xi''$ and $\Xi$ are complete simplicial fans supported on $A$ and $B$, respectively. 

  We first construct a subset $P \subset B$. For each $\xi \in \Xi_1$, let $\su_{\xi}$ be the vector of norm one on $\xi$, that is, the unique point of $B(0,1) \cap \xi$ in $B$.   Consider the
   simplicial polyhedral complex
  \[
  \Delta_{P} \coloneqq  \left\{\, \conv(0, \su_{\xi_1}, \su_{\xi_2}, \cdots, \su_{\xi_{k}}) \ \st \ \mathrm{Cone} (\xi_1, \xi_2, \cdots, \xi_{k}) \in  \Xi \right\},
  \]
  and denote by $P \subset B$ its support. Since $\Xi$ is complete, $P$ is homeomorphic to an $(n-2)$-dimensional disk $D^{n-2}$.    The set of vertices of $\Delta_P$ is given by $(\Delta_{P})_0=\{0, \su_{\xi} \,|\, \xi \in \Xi_1\}$.

  Fix $N \in \BN_{+}$ such that $N> \card{\Gamma_0}$, with $\Gamma_0$ the set of vertices of the tropical curve $\Gamma$ in $A$. For each $\xi \in \Xi_1$ and $0 \le i \le N$,  denote $\su_{\xi, i}  \coloneqq B(0,i/N) \cap \xi$. The segment $[0, \su_{\xi}]$ is thus subdivided into $N$ segments by the points $\su_{\xi,i}$.
  Applying Lemma \ref{lem:successive_stellar_subdivision}, we can subdivide $\Delta_{P}$ to a simplicial polyhedral complex $\Delta'_{P}$ such that $(\Delta'_{P})_{0} = \{0, \su_{\xi, i} \,|\, \xi \in \Xi_1, 1 \le i \le N \}$.

  \begin{claim} \label{claim:top_dimensional_simplices_in_subdivision_of_Delta_P}
  Let $\sigma \in \Delta'_{P}$ be an $(n-2)$-dimensional simplex. Then, there exists rays $\xi_1, \xi_2, \cdots, \xi_{n-2} \in \Xi_1$ such  that $\mathrm{Cone}(\xi_1, \xi_2, \cdots, \xi_{n-2}) \in \Xi_{n-2}$, and  
  \[
  \sigma =\conv(\su_{\xi_1, i_1}, \su_{\xi_1, i_1+1}, \su_{\xi_2, i_2}, \su_{\xi_3, i_3}, \cdots, \su_{\xi_{n-2}, i_{n-2}})
  \]
  \end{claim}
  \begin{proof}[Proof of Claim \ref{claim:top_dimensional_simplices_in_subdivision_of_Delta_P}]
    Since $\Delta_{P'}$ is a subdivision of $\Delta_P$, we may assume that 
    \[
    \sigma \subset \conv(0, \su_{\xi_1}, \su_{\xi_2}, \cdots, \su_{\xi_{n-2}}), \ \text{ with } \mathrm{Cone}(\xi_1, \xi_2, \cdots, \xi_{n-2}) \in \Xi_{n-2}
    \]
    Since $\sigma$ is an $(n-2)$-dimensional simplex, one can write $\sigma = \conv(\fS)$, where $\fS$ is a subset of  $\{0\} \cup \{ \su_{\xi_{j}, i} \,|\, 1 \le j \le n-2, 1 \le i \le N \}$ of cardinality $n-1$. 
    
    Suppose that there exists $1 \le j_0 \le n-2$ such that $\fS \cap \{\su_{\xi_{j_0},i} \,|\, 1 \le i \le N \}$ is empty. Then, $\sigma \subset \mathrm{Cone}( \xi_1, \cdots, \widehat{\xi_{j_0}}, \cdots, \xi_{n-2})$, which implies that $\dim \sigma \le n-3$, and hence contradicts with our assumption. Therefore, after possibly permutating $\{\su_{j}\}_{1 \le j \le n-2}$, one can write $\fS=\{ \su_{\xi_1, i_1'}, \su_{\xi_1, i_1}, \su_{\xi_2, i_2}, \su_{\xi_3, i_3}, \cdots, \su_{\xi_{n-2}, i_{n-2}} \}$, where $1 \le i_j \le N$ for each $1 \le j \le n-2$, and $0 \le i_1' < i_1 \le N$. 

    By assumption, for any $i_1'' \in \{0, \cdots, n-2\} \setminus \{i_1, i_1'\}$, the vertex $\su_{\xi_1, i_1''}$ is not contained in $\sigma$. This implies that $\su_{\xi_1, i_1''} \notin \conv(\su_{\xi_1, i_1}, \su_{\xi_1, i'_1})$, and hence we have $i_1=i_1'+1$. Therefore, the claim is concluded. 
  \end{proof}

  We now construct a subset $\widetilde{P } \subset B \times \R$. Denote $\su_{+} \coloneqq (0,1) \in B \times \R$, and $\su_{-} \coloneqq (0,-1) \in B \times \R$.
  Consdier the simplicial polyhedral complex
  \[
  \Delta_{\widetilde{P}} \coloneqq \left\{\, \sigma, \conv(\sigma, \su_{+}), \conv(\sigma, \su_{-}) \ \st \ \sigma \in \Delta_{P} \right\}
  \]
  and denote by  $\widetilde{P}$ its support in $B \times \R$. We set $(\partial \widetilde{P})_{\pm} \coloneqq \partial \widetilde{P} \cap (B \times \R_{\pm})$. We denote $i \colon P \to \widetilde{P}$ the inclusion map given by $x \mapsto (x, 0)$. Then, it is not hard to see that there exists a homeomorphism from $\widetilde{P}$ to an $(n-1)$-dimensional disk $D^{n-1}$ with $\partial P$ mapped to the equator $S^{n-3}$ of $D^{n-1}$, and $(\partial \widetilde{P})_{\pm}$ mapped to the open northern and southern hemisphere respectively. 
  
  We define a subdivision $\Delta'_{\widetilde{P}}$ of $\Delta_{\widetilde{P}}$ by:
  \[
  \Delta'_{\widetilde{P}} \coloneqq \left\{\,\conv(\sigma', \su_{+}), \conv(\sigma', \su_{-}) \ \st \ \sigma' \in \Delta'_{P} \right\}.
  \]
  
  \medskip

  \noindent $\bullet$ \textbf{Step 4. Constructions of $Q \subset A$ and a homeomorphism $h \colon Q \to \widetilde{P}$. } \label{Step_4} 
\smallskip

  Denote by $\Gamma_{1,\infty}$ the set of infinite edges of $\Gamma$. Recall that $\Upsilon$ is the recession cone of $\Gamma$, and we have chosen a map $s_{(1)} \colon \Xi_{1} \to \Upsilon_{1}$ in Step 2.  For each ray $\xi \in \Xi_1$, we choose an infinite edge $\mathrm{Leg}(s_{(1)}(\xi)) \in \Gamma_{1, \infty}$ such that $\mathrm{Leg}(s_{(1)}(\xi))=x+s_{(1)}(\xi)$ for some $x \in \Gamma$, with $s_{(1)}(\xi) \in \Upsilon_1$. Fix a vertex $\sv$ of $\Gamma$, and fix $r \gg 0$. Denote $\sv_{\xi} \coloneqq B(\sv,r) \cap \mathrm{Leg}(s_{(1)}(\xi))$. 

  Recall that $\rho_0$ is the ray of $\Upsilon_1$ which we fixed in Step 2. We choose an infinite edge $\mathrm{Leg}(-\rho_0) \in  (-\Gamma')_{1, \infty}$ such that $\mathrm{Leg}(-\rho_0)=x_0-\rho_0$ for some $x_0 \in -\Gamma'$. Let $\rho_{0, \sv}=\sv+\rho_0$ be a ray originated at $\sv$ with direction $\rho_0$. Denote $\sv_{+} \coloneqq B(\sv, r) \cap \rho_{0, \sv} \in A$, and $\sv_{-} \coloneqq B(\sv,r) \cap \mathrm{Leg}(-\rho_0)x \in A$.

  Denote the finite set $S \coloneqq \{0, +, - \} \cup \Xi_1$. We define an abstract simplicial complex on $S$ by
  \[
  K \coloneqq \{ (0, \xi_1, \cdots, \xi_{k}), (0, +, \xi_1, \cdots, \xi_{k}), (0, -, \xi_1, \cdots, \xi_{k})  \ |\ \mathrm{Cone}(\xi_1, \cdots, \xi_k) \in \Xi \}
  \]
  We view $\su \colon S \to B \times \R$ by setting $\su(0)=0$, $\su(\pm)=\su_{\pm}$, and $\su(\xi)=\su_{\xi}$; and we define $\sv \colon S \to A$ by setting $\sv(0)=\sv$, $\sv(\pm)=\sv_{\pm}$, and $\sv(\xi)=\sv_{\xi}$. 

  For any $T \in K$, we use $\su_T \subset B \times \R$ and $\sv_T \subset A$ to denote the image of $T$ under $\su$ and $\sv$ respectively.  We define simplices $|T|_{\widetilde{P}} \coloneqq \conv(\su_T) \subset B \times \R$, and $|T|_{Q} \coloneqq \conv(\su_T) \subset A$. Notice that the simplicial polyhedral complex $\Delta_{\widetilde{P}}$ equals to  $\{|T|_{\widetilde{P}}\ |\ T \in K\}$.

  Now we define a set of simplices in $A$ by
  \[
  \Delta_{Q} \coloneqq \{ |T|_{Q}  \ |\ T \in K\} \footnote{We will prove $\Delta_{Q}$ is a simplicial polyhedral complex in Claim \ref{claim:Delta_Q_is_polyhedral_complex_and_homeomorphism_of_h}.}
  \]
  with
  \[
  Q \coloneqq (\bigcup_{T \in K} | T |_{Q}) \subset A.
  \] 

  We define a continuous map $h \colon \widetilde{P} \to Q$ which is facewise linear on $\Delta_{\widetilde{P}}$ by $h(\su(\fs))=\sv(\fs)$, for $\fs \in S$, extended linearly on each face of $\Delta_{\widetilde{P}}$.

  \medskip

  \begin{claim} \label{claim:Delta_Q_is_polyhedral_complex_and_homeomorphism_of_h}
  When $r \gg 0$, $\Delta_Q$ is a simplicial polyhedral complex supported on $Q \subset A$, and $h \colon \widetilde{P} \to Q $ is a homeomorphism.
  \end{claim}

  \begin{proof}[Proof of Claim \ref{claim:Delta_Q_is_polyhedral_complex_and_homeomorphism_of_h}]

  For each ray $\xi \in \Xi_1$, denote $\sv'_{\xi} \coloneqq B(\sv,1) \cap s_{(1)}(\xi) \in A$ with $s_{(1)}(\xi) \in \Upsilon_1$. Furthermore, we denote $\sv' \coloneqq 0$, $\sv'_{+} \coloneqq e_{\rho_0}$, and $\sv'_{-} \coloneqq -e_{\rho_0}$. We view $\sv'$ as a map from $S$ to $A$ by setting $\sv'(0)=\sv'$, $\sv'(\pm)=\sv'_{\pm}$, $\sv'(\xi)=\sv'_{\xi}$. Set  $|T|_{Q'} \coloneqq \conv(\sv'_T) \subset A$ for $T \in K$, where $\sv'_T$ denotes the image of $T$ under $\sv'$. We define an auxiliary set of simplices in $A$ by
  $\Delta_{Q'} \coloneqq \{|T|_{Q'} \ |\ T \in K \}$, with support $Q' \coloneqq (\bigcup_{T \in K} |T|_{Q'}) \subset A$. 
  
  For each $\xi \in \Xi_1$, denote $\su'_{\xi} \coloneqq p(\sv'_{\xi})$, and denote $P' \coloneqq p(Q') \subset B$. Notice that $\su'_{\xi} \in \R_{+} \cdot \su_{\xi}$ for each $\xi \in \Xi_1$. Therefore, the simplicial polyhedral complex  $\Delta_{P'} \coloneqq \{|T|_{P'} \ |\ T \in K,\ T \subset \Xi_1 \}$  is isomorphic to $P$, and it has support $P'$. Here $|T|_{P'} \coloneqq \conv(\su'_T) \subset B$ for any $T \in K$ such that $T \subset \Xi_1$.  Denote $s'$ the restriction of the continuous injection $s \colon B \to A$ to $P'$. Notice that $s'(P') \subset Q'$, $s' (\partial P')  \subset \partial Q'$, and  $|T|_{Q'}=s'(|T|_{P'})$ for any $T \subset \Xi_1 \subset S$.

  We show that for any $T_1, T_2 \in K$, we have  $|T_1 \cap T_2|_{Q'}=|T_1|_{Q'} \cap |T_2|_{Q'}$.   When $T_1, T_2 \subset \Xi_1 \subset S$, the equality follows from the facts that $|T_1 \cap T_2|_{P'}=|T_1|_{P'} \cap |T_2|_{P'}$ and $s'$ is injective. Otherwise, we may assume that $+ \in T_1$ without loss of generality. Remark that each of $T_1$ and $T_2$ contains at most one of $\pm$. In particular, we have $- \notin T_1$, and hence $|T_1|_{Q'} \subset \overline{U_+}$. Therefore, there are two possible cases to consider. 
  \begin{enumerate}
    \item If $+ \notin T_2$, then $|T_2|_{Q'} \subset \overline{U_-}$. Thus, $|T_1|_{Q'} \cap |T_2|_{Q'} \subset s'(P') \subset \overline{U_+} \cap \overline{U_-}$. 
    Notice that $|T_1|_{Q'} \cap s'(P') = |T_1 \setminus \{+\}|_{Q'}$, and $|T_1|_{Q'} \cap s'(P') = |T_2 \setminus \{-\}|_{Q'}$. Hence, the equality follows from the previous case. 
    \item Suppose $+ \in T_2$, and $x \in |T_1|_{Q'} \cap |T_2|_{Q'}$. Denote $T'_i \coloneqq T_i \setminus +$ for each $i \in \{1,2\}$.  Notice that $|T_i|_{Q'}=\conv(|T_i'|_{P'}, \sv'_+)$. Thus, there exist $y_i \in |T_i'|_{P'}$ and $\lambda_i \ge 0$, for each $i \in \{1,2\}$, such that $x=s'(y_1)+\lambda_1 \cdot \sv'_+=s'(y_2)+\lambda_2 \cdot \sv'_+$. Equivalently, $s'(y_1)-s'(y_2)=(\lambda_2-\lambda_1) \cdot \sv'_+$. Recall that $p \circ s'=\mathrm{id}_{P'}$. So we have $y_1-y_2=(\lambda_1-\lambda_2)p(\sv'_{+})=0$. This implies that $y_1=y_2 \in |T'_1 |_{P'} \cap |T'_2 |_{P'}=|T'_1 \cap T'_2|_{P'}$, and hence $x \in \conv (|T'_1 \cap T'_2|_{P'}, \sv'_+)=|T_1 \cap T_2|_{Q'}$. Therefore, the equality is proved. 
  \end{enumerate}

  Now fix $\epsilon \ll 1$. We can choose $r \gg 0$ such that for any $T_1, T_2 \in K$, one has $d(\frac{1}{r} |T_i|_{Q}-\sv, \frac{1}{r} |T_i|_{Q})< \epsilon$ for each $i \in \{1,2 \}$. Therefore, by Lemma \ref{lem:intersection_of_small_pertubation_of_simplices}, $|T_1 \cap T_2|_{Q}=|T_1|_Q \cap |T_2|_Q$. 
  Thus, we proved that  $\Delta_Q$ is a simplical polyhedral complex when $r \gg 0$. Hence, $h \colon \widetilde{P} \to Q$ is a homeomorphism whose inverse is given by the facewise linear extension of the map $\sv(\fs) \mapsto \su(\fs)$ to the polyhedral complex $\Delta_Q$. The claim follows. 
  \end{proof}

\noindent $\bullet$ \textbf{Step 5. Construction of a continuous map $g \colon P \to Q$ such that $g(P) \subset \fC \cap Q$, and $(h^{-1} \circ g)|_{\partial P}=i|_{\partial P}$.} \label{Step_5}

\smallskip

  Recall that $\Gamma$ is a connected tropical curve. Therefore, when $r \gg 0$, we can take $\sv_{\xi,0}, \cdots, \sv_{\xi,N}$ in the vertex set $\Gamma_0$, for each $\xi \in \Xi_1$, such that $\sv_{\xi,0}=\sv$, $\sv_{\xi,N}=\sv_{\xi}$, and $\sv_{\xi,i} \sim  \sv_{\xi,i+1}$ is an edge in $\Gamma$ for any $0 \le i \le N-1$. 

  We define a continuous map $g \colon P \to Q$ which is facewise linear on $\Delta'_{P}$ by setting $g(\su_{\xi,i}) \coloneqq \sv_{\xi,i}$, and then extend it linearly on each face of $\Delta'_{P}$.

  Let $\sigma \in \Delta'_{P}$. By Claim \ref{claim:top_dimensional_simplices_in_subdivision_of_Delta_P}, $\sigma =\conv(\su_{\xi_1, i_1}, \su_{\xi_1, i_1+1}, \su_{\xi_2, i_2}, \su_{\xi_3, i_3}, \cdots, \su_{\xi_{n-2}, i_{n-2}})$,
  where $(\xi_1, \xi_2, \cdots, \xi_{n-2})$ generate an $(n-2)$-dimensional cone in $\Xi$.
  
  For $x \in \sigma$, we have 
  \[
  g(x) \in \conv(\sv_{\xi_1, i_1}, \sv_{\xi_1, i_1+1}, \sv_{\xi_2, i_2}, \sv_{\xi_3, i_3}, \cdots, \sv_{\xi_{n-2}, i_{n-2}}). 
  \] 

  Since $\sv_{\xi_1, i_1} \sim \sv_{\xi_1, i_1+1}$ is an edge in $\Gamma$, there exists a point $y$ inside this edge such that 
  \[
  g(x) \in \conv(y, \sv_{\xi_2, i_2}, \sv_{\xi_3, i_3}, \cdots, \sv_{\xi_{n-2}, i_{n-2}}) \subset \fC.
  \] 
  This shows that $g(P) \subset \fC$. 

  It remains to characterize $(h^{-1} \circ g) |_{\partial P}$. Notice that for any $x \in \partial P$, $x \in \tau = \conv(u_{\xi_1}, \cdots, u_{\xi_k})$ where $\mathrm{Cone}(\xi_1, \cdots, \xi_k) \in \Xi$. Then, by definition, $h^{-1} \circ g(\su_{\xi_i})=\su_{\xi_i}$, and both $g|_{\tau}$ and $h^{-1}|_{g(\tau)}$ are linear maps. Therefore, $(h^{-1} \circ g)|_{\tau}$ is an identity map and coincides with $i|_{\tau}$. 

  \medskip

  \noindent $\bullet$ \textbf{Step 6. Concluding  Proposition \ref{prop:existence_of_balanced_rays_in_2d_fan} by applying Lemma \ref{lem:application_Borsuk_Ulam_theorem}.} \label{Step_6} 

  \smallskip

  We choose an infinite edge $\mathrm{Leg}(-\rho_1) \in  (-\Gamma')_{1, \infty}$ such that $\mathrm{Leg}(-\rho_1)=x_1-\rho_1$ for some $x_1 \in -\Gamma'$. Recall that $(\partial \widetilde{P})_{\pm} \coloneqq \partial \widetilde{P} \cap (B \times \R_{\pm})$, and $(\partial Q)_{\pm} \coloneqq h^{-1}((\partial \widetilde{P})_{\pm})$. 

  \begin{claim} \label{claim:choose_of_point_in_north_hemisphere}
    When $r \gg 0$, $\mathrm{Leg}(-\rho_1) \cap (\partial Q)_{+}$ is nonempty. 
  \end{claim}
  
  \begin{proof} [Proof of Claim \ref{claim:choose_of_point_in_north_hemisphere}]
    Recall that in the proof of Claim \ref{claim:Delta_Q_is_polyhedral_complex_and_homeomorphism_of_h}, we constructed $Q' \subset A$, $P' \subset B$, and a continuous map $s' \colon P' \to Q'$ satisfying $p \circ s'=\mathrm{id}_{P'}$. Let $h' \colon \widetilde{P} \to Q'$ be the facewise linear extension of the map $\su(t) \mapsto \sv'(t)$ to $\Delta_{\widetilde{P}}$. Denote $(\partial Q')_{+} \coloneqq h'((\partial \widetilde{P})_{+})$. We first show that $-\rho_1 \cap (\partial Q')_{+}$ is nonempty.  One can define a continuous PL function $u'_{+} \colon P' \to \R$ such that $s'(y)+u'_{+}(y) \cdot e_{\rho_0} \in (\partial Q')_{+}$ for any $x \in \interior(P')$. Remark that $u'_{+}(0)=1$, and $u'_{+}(y)=0$ for any $y \in \partial P'$. 
    
    Let $y'_0 \coloneqq -\xi_1 \cap \partial P' \subset P'$ where $\xi_1 =p(\rho_1)$, and  let $I'$ be the interval connecting $0$ and $y'_0$. Define linear function $l'_{+} \colon I' \to \R$ such that $s'(y)+l'_{+}(y) \in -\rho_1$.   Recall that in Step 2, we proved $-e_{\rho_1} \in U_{+}$. This implies that $l'_{+}(y'_0)>0$. On the other hand, $l'_{+}(0)=0$. Therefore, there exists $y'_1 \in I'$ such that $l'_{+}(y'_1)=u'_{+}(y'_1)$, and hence $s'(y'_1)+u'_{+}(y'_1) \in -\rho_1 \cap (\partial Q')_{+}$. 
    
    Now we prove the claim. Fix $\epsilon \ll 1$, and we assume $r  \gg 0$. Recall that in the proof of Claim \ref{claim:Delta_Q_is_polyhedral_complex_and_homeomorphism_of_h}, we showed that $d(\frac{1}{r}Q-\sv,  Q') < \epsilon$, where $d$ denotes the Gromov-Hausdorff distance with respect to the Euclidean metric. Meanwhile, we have $d(\frac{1}{r}\mathrm{Leg}(-\rho_1)-\sv, -\rho_1) < \epsilon$. This implies that $d(p(\frac{1}{r}Q-\sv), P') < \epsilon$, $d(I, I') < \epsilon$, $\left\|u_{+} - u'_{+} \right\|_{\nC^0} < \epsilon$ and $\left\|l_{+} - l'_{+} \right\|_{\nC^0} < \epsilon$, where $I, u_{+}, l_{+}$ is defined with respect to $Q$ in a  similar way that $I', u'_{+}, l'_{+}$ is defined with respect to $Q'$. Thus, we can conclude that when $r \gg 0$, there exists $y_1 \in I$ such that $u_+(y_1)=l_{+}(y_1)$, and $\mathrm{Leg}(-\rho_1) \cap (\partial Q)_{+}$ is nonempty.
  \end{proof}

  Now we finish the proof of Proposition \ref{prop:existence_of_balanced_rays_in_2d_fan}. By Claim \ref{claim:choose_of_point_in_north_hemisphere}, there exists  $\sw_{+} \in \mathrm{Leg}(-\rho_1) \cap (\partial Q)_{+}$. Since $\Gamma'$ is a connected tropical curve, when $r \gg 0$, there exists a continuous path $\gamma' \colon [0,1] \to Q \cap (-\Gamma')$ such that $\gamma'(0)=\sw_{+}$ and $\gamma'(1)=\sv_{-}$. 

  We define continuous maps $f \colon P \to \widetilde{ P}$ by $f \coloneqq h^{-1} \circ g$, and $\gamma \colon [0,1] \to \widetilde{P}$ by $\gamma \coloneqq h^{-1} \circ \gamma'$. Recall that we proved $f|_{\partial P}=i|_{\partial P}$ in Step 5. Furthermore, $\gamma(0)=h^{-1}(\sw_{+}) \in (\partial \widetilde{P})_{+}$; and by definition, $\gamma(1)=\su_{-} \in (\partial \widetilde{P})_{-}$. Therefore, by applying Lemma \ref{lem:application_Borsuk_Ulam_theorem} to $f$ and $\gamma$, there exists $x' \in \widetilde{P}$ such that $x' \in f(P) \cap \gamma([0,1])$. Hence, $h(x') \in g(P) \cap \gamma'([0,1]) \subset \fC \cap (-\Gamma')$. Thus, we conclude that $\fC \cap (-\Gamma')$ is nonempty, and hence the proposition is proven.

%%%%%%%%%%%%%%%%%%%%%%%%%%%%%
%%%%%Relation to NA solutions
%%%%%%%%%%%%%%%%%%%%%%%%%%%%%

\section{Relation to non-Archimedean solutions} \label{sec:relation_to_SYZ}
In this section,  $\valfield = \BC((\mathsf{t}))$ denotes the valued field of Laurent power series equipped with the $\mathsf{t}$-adic valuation, and $\valring$ denotes its valuation ring. We have $\resfield \coloneqq \valfield / \valring \cong \BC$. We fix a smooth projective variety $\varX$ over $\valfield$, and we use $\varX^{\an}$  to denote its Berkovich analytification over $\valfield$.

\subsection{SNC models and dual complexes}
A \emph{model} $\fX$ of $\varX$ is a flat and proper scheme $\fX$ over $\Spec(\valring)$ together with an identification of its generic fibre with $\varX$.  We say that $\fX'$ \emph{dominates} $\fX$ if there exists a proper birational morphism $\pi \colon \fX' \to \fX$ over $\Spec \valring$ satisfying $\pi|_{\fX}=\mathrm{id}_{\fX}$. 

Let $\fX$ be a model of $\varX$. We view the special fibre $\fX_s=\div(\mathsf{t})$ as a Cartier divisor of $\fX_s$. Let $\{E_i\}_{i \in I}$ be the finite set of irreducible component of $\fX_s$. Denote $E_{J} \coloneqq \bigcap_{j \in J} E_j$ for any subset $J \subset I$. Following the convention in \cite{BFJ16}, we say that $\fX$ is an SNC model if the special fiber $\fX_s$ is a simple normal crossing divisor in $\fX$, and $E_J$ is either irreducible or empty for each subset $J \subset I$. It follows from \cite{Tem08} that for any model $\fX$ of $\varX$, there exists an SNC model $\fX'$ dominating $\fX$.

Let $\mathfrak{X}$ be an SNC model. Suppose $\mathfrak{X}_s = \div(\mathsf{t})= \sum_{i\in I} b_iE_i$ with $b_i \in \Z_{>0}$.   The \emph{dual complex} $\Delta_{\mathfrak{X}}$ of $\mathfrak{X}$ is the simplicial complex whose set of vertices equals to $I$, and whose simplices are subsets $J \subset I$ such that $E_J \neq \emptyset$. 
We associate to each such $J$, the simplex $\sigma_J$ consisting of the tuples $w = (w_j)_{j\in J}$ with $w_j \geq 0$ and $\sum_{j\in J} b_jw_j =1$. 
Geometrically, $\Delta_\X$ is realized as the union of $\sigma_J$ with $E_J\neq\emptyset$ in $\R^I$. By an abuse of notation, we use the same symbol $\Delta_\X$ for this realization.

Let $J \subset I$ be a simplex in $\Delta_{\fX}$ and take $w = (w_j)_{j\in J}\in\sigma_J$. Let $\xi_J$ be the generic point of $E_J$. Take a system of local coordinates $ \{z_j\}_{j\in J} \subset \mathcal{O}_{\mathfrak{X},\xi_J}$ with $z_j$ defining $E_j$ around $\xi_J$. By Cohen's structure theorem, $\widehat{\mathcal{O}}_{\mathfrak{X},\xi_J} \simeq \kappa(\xi_J)[[z_j, j \in J]]$, where  $\kappa(\xi_J)$ denote the residue field of $\CO_{\fX, \xi_J}$. Under this identification, denote by $v_{J,w}$ the restriction to $\mathcal{O}_{\mathfrak{X},\xi_J} \subset \widehat{\mathcal{O}}_{\mathfrak{X},\xi_J}$ of the monomial valuation with weight $w_j$ on the power series ring,  i.e.\ 
 \[
 v_{J,w} \left( \sum_{\alpha \in \mathbb{N}^J} c_\alpha z^\alpha \right) = \min \left\{ \sum_{j \in J} w_j \alpha_j \mid c_\alpha \neq 0 \right\}, \quad \text{ where } \alpha=(\alpha_j)_{j\in J}.
 \]
On the other hand, we have $\mathsf{t}=u \cdot \prod_{j \in J} z_j^{b_j}$ in $\CO_{\fX, \xi_J}$, where $u$ is a unit in $\CO_{\fX, \xi_J}$. Thus, we get $v_{J,w}(\mathsf{t})=1$, and therefore $v_{J,w}\rest{\valfield}=v_{\valfield}$. This implies that $v_{J,w} \in \varX^{\an}$, and hence yields an embedding of $\mathrm{emb}_{\mathfrak{X}}\colon \Delta_{\mathfrak{X}} \to \varX^{\an}$, for the geometric realization of $\Delta_{\mathfrak{X}}$. Moreover, the parameters $w$ equip $\Delta_{\mathfrak{X}}$ with an integral affine structure.

Now, we recall the retraction map $p_{\mathfrak{X}}\colon \varX^{\an} \to \Delta_{\mathfrak{X}}$.  Let $x \in \varX^{\an}$, and let $v_x$ denote the semivaluation defined by $x$.  Then, the \emph{center} of $v_x$,   denoted by $c_{\fX}(x)$, is defined to be the unique (scheme) point $\xi \in \fX_s$ satisfying $v_x(f) \ge 0$ for any $f \in \CO_{\fX, \xi}$, and $v_x(f)>0$ for any $f \in \fm_{\fX, \xi}$. The valuative criterion of properness guarantees the existence and the uniqueness of $c_{\fX}(x)$.  
Let $J \subset I$ be the maximal subset such that $c_{\fX}(x) \in E_J$. Then $p_{\mathfrak{X}}(x) \in \sigma_J\subset \Delta_{\mathfrak{X}}$ corresponds to the monomial valuation with weight $v_x(z_j)$, $j \in J$.  

By definition, $p_{\mathfrak{X}}\rest{\Delta_\fX} = \mathrm{id}$. If $\mathfrak{X}'$ is an SNC model dominating $\mathfrak{X}$, then we have $\Delta_{\fX} \subset \Delta_{\fX'}$, and $p_{\mathfrak{X}} \circ p_{\fX'} = p_{\mathfrak{X}}$. We recall the following result.

\begin{prop}[{\cite[Corollary 2.5]{BFJ16}}]
  $\varX^{\an}$ is homeomorphic to  $\varprojlim\limits_{\fX} \Delta_{\mathfrak{X}}$. Here,  $\fX$ ranges over all SNC models of $\varX$. 
\end{prop}

\subsection{Non-Archimedean Calabi--Yau theorem}
In \cites{BFJ15, BJ22}, a non-Archimedean pluripotential theory has been developed. We begin by introducing very briefly the conventions from this theory that will be used throughout this section. For further details on these conventions, we refer to \cite[\S 6]{APW1}, and for a more comprehensive introduction to the theory, we refer to \cite{BFJsurvey}.

Let $\varL$ be a $\Q$-line bundle on $\varL$. A \emph{model} of $(\varX, \varL)$ is a pair $(\fX, \fL)$ consisting of a model $\fX$ of $\varX$ and a  $\Q$-line bundle $\fL$ on $\fX$ whose restriction on $\varX$ (as the generic fibre of $\fX$) is isomorphic to $\varL$. We say that $(\fX', \fL')$ \emph{dominates} $(\fX, \fL)$ if $\fX'$ dominates $\fX$ by   a birational proper morphism $\pi \colon \fX' \to \fX$, and $\fL'=\pi^{*}\fL$. We say that $(\fX', \fL')$ and $(\fX'', \fL'')$ are \emph{equivalent} if there exists a model $(\fX, \fL)$ dominating both  $(\fX', \fL')$ and $(\fX'', \fL'')$. A model $(\fX, \fL)$ is called \emph{nef} if $\fL$ is relatively nef on $(\fX / \Spec \valring)$. 

A \emph{closed  $(1,1)$-form} is an equivalence class of models of $(\varX, \varL)$ for some $\varL \in \mathrm{Pic}_{\Q}(\varX)$. A closed  $(1,1)$-form $\omega_{\NA}$ is called \emph{semipositive} if some (equivalently, any) model representing $\omega_{\NA}$ is nef.  The first Chern class $c_1(\omega_{\NA}) \in N^1(\varX)$ is defined to be  $c_1(\fL|_{\varX})\in N^1(\varX)$ for some model $(\fX, \fL)$ representing $\omega_{\NA}$.

Let $\omega_{\NA}$ be a closed semipositive $(1,1)$-form on $\varX^{\an}$. For any PL function $\varphi$ on $\varX^{\an}$ (denoted by $\varphi \in \PL(\varX^{\an})$), one can associate a closed $(1,1)$-form $\omega_{\NA, \varphi}$ on $\varX^{\an}$, which serves as the non-archimedean analogue of $\omega_{\NA}+dd^c \varphi$ (see \cite[\S 6.2]{APW1}). A PL function $\varphi$ on $\varX^{\an}$ is called $\omega_{\NA}$-psh (written $\varphi \in \PSH_{\omega, \NA} \cap \PL(\varX^{\an})$) if $\omega_{\NA, \varphi}$ is semipositive.

Following \cite{BJ22, APW1}, when $c_1(\omega_{\NA}) \in N^1(\varX)$ is an ample class, a function $\varphi \colon X^{\an} \to \R \cup \{-\infty\}$ is called  $\omega$-psh (written $\varphi \in \PSH_{\omega, \NA}(\varX^{\an})$) if $\varphi$ is generically finite (\emph{i.e.} $\varphi$ is not constantly equal to $-\infty$ on each connected component of $\varX^{\an}$) and $\varphi$ can be written as the pointwise limit of a decreasing sequence of $\omega_{\NA}$-psh PL functions. In general, when $c_1(\omega_{\NA}) \in N^1(\varX)$ is a nef class, we say that $\varphi \in \PSH_{\omega, \NA}(\varX^{\an})$ if $\varphi \in \PSH_{\omega+\epsilon \theta, \NA}(\varX^{\an})$ for each $\epsilon \in \Q_{>0}$, where $\theta_{\NA}$ is a  closed semipositive $(1,1)$-form on $\varX^{\an}$ such that $c_1(\theta_{\NA})$ is ample.  We use $\CPSH_{\omega,\NA}(\varX^{\an})$ to denote the space of continuous $\omega_{\NA}$-psh functions on $\varX^{\an}$. 

Following \cite{CL06}, for $\varphi \in \PSH_{\omega, \NA} \cap \PL(\varX^{\an})$, the non-Archimedean \MongeA measure $\MA_{\omega, \NA}(\varphi)$ is is defined in terms of intersection numbers on a model representing $\omega_{\NA}$. Using the Chern--Levine--Nirenberg inequality, one can extend the definition of the non-Archimedean \MongeA measure $\MA_{\omega,\NA}(\varphi)$ to $\varphi \in \CPSH_{\omega,\NA}(\varX^{\an})$. This measure is a nonnegative Radon measure on $\varX^{\an}$ with total mass $\deg(\varL)^d$.

\smallskip

We now recall the non-Archimedean Calabi--Yau  theorem. 

\begin{thm} [NA Calabi--Yau Theorem, \cites{BFJ15, BJ22}] \label{thm:NA_Calabi_Yau}
  Let $\varX$ be a smooth polarized variety over $\valfield$, and let $\omega_{\NA}$ be a closed semipositive $(1,1)$-form on $\varX^{\an}$ such that $c_1(\omega_{\NA})$ is an ample class on $\varX$. Let $\fX$ be an SNC model of $\varX$. Let $\mu$ be a nonnegative Radon measure supported on $\Delta_{\fX}$ of total volume $\deg(\varL)^d$. Then, there exists $\varphi \in \CPSH_{\omega, \NA}(\varX^{\an})$, unique up to a constant, such that 
  \[
  \MA_{\omega, \NA}(\varphi)=\mu.
  \]
\end{thm}

We also recall the following result.
\begin{prop} [{\cite[Lemma 5.13 + Theorem 5.20]{BJ22}}] \label{prop:NA_envelope_properties_implies_birational_descend}
  Let $\varX$ be a smooth polarized variety over $\valfield$, and let $\omega_{\NA}$ be a closed semipositive $(1,1)$-form on $\varX^{\an}$ such that $c_1(\omega_{\NA})$ is an ample class on $\varX$. Suppose that  $\pi \colon \widetilde{\varX} \to \varX$ be a proper birational morphism. Then, $\pi^{*} \colon \PSH_{\omega, \NA}(\varX^\an) \to  \PSH_{\pi^{*}\omega, \NA}(\widetilde{\varX}^\an)$ is bijective. 
\end{prop}

\subsection{Comparison Theorem}
Let $\iota \colon \varX \hookrightarrow \P_{\Xi}$ be an embedding of $\varX$ into a projective toric variety $\P_{\Xi}$. Denote $\varY \coloneqq \varX \cap \varT$, where $\varT$ is the open torus orbit in $\P_{\Xi}$. Let $X=\Trop(\varX, \iota) \subset \TP_\Xi$ and $Y=\Trop(\varY, \iota)\subset \mathrm T$ be the tropicalization of $\varX$ and $\varY$, respectively, where $\TP_\Xi$ is the tropical toric variety associated to $\Xi$, and $\mathrm T$ is its open tropical torus orbit.  

Let $\varL$ be an ample toric line bundle over $\P_{\Xi}$, and let $h$ be the convex conewise linear function on $\Xi$ associated to $\varL$. The trivial model $((\P_{\Xi})_{\valring}, \varL_{\valring})$ of $(\P_{\Xi}, \varL)$ then induces a non-Archimedean closed semipositive $(1,1)$-form on $\P_{\Xi}^{\an}$, whose restriction to $\varX^{\an}$ is denoted by $\omega_{\NA}$. 

\begin{thm} \label{thm:comparison_thm_general_case}
  Suppose that there exists a bounded function $\varphi_o\colon Y \to \R$ such that $\varphi_o+h \in \LConv(Y)$. Then, $\varphi_o \circ \Trop$ can be extended to $\varphi_{\NA} \in \CPSH_{\omega, \NA}(\varX^{\an})$. Furthermore, $\MA_{\omega, \NA}(\varphi_{\NA})$ has zero mass on $\varX^{\an} \setminus \varY^{\an}$, and  $\Trop_{*}(\MA_{\omega, \NA}(\varphi_{\NA}))=\MA_{\Trop}(\varphi_o+h)$ on $Y$.
\end{thm}
\begin{proof}
  By Corollary \ref{cor:refinement_of_open_tropical_varieties_are_projective}, there exist a  refinement $\widetilde{\Xi}$ of $\Xi$ and a polyhedral structure $\Pi$ on $Y$ such that $\TP_{\widetilde{\Xi}}$ is a projective tropical toric variety, $\rec(\Pi)$ is a subfan of $\widetilde{\Xi}$, and the closure of $Y$ in $\TP_{\widetilde{\Xi}}$ denoted $\widetilde{X}$  is a tropical subvariety of $\TP_{\widetilde{\Xi}}$.  Then, $\widetilde{X}$ is a tropical variety in the sense of this  paper. Let $\widetilde{\varX}$ denote the closure of $\varY$ in $\P_{\widetilde{\Xi}}$, and $\widetilde{\iota} \coloneqq  \widetilde{ \varX} \hookrightarrow  \P_{\widetilde{\Xi}}$ the corresponding embedding. Let $\pi \colon \P_{\widetilde{\Xi}} \to \P_{\Xi}$ be the toric birational morphism. Then, we have $\iota \circ \pi = \pi \circ \widetilde{\iota}$. Note that $\widetilde{X}=\Trop(\widetilde{ \varX}, \widetilde{\iota})$, with the open part $\widetilde{X}_o=Y$. 

  Remark that $h$ is conewise linear on $\widetilde{\Xi}$. By Example \ref{example_kahler_form_embedded_tropical_varieties}, $\widetilde{\omega} \coloneqq \dm \dmb h$ is a semipositive simplicial form on $\widetilde{X}$. Let $g$ be a strictly convex conewise linear function on $\widetilde{\Xi}$. By Proposition \ref{prop:kahler_form_on_tropical_toric_varieties}, $[\widetilde{\theta}]$ is a tropical \kahler class on $\widetilde{X}$ for $\widetilde{\theta} \coloneqq \dm \dmb g$. Since $\varphi_o$ is bounded on $Y$ and $\varphi_o+h \in \LConv(Y)$, by Theorem \ref{thm:equivalent_def_psh}, $\varphi_o$ extends uniquely to $\widetilde{\varphi} \in \CPSH_{\widetilde{\omega}}(\widetilde{X})$. Let $\{\alpha_i\}_{i \ge 1}$ be a decreasing sequence of positive numbers $\{\alpha_i\}_{i \ge 1}$ such that $\lim_{i \to +\infty} \alpha_i = 0$. 
  Denote $\widetilde{\theta_i} \coloneqq \widetilde{\omega}+\alpha_i \widetilde{\theta}$ for each $i \ge 1$.  Now, applying Theorem \ref{thm:LConv_regularization_PL}, we obtain a sequence of functions $\{\widetilde{\varphi}_i\}_{i \ge 1} \subset \CPSH_{\widetilde{\theta}_i} \cap \PL(\widetilde{X})$ such that $\widetilde{\varphi}_i$ converges uniformly to $\widetilde{\varphi}$ on $\widetilde{X}$. 

  Let $\widetilde{\varL}_h$ and $\widetilde{\varL}_g$ denote the nef toric line bundles on $\P_{\widetilde{\Xi}}$ associated to $h$ and $g$ respectively. Let $\widetilde{\omega}_{\NA}$, $\widetilde{\theta}_{\NA}$  denote their corresponding closed semipositive $(1,1)$-forms on $\widetilde{\varX}^{\an}$ respectively.  
  Then,  by \cite[Thm.~6.13]{APW1}, we have  for each $i \ge 1$,
  %cross-ref in APW1: \ref{thm:comparison} 
$\widetilde{\varphi}_i \circ \Trop \in \PSH_{\widetilde{\theta}_i, \NA} \cap \PL(\widetilde{ \varX}^{\an})$, and $\Trop_{*}(\MA_{\widetilde{\theta}_i, \NA}(\widetilde{\varphi}_i \circ \Trop))=\MA_{\widetilde{\theta}_i, \Trop}(\widetilde{\varphi}_i)$. 
  Then, we define $\widetilde{\varphi}_{\NA} \coloneqq \widetilde{\varphi} \coloneqq \lim_{i \to +\infty} \widetilde{\varphi}_i \circ \Trop \in \CPSH_{\widetilde{\omega}, \NA}(\widetilde{ \varX}^{\an})$. 
  By construction, we have $\widetilde{\varphi}_{\NA}\rest{\varY^{\an}}=\varphi_o \circ \Trop$. Moreover, the non-Archimedean Bedford--Taylor theory~\cite[Thm.\,7.18]{BJ22} combined with the tropical analogue (Theorem \ref{thm:global_Bedford-Taylor_theory}) imply that $\Trop_{*}(\MA_{\widetilde{\omega}, \NA}(\widetilde{\varphi}_{\NA}))=\MA_{\widetilde{\omega}, \Trop}(\widetilde{\varphi})$ on $\widetilde{X}$.
  %cross-ref in APW1: \ref{thm:NA_Bedford_Taylor_theory_paper_1}

  Notice that $\widetilde{\varL}_h =\pi^{*}\varL$, and $\widetilde{\omega}_{\NA}=\pi^{*} \omega_{\NA}$. Then, applying Proposition \ref{prop:NA_envelope_properties_implies_birational_descend} to $\pi \colon \widetilde{\varX} \to \varX$, there exists $\varphi_{\NA} \in \CPSH_{\omega, \NA}(\varX^{\an})$ such that $\varphi_{\NA} \circ \pi= \widetilde{\varphi}_{\NA}$. Since $\pi\rest{\varY}=\mathrm{id}_{\varY}$, and the open  tropicalizations of both $\widetilde{\varX}$ and $\varX$ coincides with $Y=\Trop(\varY, \iota_o)$, we get $\varphi_o \circ \Trop=\varphi_{\NA}$. By \cite[Cor.\,7.49]{BJ22},  $\MA_{\omega, \NA}(\varphi_{\NA})$ and $\MA_{\omega, \NA}(\widetilde{\varphi}_{\NA})$ has zero mass on $\varX^{\an} \setminus \varY^{\an}$ and $\widetilde{\varX}^{\an} \setminus \varY^{\an}$ respectively. Since $\pi_{*} \MA_{\omega, \NA}(\widetilde{\varphi}_{\NA})= \MA_{\omega, \NA}(\varphi_{\NA})$, we have $\Trop_{*}(\MA_{\omega, \NA}(\varphi_{\NA}))=\MA_{\omega, \Trop}(\varphi_o)$ on $Y$.
\end{proof}

\subsection{Strongly faithful tropicalization}
In the remaining part of this section, we say that a polyhedron is of \emph{maximal dimension} if its dimension equals to $\dim_{\valfield}(\varX)$. 
\begin{definition} \label{def:faithful_tropicalization}
  Let $\Delta_{\fX}$ be the dual complex of an SNC model $\fX$ of $\varX$ with $\dim(\Delta_{\fX})=\dim_{\valfield}(\varX)$. Let $\{f_1, \cdots, f_k\} \subset \valfield(\varX)$ be a finite set of rational functions,  and let 
  \[
  f=(f_1, \cdots, f_k) \colon \varX \dashrightarrow \varG_m^k 
  \]
  be the corresponding rational map.  We say that the  tropicalization $\Trop_{f} \colon \varX^{\val}  \to \Trop(\varX, f)\subset \R^k$ is \emph{strongly faithful} with respect to $\Delta_{\fX}$ if $(\Trop_{f})\rest{\Delta_{\fX}}$ is injective, and $\weight (\delta)=1$ for any maximal dimension face $\delta \subset (\Trop_{f})({\Delta_{\fX}})$.
\end{definition}

\begin{remark}
  In the above definition, we do not require $\Trop_{f}$ be affine on each maximal dimensional face $\delta \subset \Delta_{\fX}$. In general, $\Trop_{f}$ is only  piecewise affine on  $\delta$. 
\end{remark}

\begin{remark}
  This definition of strongly faithful tropicalization is a refinement of the one considered in \cite{GRW16}. However, these two notions are closely related. This is justified in the proof of the following proposition. 
\end{remark}

\begin{thm} \label{thm:existence_of_faithful_tropicalization}
  Let $\fX$ be an SNC model  of $\varX$, and $\Delta_{\fX}$ be the dual complex of $\fX$. Suppose that $\dim(\Delta_{\fX})=\dim_{\valfield}(\varX)$. Then, there exists a finite subset $\{f_1, \cdots, f_k\} \subset \valfield(\varX)$ such that  $\Trop_f$ is a strongly faithful tropicalization with respect to $\Delta_{\fX}$ for $f=(f_1, \cdots, f_k)$. 
\end{thm}
\begin{proof}
  Suppose that $\dim \varX=d$. 
  Let  $\{s_0, s_1 \cdots, s_r\}$ be a finite subset of 
  $H^0(\fX, \CL^{\otimes k})$ with $k \in \N$ such that $\iota=[(s_1)_{\valfield} \colon \cdots \colon (s_r)_{\valfield}] \colon \varX \to \P^{r-1}_{\valfield}$ is a closed embedding. 
  We use $(\Delta_{\fX})_{d}$ to denote the set of $d$-dimensional faces of $\Delta_{\fX}$. For any  $\delta \in (\Delta_{\fX})_{d}$  corresponding to a $0$-dimensional stratum  $\{x\}=\bigcap_{i \in I_{\delta}} E_i$ with $I_{\delta}$ a finite index set of cardinality $d+1$, we take $\{g_{\delta,i}\}_{i \in I_{\Delta}} $  a finite subset of  $\valfield (\varX)=\valfield(\fX)$ such that $g_{\delta,i}$ defines $E_i$ in $\CO_{\fX, x}$ for each $i \in I_{\delta}$. Define the finite subset $g$ of $\valfield(\varX)$ by 
  \[
  g \coloneqq \{s_1/s_0, \cdots, s_r/s_0, g_{\delta,i} \st \delta \in (\Delta_{\fX})_{d} \,, {i \in I_{\delta}}\}.
  \]
  Let $\{\CH_i\}_{i \in I}$ be the set of irreducible components of $\bigcup_{g_k \in g} \div(g_k) \subset \fX$. Denote $\CH \coloneqq \sum_{i \in I} \CH_i$ and $\varH \coloneqq \CH \cap \varX$. Slightly abuse of notation, we also denote by $g \colon \varX \dashrightarrow \varT$ the rational map defined by $g$, where $\varT$ denotes a torus. 
  Then, $g$ is a closed embedding on $\varX \setminus \mathrm{Supp}(\varH)$.  Also, $\Trop_g\rest{\delta}$ is injective for any $\delta \in (\Delta_{\fX})_d$. 
  
  By \cite{Tem08}, there exists a proper birational morphism $\pi \colon \fX' \to \fX$ such that $\pi^{-1}(\mathrm{Supp}(\CH) \cup \fX_s)=\sum_{i \in I'} \CH'_i+\fX'_s$ is SNC, and $\pi$ is an isomorphism on $\pi^{-1}(\fX \setminus (\mathrm{Supp}(\CH) \cup \fX_s))$.  Let $\varX' \coloneqq \fX'_{\valfield}$, and $\varH'_i \coloneqq \CH'_i \cap \varX'$. Then, $g \colon \varX' \setminus \mathrm{Supp}(\varH')=\varX \setminus \mathrm{Supp}(\varH) \hookrightarrow \varT$ is a closed embedding. On the other hand, $\Delta_{\fX} \subset \Delta_{\fX'}$ inside $\varX^{\val}=\varX'^{\val}$.   

  In \cite{GRW16}, the authors introduced an integral affine polyhedral complex denoted $S(\fX', \CH')$ for the SNC pair $(\fX', \CH')$. Furthermore, there is an inclusion map $\mathrm{emb}_{\fX', \CH'} \colon S(\fX', \CH') \hookrightarrow \varX^{\val}$, and a retraction map $r_{\fX',\CH'} \colon \varX^{\val} \to (\fX', \CH')$. By the construction,  $\Delta_{\fX}$ is a subcomplex of $S(\fX', \CH')$, and $(\mathrm{emb}_{\fX', \CH'})\rest{\Delta_{\fX}}=\mathrm{emb}_{\fX}$.

  We say a tropicalization $\Trop_f$ is unimodular on $S(\fX', \CH')$ if there exists a rational polyhedral subdivision $S''(\fX', \CH')$ of $S(\fX', \CH')$ such that $\Trop_f$ is a unimodular integral affine map on each maximal dimensional face in $S''(\fX', \CH')$.  
   By Theorem 9.5 in \textit{loc.\,cit.}, there exists a finite subset $\{f_1, \cdots, f_m\} \subset \valfield (\varX)$ containing $g$ such that $(\Trop_f)$ is injective and  unimodular on $S(\fX', \CH')$.

  By Proposition 8.2 in \cite{GRW16}, $\Trop\rest{S(\fX', \CH')} \colon S(\fX', \CH') \to \Trop(\varX, g)$ is surjective. Let $\pi_{g} \colon \Trop(\varX,f) \to \Trop(\varX, g)$ be the projection map to the components corresponding to $g$.  Then we have the following commutative diagram
  \[
  \begin{tikzcd}
      & \Trop(\varX,f) \arrow[d, "\pi_{g}"] \\
      S(\fX', \CH') \arrow[ru, "\Trop_{f}"] \arrow[r, "\Trop_{g}"'] & \Trop(\varX, g)
  \end{tikzcd}
  \]

  Now we fix a rational polyhedral subdivision $S'(\fX', \CH')$ of $S(\fX', \CH')$ such that $(\Trop_f)\rest{S'(\fX',\CH')}$ is a facewise unimodular integral affine. We also fix integral polyhedral structures $\Pi_f$ and $\Pi_g$ on $\Trop(\fX, f)$ and $\Trop(\fX, g)$, respectively, such that for any $\delta \in \Pi_g$ and any $\delta' \in \Pi_f$, the preimage of $\delta$ in $\Trop(\varX, f)$  and the preimage of $\delta'$ in $S'(\fX',\CH')$  are both unions of faces. Then, for any maximal dimensional face $\delta \in \Trop(\varX, g)$, we use $\pi^{-1}_{g}(\delta)$ to denote the set of maximal dimensional faces of $\Trop(\varX, f)$ contained in the preimage of $\delta$, and $\Trop_{g}^{-1}(\delta)$ to denote the set of maximal dimensional faces of $S'(\fX', \CH')$ contained in the preimage of $\delta$.

  For any maximal dimensional face $\delta$ in $\Pi$, the skeletal Sturmfels--Tevelev multiplicity formula~\cite[Thm.\,8.4]{GRW16} implies that 
  \[
  \weight_{\Trop(\varX, g)}(\delta)=\sum_{\delta_{S'} \in \Trop_{g}^{-1}(\delta)} [N_{\delta} \colon N_{\delta_{S'}}].
  \]
  On the other hand, applying the tropical Sturmfels--Tevelev multiplicity formula \cite[Thm.\,1.1]{ST08} to $\pi_{g}$, we get 
  \bens
  \weight_{\Trop(\varX, g)}(\delta)&=&\sum_{\delta' \in \pi_{g}^{-1}(\delta)} \weight_{\Trop(\varX, f)}(\delta') [N_{\delta} \colon N_{\delta'}] \\
  &=&\sum_{\delta_{S'} \in \Trop_{g}^{-1}(\delta)} \weight_{\Trop(\varX, f)}(\Trop_f(\delta_{S'})) [N_{\delta} \colon N_{\delta_{S'}}] \\
  && + \sum_{\delta'' \in \pi_{g}^{-1}(\delta) \setminus \Trop_f(\Trop_{g}^{-1}(\delta))} \weight_{\Trop(\varX, f)}(\delta'') [N_{\delta} \colon N_{\delta''}].
  \eens
  By comparing the above two formulas, we get 
  \[
  \Trop_f(\Trop_{g}^{-1}(\delta))=\pi_{g}^{-1}(\delta), \quad \text{and} \quad  \weight_{\Trop(\varX, f)}(\Trop_f(\delta_{S'}))=1 \text{ for any }  \delta_{S'} \subset \Trop^{-1}_{g}(\delta).
  \]
  For any $d$-dimensional face $\delta_{S'} \in S'(\fX', \CH')$ such that $\delta_{S'}$ is contained in $\Delta_{\fX}$, we have  
  \[
  \weight_{\Trop(\varX, f)}(\Trop_f(\delta_{S'}))=1
  \]
  since $\Trop_g$ is injective on $\delta_{S'}$. In summary, we have proven that $(\Trop_f)\rest{\Delta_{\fX}}$ is injective, and $\weight_{\Trop_f}(\delta')=1$ for any maximal dimensional face $\delta' \subset \Trop_f(\Delta_\fX)$. Therefore,   $\Trop_f$ is a strongly faithful tropicalization with respect to $\Delta_{\fX}$. 
\end{proof}

\subsection{Tropical \MongeA solutions and the comparison property}
We recall the following definition given in \cite{Li23WeakComparison}.
\begin{definition} \label{def:partial_retraction_invariance}
  Let $\varphi \in \nC^0(\varX^{\an})$. We say that $\varphi$ satisfies the \textit{comparison property} if there exists   an SNC model $\fX$ of $\varX$ such that $\dim(\Delta_{\fX})=\dim_{\valfield}(\varX)$, and $\varphi=\varphi \circ r_{\fX}$ on $r^{-1}_{\fX}(\relint(\delta))$ for each maximal dimensional face $\delta$ in $\Delta_{\fX}$.
\end{definition}

The comparison property was motivated by the study of the metric SYZ conjecture. A weaker version of this property was established in \cite{Li26Valuative} for the generic fibres of maximally degenerating families of polarized Calabi--Yau manifolds, leading to a proof of the metric SYZ conjecture.

\smallskip

The following result shows that every continuous function on $\varX^{\an}$ which factors through a tropicalization satisfies the comparison property. 

\begin{prop} \label{prop:tropicalization_implies_partial_retraction_invariance}
  Let $\varphi \in C^0(\varX^{\an})$. Suppose that there exists a rational map 
  \[
  f=(f_1, \cdots, f_k) \colon \varX \dashrightarrow \varG_m^k
  \]
  such that $\varphi =\varphi \circ \Trop_f$ on $\varX^{\val}$. Then, $\varphi$ satisfies the retraction invariance property. 
\end{prop}
\begin{proof}
  Let $\dim(\varX)=d$. Fix an SNC model $\fX$ of $\varX$ with $\dim(\Delta_{\fX})=d$. 
  Let $\delta$ be a $d$-dimensional face of $\Delta_{\fX}$ corresponding to a $0$-dimensional strata $\{x\}=\bigcap_{0 \le i \le d}E_i$. Let $z_i \in \widehat{\CO}_{\fX, x}$ such that  $E_i$ is defined locally by $z_i=0$. Suppose that  $\uniform=z_0^{b_0} z_1^{b_1} \cdots z_d^{b_d}$ for $b_0, \cdots, b_d \in \BZ_{>0}$.  Cohen's structure theorem yields an isomorphism 
  \[
  \widehat{\CO}_{\fX, x} \cong \kappa [[z_0, z_1, \cdots, z_{d}]] \cong \kappa [\mathsf{t}] [[z_0, z_1, \cdots, z_{d}]] / (\uniform-z_0^{b_0} z_1^{b_1} \cdots z_d^{b_d})
  \]
  By construction, $r_{\fX}(v)$ equals to the tropicalization map $(v(z_0), \cdots, v(z_{d}))$ for any $v$ in the set $r_{\fX}^{-1}(\relint(\delta))$, with $\sum_{i} b_i v(z_i)=v(\mathsf{t})=1$.   Thus we can identify the restriction of $r_{\fX}$ over $\relint(\delta)$ with the restriction of the tropicalization map $(\BG_m^d)^{\an} \to \R^d$ over the standard $d$-dimensional open simplex $\relint(\delta_d^{\std}) \coloneqq \{u \in \R^{d+1}_{>0} \,|\, b_0u_0 + \cdots +b_du_d=1\}$.

  Consider the tropicalizations $\Trop_{(f,z)} \colon \varX^{\val} \to \Trop(\varX, (f,z))$, $\Trop_{z} \colon \varX^{\val} \to \Trop(\varX, (z))$ and the projection map $\pi_{z} \colon \Trop(\varX, (f,z)) \to \Trop(\varX, (z))$. Then, by assumption, we have $\varphi =\varphi \circ \Trop_{(f,z)}$ on $r_{\fX}^{-1}(\delta) \cap \varX^{\val}$. Take a simplicial rational polyhedral structure $\Pi_{\delta}$ on  $\delta$ such that for any $\delta' \in (\Pi_{\delta})_d$, $\pi_z^{-1}(\delta')$ is a union of polyhedra, and the restriction of $\pi_z$ on the relative interior each of these polyhedra is injective. 

  Now, we view $f=(f_1, \cdots, f_k)$ as elements in $\widehat{\CO}_{\fX, x}$, and we identify $\relint(\delta)$ with $\relint(\delta_n^{\std})$. Note that for a sufficiently irrational point\,\footnote{Sufficiently irrational here means that 
  the coordinates $y_1, \cdots, y_d$ of $y$ are linearly independent over $\Q$.} $y \in \relint(\delta_d^{\std})$, the cardinality of $\pi^{-1}(y) \subset \Trop((\BG_m^d)^{\an}, (f,z))$ equals to $1$.  This implies that for any $d$-dimensional polyhedron $\delta'$ in $(\Pi_{\delta})_d$, the restriction of $\pi_z$ on $\pi_z^{-1}(\relint(\delta'))$ is a bijection. Therefore, $\varphi=\varphi \circ \Trop_z$ on $r_{\fX}^{-1}(\relint(\delta')) \cap \varX^{\val}$.

  There exists a successive weighted blowup $\fX'$ of the strata in $\fX_s$ such that the support of the dual complex $\Delta_{\fX'}=\Delta_{\fX}$, and the restriction of $\Delta_{\fX'}$ on $\delta$ equals to $\Pi_{\delta}$ for each $d$-dimensional face $\delta$ in $\Delta_{\fX}$. Moreover, we have $r_{\fX}=r_{\fX'}$. Therefore, $\varphi=\varphi \circ r_{\fX'}$ on $r_{\fX'}^{-1}(\relint(\delta'))$ for each $d$-dimensional face $\delta'$ in $\Delta_{\fX'}$. Since $\varphi, r_{\fX'}$ are continuous and $\varX^{\val}$ is dense in $\varX^{\an}$, this implies that $\varphi$ satisfies the partial retraction invariace property with respect to $\fX'$. 
\end{proof}

We now relate the tropical \MongeA solution on a strongly faithful tropicalization to a non-Archimedean \MongeA solution on $\varX^{\an}$ which satisfies the comparison property.

\begin{thm} \label{thm:faithful_tropicalization_partial_retraction_invariance}
  Under the assumptions of Theorem \ref{thm:comparison_thm_general_case}, suppose further that $Y$ be a strongly faithful open tropicalization with respect to $\Delta_{\fX}$, where $\Delta_{\fX}$ is the dual complex of an SNC model $\fX$ of $\varX$ with $\dim(\Delta_{\fX})=\dim_{\valfield} (\varX)$. Let $\mu$ be any propability measure which is  supported on $\Delta_{\fX}$ and absolute continuous with respect to the Lebesgue measure. Suppose further that the bounded function $\varphi_o\colon Y \to \R$ satisfies $\varphi_o+h \in \LConv(Y)$, and $\MA_{\Trop}(\varphi_o+h)=\deg(\varL^d) \, \Trop_{*}(\mu)$ on $Y$. Then $\varphi_o \circ \Trop$ can be extended to $\varphi_{\NA} \in \CPSH_{\omega, \NA}(\varX^{\an})$ satisfying $\MA_{\omega, \NA}(\varphi_{\NA})=\deg(\varL^d) \,  \mu$. Moreover, $\varphi_{\NA}$ satisfies the comparison property. 
\end{thm}
\begin{proof}
  Write $Y=\Trop(\varX, g)$ for $g=(g_1, \cdots, g_k)$. By Theorem \ref{thm:comparison_thm_general_case}, $\varphi_o \circ \Trop$ can be extended to $\varphi_{\NA} \in \CPSH_{\omega, \NA}$ satisfying $(\Trop_f)_{*}(\MA_{\omega, \NA}(\varphi_{\NA}))=\MA_{\Trop}(\varphi_o+h)$. That $\varphi_{\NA}$ satisfies the comparison property follows from Proposition \ref{prop:tropicalization_implies_partial_retraction_invariance}.  By \cite{Pay09}, $\varX^{\val}$ is homeomorphic to $\varprojlim_f \Trop(\varX, f)=\varprojlim_f \Trop(\varX, (g,f))$, where the inverse limit is over the set of rational maps of $\varX$ to tori. 

  Denote $\nu \coloneqq \MA_{\omega, \NA}(\varphi_{\NA})$ on $\varX^{\an}$. Denote $\nu_g \coloneqq \MA_{\Trop}(\varphi_o+h)$ on $\Trop(\varX, g)$, and denote $\nu_{g,f } \coloneqq \MA_{\Trop}((\varphi_o+h) \circ \pi_{g,f})$ on $\Trop(\varX, (g,f))$. Here $\pi_{g,f} \colon \Trop(\varX, (g,f)) \to \Trop(\varX, f)$ denotes the projection map. Then, \cite[Prop.\,5.20]{APW1} 
  %cross-ref in APW1: \ref{prop:compatibility_Monge-Ampere_projections} 
   and Theorem \ref{thm:comparison_thm_general_case} imply that $(\pi_{g,f})_{*} \nu_{g,f}=\nu_g$ and $(\Trop_{g,f})_{*} \nu=\nu_{g,f}$.
  
  Since $\weight(\delta)=1$ for any maximal dimensional face $\delta$ in $\Trop_g(\Delta_{\fX})$, the tropical Strumfels--Tevelev multiplicity formula implies that $\card{\pi_{g,f}^{-1}(y)}=1$ for a sufficiently irrational point $y \in \delta$.  This implies that there exists a rational simplicial polyhedral structure $\Pi$ on $\Delta_{\fX}$ such that the restriction of $\pi_{g,f}$ on $\pi_{g,f}^{-1}(\bigcup_{\delta' \in \Pi_d}\relint(\delta'))$ is a  bijection. Therefore, $\pi_{g,f}^{-1}(\bigcup_{\delta' \in (\Pi_d)}\relint(\delta')) \subset \Trop_{(g,f)}(\Delta_{\fX})$. On the other hand, since $\mu_g$ is absolute continuous with respect to the Lebesgue measure, $\nu_g(\Delta_{\fX} \setminus \bigcup_{\delta' \in (\Pi_d)}\relint(\delta'))=0$. This implies that $\mathrm{Supp}(\nu_{g,f}) \subset \Trop_{(g,f)}(\Delta_{\fX})$. Passing to the limit over all rational maps $f$, we obtain that $\mathrm{Supp}(\nu) \subset \Delta_{\fX}$.  Since $\Trop_f$ is injective on $\Delta_{\fX}$ and $(\Trop_f)_{*}(\mu)=(\Trop_f)_{*}(\nu)$, we obtain $\mu=\nu$ on $\varX^{\an}$.
\end{proof}

\appendix
%%%%%%%%%%%%%%%%%%%%%%%%%%%%%
%%%Uniform Lipschitz property
%%%%%%%%%%%%%%%%%%%%%%%%%%%%%

\section{Uniform Lipschitz property of tropical psh functions} \label{appendix:uniform_lipschitz}
Let $X$ be a compact tropical variety with open part $\xo$. Let $\omega$ be a semipositive (1,1)-form on $X$. We prove the following theorem.
 
\begin{thm}  \label{thm:appendix_lipschitz}
  Let $K$ be a compact subset of $\xo$. There exists a constant $C_K$ such that each $\varphi \in \PSH_{\omega}(X)$ is $C_K$-Lipschitz.
\end{thm} 
Here, the notion of Lipschitz function is defined as follows. We fix a finite open cover $\xo=\bigcup_{i \in I} W_i$ 
by tropical charts $(W_i, \phi_i, W_i \to U_i)$.  For each $i\in I$, we also 
fix an Euclidean metric on $U_i$, and denote by $d_i$ the induced metric on $W_i$ via $\phi_i$.
Then $\varphi \colon S\to \R$, for $S\subset \xo$ is called $C$-Lipschitz if $\varphi\rest{S\cap W_i}$ is 
is $C$-Lipschitz with respect to $d_i$, for each $i\in I$. 

\smallskip
Let $\olPi$ be a polyhedral structure on $X$, with induced polyhedral structure $\Pi$ on $\xo$ (see Section~\ref{subsubsec:polyhedral-complex-structures}).
For each point $x \in \xo$, let $\vec{T}_xX$ be the \emph{pointed tangent space of $X$ at $x$} defined by 
\[
\vec T_x X \coloneqq \bigcup_{\substack{\sigma \ni x \\\sigma \in \Pi}} \vec T_{x}{\sigma},
\] 
where $\vec T_x\sigma$ is the set of all $\fe$ in the tangent space $T_\sigma$ such that $x + t\fe \in \sigma$ for an small positive real $t$. 
The pointed tangent space $\vec T_xX$ is thus a subset of the tangent space of $x$ in $X$ 
(defined by taking the linear span of the faces $\sigma \in \Pi$ than contain $x$). Note that the pointed tangent space does not depend on the choice of the polyhedral structure. 

For a point $x$ and $f \in \nC^{0}(U)$ defined on a neighborhood $U$ of $x$, we set 
\[
\partial_{\fe} f (x) \coloneqq \lim_{t \to 0+} \frac{f(x+t\fe)-f(x)}t, \qquad \fe \in \vec T_xX, 
\] 
if this limit exists. 

\smallskip

Note that the 1-skeleton $\Pi_{(1)}=\Pi_0\cup \Pi_1$ of $\Pi$ has vertex set $\Pi_0$, a set of finite edges given by $\Pi_{\f,1}$, and a set of infinite edges called \emph{legs} corresponding to nonempty intersections 
 $L\cap \xo$ for $L\in \olPi_1$. On each edge  of the 1-skeleton, we assign length one to 
 the primitive vector of that edge. In this way, we obtain a metric graph structure on the 1-skeleton that we denote by $\Gamma$. Topologically, $\Gamma$ is the support $\supp{\Pi_{(1)}}$, and is endowed with the shortest path metric.
Note that any function $\varphi$ on $X$ yields a function on $\Gamma$ that we dentoe by $\varphi\rest{\Gamma}$.

 We first prove the theorem for the restriction of $\omega$-psh functions to $\Gamma$. 
\begin{prop} \label{prop:1-dimensional_uniform_Lipschitz}
  Let $X$ be a tropical variety of dimension $d$, and let $\olPi$ be a polyhedral structure on $X$, with induced polyhedral structure $\Pi$ on $\xo$. Let $\Gamma$ be the corresponding metric graph obtained from the one-skeleton $\Pi_{(1)}$ 
  by assigning length one to the primitive vector of each edge in $\Pi_1$. 
  Then, there exists a positive constant $C$ such that $\varphi\rest{\Gamma}$ is $C$-Lipschitz for any $\varphi \in \PSH_{\omega}(X)$. 
\end{prop}

For two vertices $v$ and $u$ in $\Pi_0$, we write $v\sim u$ if $v$ and $u$ are adjacent. In this case, we denote by $vu$ the corresponding oriented finite edge with source $v$ and tail $u$,
 and let $\e_{vu} \in \vec T_v\Gamma$ be the unit tangent vector (corresponding to the primitive vector of $\vec T_vX$ tangent to the edge $vu$ in the direction of $u$) based at $v$ and pointing toward $u$. By an abuse of notation, for each point $x$ in the interval edge $vu$, $x\neq u$, we dentoe by $\fe_{vu}$ the unit tangent vector in $\vec T_xX$ in the direction of $v$. We denote by $\rec(\Gamma)=\rec(\Pi)$ the recession fan of $\Pi$ and $\Gamma$. 
 For $\rho \in \rec(\Gamma)$, we write $v \sim \rho$ if there exists an infinite edge of $\Gamma$ based at $v$ and pointing in the direction of $\rho$. In this case,  we denote this oriented infinite edge by $v \rho$, and the corresponding unite tangent vector by $\e_{v\rho}$ (corresponding to the primitive vector of $\vec T_vX$ in the direcion of $\rho$). 
 To simplify the notation, set $\partial_{vu} \coloneqq \partial_{\fe_{vu}}$, and $\partial_{v\rho} \coloneqq \partial_{\fe_{v \rho}}$.

\begin{lemma} \label{lem:inequalities_slopes} Notation as in Proposition~\ref{prop:1-dimensional_uniform_Lipschitz},
  there exist  constants $A, B >0$  such that for any $\varphi \in \PSH_{\omega}(X)$, the following inequalities hold.
  \begin{itemize}
    \item For any vertex $v$ of\/ $\Gamma$,
    \begin{equation} \label{inequality:A1}
      \max_{u \sim v}\partial_{vu}\varphi(v) \geq -A\min_{u\sim v}\partial_{vu}\varphi(v)-B.
    \end{equation}
    \item For any infinite edge $v \rho$ of $\Gamma$,
    \begin{equation} \label{inequality:A2}
      \partial_{v\rho}\varphi(v) \ge -A\min_{\rho\sim v}\partial_{v\rho}\varphi(v)-B
    \end{equation}
    \item For any oriented finite edge $vu$ of $\Gamma$ and any point $x$ on the segment $vu$, $x\neq u$,
    \begin{equation} \label{inequality:A3}
      \partial_{vu}\varphi(v)  \le \partial_{vu}\varphi(x)+ B/2 \le -\partial_{uv}\varphi(u)+B.
    \end{equation}
    \item For any oriented infinite edge $v \rho$ of $\Gamma$ and $x$ on the leg $v \rho$ of $\Gamma$,
    \begin{equation} \label{inequality:A4}
      \partial_{v\rho}\varphi(v)  \le \partial_{v\rho}\varphi(x)+ B/2 \le B.
    \end{equation}
  \end{itemize}
\end{lemma}
\begin{proof}
    Fix a vertex $v$ of $\Gamma$. Let $h$ be a local potential of $\omega$ in a neighborhood of $v$. Then, Proposition \ref{prop:equivalent_def_LConv} implies that 
  \begin{equation*} 
  \sum_{u \sim v} \partial_{vu}(\varphi+h)(v) +\sum_{v \sim \rho} \partial_{v\rho}(\varphi+h)(v) \ge 0 
  \end{equation*}
  Note that $\varphi+h$ is a convex function bounded above on each infinite edge $v \rho$, Lemma \ref{lem:extension_bounded_convex_function_polyhedron} implies that $\partial_{v \rho}(\varphi+h)(v) \le 0$. On the other hand, $|\partial_{vu}h(x)|$ (\resp $|\partial_{v\rho}h(x)|$)  is uniformly bounded on $vu$ (\resp $v \rho$), and the number of edges in $\Gamma$ is finite. Therefore, there exist constants $A, B>0$ such that inequality \ref{inequality:A1} and inequality \ref{inequality:A2} hold.  

  Now fix an oriented finite edge $vu$ of $\Gamma$. Let $h$ be a local potential of $\omega$ in a neighborhood of $vu$. The convexity of $\varphi+h$ on the segment $vu$ implies that 
  \begin{equation*} 
  \partial_{vu}(\varphi+h)(v) \le \partial_{vu}(\varphi+h)(x) \le - \partial_{uv}(\varphi+h)(u).
  \end{equation*}
  Applying the above inequality to the case when $\varphi=0$, we get $\partial_{vu}h(v) \le \partial_{vu}h(x) \le \partial_{vu}h(u)$. Therefore, there exists $B_{vu}>0$ such that for any $x \in vu$, $x\neq u$, we have
  \[
  0 \le \partial_{vu}h(x)-\partial_{vu}h(v) \le B_{vu}/2, \quad \text{and} \quad 0 \le  -\partial_{uv}h(u)-\partial_{vu}h(x) \le B_{vu}/2.
  \] 
  Since  the number of  edges in $\Gamma$ is finite, there exists $B>0$ such that  inequality \ref{inequality:A3} follows.

  Now fix an oriented infinite edge $v\rho$ of $\Gamma$. Let $h$ be a local potential of $\omega$ in a neighborhood of $v\rho$. Since $\varphi+h$ is a convex function bounded above on $v \rho$, Lemma \ref{lem:extension_bounded_convex_function_polyhedron} implies that 
  \begin{equation*} 
  \partial_{v\rho}(\varphi+h)(v) \le \partial_{v\rho}(\varphi+h)(x) \le 0
  \end{equation*}
  Inequality \ref{inequality:A4} follows from an argument similar to the one in the previous paragraph.
\end{proof}

\begin{proof}[Proof of Proposition~\ref{prop:1-dimensional_uniform_Lipschitz}]
Let $N$ be the number of vertices of $\Gamma$. Define the sequence $s_k, k=0,1, \dots, N+1$ recursively by setting $s_0 \coloneqq  B+1$, and $s_{k+1} \coloneqq  B+ (s_k + B)/A$. Note that $s_k > B$ for all $k$. 

By inequality \ref{inequality:A2} and inequality \ref{inequality:A4}, it is sufficient to prove that  $\PSH_{\omega}(X)$ is uniformly Lipschitz on finite edges of $\Gamma$. We claim that for any $\varphi \in \PSH_{\omega}(X)$, $\varphi$ is $(s_{N+1}-B/2)$-Lipschitz on each finite edge of $\Gamma$.  Suppose this is not the case. 
Then, there exist $\varphi \in \PSH_{\omega}(X)$, an oriented finite edge $v_0v_1$, and a point $x$ on the segment $v_0 v_1$, $x\neq v_1$,
 such that $\partial_{v_0 v_1} \varphi(x)   \ge s_{N+1}-B/2$. 
By inequality \ref{inequality:A3}, we have
\[ 
-\partial_{v_1v_0}\varphi(v_1) \geq s_{N+1} - B > 0. 
\]
Then, inequality \ref{inequality:A1} ensures the existence of a vertex $v_2$ adjacent to  $v_1$ such that
\[ 
\partial_{v_1v_2}\varphi(v_1) \geq -A\partial_{v_1v_0}\varphi(v_1) - B \geq A(s_{N+1}-B)-B = s_{N} > B. 
\]
One can pursue the argument and find vertices $v_3, \dots, v_{N+1}$ such that $\partial_{v_kv_{k+1}}\varphi(v_k) \geq s_{N+1-k} > B$. Therefore, inequality \ref{inequality:A3} implies that  for any $y \in v_kv_{k+1}$,
\[
\partial_{v_kv_{k+1}}\varphi(y) > \partial_{v_kv_{k+1}}\varphi(v_k)-B/2 > B/2>  0
\]
Therefore, we have $\varphi(v_{k+1})> \varphi(v_k)$ for each $1 \le k \le N$. On the other hand, the pigeonhole principle implies that there exist $1 \le i < j \le N+1$ such that $v_{i}=v_{j}$. This is clearly a contradiction and the proposition follows.
\end{proof}

We are prepared to prove Theorem \ref{thm:appendix_lipschitz}.

\begin{proof} [Proof of Theorem \ref{thm:appendix_lipschitz}] \label{proof:thm:appendix_lipschitz}
   Enlarging $K$ if necessary, and fixing a point $x\in K$, 
   we can choose a simplicial polyhedral structure $\overline{\Pi}$ on $X$, with induced polyhedral structure $\Pi$ on $\xo$, such that $x \in \Pi_0$, $K \subset \supp{\Pi_{\f}}$, 
   and such that for any $\tau \in \Pi$, there exists $i \in I$ such that  $\tau$ is contained in the tropical chart $W_i$. Here, $(W_i,\phi_i)_i$ refers to the finite cover by charts with respect to which the Lipschitz property is defined. 
   
     Let $\PSH_{\omega,x }(X)=  \{\varphi \in \PSH_{\omega}(X) \,\st\, \varphi(x)=0 \}$. We only need to prove that $\varphi\rest K$ is uniformly Lipschitz for $\varphi \in \PSH_{\omega,x }(X)$.  
     By Proposition \ref{prop:Lipschitz_bounds_in_terms_of_C0_bounds}, it suffices to show that for any $\tau \in \Pi_{\f}$,  $\left\| \varphi \right\|_{\nC^{0}(\tau)}$  is uniformly bounded for all  $\varphi \in \PSH_{\omega,x}(X)$. 
  
     Proposition \ref{prop:1-dimensional_uniform_Lipschitz} implies that there exists $M>0$ such that $|\varphi(v)| \le M$ for any $\varphi \in \PSH_{\omega,x}(X)$ and any vertex $v \in \Pi_0$. 
 Since $\varphi\rest{\tau}$ is convex, we have $\varphi\rest{\tau} \le M$ for any $\varphi \in \PSH_{\omega,x}(X)$.

 It remains to prove a uniform lower bound for $\varphi\rest{\tau}$, for any $\tau \in \Pi_{\f}$.
 Fix such a face $\tau$ and let $i\in I$ be such that $\tau \subset W_i$. We use $\tau_{(0)}$ to denote the set of vertices of $\tau$. 

   Fix $v \in \tau_{(0)}$. Let $J$ be an index set, and let $\{e_j\}_{j \in J}$ be the set of edges in the graph $\Pi_1$ adjacent to $v$. For $e_j=vu$ (\resp $e_j=v\rho$) with $j \in J$, denote by $\fe_{j} \coloneqq \fe_{vu}$ (\resp $\fe_{j} \coloneqq \fe_{v\rho}$), \textit{i.e.} the primitive tangent vector to the edge $e_j$ based at $v$. 
  The balancing condition ensures the existence of $\{a_j\}_{j \in J}$ with $a_j>0$, such that $\sum_{j \in J} a_j \fe_j=0$ (in the ambient vector spacce containing $\phi_i(W_i)$). 
  
  Denote by $J_1 \coloneqq \{j \in J \,|\, e_j \subset \tau \} \subset J$, and denote $J_2 \coloneqq J \setminus J_1$. 
  Note that $\{\fe_{j}\}_{j \in J_1}$ is linearly independent since $\tau$ is a simplex. Let $T_{\tau}$ denote the tangent space of $\tau$, and $\left\| \cdot \right\|$ denote the Euclidean norm on $T_{\tau}$ induced by the Euclidean metric on $W_i$. Set 
  \[
  S \coloneqq \{\fe \in T_{\tau} \,\st \, \left\| \fe \right\|=1, \, v+t \fe \in \tau \text{ for some }t>0 \}.
  \]  
  Then, there exists $\underline B>0$ such that for any $\fe \in S$, we have $\fe=\sum_{j \in J_1} b_j\fe_{j}$ for some $0 \le b_{j} < \underline B$. Set $\underline A \coloneqq \max_{j \in J_1} \underline B/a_j$. 
  
  Fix $\fe=\sum_{j \in J_1} b_j\fe_{j} \in S$. Then, we have 
  \[
  \fe+\sum_{j \in J_1}(\underline Aa_j-b_j) \fe_j + \sum_{j \in J_2} (\underline Aa_j) \fe_j=0,  
  \]
  with   $(\underline Aa_j-b_j)>0$ for each $j \in J_1$, and $\underline Aa_j>0$ for each $j \in J_2$. 
  Let $h$ be a local potential of $\omega$ on $W_i$. Then, by Proposition \ref{prop:equivalent_def_LConv},  we have
  \[
  \partial_{\fe} (\varphi+h) (v)+ \sum_{j \in J_1} (\underline Aa_j-b_j) \partial_{ \fe_j} (\varphi+h) (v) + \sum_{j \in J_2} (\underline Aa_j) \partial_ {\fe_j} (\varphi+h) (v) \ge 0
  \]
  for every $\varphi \in \PSH_{\omega,x}(X)$. On the other hand, Proposition \ref{prop:1-dimensional_uniform_Lipschitz} implies that there exists $C>0$ such that $|\partial_{\fe_j} \varphi (v)| \le C$ for each $j \in J$. Since $h$ is Lipschitz on $\tau$ (Proposition \ref{prop:Lipschitz_bounds_in_terms_of_C0_bounds}), there exists $C'>0$ such that $|\partial_{\fe}h(v)|, \, |\partial_{\fe_j}h(v)| \le C'$ for each $\fe \in S$ and each $j \in J$.   Thus, $\partial_{\fe}\varphi(v) \ge M'$ for every $\varphi \in \PSH_{\omega,x}(X)$ and every $\fe \in S$, with the constant $M'$ defined by
  \[
  M'=-\underline A(C+C') \left( \sum_{j \in J}a_j \right) -(\underline B+1)C'. 
  \]
  Therefore, $\varphi(y) \ge \varphi(v)+\partial_{(y-v)}\varphi(v) \ge -M+ \mathrm{diam}(\tau ) \cdot M'$ for every $\varphi \in \PSH_{\omega,x}(X)$ and $y \in \tau$. Here, $y-v \in T_{\tau}$, and $\mathrm{diam}(\tau ) $ denotes the diameter of $\tau$. This proves that $\PSH_{\omega,x}(X)$ has a uniform lower bound on $\tau$. 
\end{proof}

%%%%%%%%%%%%%%%%%%%%%%%%%%%%%%%
%%Lemmas in polyhedral geometry
%%%%%%%%%%%%%%%%%%%%%%%%%%%%%%%

\section{Polyhedral geometry toolkit} \label{appendix:polyhedral_geometric_lemmas}
This appendix is devoted to stating and proving four polyhedral-geometric lemmas used in the proof of Proposition \ref{prop:existence_of_balanced_rays_in_2d_fan}.
\begin{lemma} \label{lem:successive_stellar_subdivision} 
  Let $\Delta$ be any simplicial polyhedral complex in $\R^n$, and let $\{x_i\}_{i \in I}$ be a finite set of points in the support of $\Delta$. Then, there exists a refinement $\Delta'$ of $\Delta$ such that $\Delta'$ is simplicial  and $\Delta'_{0}=\Delta_0 \cup \{x_i\}_{i \in I}$.
\end{lemma}
\begin{proof}
  The refinement can be realized by performing stellar subdivisions successively at $x_i$, see \cite[\S\, III.2]{Ewa96}.
\end{proof}

We say that a one-dimensional fan is \emph{balanceable} if the set of its rays is positively dependent (see Definition \ref{def:nonnegatively_depedent}). In particular (if the fan is rational), there exists a tropical weight turning it into a tropical fan.

\begin{lemma} \label{lem:completion_of_fans_with_prescribed_rays} 
  Let $\CL$ be a balanceable $1$-dimensional fan in $E=\R^n$ with $\mathrm{span}(\CL)=E$. Let $\CF'$ be a simplicial  (not necessarily complete) fan supported on $E$ such that 
  \begin{enumerate}
    \item $\CF'_1 \subset \CL_1$.
    \item $|\CF'|=\conv(\CF'_1)$, where $\conv$ denotes the convex hull of $\CF'_1$.
    \item $\rho \nsubseteq |\CF'|$ for any $\rho \in \CL_1 \setminus \CF'_1$.
  \end{enumerate}
  Then, there is a  simplicial fan $\CF$ supported on $E$ satisfying 
  \begin{enumerate}
    \item $\CF'$ is a subfan of $\CF$. 
    \item $\CF_1 = \CL_1$. 
    \item $\CF$ is complete. 
  \end{enumerate}
\end{lemma}
\begin{proof}
  We will prove inductively that for any subset of rays $\CF'_1 \subset \CR \subset \CL_1$, there exists a simplicial fan $\CF^{\CR}$ supported on $E$  such that 
  \begin{enumerate}
    \item[($i$)] $\CF'$ is a subfan of $\CF^{\CR}$.
    \item[($ii$)] $(\CF^{\CR})_1=\CR$.
    \item[($iii$)] $|\CF^{\CR}|=\conv(\CR)$.
  \end{enumerate}
  Granted this, we take $\CF \coloneqq \CF^{\CL_1}$. Then, $\CF$ satisfies  the first two conditions listed in the statement of the lemma. Since $\CL$ is balanceable, we have $|\CF^{\CL_1}|=\conv(\CL_1)=E$, as desired. 

  \smallskip

  When $\CR=\CF'_1$, we can simply take $\CF^{\CR} = \CF'$. Suppose we have constructed $\CF^{\CR_{-}}$ for $\CF'_1 \subset \CR_{-} \subset \CL_1$ satisfying Conditions ($i$)-($iii$). Let $\CR=\CR_-\sqcup\{\rho\} \subset \CL_1$, we
   wish to construct $\CF^{\CR}$ with the desired properties.
  
  If $\rho \subset \supp{\CF^{\CR_{-}}}$, we take $\CF^{\CR}$ to be the stellar subdivision of $\CF^{\CR_{-}}$ at $\rho$. Then,  $\CF^{\CR}$ is a simplicial fan satisfying $( \CF^{\CR})_1=\CR$ 
  and $\supp{\CF^{\CR}}=\supp{\CF^{\CR_{-}}} =\conv(\CR)$. Since $\rho \nsubseteq \supp{\CF'}$ by assumption, the stellar subdivision does not change any cone in $\CF'$. Therefore, $\CF'\subset \CF^{\CR}$, as desired.

  It remains to treat the case when $\rho \nsubseteq \supp{\CF^{\CR_{-}}}$. Let $\CK \coloneqq \supp{\CF^{\CR_{-}}} \subset E$. Denote by $l$ the $1$-dimensional subspace of $E$ generated by $\rho$. 
  We take an $(n-1)$-dimensional subspace $E_1 \subset E$ such that $E = E_1+l$. Fix an Euclidean metric on $E$, and denote by $\BS_{E_1}$ the unit sphere in $E_1$. 
  
  Fix $e_0 \in \rho \setminus \{0\}$. We define a function $\ft \colon \BS_{E_1} \to \R \cup \{+\infty\}$ by 
  \[
  \ft(y) \coloneqq \inf \{\lambda \in \R \,|\, y-\lambda e_0 \in \CK \}
  \]
  Set $ \sC  \coloneqq \{y \,|\, \ft(y) < +\infty \} \subset \BS_{E_1}$, and define $\pi  \colon \sC  \to \partial \CK$ by 
  \[
  \pi(y) \coloneqq y-\ft(y) e_0 \qquad \text{for }y\in \sC.
  \]

  \medskip

  \begin{claim} \label{claim:image_of_projection_is_the_support_of_a_fan}
    There is a subfan $\CE$ of $\CF^{\CR_{-}}$ such that $\mathscr{C}(\pi(\sC))=\supp{\CE}$, where $\mathscr{C}(\pi(\sC))$ denotes the cone generated by $\pi(\sC)$\,\footnote{For a subset $S  \subset E$, the cone $\mathscr{C}(S)$ generated by $S$ is defined to be $\mathscr{C}(S) \coloneqq \{\lambda s\st \lambda \in \R_{\ge 0}, s\in S\}$. It should be distinguished from the convex cone generated by a subset $S$, which equals the cone generated by the convex hull of $S$, and is denoted by $\mathrm{Cone}(S)$.}. Furthermore, for any $\sigma \in \CE$, $\mathrm{Cone}(\sigma, \rho)$ is a simplicial cone, and $\mathrm{Cone}(\sigma, \rho) \cap \CK=\sigma$.  
  \end{claim}
  \begin{proof} [Proof of Claim \ref{claim:image_of_projection_is_the_support_of_a_fan}]
    Fix $x_0 \in \pi(\sC)$.  Suppose that $x_0 \in \relint(\sigma)$ for some $\sigma \in \CF^{\CR_{-}}$ with 
    \[
    \sigma = \mathrm{Cone}(x_1, \cdots, x_k), \quad \text{ for }  x_1, \dots, x_k \in \sigma. 
    \]

    We first show that $\{e_0, x_1, \dots, x_k\}$ is linearly indepedent in $E$. Suppose not. Then $e_0 \in T_{\sigma}$, where $T_{\sigma}$ denotes the tangent space of $\sigma$. Write $x_0=y_0- \ft(y_0)e_0$ for  $y_0 \in \BS_{E_1}$.  Let $\{c_i\}_{i \ge 1}$ be an increasing sequence of real numbers such that $\lim_{i \to +\infty}c_i=\ft(y_0)$. Then, we have $y_0-c_i e_0 \to x_0$, and  $y_0-c_i e_0 \in T_{\sigma}$ for each $i \ge 1$. On the other hand, $y_0-c_i e_0 \notin \CK$ and hence $y_0-c_i e_0 \notin \sigma$ for each $i \ge 1$. This contradicts the assumption that $x_0 \in \relint(\sigma)$. 
    Therefore, we have shown that $\{e_0, x_1, \dots, x_k\}$ is linearly independent in $E$.

    Fix $x \in \sigma \setminus \{0\}$. Since $E=E_1+l$ and $e_0 \notin T_{\sigma}$, one can always write $x=b(y-\lambda e_0)$ with $y \in \BS_{E_1}$, $b \in \R_{>0}$ and $\lambda \in \R$.  We would like to show that $\ft(y)=\lambda$. Suppose not.   Then, there exists $\lambda' < \lambda$ such that $y-\lambda' e_0 \in \CK$. Since $\CK$ is a convex cone, we have $\mathrm{Cone}(y-\lambda' e_0 , x_1, \dots, x_k) \subset \CK$. Note that $\sigma = \bigcup_{1 \le i \le k} \mathrm{Cone}(x, x_1, \dots, \widehat{x_i}, \dots, x_k)$.  Hence, we may assume that $x_0 \in \mathrm{Cone}(x, x_2, \dots, x_k)$ without loss of generality. 
    Furthermore, $x_0 \in \relint(\sigma)$ implies that $x_0 \notin \mathrm{Cone}(x_2, \dots x_k)$. Then, we can write $x_0=ax+\sum_{j=2}^{k} a_j x_j$, with $a>0$ and $a_j \ge 0$ for each $2 \le j \le k$, and we have 
    \begin{align*}
    ab&(y-\lambda' e_0)+\sum_{j=2}^{k} a_j x_j 
    = ax+\sum_{j=2}^{k} a_j x_j +ab(\lambda-\lambda')e_0 \\
    &= x_0+ab(\lambda-\lambda')e_0=y_0-\ft(y_0)e_0 +ab(\lambda-\lambda') e_0 \in \CK
    \end{align*} 
    with $ab(\lambda-\lambda')>0$. This contradicts the definition of $\ft(y_0)$. Therefore, we conclude that $\ft(y)=\lambda$ for any $b(y-\lambda e_0) \in \sigma \setminus \{0\}$. 

    Now, we define the set of cones $\CE$ by 
    \[
    \CE \coloneqq \{\sigma \in \CF^{\CR_{-}} \,\st \ \relint(\sigma) \cap \pi(\sC) \text{ is nonempty} \}.
    \] 
    By construction, we have $\mathscr{C}(\pi(\sC)) \subset \supp{\CE}$. On the other hand, the previous paragraph implies that $\sigma \subset \mathscr{C}(\pi(\sC))$ for any $\sigma \in \CE$. In particular, this implies  that $\supp{\CE} \subset \mathscr{C}(\pi(\sC))$, and we have $\tau \in \CE$ for each element $\sigma \in \CE$ and each face $\tau \subface \sigma$.  Therefore, $\CE$ is a fan with support $\mathscr{C}(\pi(\sC))$. 
    
    Since $\{e_0, x_1, \cdots, x_k\}$ is linearly independent in $E$, $\mathrm{Cone}(\sigma, \rho)$ is simplicial. On the other hand, for any $x=b(y-\lambda e_0) \in \sigma \setminus \{0\}$, the equality $\ft(y)=\lambda$ implies $\mathrm{Cone}(x, \rho) \cap \CK=\mathrm{Cone}(x)$. Hence, 
    \[
    \mathrm{Cone}(\sigma, \rho) \cap \CK=\bigcup_{x \in \sigma \setminus \{0\}} (\mathrm{Cone}(x, \rho) \cap \CK)=\bigcup_{x \in \sigma \setminus \{0\}} \mathrm{Cone}(x) =\sigma \qedhere
    \]
  \end{proof}

  \medskip

  We take $\CF^{\CR} \coloneqq \CF^{\CR_{-}} \cup \{ \mathrm{Cone}(\sigma, \rho) \,\st \, \sigma \in \CE \}$. We first verify that $\CF^{\CR}$ is a simplicial fan. It is enough to show that $\tau_1 \cap \tau_2 \in \CF^{\CR}$ for any $\tau_1, \tau_2 \in \CF^{\CR}$. There are three possible cases to consider:
  \begin{enumerate}
    \item If $\tau_1, \tau_2 \in \CF^{\CR_{-}}$, then $\tau_1 \cap \tau_2 \in \CF^{\CR_{-}} \subset \CF^{\CR}$ holds trivially. 
    \item If $\tau_i=\mathrm{Cone}(\sigma_i, \rho)$ for $\sigma_i \in \CE$ for each $i \in \{1,2\}$, we show that $\tau_1 \cap \tau_2 =\mathrm{Cone}(\sigma_1 \cap \sigma_2, \rho) \in \CF^{\CR}$. The inclusion $\mathrm{Cone}(\sigma_1 \cap \sigma_2, \rho) \subset \tau_1 \cap \tau_2$ is evident. Thus, it is sufficient to prove that  $\tau_1 \cap \tau_2 \subset \mathrm{Cone}(\sigma_1 \cap \sigma_2, \rho) $.  Let $z \in \tau_1 \cap \tau_2$. Then, one can write $z = x_1+\mu_1 e_{0}=x_2+\mu_2 e_0$ with $x_i \in \sigma_i$ and $\mu_i \ge 0$ for each $i \in \{1,2\}$. Without loss of generality, we assume $\mu_1 \le \mu_2$. Then, one has $x_1=x_2+(\mu_2-\mu_1)e_0 \in \sigma_1 \cap \mathrm{Cone}(\sigma_2, \rho)$. By Claim \ref{claim:image_of_projection_is_the_support_of_a_fan},  $\sigma_1 \cap \mathrm{Cone}(\sigma_2, \rho)=\sigma_1 \cap (\mathrm{Cone}(\sigma_2, \rho) \cap \CK)=\sigma_1 \cap \sigma_2$. Therefore, $z \in \mathrm{Cone}(\sigma_1 \cap \sigma_2, \rho)$, which implies that $\tau_1 \cap \tau_2 \subset \mathrm{Cone}(\sigma_1 \cap \sigma_2, \rho)$. 
    \item Otherwise, we may assume that $\tau_1 \in \CF^{\CR_{-}}$ and $\tau_2=\mathrm{Cone}(\sigma, \rho)$ for $\sigma \in \CE$. Then $\tau_1 \cap \tau_2=\tau_1 \cap (\tau_2 \cap \CK)=\tau_1 \cap \sigma \in \CF^{\CR_{-}}$ by Claim \ref{claim:image_of_projection_is_the_support_of_a_fan}. 
    \end{enumerate}

    We now check that $\CF^{\CR}$ satisfies the three properties ($i$)-($iii$).  We have $\CF' \subset \CF^{\CR_{-}} \subset \CF^{\CR}$, the first inclusion coming from the induction hypothesis, and the second from the construction. 
    Also, $(\CF^{\CR})_1=\CR$ by construction. As for the third property, it is clear that $\supp{\CF^{\CR}} \subset \conv(\CR)$. To prove $\conv(\CR) \subset \supp{\CF^{\CR}}$, fix $z=x+\mu e_0 \in \conv(\CR)$ with $x \in \CK$ and $\mu \ge 0$. Since $E=E_1+l$, one can write $x=by-a e_0$ for  $y \in \BS_{E_1}$,  $b \in \R_{\ge 0}$, and $a \in \R$.  
    \begin{enumerate}
      \item If $b = a = 0$, then we have $x=0$ and $z=\mu e_0 \in \rho \subset \supp{\CF^{\CR}}$.
      \item If $b=0$ and $a \neq 0$, then the assumption that $\rho \nsubseteq \supp{\CF^{\CR_{-}}}$ and $x \in \CK=\supp{\CF^{\CR_{-}}}$ imply that $a >0$. Therefore, we have $-e_0=a^{-1} x \in \CK$, and hence $z=(\mu-a)e_0 \in l \subset \CK \cup \rho \subset \supp{\CF^{\CR}}$.
      \item If $b > 0$, then $\ft(y) \le \frac{a}{b}< +\infty$. Since $\CK$ is convex, $\mathrm{Cone}(x, y-\ft(y)e_0) \subset \CK$. On the other hand, $\mathrm{Cone}(e_0, y-\ft(y)e_0) \subset \mathrm{Cone}(\sigma, \rho) \subset \supp{\CF^{\CR}}$ for some $\sigma \in \CE$. Therefore, $z \in \mathrm{Cone}(e_0, x)=\mathrm{Cone}(e_0, y-\ft(y)e_0) \cup \mathrm{Cone}(x, y-\ft(y)e_0) \subset \supp{\CF^{\CR}}$. \qedhere
    \end{enumerate}
\end{proof}

\begin{lemma} \label{lem:intersection_of_small_pertubation_of_simplices}
  Let $E$ be a Euclidean space of dimension $n$. Let $T$ be a finite set, with subsets $T_1, T_2 \subset T$, and $S=T_1 \cap T_2$. We use $F^{\mathfrak{S} }(T)$ to denote the set of injective maps $g \colon T \to E$ such that $\sigma(g) \coloneqq \conv(g(S))$, $\tau_1(g) \coloneqq \conv(g(T_1))$, and $\tau_2(g) \coloneqq \conv(g(T_2))$ are simplices in $E$.  Let $f \in F^{\fS}(T)$ satisfy $\tau_1(f) \cap \tau_2(f) =\sigma(f)$.  Then, there exists $\epsilon >0$ such that for any $f' \in F^{\mathfrak{S} }(T)$ with $d(f, f') < \epsilon$, one has $\tau_1(f') \cap \tau_2(f')=\sigma(f')$. Here $d(f,f') \coloneqq \sup_{t \in T} d(f(t), f'(t))$.  
\end{lemma}
\begin{proof}
  Choose $r \gg 0$ such that $\tau_1(f'), \tau_2(f') \subset B(0,r)$ for any $f' \in F^{\mathfrak{S} }(T)$ with $d(f, f')<1$. For any $0 < \epsilon_1 \ll 1$, there exists $0<\epsilon_2< \epsilon_1/2$ such that for any  $f' \in F^{\mathfrak{S} }(T)$ with $d(f, f')< \epsilon_2$, there exists an isometry $\phi_{f'} \colon E \to E$ satisfying 
  \begin{enumerate}
    \item $d(\phi_{f'}(x),x)<\epsilon_1/2$ for any $x \in B(0,r)$.
    \item $\phi_{f'}(\langle \sigma(f') \rangle_{\aff})=\langle \sigma(f) \rangle_{\aff}$, where $\langle \cdot \rangle_{\aff}$ denotes the affine subspace generated by a polyhedron. 
  \end{enumerate}
  Note that $d(f, \phi_{f'} \circ f')< \epsilon_2 + \epsilon_1/2<\epsilon_1$ under the above conditions.  Therefore, it suffices to prove the lemma in the case when $\langle \sigma(f') \rangle_{\aff}=\langle \sigma(f) \rangle_{\aff}$. 
  
  After translation, we may assume that $E_{\sigma(f)} \coloneqq \langle \sigma(f) \rangle_{\aff}$ is a vector subspace of $E$. Write $E=E_{\sigma(f)} \oplus E_{\sigma(f)}^{\bot}$, and denote by $p \colon E \to E_{\sigma(f)}^{\bot}$ the projection map. Then by assumption, $p(\tau_1(f)) \cap p(\tau_2(f))=\{0\}$. Observe that for any $f' \in F^{\mathfrak{S} }(T)$ satisfying $\langle \sigma(f') \rangle_{\aff}=\langle \sigma(f) \rangle_{\aff}$, $p(\tau_i(f'))$ are simplices in $E_{\sigma(f)}^{\bot}$, and further we have $p^{-1}(0) \cap \tau_1(f') \cap \tau_2(f')=\sigma(f')$. Hence, showing $p(\tau_1(f')) \cap p(\tau_2(f'))=\{0\}$ would imply $\tau_1(f') \cap \tau_2(f')=\sigma(f')$. Since $d(p(x), p(x')) \le d(x, x')$ for any $x, x' \in E$, it is enough to prove the lemma for $\tau_1(p \circ f), \tau_2(p \circ f) \subset E_{\sigma(f)}^{\bot}$ with $\tau_1(p \circ f) \cap \tau_2(p \circ f)=\{0\}$, where we view $p \circ f$ as an injective map from $T/\{S \to *\}$ to $E_{\sigma(f)}^{\bot}$. Therefore, we may henceforth assume that $T_1 \cap T_2=\{s\}$, and $\tau_1(f) \cap \tau_2(f)=\{0\}$. 

  Denote by $\BS_{E}$ the unit sphere in $E$. For each $f' \in F^{\fS}(T)$ with $f'(s)=0$ and for each $i \in \{1,2\}$,  define $P_i(f') \coloneqq \BS_{E} \cap \mathrm{Cone}(f'(T_i))$.  Then, $\tau_1(f') \cap \tau_2(f')=\{0\}$ is equivalent to $P_1(f') \cap P_2(f') = \emptyset$. Thus, by assumption $d(P_1(f), P_2(f))\eqqcolon \epsilon_3 >0$. Take $\epsilon \coloneqq  \epsilon_3 M/6$, where $M \coloneqq \min_{t \in T \setminus \{s\} } \left\|f(t)\right\|$. Then, for any $f' \in F^{\fS}(T)$ with $f'(s)=0$ and $d(f,f') < \epsilon$, we have $d(P_1(f'), P_2(f'))>\epsilon_3 /3$. This proves the lemma.
\end{proof}

\bibliographystyle{alpha}
\bibliography{bibliography}

\end{document}